\documentclass[11pt]{article}
\usepackage{amsmath,amssymb}
\usepackage{amsthm,tocvsec2}
\usepackage{graphicx,subfigure,float}
\usepackage{cite}
\usepackage{setspace}
\usepackage{mdwlist,paralist}
\usepackage{enumitem}
\usepackage{indentfirst}
\usepackage{empheq}
\usepackage{grffile}
\usepackage{color}
\usepackage{booktabs}
\usepackage{leftindex}
\usepackage{appendix}

\date{}

\numberwithin{figure}{section}
\numberwithin{table}{section}
\numberwithin{footnote}{section}

\theoremstyle{plain}
\newtheorem{thm}{Theorem}[section]
\newtheorem*{thm*}{Theorem}

\theoremstyle{remark}
\newtheorem{rem}{Remark}[section]

\newtheorem{alg}{Algorithm}[section]

\numberwithin{equation}{section}

\def\cA{\mathcal{A}}
\def\cB{\mathcal{B}}

\def\cL{\mathcal{L}}
\def\cM{\mathcal{M}}

\def\cO{\mathcal{O}}

\def\cX{\mathcal{X}}

\def\bn{\mathbf{n}}

\def\bx{\mathbf{x}}

\def\bJ{\mathbf{J}}

\def\bQ{\mathbf{Q}}
\def\bX{\mathbf{X}}
\def\BN{\mathbb{N}}

\def\bsO{\boldsymbol{O}}
\def\half{\frac{1}{2}}

\begin{document}
	\title{Onsager Principle-Based Domain Embedding for Thermodynamically Consistent Cahn-Hilliard Model in Arbitrary Domain}
	\author{Wenkai Yu\footnote{Department of Mathematics, Hong Kong University of Science and Technology, Clear Water Bay, Kowloon, Hong Kong, P. R. China}, Qi Wang\footnote{Department of Mathematics, University of South Carolina, Columbia, SC 29208, USA}, Zhen Zhang\footnote{Southern University of Science and Technology, Shenzhen, P. R. China} and Tiezheng Qian\footnote{Department of Mathematics, Hong Kong University of Science and Technology, Clear Water Bay, Kowloon, Hong Kong, P. R. China}\;\footnote{To whom correspondence should be addressed. Email: maqian@ust.hk}}
	
	\maketitle
	
	\begin{abstract}
		The original Cahn-Hilliard model in an arbitrary domain with two prescribed boundary conditions is extended to a Cahn-Hilliard-type model in a larger, regular domain with homogeneous Neumann boundary conditions. The extension is based on the Onsager principle-based domain embedding (OPBDE) method, which has been developed as a systematic domain embedding framework to ensure thermodynamic consistency. By introducing a modified conservation law, the flux at the boundary of the original domain is incorporated into the conservation law as a source term. Our variational approach demonstrates that, even without a prior knowledge on the specific form of the rate of free energy pumped into the system, the Onsager principle remains an effective instrument in deriving the constitutive equation of the extended system. This approach clarifies the intrinsic structure of the extended model in the perspectives of free energy and its dissipation. Asymptotic analysis is carried out for the extended OPBDE Cahn-Hilliard model, demonstrating that the original Cahn–Hilliard model, including its boundary conditions, can be fully recovered. To validate our approach, a structure-preserving numerical scheme is developed to discretize the extended model. Numerical results show that the OPBDE Cahn-Hilliard model is accurate, effective, and robust, highlighting the capability of the OPBDE method in handling gradient flow problems in arbitrary domain geometries.
	\end{abstract}
	
	
	\section{Introduction}
	In recent decades, gradient flows have emerged as a powerful and versatile framework to study the evolution of systems driven by the minimization of free energy functionals. In a gradient flow, the state of a system evolves in the direction of the steepest descent of a free energy landscape and therefore approaches equilibrium progressively. This framework not only provides insight into the underlying dynamics of various physical processes, but also offers a rigorous mathematical formulation to describe various complex phenomena, especially in fluid dynamics and materials science \cite{Cahn1958FreeEnergyNonuniform,Cahn1959FreeEnergyNonuniform,Allen1979MicroscopicTheoryAntiphase,Gurtin1996TwophaseBinaryFluids,Hong2021EnergyproductionratePreservingNumerical}.
	
	Phase field models represent a significant category of gradient flows. These models employ phase variables to describe different phases of multi-phase systems, where thin transition layers are introduced for the interfaces between different phases. When a phase field variable represents a non-conserved order parameter in the bulk region, the Allen-Cahn model can be developed. When a boundary condition is needed at the solid boundary, \cite{Yu2025OnsagerPrinciple-BasedDomain} shows that different forms of boundary conditions correspond to different interfacial free energies. Moreover, a dissipative process can exist at the boundary and be described by a dynamical boundary condition. When the local relaxation at the boundary becomes extremely fast, the dynamical boundary condition can further reduce to other boundary conditions corresponding to different interfacial energies.
	
	When a conserved order parameter is represented by a phase field variable in the bulk region, the Cahn-Hilliard model is formulated. Widely used to describe the phase separation and diffusion processes, the Cahn-Hilliard model has been extensively studied for decades with numerous applications  \cite{Novick-Cohen2008Chapter4Cahn,Provatas2011PhaseFieldMethodsMaterials,Weber2019PhysicsActiveEmulsions,Yang2017LinearUnconditionallyEnergy,Patzold1995NumericalSimulationPhase,Shen2021ThermodynamicallyConsistentAlgorithms,Xu2019EfficientLinearSchemes,Yang2018EfficientSchemesUnconditionally,Shen2021UnconditionallyPositivityPreserving}. Mathematically, it can be viewed as a variational model that can be derived using an $H^{-1}$ norm variation of the free energy functional \cite{Fife2000ModelsPhaseSeparation,Cowan2005CahnHilliardEquationGradient}.
	The Allen-Cahn model and the Cahn-Hilliard model can be derived from the same free energy functional, except that the Allen-Cahn model corresponds to the $L^2$ norm variation. Therefore, the generic form of the free energy, first given in \cite{Yu2025OnsagerPrinciple-BasedDomain}, is adopted here:
	\begin{equation}
		F[\phi] = \int_{\Omega_1}[f(\phi)+\half K|\nabla \phi|^2]d\bx + \int_{\partial \Omega_1}[\half \alpha (\phi-h_1)^2-h_2\phi]ds,
	\end{equation}
	where $\phi$ is the phase field which usually denotes the volume or mass fraction, $K$ is a positive parameter quantifying the free energy cost due to the inhomogeneity of $\phi$, and $f(\phi)$ is the bulk free energy, typically a double well potential. The standard Cahn-Hilliard model in domain $\Omega_1$ with the dynamic boundary condition is given by \cite{Qian2003MolecularScaleContact,Jing2022ThermodynamicallyConsistentModels,Jing2023ThermodynamicallyConsistentDynamic}
	\begin{subequations}\label{CH1}
		\begin{empheq}[left = \empheqlbrace]{align}
			&\phi_t = - \nabla\cdot \bJ = \nabla \cdot (M \nabla \mu), \quad \mu = \frac{\delta F}{\delta \phi}, \text{ in } \Omega_1, \label{CH1.1}\\
			&\phi_t = -\Gamma\mathcal{L}(\phi) = - \Gamma[\alpha(\phi-h_1)+\bn_1\cdot K\nabla \phi-h_2], \text{ on } \partial\Omega_1, \label{CH1.2}\\
			&\bJ \cdot \bn_1 = - M \nabla \mu \cdot \bn_1 = h_3, \text{ on } \partial \Omega_1. \label{CH1.3}
		\end{empheq}
	\end{subequations}
	where $\mu = \frac{\delta F}{\delta \phi}$ is chemical potential, $\bn_1$ is the unit outward normal of $\partial \Omega_1$, $M > 0$ is the bulk mobility, $\Gamma > 0$ is the boundary mobility (i.e., a rate coefficient), and $\bJ$ is the flux. The dynamic boundary condition \eqref{CH1.2} can reduce to other boundary conditions under different limits for $\Gamma$ and $\alpha$:
	\begin{itemize}
		\item If $\Gamma \rightarrow +\infty$, \eqref{CH1.2} reduces to
		\begin{equation}
			\alpha(\phi-h_1)+\bn_1\cdot K\nabla \phi = h_2,
		\end{equation}
		which is the Robin boundary condition.
		\item If $\Gamma \rightarrow +\infty$ and $\alpha = 0$, \eqref{CH1.2} reduces to
		\begin{equation}
			\bn_1\cdot K\nabla \phi = h_2,
		\end{equation}
		which is the Neumann boundary condition.
		\item If $\Gamma \rightarrow +\infty$ and $\alpha \rightarrow + \infty$, \eqref{CH1.2} reduces to
		\begin{equation}
			\phi = h_1,
		\end{equation}
		which is the Dirichlet boundary condition.
	\end{itemize}
	
	Real-world problems typically involve irregular domains. The complex domain geometry can make the application of a conventional numerical method  intricate. Substantial modification may be needed and sometime the whole method may still fail. To deal with the challenge of solving PDE models in arbitrary domains, several methods have been proposed recently. Within the finite element framework, the key feature of these methods is the use of a partition of unity \cite{Melenk1996PartitionUnityFinite,Duarte1996HpCloudsHp,Oden1998NewCloudbasedHp,Duarte2000GeneralizedFiniteElement}. On the other hand, finite difference methods often require localized modifications of the numerical schemes near the boundaries, as seen in the immersed boundary method \cite{Peskin2002ImmersedBoundaryMethod,OBrien2018VolumeoffluidGhostcellImmersed}, immersed interface method \cite{Chen2008FastFiniteDifference,Li2003OverviewImmersedInterface}, and the smoothed boundary method \cite{Bueno-Orovio2006FourierEmbeddedDomain,Bueno-Orovio2006SpectralMethodsPartial,Bueno-Orovio2006SpectralSmoothedBoundary,Yu2012ExtendedSmoothedBoundary}. A new development in this area has recently been made by Li et al. \cite{Li2009SolvingPDESComplex}, who introduced a general method for solving PDEs in complex geometries subject to three types of boundary conditions. The core idea is to reformulate the problem within a larger, regular domain using the weak formulation, facilitated by an auxiliary phase field function. Known as the diffuse domain method (DDM) \cite{Li2009SolvingPDESComplex}, this approach has been applied and developed since its inception  \cite{Teigen2009DiffuseinterfaceApproachModeling,Teigen2011DiffuseinterfaceMethodTwophase,Aland2010TwophaseFlowComplex,Chen2014TumorGrowthComplex,Chen2019TumorGrowthCalcification,Yu2020HigherorderAccurateDiffusedomain,Lowengrub2016NumericalSimulationEndocytosis,Guo2021DiffuseDomainMethod}.
	
	For gradient flows in arbitrary domains, e.g., those formulated as phase field models, an important issue is whether the gradient flow structure can still be preserved in the extended or modified models. It is worth emphasizing that the extended models are desired to maintain the thermodynamic consistency possessed by the original model. When the DDM is used to derive the extended model, the formulation is directly obtained through the weak solution approach, without explicitly considering the free energy, dissipation or other structures of the extended system.  Recently, Yu et al. introduced a new approach called the Onsager Principle-Based Domain Embedding (OPBDE) method \cite{Yu2025OnsagerPrinciple-BasedDomain}. This method has been applied to develop an extended Allen-Cahn-type model with thermodynamic consistency fully preserved. A key aspect of this approach is the application of the Onsager principle \cite{Onsager1931ReciprocalRelationsIrreversible,Onsager1931ReciprocalRelationsIrreversiblea,Qian2006VariationalApproachMoving,Wang2021GeneralizedOnsagerPrinciple,Xu2017HydrodynamicBoundaryConditions}, which is of fundamental importance to the description of linear thermodynamics. It is anticipated that this method can be applied to extend other complex models in arbitrary domains.
	
	In the present work, we consider a general Cahn-Hilliard-type model that is thermodynamically consistent and supplemented with a dynamic boundary condition of $\phi$ and an inhomogeneous Neumann boundary condition of $\mu$ in an arbitrary irregular domain. There are two main objectives to accomplish. The first is to develop an extended Cahn-Hilliard model in a larger, regular domain using the OPBDE method. Here the Variational Domain Embedding Method (VDEM) is utilized as the original model is dissipative and excludes reversible processes. Thermodynamic consistency is naturally ensured by utilizing the VDEM. The second objective is to allow the existence of flux across the boundary of the original model system. This requires us to reexamine the framework of modeling in which the variational Onsager principle applied. To accomplish the two main objectives outlined above, resolution of several issues are necessary. Firstly, the conservation law needs to be modified in order to work in the extended domain \cite{Zhou2025NewPhaseFieldModel}. Secondly, caused by the flux across the boundary of the original model system, the rate of the free energy pumped into the system needs to be clarified and properly introduced in the VDEM. Thirdly, guided by the Onsager principle, a structure-preserving numerical scheme needs to be developed for the thermodynamically consistent extended model.
	
	The paper is organized as follows. In section 2, the original model with a dynamic boundary condition is extended to a larger, regular domain with thermodynamic consistency being preserved. In section 3, the OPBDE modeling for closed systems is compared to that for open systems regarding the application of the variational Onsager principle. In section 4, an asymptotic analysis is presented for the extended OPBDE Cahn-Hilliard model, with the original model being recovered at the leading order in the original domain. In section 5, an energy dissipation-rate-preserving numerical scheme is developed for the extended model which is thermodynamically consistent. In section 6, numerical simulations are carried out for coarsening dynamics in phase separation and droplet spreading on substrates. Numerical results from the extended OPBDE model are compared with those from the original model and the DDM model, showing remarkable agreement and therefore demonstrating the efficacy and robustness of the numerical scheme in implementing the extended model. Finally, concluding remarks are given in section 7.
	
	\section{The OPBDE Cahn-Hilliard model with a dynamic boundary condition} \label{Sec Dynamic}
	We begin from the model defined in an arbitrary, piecewise smooth domain $\Omega_1$, where the free energy functional is given by
	\begin{equation}
		F(\phi) = \int_{\Omega_1}[f(\phi)+\frac{K}{2}|\nabla \phi|^2]d\bx+\int_{\partial \Omega_1} [\half \alpha (\phi-h_1)^2-h_2\phi]ds,
	\end{equation}
	from which we have
	\begin{equation}
		\delta F = \int_{\Omega_1}\mu \delta \phi d\bx+\int_{\partial \Omega_1}[\alpha(\phi-h_1)+\bn_1\cdot K\nabla \phi-h_2]\delta \phi ds,
	\end{equation}
	where the chemical potential, $\mu$, in the bulk region is defined by
	\begin{equation}
		\mu = \frac{\delta F}{\delta \phi} = f^{\prime}(\phi)-K\nabla^2 \phi.
	\end{equation}
	The rate of change of free energy $F$ is calculated as
	\begin{equation}
		\begin{split}
			F_t
			&=\int_{\Omega_1} [f^{\prime}(\phi)-K\nabla^2 \phi]\phi_t d\bx+\int_{\partial \Omega_1} [\alpha(\phi-h_1)+\bn_1\cdot K\nabla \phi-h_2]\phi_t ds\\
			&=\int_{\Omega_1} \mu \phi_t d\bx+\int_{\partial \Omega_1} \cL(\phi)\phi_t ds,
		\end{split}
	\end{equation}
	where operator $\cL$ at the boundary is defined by
	\begin{equation}
		\cL(\phi)=\alpha(\phi-h_1)+\bn_1\cdot K\nabla \phi-h_2.
	\end{equation}
	The local equilibrium boundary condition,
	\begin{equation}
		\cL(\phi)=0,
	\end{equation}
	means that the equilibration at the boundary is extremely fast so that the dynamics reaches equilibrium instantly. When the equilibration at $\partial\Omega_1$ is no longer extremely fast, the local equilibrium boundary condition is invalidated and must be replaced by a dynamic boundary condition
	\begin{equation}
		\phi_t = - \Gamma\cL(\phi),
	\end{equation}
	which is in the linear response regime and will be derived later. The conservation law of the phase field variable $\phi$ is expressed as
	\begin{equation} \label{Conservation law}
		\phi_t = -\nabla \cdot \bJ,   \text{ in } \Omega_1,
	\end{equation}
	where $\bJ$ is the flux (also called the current density). With the conservation law of $\phi$, $F_t$ can be expressed as follows
	\begin{equation}
		\begin{split}
			F_t&=\int_{\Omega_1}-\mu \nabla \cdot \bJ d\bx +\int_{\partial \Omega_1} \cL(\phi)\phi_t ds\\
			&=\int_{\Omega_1} \nabla \mu \cdot \bJ d\bx + \int_{\partial \Omega_1}\cL(\phi)\phi_t - \mu \bJ\cdot \bn_1 ds.
		\end{split}
	\end{equation}
	
	For $\bJ \cdot \bn_1 = h_3(\bx)$ at $\partial \Omega_1$, which is the inflow-outflow condition, we introduce
	\begin{equation} \label{form of W_t}
		W_t = -\int_{\partial\Omega_1} \mu h_3 ds,
	\end{equation}
	which is the rate of the free energy pumped into the system across the boundary. To apply the variational Onsager principle, we consider the Rayleighian defined by
	\begin{equation}
		R = F_t-W_t+\Phi_{F},
	\end{equation}
	where the dissipation functional $\Phi_{F}$ is a half the rate of free energy dissipation, which is given by
	\begin{equation} \label{Dissipation rate}
		\Phi_{F} = \int_{\Omega_1} \frac{\bJ^2}{2M} d\bx+\int_{\partial \Omega_1}\frac{\phi_t^2}{2\Gamma} ds,
	\end{equation}
	which includes a boundary integral, representing the dissipation taking place at the boundary.
	
	Taking the variation of $R$ with respect to $\bJ$ in $\Omega_1$ and $\phi_t$ at $\partial \Omega_1$, we have
	\begin{equation}
		\begin{split}
			&\frac{\delta R}{\delta \bJ} = 0 \implies \bJ = -M \nabla \mu, \text{ in }\Omega_1, \\
			&\frac{\delta R}{\delta \phi_t} = 0 \implies \phi_t = -\Gamma \cL(\phi), \text{ at }\partial\Omega_1.
		\end{split}
	\end{equation}
	The Cahn-Hilliard model with a dynamic boundary condition is summarized as follows:
	\begin{subequations} \label{CH2}
		\begin{empheq}[left = \empheqlbrace]{align}
			&\phi_t = - \nabla \cdot \bJ = \nabla \cdot (M \nabla \mu ), \quad \text{ in } \Omega_1, \label{CH2.1}\\
			&\mu = f^{\prime}(\phi)-K\nabla^2 \phi , \quad \text{ in } \Omega_1,\label{CH2.2}\\
			&\phi_t = -\Gamma\cL(\phi), \text{ at } \partial\Omega_1,\label{CH2.3}\\
			&\bJ \cdot \bn_1 = -M \nabla \mu \cdot \bn_1 = h_3, \text{ at } \partial\Omega_1. \label{CH2.4}
		\end{empheq}
	\end{subequations}
	The energy dissipation law associated with the model is
	\begin{equation}
		F_t-W_t = -2\Phi_F = -\int_{\Omega_1} \frac{\bJ^2}{M} d\bx-\int_{\partial\Omega_1}\frac{\phi_t^2}{\Gamma} ds.
	\end{equation}

	\begin{figure}[H]
		\centering
		\includegraphics[width=0.5\textwidth]{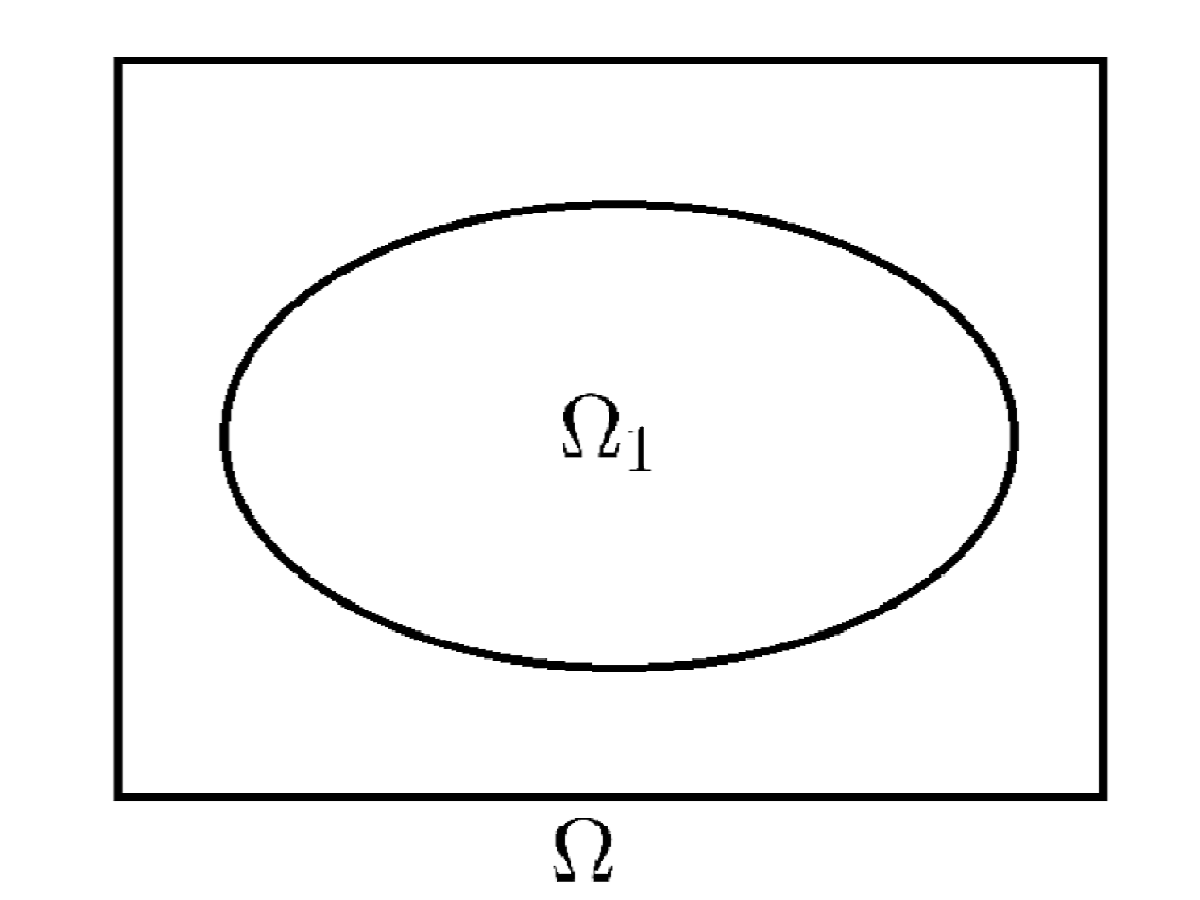}
		\caption{Given domain $\Omega_1$ and its embedded domain $\Omega$.}
	\end{figure}
	
	Next, we extend the model to a larger ``regular'' domain, $\Omega \supset \Omega_1$. This extended model is expected to reproduce the original model in $\Omega_1$ along with the two boundary conditions. We use $\psi$ to denote the characteristic function of domain $\Omega_1$, defined by
	\begin{equation}
		\psi(\bx) =
		\begin{cases}
			1, & \bx \in \Omega_1\backslash\partial\Omega_1, \\
			\half, & \bx \in \partial\Omega_1, \\
			0, & \hbox{otherwise}.
		\end{cases}
	\end{equation}
	The $\delta$ function for piecewise smooth boundary $\partial \Omega_1$, denoted by $\delta_{\partial \Omega_1}$, is represented by $\delta_{\partial \Omega_1} = -\bn_1 \cdot \nabla \psi = \frac{\nabla \psi}{|\nabla \psi|} \cdot\nabla\psi = |\nabla \psi|$. Extensions of $h_1(\bx)$, $h_2(\bx)$ and $h_3(\bx)$ to the neighborhood of $\partial \Omega_1$ are constants in the normal direction as the actual values required by the model are concentrated in a neighborhood of $\partial\Omega_1$.
	
	Multiplying $\psi$ on both sides of \eqref{Conservation law} and incorporating $\bJ \cdot \bn_1 = h_3$ at $\partial \Omega_1$, we obtain a modified conservation law
	\begin{equation} \label{Modified conservation law1}
		\psi\phi_t = -\psi\nabla \cdot \bJ = -\nabla\cdot(\psi \bJ) + \nabla \psi \cdot \bJ = -\nabla\cdot(\psi \bJ)-|\nabla \psi|\bJ\cdot \bn_1 = -\nabla\cdot(\psi \bJ)-|\nabla \psi|h_3.
	\end{equation}
	
	\begin{rem}
		When \eqref{Modified conservation law1} is examined from an integral perspective, the following equations are obtained:
		\begin{equation}
			\begin{cases}
				\int_{\Omega_1^-}\psi \phi_t d\bx = \int_{\Omega_1^-}[-\nabla \cdot (\psi \bJ)-|\nabla \psi|h_3 ]d\bx = -\int_{\partial \Omega_1^-}\bJ \cdot \bn_{\partial \Omega_1^-} ds, \\
				\int_{\Omega_1}\psi \phi_t d\bx = \int_{\Omega_1}[-\nabla \cdot (\psi \bJ)-|\nabla \psi|h_3 ]d\bx = -\half\int_{\partial \Omega_1}\bJ \cdot \bn_{\partial \Omega_1} ds - \half\int_{\partial \Omega_1}h_3 ds = - \int_{\partial \Omega_1}h_3 ds, \\
				\int_{\Omega_1^+}\psi \phi_t d\bx = \int_{\Omega_1^+}[-\nabla \cdot (\psi \bJ)-|\nabla \psi |h_3] d\bx = -\int_{\partial \Omega_1}h_3 ds,
			\end{cases}
		\end{equation}
		where $\Omega_1^- \subset \Omega_1 \subset \Omega_1^+$ and $\bn_{\partial \Omega_1^-}$ is the unit outward normal of $\partial \Omega_1^-$. When $\Omega_1^- \rightarrow \Omega_1$, we have $\bJ \cdot \bn_{\partial \Omega_1^-} \rightarrow h_3$. Therefore, in the integral form, the rate of change of the volume integral of $\psi\phi$ is {\it continuous} as the domain of integration expands. The rate of change depends only on the inflow at $\partial\Omega_1$ when the integration domain contains $\partial \Omega_1$, consistent with the original model in $\Omega_1$.
	\end{rem}
		
	In practice, it is computationally convenient to employ a smooth approximation of $\psi$. Motivated by the previous studies \cite{Yu2025OnsagerPrinciple-BasedDomain,Ratz2006PDEsSurfacesaDiffuse}, we approximate $\psi$ by using the following phase field function :
	\begin{equation}
		\psi_{\varepsilon}(\bx) = \half(1-\tanh (\frac{3r(\bx)}{\varepsilon})) = \frac{1}{e^{\frac{6r(\bx)}{\varepsilon}}+1}, \quad \bx \text{ in } \Omega.
	\end{equation}
	where $r(\bx)$ denotes the signed distance from point $\bx$ to boundary $\partial \Omega_1$. The value of $r(\bx)$ is negative in $\Omega_1$ and positive in $\Omega \backslash \Omega_1$, and $\varepsilon$ measures the width of the diffuse interface.
	
	We define the free energy, $\tilde{F}$, in extended domain $\Omega$ as follows:
	\begin{equation}
		\tilde{F}(\tilde{\phi}) = \int_{\Omega} \left\{\psi_{\varepsilon} [f(\tilde{\phi})+\frac{K}{2}|\nabla \tilde{\phi}|^2]+|\nabla \psi_{\varepsilon}|[\half(\tilde{\phi}-h_1)^2-h_2\tilde{\phi}]\right\}d\bx,
	\end{equation}
	where $\tilde{\phi}$ is the phase field variable in the extended model. The corresponding chemical potential $\tilde{\mu}$ is defined by
	\begin{equation}
		\tilde{\mu} = \frac{\delta \tilde{F}}{\delta \tilde{\phi}} = \psi_{\varepsilon} f^{\prime}(\tilde{\phi})-K\nabla \cdot(\psi_{\varepsilon} \nabla \tilde{\phi})+|\nabla \psi_{\varepsilon}|[\alpha (\tilde{\phi}-h_1)-h_2].
	\end{equation}
	The rate of change of the free energy is given by
	\begin{equation} \label{rate1}
		\begin{split}
			\tilde{F}_t
			&=\int_{\Omega}\left\{\psi_{\varepsilon} f^{\prime}(\tilde{\phi})-K\nabla \cdot(\psi_{\varepsilon} \nabla \tilde{\phi})+|\nabla \psi_{\varepsilon}|[\alpha(\tilde{\phi}-h_1)-h_2]\right\}\tilde{\phi}_t d\bx+\int_{\partial \Omega} \bn \cdot \psi_{\varepsilon} K \nabla \tilde{\phi} \tilde{\phi}_t ds\\
			&=\int_{\Omega} \tilde{\mu} \tilde{\phi}_t d\bx+\int_{\partial \Omega} \bn \cdot \psi_{\varepsilon} K \nabla \tilde{\phi} \tilde{\phi}_t ds,
		\end{split}
	\end{equation}
	where $\bn$ is the unit outward normal of the extended domain $ \Omega$.
	
	We introduce $P = \psi_{\varepsilon} \tilde{\phi}$ and then obtain
	\begin{equation}
		\frac{\delta \tilde{F}}{\delta P} = \frac{\delta \tilde{F}}{\delta \tilde{\phi}}\frac{\delta \tilde{\phi} }{\delta P} = \tilde{\mu}\frac{1}{\psi_{\varepsilon}} = \chi_{\varepsilon}\tilde{\mu},
	\end{equation}
	where $\chi_{\varepsilon} = \frac{1}{\psi_{\varepsilon}}$. Motivated by \eqref{Modified conservation law1}, we arrive at a modified conservation law in $\Omega$ for $P$:
	\begin{equation} \label{Modified conservation law2}
	P_t = -\nabla\cdot \bQ - |\nabla \psi_{\varepsilon}|h_3,
	\end{equation}
	which includes a source term along $\partial \Omega_1$ due to the prescribed inflow-outflow boundary condition in the original model. Here $\bQ$ is the flux corresponding to $P$ in $\Omega$, and then $\tilde{\bJ} = \chi_{\varepsilon}\bQ$ can be introduced as the counterpart of $\bJ$ in $\Omega$. 
	
	Substituting \eqref{Modified conservation law2} into \eqref{rate1}, we have
	\begin{equation} \label{rate2}
		\begin{split}
			\tilde{F}_t
			&=\int_{\Omega} \chi_{\varepsilon}\tilde{\mu} P_t d\bx+\int_{\partial \Omega} \bn \cdot \psi_{\varepsilon} K \nabla \tilde{\phi} \tilde{\phi}_t ds \\
			&=\int_{\Omega} \frac{\delta \tilde{F}}{\delta P}\left(- \nabla \cdot \bQ - |\nabla \psi_{\varepsilon}|h_3\right)d\bx+\int_{\partial \Omega} \bn \cdot \psi_{\varepsilon} K \nabla \tilde{\phi} \tilde{\phi}_t ds\\
			&=\int_{\Omega} \left(\nabla \frac{\delta \tilde{F}}{\delta P}\cdot \bQ - \frac{\delta \tilde{F}}{\delta P}|\nabla \psi_{\varepsilon}|h_3\right) d\bx+\int_{\partial \Omega} \bn \cdot \left(\psi_{\varepsilon} K \nabla \tilde{\phi} \tilde{\phi}_t - \frac{\delta \tilde{F}}{\delta P}\bQ\right) ds.
		\end{split}
	\end{equation}
	We prescribe the homogeneous Neumann boundary condition $\bn \cdot K \nabla \tilde{\phi} = 0$ and the no flux boundary condition $\bQ \cdot \bn = 0$, which is equivalence to $\tilde{\bJ} \cdot \bn = 0$ at $\partial \Omega$. Now we consider the Rayleighian 
	\begin{equation}\label{the_Rayleighian_tilde_R}
		\tilde{R} = \tilde{F}_t - \tilde{W}_t + \Phi_{\tilde{F}},
	\end{equation}
	where $\tilde{W}_t$ is the rate of the free energy pumped into the system across $\partial \Omega_1$, and $\Phi_{\tilde{F}}$ is the dissipation functional in the extended domain. When the extended model needs to describe the dynamic boundary condition and the inflow-outflow condition at $\partial\Omega_1$, it is not viable to have prior knowledge regarding the specific form of $\tilde{W}_t$. But this does not prevent us from deriving a local constitutive equation below. We consider that $\tilde{W}_t$ satisfies
	\begin{equation} \label{Intrinsic property}
		\frac{\delta \tilde{W}_t}{\delta \tilde{\bJ}} = \psi_{\varepsilon}\frac{\delta \tilde{W}_t}{\delta \bQ} = 0, \quad \frac{\delta \tilde{W}_t}{\delta \tilde{\phi}_t} = \psi_{\varepsilon}\frac{\delta \tilde{W}_t}{\delta P_t} = 0.
	\end{equation}
This is to state that the presence of the normal flux $h_3$, which is extrinsic, does not influence the form of the local constitutive equation for $\bQ$ or $\tilde{\bJ}$, which is to be derived as an intrinsic property of the extended system. 
Note that the normal flux $h_3$, originally introduced at $\partial\Omega_1$, is prescribed and hence not treated as a variable. 
As to the dissipation functional, $\Phi_{\tilde{F}}$ is obtained by extending the dissipation functional in \eqref{Dissipation rate} as follows
	\begin{equation}
		\begin{split}
			\Phi_{\tilde{F}} 
			&= \int_{\Omega} (\psi_{\varepsilon} \frac{\tilde{\bJ}^2}{2M}+|\nabla \psi_{\varepsilon}| \frac{\tilde{\phi}_t^2}{2\Gamma})d\bx \\
			&= \int_{\Omega} (\frac{\bQ^2}{2\psi_{\varepsilon}M}+|\nabla \psi_{\varepsilon}| \frac{P_t^2}{2\psi_{\varepsilon}^2\Gamma})d\bx \\
			&= \int_{\Omega} \left[\frac{\bQ^2}{2\psi_{\varepsilon}M}+|\nabla \psi_{\varepsilon}| \frac{(\nabla\cdot \bQ + |\nabla \psi_{\varepsilon}|h_3)^2}{2\psi_{\varepsilon}^2\Gamma}\right]d\bx .
		\end{split}
	\end{equation}
	We note that $\bQ$ is the only flux variable in $\Phi_{\tilde{F}}$. Applying the variational Onsager principle, we arrive at
	\begin{equation} \label{Constitutive equation1}
		\begin{split}
			\frac{\delta \tilde{R}}{\delta \bQ} = 0 &\implies \nabla \frac{\delta \tilde{F}}{\delta P} + \frac{\bQ}{\psi_{\varepsilon}M}-\nabla\left(|\nabla \psi_{\varepsilon}| \frac{\nabla \cdot \bQ + |\nabla \psi_{\varepsilon}|h_3}{\psi_{\varepsilon}^2\Gamma}\right) = 0\\
			&\implies \bQ = -\psi_{\varepsilon}M\nabla\left (\frac{\delta \tilde{F}}{\delta P} +|\nabla \psi_{\varepsilon}| \frac{P_t}{\psi_{\varepsilon}^2\Gamma}\right) , \quad \text{ in } \Omega.
		\end{split}
	\end{equation}
	Then
	\begin{equation} \label{Transport equation1}
		P_t = -\nabla\cdot \left[- \psi_{\varepsilon} M \nabla\left(\frac{\delta \tilde{F}}{\delta P} +|\nabla \psi_{\varepsilon}| \frac{P_t}{\psi_{\varepsilon}^2\Gamma}\right)\right]- |\nabla\psi_{\varepsilon}|h_3 = \nabla\cdot \left[\psi_{\varepsilon} M \nabla\left(\frac{\delta \tilde{F}}{\delta P} +|\nabla \psi_{\varepsilon}| \frac{P_t}{\psi_{\varepsilon}^2\Gamma}\right)\right]- |\nabla\psi_{\varepsilon}|h_3.
	\end{equation}
	Introducing 
	\begin{equation} \label{mu_star}
		\mu_* = \psi_{\varepsilon}\left(\frac{\delta \tilde{F}}{\delta P} +|\nabla \psi_{\varepsilon}| \frac{P_t}{\psi_{\varepsilon}^2\Gamma}\right),
	\end{equation} 
	we have $\bQ = -\psi_{\varepsilon}M\nabla(\chi_{\varepsilon} \mu_*)$, which gives
	\begin{equation} \label{Constitutive equation2}
		\tilde{\bJ} = - M\nabla(\chi_{\varepsilon} \mu_*), 
	\end{equation}
	\begin{equation} \label{Transport equation2}
		\tilde{\phi}_t = \chi_{\varepsilon} \nabla\cdot [\psi_{\varepsilon} M \nabla(\chi_{\varepsilon} \mu_*)]-\chi_{\varepsilon} |\nabla\psi_{\varepsilon}|h_3. 
	\end{equation}
	With the help of the above equations, the rate of change of free energy $\tilde{F}$ is given by
	\begin{equation} \label{rate3}
		\begin{split}
			\tilde{F}_t
			&=\int_{\Omega} [\nabla(\chi_{\varepsilon} \mu_*)\cdot\psi_{\varepsilon} \tilde{\bJ}-\nabla(\chi_{\varepsilon} |\nabla \psi_{\varepsilon}|\Gamma^{-1}\tilde{\phi}_t)\cdot\psi_{\varepsilon} \tilde{\bJ}-\chi_{\varepsilon} \tilde{\mu}|\nabla \psi_{\varepsilon}|h_3]d\bx \\
			&= \int_{\Omega}[-\psi_{\varepsilon} \frac{\tilde{\bJ}^2}{M} + \chi_{\varepsilon} |\nabla \psi_{\varepsilon}|\Gamma^{-1}\tilde{\phi}_t \nabla \cdot (\psi_{\varepsilon} \tilde{\bJ}) - \chi_{\varepsilon} \tilde{\mu}|\nabla \psi_{\varepsilon}|h_3] d\bx \\
			&=\int_{\Omega} [-\psi_{\varepsilon} \frac{\tilde{\bJ}^2}{M} - |\nabla \psi_{\varepsilon}|\Gamma^{-1}\tilde{\phi}_t (\tilde{\phi}_t+\chi_{\varepsilon} |\nabla \psi_{\varepsilon}|h_3) - \chi_{\varepsilon} \tilde{\mu}|\nabla \psi_{\varepsilon}|h_3] d\bx \\
			&=\int_{\Omega}[-\psi_{\varepsilon} \frac{\tilde{\bJ}^2}{M} - |\nabla \psi_{\varepsilon}|\frac{\tilde{\phi}_t^2}{\Gamma} - \chi_{\varepsilon}(\tilde{\mu} + |\nabla \psi_{\varepsilon}|\Gamma^{-1}\tilde{\phi}_t) |\nabla \psi_{\varepsilon}|h_3]d\bx \\
			&=-\int_{\Omega} (\psi_{\varepsilon} \frac{\tilde{\bJ}^2}{M}+|\nabla \psi_{\varepsilon}|\frac{\tilde{\phi}_t^2}{\Gamma})d\bx - \int_{\Omega} \chi_{\varepsilon}\mu_* |\nabla \psi_{\varepsilon}|h_3d\bx. \\
			&=-2\Phi_{\tilde{F}} - \int_{\Omega} \chi_{\varepsilon}\mu_* |\nabla \psi_{\varepsilon}|h_3d\bx.
		\end{split}
	\end{equation}
	Combining $\tilde{F}_t$ in \eqref{rate3} with the energy law based on the variational Onsager principle
	\begin{equation} \label{Energy dissipation law}
		\tilde{F}_t - \tilde{W}_t = -2 \Phi_{\tilde{F}},
	\end{equation}
	we obtain $\tilde{W}_t$ as follows
	\begin{equation}
		\tilde{W}_t = - \int_{\Omega} \chi_{\varepsilon}\mu_* |\nabla \psi_{\varepsilon}|h_3d\bx.
	\end{equation}
	Below we show that $\tilde{W}_t$ is a functional of $\tilde{\phi}$, $\psi_{\varepsilon}$ and $h_3$ to validate equation \eqref{Intrinsic property}, which ensures that the local constitutive equation \eqref{Constitutive equation1} derived for $\bQ$ is an intrinsic property that is independent of the extrinsic and prescribed $h_3$.
	
	We define two operators: $\cA = - \chi_{\varepsilon} \nabla [\psi_{\varepsilon} M \nabla(\chi_{\varepsilon} \cdot)]$ and $\cB = |\nabla \psi_{\varepsilon}|\Gamma^{-1}$. It follows that \eqref{Transport equation1} can be formally written as follows:
	\begin{equation} \label{Transport equation3}
		\begin{split}
			\tilde{\phi}_t
			&= -\left\{1 -\chi_{\varepsilon} \nabla \cdot [\psi_{\varepsilon} M \nabla (\chi_{\varepsilon} |\nabla \psi_{\varepsilon}|\Gamma^{-1} \cdot )]\right\}^{-1}\left\{- \chi_{\varepsilon} \nabla [\psi_{\varepsilon} M \nabla(\chi_{\varepsilon}\tilde{\mu})] +\chi_{\varepsilon}|\nabla \psi_{\varepsilon}|h_3\right\} \\
			&= - (1 + \cA\cB)^{-1}\cA\tilde{\mu} - (1 + \cA\cB)^{-1}\chi_{\varepsilon}|\nabla \psi_{\varepsilon} |h_3.
		\end{split}
	\end{equation}
	Introducing  $\beta = - (1 + \cA\cB)^{-1}\chi_{\varepsilon}|\nabla \psi_{\varepsilon} |h_3$ with the boundary condition, $\bn \cdot M \nabla (\cB \beta) = 0$, we have
	\begin{equation}
		\beta = - (\chi_{\varepsilon}|\nabla\psi_{\varepsilon}|h_3 + \cA \cB\beta),
	\end{equation}
	and
	\begin{align}
		\begin{array}{rcl}
			\tilde{W}_t
			&=& - \int_{\Omega} \chi_{\varepsilon}\mu_* |\nabla \psi_{\varepsilon}|h_3d\bx \\
			&=& \int_{\Omega}(\tilde{\mu}+\cB\tilde{\phi}_t)(\beta + \cA\cB\beta)d\bx \\
			&=& \int_{\Omega}(\tilde{\mu}\beta + \cB\tilde{\phi}_t\beta + \mu_*\cA\cB\beta)d\bx \\
			&=& \int_{\Omega}(\tilde{\mu}\beta + \cA\mu_*\cB\beta + \tilde{\phi}_t\cB\beta)d\bx \\
			&=& \int_{\Omega}[\tilde{\mu}\beta + (\cA\mu_* + \tilde{\phi}_t)\cB\beta]d\bx \\
			&=& \int_{\Omega}(\tilde{\mu}\beta - \chi_{\varepsilon}|\nabla \psi_{\varepsilon} |h_3\cB\beta) d\bx,
		\end{array}
	\end{align}
	which expresses $\tilde{W}_t$ as a functional of $\tilde{\phi}$, $\psi_{\varepsilon}$ and $h_3$ with the help of operators. 
Had we known this expression beforehand, we would have derived the constitutive equation \eqref{Constitutive equation1} as well. Such a derivation would avoid the explicit use of equation \eqref{Intrinsic property} and achieve full self-consistency in the sense that the known expression of $\tilde{W}_t$, used as an input into the Rayleighian $\tilde{R}$ in \eqref{the_Rayleighian_tilde_R}, will be identical to that obtained later from the energy law \eqref{Energy dissipation law}.

	\begin{rem}
		We note that $\beta$ can be solved from the equation $(1 + \cA\cB)\beta = -\chi_{\varepsilon}|\nabla \psi_{\varepsilon}|h_3$ with the boundary condition. This is an elliptic type equation. The Lax-Milgram theorem can be used to show that there exists a unique solution.
	\end{rem}
	
	\begin{rem}
		When $\Gamma \rightarrow +\infty$, $\beta = - \chi_{\varepsilon}|\nabla \psi_{\varepsilon}|h_3$ and $\cB = 0$, then 
		\begin{equation}
			\tilde{W}_t = -\int_{\Omega}\chi_{\varepsilon}\tilde{\mu}|\nabla \psi_{\varepsilon} |h_3 d\bx
=-\int_{\Omega}\displaystyle\frac{\delta\tilde F }{\delta P}|\nabla \psi_{\varepsilon} |h_3 d\bx,
		\end{equation}
		which appears explicitly in equation \eqref{rate2}.
	\end{rem}
	The OPBDE Cahn-Hilliard model is summarized as follows:
	\begin{subequations} \label{OPBDE CH}
		\begin{empheq}[left = \empheqlbrace]{align}
			&\tilde{\phi}_t = -\chi_{\varepsilon}\nabla\cdot(\psi_{\varepsilon} \tilde{\bJ})-\chi_{\varepsilon}|\nabla \psi_{\varepsilon}|h_3 = \chi_{\varepsilon} \nabla \cdot [\psi_{\varepsilon} M \nabla (\chi_{\varepsilon} \mu_*)]-\chi_{\varepsilon} |\nabla \psi_{\varepsilon} | h_3, \quad \text{ in } \Omega, \\
			&\mu_* = \psi_{\varepsilon} f^{\prime}(\tilde{\phi})-K\nabla \cdot(\psi_{\varepsilon} \nabla \tilde{\phi})+|\nabla \psi_{\varepsilon}|[\alpha (\tilde{\phi}-h_1)-h_2 +\Gamma^{-1} \tilde{\phi}_t], \quad \text{ in } \Omega,\\
			&\bn \cdot K \nabla \tilde{\phi} = 0, \quad \tilde{\bJ} \cdot \bn = -M \nabla (\chi_{\varepsilon} \mu_*) \cdot \bn = 0 , \text{ at } \partial\Omega.
		\end{empheq}
	\end{subequations}
	
	\begin{rem}
		We notice that $\tilde{W}_t$ can be decomposed into two parts:
		\begin{equation} \label{Divided W_t}
			\tilde{W}_t = - \int_{\Omega} \chi_{\varepsilon}\mu_* |\nabla \psi_{\varepsilon}|h_3d\bx = - \int_{\Omega} [f^{\prime}(\tilde{\phi}) - K\nabla^2\tilde{\phi}]|\nabla \psi_{\varepsilon}|h_3d\bx - \int_{\Omega} \chi_{\varepsilon} |\nabla \psi_{\varepsilon}|(\cL(\tilde{\phi}) + \Gamma^{-1}\tilde{\phi}_t)|\nabla \psi_{\varepsilon}|h_3d\bx.
		\end{equation}
		According to equation \eqref{asymptotic_dynamic_BC} from the asymptotic analysis in Sec. \ref{Sec Asymptotic}, we note that the last term  in \eqref{Divided W_t}, 
		\begin{equation}
			-\int_{\Omega} \chi_{\varepsilon} |\nabla \psi_{\varepsilon}|(\cL(\tilde{\phi}) + \Gamma^{-1}\tilde{\phi}_t)|\nabla \psi_{\varepsilon}|h_3d\bx ,
		\end{equation}
		equals zero at the leading order.   
   		Furthermore, the other term, $- \int_{\Omega} [f^{\prime}(\tilde{\phi}) - K\nabla^2\tilde{\phi}]|\nabla \psi_{\varepsilon}|h_3d\bx$ in \eqref{Divided W_t} is identical to $W_t$ in \eqref{form of W_t} as $-\int_{\Omega} [f^{\prime}(\tilde{\phi}) - K\nabla^2\tilde{\phi}]|\nabla \psi_{\varepsilon}|h_3d\bx = -\int_{\partial\Omega_1}[f^{\prime}(\tilde{\phi}) - K\nabla^2\tilde{\phi}]h_3ds$ (with $|\nabla \psi_{\varepsilon}|\to \delta_{\partial \Omega_1}$ for $\varepsilon \to 0$). Therefore,  the rate of change of the free energy, $F_t$, the dissipation functional $\Phi_{F}$, and the rate of the free energy pumped into the system $W_t$ in the original model can all be consistently represented in the extended model.
	\end{rem}
	
	\begin{rem}
		In equation \eqref{Transport equation3}
		\begin{equation}
			\tilde{\phi}_t = - (1 + \cA\cB)^{-1}\cA\tilde{\mu} - (1 + \cA\cB)^{-1}\chi_{\varepsilon}|\nabla \psi_{\varepsilon} |h_3,
		\end{equation}
		if we define the operator, $\cM = (1+\cA\cB)^{-1}\cA >0$, which is the mobility operator for the extended system, then its inverse, usually called the damping coefficient, is
		\begin{equation}
			\cM^{-1} = \cA^{-1} + \cB = \{-\chi_{\varepsilon} \nabla \cdot [\psi_{\varepsilon} M \nabla (\chi_{\varepsilon} \cdot )]\}^{-1} + |\nabla \psi_{\varepsilon}|\Gamma^{-1}.
		\end{equation}
		We have
		\begin{equation}
			\begin{split}
				\tilde{F}_t
				&=\int_{\Omega} \tilde{\mu}\tilde{\phi}_t d\bx+\int_{\partial \Omega} \bn \cdot K \nabla \tilde{\phi}\tilde{\phi}_t ds\\
				&=\int_{\Omega}-\tilde{\mu}\cM\tilde{\mu}d\bx + \int_{\Omega}\tilde{\mu}\beta d\bx \\
				&=\int_{\Omega}(-\tilde{\mu}\cM\tilde{\mu} + \chi_{\varepsilon}|\nabla \psi_{\varepsilon}|h_3\cB\beta) d\bx + \int_{\Omega}(\tilde{\mu}\beta - \chi_{\varepsilon}|\nabla \psi_{\varepsilon}|h_3\cB\beta)d\bx \\
				&=\int_{\Omega}(-\tilde{\mu}\cM\tilde{\mu} + \chi_{\varepsilon}|\nabla \psi_{\varepsilon}|h_3\cB\beta) d\bx + \tilde{W}_t,
			\end{split}
		\end{equation}
		which gives $\Phi_{\tilde{F}}$ in another form:
		\begin{equation}
			\Phi_{\tilde{F}} = \int_{\Omega}[\half(\tilde{\phi}_t - \beta)\cM^{-1}(\tilde{\phi}_t - \beta) - \half\chi_{\varepsilon}|\nabla \psi_{\varepsilon}|h_3\cB\beta] d\bx = \int_{\Omega} (\psi_{\varepsilon} \frac{\tilde{\bJ}^2}{2M}+|\nabla \psi_{\varepsilon}| \frac{\tilde{\phi}_t^2}{2\Gamma})d\bx
		\end{equation}
		according to energy law \eqref{Energy dissipation law}. In fact, if $\tilde{\phi}_t$ is instead used as the only flux variable of the system, the same extended model can be derived. Details are presented in Appendix I.
		
		\begin{rem}
			In addition, we can derive the extended model by applying a Lagrange multipliers' method. Details are presented in Appendix II.
		\end{rem}
	\end{rem}

	\section{Onsager principle in open vs closed systems}
	If a model is said to be thermodynamically consistent in a closed or periodic system in a domain $\Omega$, we typically have
	\begin{equation} \label{close system}
		F_t = -2 \Phi_{F}.
	\end{equation}
	This corresponds to our model with $h_3 = 0$.
	
	For an open system that includes inflow and outflow, i.e., $h_3 \neq 0$, \eqref{close system} is no longer valid. It is apparent that we need to introduce another rate for the free energy pumped into the system. For the models considered with infinite $\Gamma$, $W_t$ and $\tilde{W}_t$ appear explicitly in $F_t$ or $\tilde{F}_t$, and we can easily find them. However, for the models considered with finite $\Gamma$, though $W_t$ remains the same, it is challenging to predict in advance the form of $\tilde{W}_t$ in the extended domain. But this does not prevent us from constructing the corresponding Rayleighian $\tilde{R}$. The variational approach can proceed under the intrinsic property from thermodynamic laws that $\tilde{W}_t$ is independent of the flux variable in the system. Therefore, the constitutive equation, \eqref{Constitutive equation2}, can still be obtained by applying the Onsager principle without knowing the explicit form of $\tilde{W}_t$. Finally, the specific form of $\tilde{W}_t$ is validated by the energy law,
		\begin{equation}
			\tilde{F}_t - \tilde{W}_t = 2 \Phi_{\tilde{F}}.
		\end{equation}
	In the next section, we will show the validity of the OPBDE Cahn-Hilliard model through asymptotic analysis. Specifically, we show that as $\varepsilon \rightarrow 0$, the OPBDE Cahn-Hilliard model reduces to the original model in $\Omega_1$ and recovers the prescribed boundary conditions at $\partial \Omega_1$.

	\section{Asymptotic analysis} \label{Sec Asymptotic}
	In this section, we analyze the asymptotic limits of the OPBDE Cahn-Hilliard model \eqref{OPBDE CH} as the interface width $\varepsilon \to 0$. We expand the unknowns in powers of $\varepsilon$ in regions close to and far from the boundary in the inner and outer expansions, respectively. Firstly, we introduce a local coordinate system for the inner expansion. Define $\bX: S \rightarrow \mathbb{R}^{n}$ as a parametric representation of $\partial \Omega_{1}$, where $S$ is an oriented manifold of dimension $n-1$, and define $\omega = \{\bx \in \Omega:|r(\bx)|<d, \; 0\leq d \ll 1\}$. Then $\forall \bx\in \omega$, we write $\bx = \bX(s)+r(\bx) \mathbf{n}(s)$, where $s$ is the coordinate at $S$. We rewrite $\tilde{\phi}$, $\mu_*$, $\psi_{\varepsilon}$, $\chi_{\varepsilon}$, and $h_i$ (for $i = 1, 2, 3$) in the new coordinate system as follows:
	\begin{equation}
		\begin{aligned}
			&\tilde{\phi}(s,r) = \tilde{\phi}(\bx) = \tilde{\phi}(\bX(s)+r(\bx) \bn(s)), & \bx \in \omega, \\
			&\mu_{*}(s,r) = \mu_{*}(\bx) = \mu_{*}(\bX(s)+r(\bx) \bn(s)), & \bx \in \omega,\\
			&\psi_{\varepsilon}(s,r) = \psi_{\varepsilon}(\bx) = \psi_{\varepsilon}(\bX(s)+r(\bx) \bn(s)), & \bx \in \omega,\\
			&\chi_{\varepsilon}(s,r) = \chi_{\varepsilon}(\bx) = \chi_{\varepsilon}(\bX(s)+r(\bx) \bn(s)), & \bx \in \omega,\\
			&h_i(s,r) = h_i(\bx) = h_i(\bX(s)+r(\bx) \bn(s)), \quad i = 1,2,3, & \bx \in \omega.\\
		\end{aligned}
	\end{equation}
	For the outer expansion, we focus on the region with $r < 0$ (in $\Omega_1$) and $\psi_{\varepsilon} = \chi_{\varepsilon} = 1$. We expand $\tilde{\phi}_{\varepsilon}$, $\mu_{*,\varepsilon}$, and $h_i$ (for $i = 1, 2, 3$) in powers of $\varepsilon$:
	\begin{equation}
		\begin{split}
			&\tilde{\phi}(s,r) = \tilde{\phi}_0(s,r)+\varepsilon\tilde{\phi}_1(s,r)+\cdots,\\
			&\mu_{*,}(s,r) = \mu_{*,0}(s,r)+\varepsilon \mu_{*,1}(s,r)+\cdots,\\
			&h_i(s,r) = h_{i,0}(s,r)+\varepsilon h_{i,1}(s,r)+\cdots, \quad i = 1,2,3.
		\end{split}
	\end{equation}
	For the inner expansion, the stretched variable $z =\frac {r}{\varepsilon}$ is introduced, and the derivatives are scaled as follows:
	\begin{equation}
		\begin{split}
			&\nabla = \frac{1}{\varepsilon}\bn_1 \partial_z+\frac{1}{1+\varepsilon z \kappa}\nabla_s,\\
			&\nabla^2 = \frac{1}{\varepsilon^2}\partial_{zz}+\frac{1}{\varepsilon}\frac{\kappa}{1+\varepsilon z \kappa}\partial_z+ \frac{1}{1+\varepsilon z \kappa}\nabla_s\cdot(\frac{1}{1+\varepsilon z \kappa}\nabla_s)
		\end{split}
	\end{equation}
	where $\kappa = \nabla_s\cdot\bn_1$ is the mean curvature of $\partial \Omega_1$. We use $\Phi$, $\bar{\mu}_*$, $\Psi$, $\cX$, and $H_i$ (for $i = 1,2,3$) to represent the functions in $\omega$:
	\begin{equation}
		\begin{aligned}
			&\Phi(s,z) = \tilde{\phi}(s,r),\quad &&\bar{\mu}_*(s,z) = \mu_*(s,r),\quad \Psi(s,z) = \psi_{\varepsilon}(s,r), \\
			&\cX(s,z) = \chi_{\varepsilon}(s,r),&&H_i(s,z) = h_i(s,r), \quad i = 1,2,3.
		\end{aligned}
	\end{equation}
	These variables are expanded in powers of $\varepsilon$:
	\begin{equation}
		\begin{split}
			&\Phi(s,z) = \Phi_0(s,z)+\varepsilon\Phi_1(s,z)+\cdots,\\
			&\bar{\mu}_*(s,z) = \bar{\mu}_{*,0}(s,z)+\varepsilon \bar{\mu}_{*,1}(s,z)+\cdots,\\
			&\Psi = \Psi_0 = \half(1-\tanh(3z)) = \frac{1}{e^{6z}+1},\\
			&\cX = \cX_0 = \frac{e^{6z}+1}{\delta e^{6z}+1}, \\
			&H_i(s,z) = H_{i,0}(s,z)+\varepsilon H_{i,1}(s,z)+\cdots, \quad i = 1,2,3.
		\end{split}
	\end{equation}
	It is assumed that $h_i$, $i = 1,2,3$, are independent of $z$ in $\omega$, and hence $H_{i,j}$, $i = 1,2,3$, are independent of $z$ as well, i.e.,
	\begin{equation}
		\begin{split}
			H_i(s) = H_{i,0}(s)+\varepsilon H_{i,1}(s)+\cdots.
		\end{split}
	\end{equation}
	
	Now we consider the matching conditions. For simplicity, the generic label $\hat{V}$ is used to represent the outer expansions of $\tilde{\phi}$, $\mu_*$, $\psi_{\varepsilon}$, $\chi_{\varepsilon}$, and $h_i$ (for $i = 1,2,3$), and the generic label $\bar{V}$ is used to represent the inner expansions of $\Phi$, $\bar{\mu}_*$, $\Psi$, $\cX$, and $H_i$ (for $i = 1,2,3$). In the overlapping region where both expansions are valid, we apply the matching conditions
	\begin{equation}
		\lim_{z\rightarrow \pm\infty} \bar{V}(s,z) \simeq \lim_{r\rightarrow \pm 0} \hat{V}(\bX+r \bn_1),
	\end{equation}
	which lead to the following specific matching conditions \cite{Fife1988DynamicsInternalLayers,Caginalp1988DynamicsLayeredInterfaces,Pego1989FrontMigrationNonlinear,Holmes2013IntroductionPerturbationMethods}:
	\begin{subequations} \label{Matching condition}
		\begin{align}
			&\lim_{z\rightarrow \pm\infty}\bar{V}_{0} = \lim_{r\rightarrow \pm 0}\hat{V}_0, \label{Matching condition1}\\
			&\lim_{z\rightarrow \pm\infty}\partial^i_z\bar{V}_{0} = 0,\quad i = 1,2,\cdots, \label{Matching condition2} \\
			&\lim_{z\rightarrow \pm\infty}\partial_z \bar{V}_{1} = \lim_{r\rightarrow \pm 0}\bn_1 \cdot \nabla\hat{V}_0. \label{Matching condition3}
		\end{align}
	\end{subequations}
	
	In the following, an asymptotic analysis based on model \eqref{OPBDE CH} will be carried out to derive the dynamic boundary condition in the leading order. It can reduce to the Robin boundary condition when $\Gamma \rightarrow \infty$, and the Robin condition can reduce to the Neumann boundary condition when $\alpha = 0$ and the Dirichlet boundary condition when $\alpha \rightarrow \infty$.
	
	\begin{itemize}
		\item Outer expansion: At the leading order $\bsO(1)$, we obtain
		\begin{subequations}
			\begin{empheq}[left = \empheqlbrace]{align}
				&\tilde{\phi}_{0,t} = \nabla \cdot ( M \nabla \mu_{*,0}), \\
				&\mu_{*,0} = f^{\prime}(\tilde{\phi}_{0})-K\nabla^2\tilde{\phi}_{0},
			\end{empheq}
		\end{subequations}
		in which \eqref{CH2.1} and \eqref{CH2.2} in the Cahn-Hilliard model are recovered.
		\item Inner expansion:
		At the leading order $\bsO(\varepsilon^{-2})$ , we obtain
		\begin{subequations}
			\begin{empheq}[left = \empheqlbrace]{align}
				&\cX_0 \partial_z \left[\Psi_0 M \partial_z (\cX_0 \bar{\mu}_{*,0})\right] = 0,\\
				&- K \partial_z(\Psi_0 \partial_z \Phi_0) = 0.
			\end{empheq}
		\end{subequations}
		Using an integration and the matching condition \eqref{Matching condition2} for $i=1$, we have
		\begin{subequations}
			\begin{empheq}[left = \empheqlbrace]{align}
				&\partial_z (\cX_0 \bar{\mu}_{*,0}) = 0,\\
				&\partial_z \Phi_0 = 0.
			\end{empheq}
		\end{subequations}
		At the next order $\bsO(\varepsilon^{-1})$, we obtain
		\begin{subequations}
			\begin{empheq}[left = \empheqlbrace]{align}
				&\cX_0 \partial_z \left[\Psi_0 M \partial_z (\cX_0 \bar{\mu}_{*,1})\right]+\cX_0\partial_z\Psi_0 H_{3,0} = 0,\\
				&- K \partial_z(\Psi_0 \partial_z \Phi_1)-[\alpha(\Phi_0-H_{1,0})-H_{2,0}+\Gamma^{-1}\Phi_{0,t}] \partial_z \Psi_0 = 0.
			\end{empheq}
		\end{subequations}
		Integrating the above equations from $-\infty$ to $+\infty$, we have
		\begin{subequations}
			\begin{empheq}[left = \empheqlbrace]{align}
				&\lim\limits_{z \rightarrow -\infty} M \partial_z (\cX_0 \bar{\mu}_{*,1}) = \lim\limits_{z \rightarrow -\infty}-H_{3,0},\\
				&		\lim\limits_{z \rightarrow -\infty} \alpha(\Phi_0-H_{1,0})+K\partial_z\Phi_1-H_{2,0} = \lim\limits_{z \rightarrow -\infty} -\Gamma^{-1}\Phi_{0,t}.
			\end{empheq}
		\end{subequations}
		Using the matching conditions \eqref{Matching condition1} and \eqref{Matching condition3}, we have
		\begin{subequations}
			\begin{empheq}[left = \empheqlbrace]{align}
				&\lim\limits_{r \rightarrow 0^-} \bn_1 \cdot M \nabla \mu_{*,0} = \lim\limits_{z \rightarrow -\infty} -h_{3,0},\\
				&\lim\limits_{r \rightarrow 0^-}\alpha(\tilde{\phi}_0-h_{1,0})+\bn_1 \cdot K\nabla \tilde{\phi}_0-h_{2,0} = \lim\limits_{r \rightarrow 0^-}-\Gamma^{-1}\tilde{\phi}_{0,t},   \label{asymptotic_dynamic_BC}
			\end{empheq}
		\end{subequations}
		in which \eqref{CH2.3} and \eqref{CH2.4} are recovered.
	\end{itemize}
	
	In summary, the asymptotic analysis presented above has established the connection between the OPBDE Cahn-Hilliard model and the original model defined in an arbitrary domain. For the extended model in a regular domain, a thermodynamically consistent structure-preserving numerical scheme can be developed.
	
	\section{Numerical algorithm}
	In this section, we present a structure-preserving numerical scheme for the OPBDE Cahn-Hilliard model \eqref{OPBDE CH}. For simplicity, we drop the subscript $\varepsilon$ in $\psi_{\varepsilon}$ and $\chi_{\varepsilon}$ in the presentation below. We briefly introduce the EQ method \cite{Zhao2022GeneralFrameworkDerive,Zhao2018GeneralStrategyNumerical,Xu2019EfficientLinearSchemes} to linearize the model by introducing an auxiliary variable
	\begin{equation}
		q(\bx, t) = \sqrt{2f(\tilde{\phi})+2A},
	\end{equation}
	where $A>0$ is a constant large enough to make $q$ a well-defined real-valued function for all $\tilde{\phi} \in (-\infty, \infty)$.
	We denote $g(\tilde{\phi}) = \frac{\partial q}{\partial \tilde{\phi}} = \frac{ f^{\prime}(\tilde{\phi})}{\sqrt{2f(\tilde{\phi})+2A}}$. The free energy of the OPBDE Cahn-Hilliard model \eqref{OPBDE CH} can be written as
	\begin{equation}
		\bar{F} = \int_{\Omega}\psi (\half q^2+\half K\|\nabla \tilde{\phi}\|^2 - A)+|\nabla \psi|[\half(\tilde{\phi}-h_1)^2-h_2\tilde{\phi}] d\bx.
	\end{equation}
	Below we present a finite difference scheme with 2nd order central difference in space and Crank-Nicolson discretization in time.
	
	To simplify the presentation, we introduce the following notations. For the extended regular domain $\Omega = [-\half L_x,\half L_x]\times[-\half L_y,\half L_y]$, we divide it into a rectangular mesh with mesh size $\Delta_x = L_x/N_x$ and $\Delta_y = L_y/N_y$, where $N_x$ and $N_y$ are the numbers of the grids in the $x$ and $y$ directions, respectively.
	
	We define the interpolation and extrapolation of variables as follows
	\begin{equation}
		\begin{aligned}
			& (\bullet)^{n+\frac{1}{2}} = \frac{1}{2}(\bullet)^{n+1}+\frac{1}{2}(\bullet)^n, n \in \BN, \\
			& \bar{(\bullet)}^{n+\frac{1}{2}} = \frac{3}{2}(\bullet)^n-\frac{1}{2}(\bullet)^{n-1}, n \in \BN\backslash\{0\}, \quad \bar{(\bullet)}^{n+\frac{1}{2}} = (\bullet)^n, n = 0.
		\end{aligned}
	\end{equation}
	The average and difference operators are defined by $A_x$, $A_y$, $D_x$, $D_y$, $D_h$:
	\begin{equation}
		\begin{aligned}
			&A_x(\bullet)_{i+\half,j} = \frac{1}{2}[(\bullet)_{i+1,j}+(\bullet)_{i,j}], \quad A_y(\bullet)_{i,j+\half} = \frac{1}{2}[(\bullet)_{i,j+1}+(\bullet)_{i,j}]; \\
			&D_x(\bullet)_{i+\frac{1}{2},j} = \frac{1}{\Delta x}[(\bullet)_{i+1,j}-(\bullet)_{i,j}], \quad D_y(\bullet)_{i,j+\frac{1}{2}} = \frac{1}{\Delta y}[(\bullet)_{i,j+1}-(\bullet)_{i,j}];\\
			&D_h(\bullet,\star)_{i,j} = D_x[A_x (\bullet) D_x (\star)]_{i,j}+ D_y[A_y (\bullet) D_y (\star)]_{i,j}.
		\end{aligned}
	\end{equation}
	Here $i$ and $j$ can take either integer or half-integer values.
	
	Based on the above definitions, we define the discrete 2D weighted inner products as follows
	\begin{equation}
		(u,v) = \Delta x \Delta y \sum_{i = 1}^{N_x}\sum_{j = 1}^{N_y}u_{i,j}v_{i,j}.
	\end{equation}
	
	For the OPBDE Cahn-Hilliard model \eqref{OPBDE CH}, the corresponding full discrete scheme is presented as follows.
	
	\begin{alg}[Second Order Energy-Dissipation-Rate-Preserving Scheme] \label{Algorithm} Given $\tilde{\phi}_{i,j}^n,q_{i,j}^n$, we can update $\tilde{\phi}_{i,j}^{n+1}$ via
		\begin{subequations} \label{EQ scheme}
			\begin{empheq}[left = \empheqlbrace]{align}
				&\frac{\tilde{\phi}_{i,j}^{n+1}-\tilde{\phi}_{i,j}^n}{\Delta t} = \chi_{i,j} D_h(\psi\bar{M}^{{n+\frac{1}{2}}}, \chi\mu_{*}^{{n+\frac{1}{2}}})_{i,j}-\half \chi_{i,j}
				\left\{|D_x\psi_{i-\half,j}|A_x h_{3,i-\half,j}+ \right. \nonumber\\
				&\left.|D_x\psi_{i+\half,j}|A_x h_{3,i+\half,j} +|D_y\psi_{i,j-\half}|A_y h_{3,i,j-\half} +|D_y\psi_{i,j+\half}|A_y h_{3,i,j+\half}\right\}, \label{EQ scheme1} \\
				&\mu_{*, i,j}^{n+\frac{1}{2}} = \psi_{i,j}\bar{g}_{i,j}^{^{n+\frac{1}{2}}} q_{i,j}^{^{n+\frac{1}{2}}}-K D_h(\psi,\tilde{\phi}^{^{n+\frac{1}{2}}})_{i,j}+\half \left\{|D_x\psi_{i-\half,j}|A_x[\alpha(\tilde{\phi}^{n+\half}-h_{1})-h_2 \right. \nonumber \\
				&+\overline{\Gamma^{-1}}^{n+\half}\frac{\tilde{\phi}^{n+1}-\tilde{\phi}^n}{\Delta t}]_{i-\half,j} +|D_x\psi_{i+\half,j}|A_x[\alpha(\tilde{\phi}^{n+\half}-h_{1})-h_2+ \nonumber\\
				& \overline{\Gamma^{-1}}^{n+\half}\frac{\tilde{\phi}^{n+1}-\tilde{\phi}^n}{\Delta t}]_{i+\half,j}+|D_y\psi_{i,j-\half}|A_x[\alpha(\tilde{\phi}^{n+\half}-h_{1})-h_2 +\overline{\Gamma^{-1}}^{n+\half}\frac{\tilde{\phi}^{n+1}-\tilde{\phi}^n}{\Delta t}]_{i,j-\half} \nonumber\\
				&\left. + |D_y\psi_{i,j+\half}|A_x[\alpha(\tilde{\phi}^{n+\half}-h_{1})- h_2+\overline{\Gamma^{-1}}^{n+\half}\frac{\tilde{\phi}^{n+1}-\tilde{\phi}^n}{\Delta t}]_{i,j+\half}\right\},\label{EQ scheme2} \\
				&q_{i,j}^{n+1}-q_{i,j}^n = \bar{g}_{i,j}^{^{n+\frac{1}{2}}}(\tilde{\phi}_{i,j}^{n+1}-\tilde{\phi}_{i,j}^n).\label{EQ scheme3}
			\end{empheq}
		\end{subequations}
		All variables are discretized at the cell-center points with the following boundary conditions:
		\begin{equation} \label{BC DIS}
			\begin{split}
				\tilde{\phi}_{i,0}^k = \tilde{\phi}_{i,1}^k, \quad \tilde{\phi}_{i,N_y}^k = \tilde{\phi}_{i,N_y+1}^k, \quad, \forall i\in \{1,2,\cdots,N_x\}, k = n,n+1;\\
				\tilde{\phi}_{0,j}^{k} = \tilde{\phi}_{1,j}^{k}, \quad \tilde{\phi}_{N_x,j}^{k} = \tilde{\phi}_{N_x+1,j}^{k}, \quad \forall j\in \{1,2,\cdots,N_y\}, k = n,n+1;\\
				\bar{M}_{i,\half}^{k-\half}\mu_{i,0}^{k} = \bar{M}_{i,\half}^{k-\half}\mu_{i,1}^{k}, \quad \bar{M}_{i,N_y+\half}^{k-\half}\mu_{i,N_y}^{k} = \bar{M}_{i,N_y+\half}^{k-\half}\mu_{i,N_y+1}^{k}, \forall i\in \{1,2,\cdots,N_x\}, k = n,n+1;\\
				\bar{M}_{\half,j}^{k-\half}\mu_{0,j}^{k} = \bar{M}_{\half,j}^{k-\half}\mu_{1,j}^{k}, \quad \bar{M}_{N_x+\half,j}^{k-\half}\mu_{N_x,j}^{k} = \bar{M}_{N_x+\half,j}^{k-\half}\mu_{N_x+1,j}^{k}, \forall j\in \{1,2,\cdots,N_y\}, k = n,n+1.
			\end{split}
		\end{equation}
        \begin{rem}
            To facilitate the numerical algorithm, we need to incorporate a boundary condition of $\psi$. Since $\psi$ is almost zero near $\partial\Omega$, we set $\bn \cdot \nabla \psi = 0$ at $\partial \Omega$, which can be discretized in the same way as $\bn \cdot \nabla \tilde\phi = 0$. 
        \end{rem}
	\end{alg}
	
	\begin{thm}\label{Thm1}
		If $h_3 = 0$, i.e., there is no flux pumped into the system, the solution of algorithm \ref{Algorithm} satisfies the discrete volume conservation law
		\begin{equation}
			(\psi\phi^{n+1},1) = (\psi\phi^{n},1) = (\psi\phi^{0},1), \quad n = 1,2,\cdots.
		\end{equation}
	\end{thm}
	\begin{proof} We briefly prove this theorem. Firstly, taking the inner product of \eqref{EQ scheme1} with $\psi$, we get
		\begin{equation}
			(\psi\frac{\tilde{\phi}^{n+1}-\tilde{\phi}^n}{\Delta t},1) = (1,  D_h(\psi\bar{M}^{{n+\frac{1}{2}}}, \chi\mu_{*}^{{n+\frac{1}{2}}})).
		\end{equation}
		Expanding the right-hand side of the above equation, we have
		\begin{equation} \label{eqn1}
			\begin{split}
				&(1, D_h(\psi\bar{M}^{{n+\frac{1}{2}}}, \chi\mu_{*}^{{n+\frac{1}{2}}})) \\
				= & \sum_{i = 1}^{N_x}\sum_{j = 1}^{N_y}\left[D_x(A_x (\psi\bar{M}^{{n+\frac{1}{2}}}) D_x (\chi\mu_{*}^{{n+\frac{1}{2}}}))_{i,j}+ D_y(A_y (\psi\bar{M}^{{n+\frac{1}{2}}}) D_y (\chi\mu_{*}^{{n+\frac{1}{2}}}))_{i,j}\right]\\
				= & \sum_{j = 1}^{N_y}\frac{\left[A_x (\psi\bar{M}^{{n+\frac{1}{2}}}) D_x (\chi\mu_{*}^{{n+\frac{1}{2}}})\right]_{N_x+\half,j}-\left[A_x (\psi\bar{M}^{{n+\frac{1}{2}}}) D_x (\chi\mu_{*}^{{n+\frac{1}{2}}})\right]_{\half,j}}{\Delta x} \\
				+&\sum_{i = 1}^{N_x}\frac{\left[A_y (\psi\bar{M}^{{n+\frac{1}{2}}}) D_y (\chi\mu_{*}^{{n+\frac{1}{2}}})\right]_{i,N_y+\half}-\left[A_y (\psi\bar{M}^{{n+\frac{1}{2}}}) D_y (\chi\mu_{*}^{{n+\frac{1}{2}}})\right]_{i,\half}}{\Delta y}.
			\end{split}
		\end{equation}
		Using the boundary conditions \eqref{BC DIS}, we find that the right-hand side of equation \eqref{eqn1} is equal to 0, which means
		\begin{equation} \label{eqn2}
			(\psi\frac{\tilde{\phi}^{n+1}-\tilde{\phi}^n}{\Delta t},1) = 0 \implies (\psi \tilde{\phi}^{n+1},1) = (\psi \tilde{\phi}^{n},1), \quad \forall \; n = 0, 1, 2 \cdots.
		\end{equation}
		This gives the discrete volume conservation law
		\begin{equation}
			(\psi \tilde{\phi}^{n+1},1) = (\psi \tilde{\phi}^{n},1) = \cdots = (\psi \tilde{\phi}^{0},1).
		\end{equation}
	\end{proof}
	
	\begin{thm}\label{Thm2}
		The solution of algorithm \ref{Algorithm} satisfies the discrete energy law
		\begin{small}
			\begin{multline}\label{DIS EDL}
				\frac{\tilde{F}^{n+1}-\tilde{F}^n}{\Delta t} = - \left[(D_x(\chi\mu_*^{n+\half}), A_x(\psi\bar{M}^{n+\half})D_x(\chi\mu_*^{n+\half})) + (D_y(\chi\mu_*^{n+\half}),A_y(\psi\bar{M}^{n+\half})D_y(\chi\mu_*^{n+\half}))\right] \\
				- \sum_{i = 1}^{N_x}\sum_{j=1}^{N_y} \half \left\{|D_x\psi_{i-\half,j}|\left[\overline{\Gamma^{-1}}^{n+\half}(A_x\frac{\tilde{\phi}^{n+1} - \tilde{\phi}^n}{\Delta t})^2\right]_{i-\half,j}
				+|D_x\psi_{i+\half,j}|\left[\overline{\Gamma^{-1}}^{n+\half}(A_x\frac{\tilde{\phi}^{n+1} - \tilde{\phi}^n}{\Delta t})^2\right]_{i+\half,j}\right.\\
				\left.+|D_y\psi_{i,j-\half}|\left[\overline{\Gamma^{-1}}^{n+\half}(A_y\frac{\tilde{\phi}^{n+1} - \tilde{\phi}^n}{\Delta t})^2\right]_{i,j-\half} +|D_y\psi_{i,j+\half}|\left[\overline{\Gamma^{-1}}^{n+\half}(A_y\frac{\tilde{\phi}^{n+1} - \tilde{\phi}^n}{\Delta t})^2\right]_{i,j+\half}\right\} -\sum_{i = 1}^{N_x}\sum_{j = 1}^{N_y} \\
				\half \mu_{*,j}^{n+\half}\chi_{i,j} \left\{|D_x\psi_{i-\half,j}| h_{3,i-\half,j}+|D_x\psi_{i+\half,j}| h_{3,i+\half,j} +|D_y\psi_{i,j-\half}| h_{3,i,j-\half} +|D_y\psi_{i,j+\half}| h_{3,i,j+\half}\right\}.
			\end{multline}
		\end{small}
		for any $\Delta t >0$, where the discrete energy is defined as
		\begin{multline}
			\tilde{F}^n =\sum_{i = 1}^{N_x}\sum_{j=1}^{N_y}\left\{\psi_{i,j}|q_{i,j}^n|^2 + \half K \left[A_x \psi_{i+\half,j} (\frac{\tilde{\phi}_{i+1,j}^{n}-\tilde{\phi}_{i,j}^{n}}{\Delta x})^2 + A_y \psi_{i,j+\half} (\frac{\tilde{\phi}_{i,j+1}^{n+1}-\tilde{\phi}_{i,j}^{n+1}}{\Delta y})^2\right]\right. \\
			+ \half \left\{|D_x\psi_{i-\half,j}|[\half\alpha (A_x\tilde{\phi}^{n} - h_{1})^2 - h_{2}A_x\tilde{\phi}^n]_{i-\half,j} +|D_x\psi_{i+\half,j}|[\half\alpha (A_x\tilde{\phi}^{n} - h_{1})^2 - h_{2}A_x\tilde{\phi}^n]_{i+\half,j}\right.\\
			\left.\left. + |D_y\psi_{i,j-\half}|[\half\alpha (A_y\tilde{\phi}^{n} - h_{1})^2 - h_{2}A_y\tilde{\phi}^n]_{i,j-\half} +|D_y\psi_{i,j+\half}|[\half\alpha (A_y\tilde{\phi}^{n} - h_{1})^2 - h_{2}A_y\tilde{\phi}^n]_{i,j+\half}\right\}\right\}.
		\end{multline}
		Therefore, the scheme is unconditionally energy stable.
	\end{thm}
	
	\begin{proof} Firstly, taking the inner products of \eqref{EQ scheme1} with $\mu_*^{n+{\frac{1}{2}}}$ and also \eqref{EQ scheme3} with $ \frac{1}{\Delta t}\psi q^{n+{\frac{1}{2}}}$, we get
		\begin{equation}
			\begin{split}
				&(\mu_*^{^{n+\frac{1}{2}}},\frac{\tilde{\phi}^{n+1}-\tilde{\phi}^n}{\Delta t}) = (\mu_*^{^{n+\frac{1}{2}}}, \chi D_x[A_x(\psi\bar{M}^{^{n+\frac{1}{2}}})D_x(\chi\mu_*^{^{n+\frac{1}{2}}})]) \\
				&+ (\mu_*^{^{n+\frac{1}{2}}}, \chi D_y[A_y(\psi\bar{M}^{^{n+\frac{1}{2}}})D_y(\chi\mu_*^{^{n+\frac{1}{2}}})])- \sum_{i = 1}^{N_x}\sum_{j = 1}^{N_y}\half \mu_{*,j}^{n+\half}\chi_{i,j} \left\{|D_x\psi_{i-\half,j}|h_{3,i-\half,j} \right.\\
				&\left. +|D_x\psi_{i+\half,j}| h_{3,i+\half,j}+|D_y\psi_{i,j-\half}|h_{3,i,j-\half} +|D_y\psi_{i,j+\half}|h_{3,i,j+\half}\right\},\\
				&(\psi q^{n+{\frac{1}{2}}},\frac{q^{n+1}-q^n}{\Delta t}) = (\psi q^{n+{\frac{1}{2}}}, \bar{g }^{^{n+\frac{1}{2}}}\frac{\tilde{\phi}^{n+1}-\tilde{\phi}^n}{\Delta t}).
			\end{split}
		\end{equation}
		Adding up the above two equations and substituting \eqref{EQ scheme2} into the sum, we obtain
		\begin{multline}
			(\psi q^{n+{\frac{1}{2}}},\frac{q^{n+1}-q^n}{\Delta t})+(- K D_h(\psi,\tilde{\phi}^{^{n+\frac{1}{2}}}),\frac{\tilde{\phi}^{n+1}-\tilde{\phi}^n}{\Delta t})+\sum_{i = 1}^{N_x}\sum_{j = 1}^{N_y}\half \left \{|D_x\psi_{i-\half,j}\right.\\
			A_x[\alpha(\tilde{\phi}^{n+\half}-h_{1})-h_2+\overline{\Gamma^{-1}}^{n+\half}\frac{\tilde{\phi}^{n+1}-\tilde{\phi}^n}{\Delta t}]_{i-\half,j} +|D_x\psi_{i+\half,j}|A_x[\alpha(\tilde{\phi}^{n+\half}-h_{1})-h_2\\
			+ \overline{\Gamma^{-1}}^{n+\half}\frac{\tilde{\phi}^{n+1}-\tilde{\phi}^n}{\Delta t}]_{i+\half,j}+ |D_y\psi_{i,j-\half}|A_y[\alpha(\tilde{\phi}^{n+\half}-h_{1})-h_2+\overline{\Gamma^{-1}}^{n+\half}\frac{\tilde{\phi}^{n+1}-\tilde{\phi}^n}{\Delta t}]_{i,j-\half}\\
			\left.+ |D_y\psi_{i,j+\half}|A_y[\alpha(\tilde{\phi}^{n+\half}-h_{1})- h_2+\overline{\Gamma^{-1}}^{n+\half}\frac{\tilde{\phi}^{n+1}-\tilde{\phi}^n}{\Delta t}]_{i,j+\half}\right\}\frac{\tilde{\phi}_{i,j}^{n+1}-\tilde{\phi}_{i,j}^n}{\Delta t} \\
			= (\mu_*^{^{n+\frac{1}{2}}}, \chi D_x[A_x(\psi\bar{M}^{^{n+\frac{1}{2}}})D_x(\chi\mu_*^{^{n+\frac{1}{2}}})]) + (\mu_*^{^{n+\frac{1}{2}}}, \chi D_y[A_y(\psi\bar{M}^{^{n+\frac{1}{2}}})D_y(\chi\mu_*^{^{n+\frac{1}{2}}})]) -\sum_{i = 1}^{N_x}\sum_{j = 1}^{N_y}\\
			\half \mu_{*,j}^{n+\half}\chi_{i,j}\left\{|D_x\psi_{i-\half,j}|h_{3,i-\half,j}
			+|D_x\psi_{i+\half,j}|h_{3,i+\half,j}+|D_y\psi_{i,j-\half}|h_{3,i,j-\half} +|D_y\psi_{i,j+\half}|h_{3,i,j+\half}\right\}.
		\end{multline}
		With the help of the boundary conditions \eqref{BC DIS}, we arrive at the following equations :
		\begin{small}
			\begin{multline} \label{eqn3}
				\sum_{i = 1}^{N_x}\sum_{j = 1}^{N_y}\frac{\psi_{i,j} |q_{i,j}^{n+1}|^2-\psi_{i,j} |q_{i,j}^n|^2}{\Delta t}+\sum_{i = 1}^{N_x}\sum_{j = 1}^{N_y} \half K \left[A_x \psi_{i+\half,j} \frac{(\frac{\tilde{\phi}_{i+1,j}^{n+1}-\tilde{\phi}_{i,j}^{n+1}}{\Delta x})^2-(\frac{\tilde{\phi}_{i+1,j}^{n}-\tilde{\phi}_{i,j}^{n}}{\Delta x})^2}{\Delta t} \right. \\
				\left.+ A_y \psi_{i,j+\half} \frac{(\frac{\tilde{\phi}_{i,j+1}^{n+1}-\tilde{\phi}_{i,j}^{n+1}}{\Delta y})^2-(\frac{\tilde{\phi}_{i,j+1}^{n+1}-\tilde{\phi}_{i,j}^{n+1}}{\Delta y})^2 }{\Delta t}\right]+ \sum_{i = 1}^{N_x}\sum_{j=1}^{N_y} \half \left\{|D_x\psi_{i-\half,j}|[\half\alpha\frac{(A_x\tilde{\phi}^{n+1} - h_{1})^2 - (A_x\tilde{\phi}^{n} - h_{1})^2 }{\Delta t} \right.\\
				- \frac{h_2A_x(\tilde{\phi}^{n+1}-\tilde{\phi}^n)}{\Delta t} +\overline{\Gamma^{-1}}^{n+\half}(A_x\frac{\tilde{\phi}^{n+1} - \tilde{\phi}^n}{\Delta t})^2]_{i-\half,j} + |D_x\psi_{i+\half,j}|[\half\alpha\frac{(A_x\tilde{\phi}^{n+1} - h_{1})^2 - (A_x\tilde{\phi}^{n} - h_{1})^2 }{\Delta t} \\
				- \frac{h_2A_x(\tilde{\phi}^{n+1}-\tilde{\phi}^n)}{\Delta t} +\overline{\Gamma^{-1}}^{n+\half}(A_x\frac{\tilde{\phi}^{n+1} - \tilde{\phi}^n}{\Delta t})^2]_{i+\half,j} + |D_x\psi_{i,j-\half}|[\half\alpha\frac{(A_y\tilde{\phi}^{n+1} - h_{1})^2 - (A_y\tilde{\phi}^{n} - h_{1})^2 }{\Delta t} \\
				- \frac{h_2A_y(\tilde{\phi}^{n+1}-\tilde{\phi}^n)}{\Delta t} +\overline{\Gamma^{-1}}^{n+\half}(A_y\frac{\tilde{\phi}^{n+1} - \tilde{\phi}^n}{\Delta t})^2]_{i,j-\half} + |D_y\psi_{i,j+\half}|[\half\alpha\frac{(A_y\tilde{\phi}^{n+1} - h_{1})^2 - (A_y\tilde{\phi}^{n} - h_{1})^2 }{\Delta t} \\
				\left.- \frac{h_2A_y(\tilde{\phi}^{n+1}-\tilde{\phi}^n)}{\Delta t} +\overline{\Gamma^{-1}}^{n+\half}(A_y\frac{\tilde{\phi}^{n+1} - \tilde{\phi}^n}{\Delta t})^2]_{i,j+\half}\right\} = -\left[(D_x(\chi\mu_*^{n+\half}),A_x(\psi\bar{M}^{n+\half})D_x(\chi\mu_*^{n+\half})) \right.\\
				\left.+ (D_y(\chi\mu_*^{n+\half}),A_y(\psi\bar{M}^{n+\half})D_y(\chi\mu_*^{n+\half}))\right] - \sum_{i = 1}^{N_x}\sum_{j = 1}^{N_y}\half \mu_{*,j}^{n+\half}\chi_{i,j}
				\left\{|D_x\psi_{i-\half,j}|h_{3,i-\half,j} \right. \\
				\left.+|D_x\psi_{i+\half,j}|h_{3,i+\half,j}+|D_y\psi_{i,j-\half}|h_{3,i,j-\half} +|D_y\psi_{i,j+\half}|h_{3,i,j+\half}\right\}.
			\end{multline}
		\end{small}
		Defining
		\begin{multline}
			\tilde{F}^n =\sum_{i = 1}^{N_x}\sum_{j=1}^{N_y}\left\{\psi_{i,j}|q_{i,j}^n|^2 + \half K \left[A_x \psi_{i+\half,j} (\frac{\tilde{\phi}_{i+1,j}^{n}-\tilde{\phi}_{i,j}^{n}}{\Delta x})^2 + A_y \psi_{i,j+\half} (\frac{\tilde{\phi}_{i,j+1}^{n+1}-\tilde{\phi}_{i,j}^{n+1}}{\Delta y})^2\right]\right. \\
			+ \half \left\{|D_x\psi_{i-\half,j}|[\half\alpha (A_x\tilde{\phi}^{n} - h_{1})^2 - h_{2}A_x\tilde{\phi}^n]_{i-\half,j} +|D_x\psi_{i+\half,j}|[\half\alpha (A_x\tilde{\phi}^{n} - h_{1})^2 - h_{2}A_x\tilde{\phi}^n]_{i+\half,j}\right.\\
			\left.\left. + |D_y\psi_{i,j-\half}|[\half\alpha (A_y\tilde{\phi}^{n} - h_{1})^2 - h_{2}A_y\tilde{\phi}^n]_{i,j-\half} +|D_y\psi_{i,j+\half}|[\half\alpha (A_y\tilde{\phi}^{n} - h_{1})^2 - h_{2}A_y\tilde{\phi}^n]_{i,j+\half}\right\}\right\}
		\end{multline}
		and substituting $\tilde{F}^n$ into \eqref{eqn3}, we obtain
		\begin{small}
			\begin{multline}
				\frac{\tilde{F}^{n+1}-\tilde{F}^n}{\Delta t} = - \left[(D_x(\chi\mu_*^{n+\half}), A_x(\psi\bar{M}^{n+\half})D_x(\chi\mu_*^{n+\half})) + (D_y(\chi\mu_*^{n+\half}),A_y(\psi\bar{M}^{n+\half})D_y(\chi\mu_*^{n+\half}))\right] \\
				- \sum_{i = 1}^{N_x}\sum_{j=1}^{N_y} \half \left\{|D_x\psi_{i-\half,j}|\left[\overline{\Gamma^{-1}}^{n+\half}(A_x\frac{\tilde{\phi}^{n+1} - \tilde{\phi}^n}{\Delta t})^2\right]_{i-\half,j}
				+|D_x\psi_{i+\half,j}|\left[\overline{\Gamma^{-1}}^{n+\half}(A_x\frac{\tilde{\phi}^{n+1} - \tilde{\phi}^n}{\Delta t})^2\right]_{i+\half,j}\right.\\
				\left.+|D_y\psi_{i,j-\half}|\left[\overline{\Gamma^{-1}}^{n+\half}(A_y\frac{\tilde{\phi}^{n+1} - \tilde{\phi}^n}{\Delta t})^2\right]_{i,j-\half} +|D_y\psi_{i,j+\half}|\left[\overline{\Gamma^{-1}}^{n+\half}(A_y\frac{\tilde{\phi}^{n+1} - \tilde{\phi}^n}{\Delta t})^2\right]_{i,j+\half}\right\} -\sum_{i = 1}^{N_x}\sum_{j = 1}^{N_y} \\
				\half \mu_{*,j}^{n+\half}\chi_{i,j} \left\{|D_x\psi_{i-\half,j}| h_{3,i-\half,j}+|D_x\psi_{i+\half,j}| h_{3,i+\half,j} +|D_y\psi_{i,j-\half}| h_{3,i,j-\half} +|D_y\psi_{i,j+\half}| h_{3,i,j+\half}\right\}.
			\end{multline}
		\end{small}
		which is the discrete energy law \eqref{DIS EDL}.
	\end{proof}
	
	\section{Results and discussion}
	In this section, we simulate the OPBDE Cahn-Hilliard model \eqref{OPBDE CH} and make comparisons with the results from the original Cahn-Hilliard model in $\Omega_1$ and also those from the extended Cahn-Hilliard model using the DDM method (named the DDM Cahn-Hilliard model) in $\Omega$ \cite{Li2009SolvingPDESComplex}. Two examples are presented below: the first shows the coarsening dynamics, and the second shows the dynamics of droplet spreading. 
		
	In the actual computation, although $\psi_{\varepsilon}$ remains positive throughout the entire domain, an excessively large $\chi_{\varepsilon} = \frac{1}{\psi_{\varepsilon}}$ in $\Omega\backslash\Omega_1$ can render the computation numerically unstable. Therefore, we adopt the following form: 
	\begin{equation}
		\chi_{\varepsilon} = \frac{1}{\psi_{\varepsilon} + (1-\psi_{\varepsilon})10^{-6}}.
	\end{equation}
	This modification primarily affects volume conservation in the computation as the truely conserved volume is $\int_{\Omega}\frac{1}{\chi_{\varepsilon}}\tilde{\phi}d\bx$. However, we still present $\int_{\Omega}\psi_{\varepsilon}\tilde{\phi}d\bx$ in numerical results because it is the physically meaningful quantity.

	\subsection{Coarsening dynamics}
	The simulation is to show the coarsening dynamics governed by the Cahn-Hilliard model with the homogeneous Neumann condition $\bn_1 \cdot K\nabla \phi = 0$ and the no flux boundary condition $\bn_1 \cdot M\nabla \mu = 0$ applied at $\partial \Omega_1$ in the original model. We use $f(\phi) = \frac{1}{4}(\phi^2-1)^2$, $A = 0$, $K = 10^{-4}$, $M = 0.01$, $\Delta x = \Delta y = \frac{1}{128}$, and $\Delta t = 10^{-5}$. We simulate the extended model using $\varepsilon = 10^{-2}$ and $2\times10^{-3}$.
	
	We first simulate the coarsening dynamics of the Cahn-Hilliard model in the original domain $\Omega_1 = [-0.5,0.5]\times [-0.5,0.5]$. In the extended domain $\Omega = [-0.625,0.625]\times [-0.625,0.625]$, we carry out simulations using our OPBDE Cahn-Hilliard model \eqref{OPBDE CH} and the DDM Cahn-Hilliard model. For the original model, the initial value for $\phi$ is given by
	\begin{equation}
		\phi_{0} = 0.001 \times rand(x,y),
	\end{equation}
	where $rand(x,y)$ generates random numbers in $[-1,1]$. For the two extended models, the initial value for $\tilde{\phi}$ is
    \begin{equation}
        \tilde{\phi}_0 = \psi \times
        \begin{cases}
          \phi_0, & \mbox{if $(x,y) \in \Omega_1$},\\
          0, & \mbox{otherwise}.
        \end{cases}
    \end{equation}
	
	\begin{figure}[H]
		\centering
		\subfigure[]{
			\includegraphics[width = 1.5in]{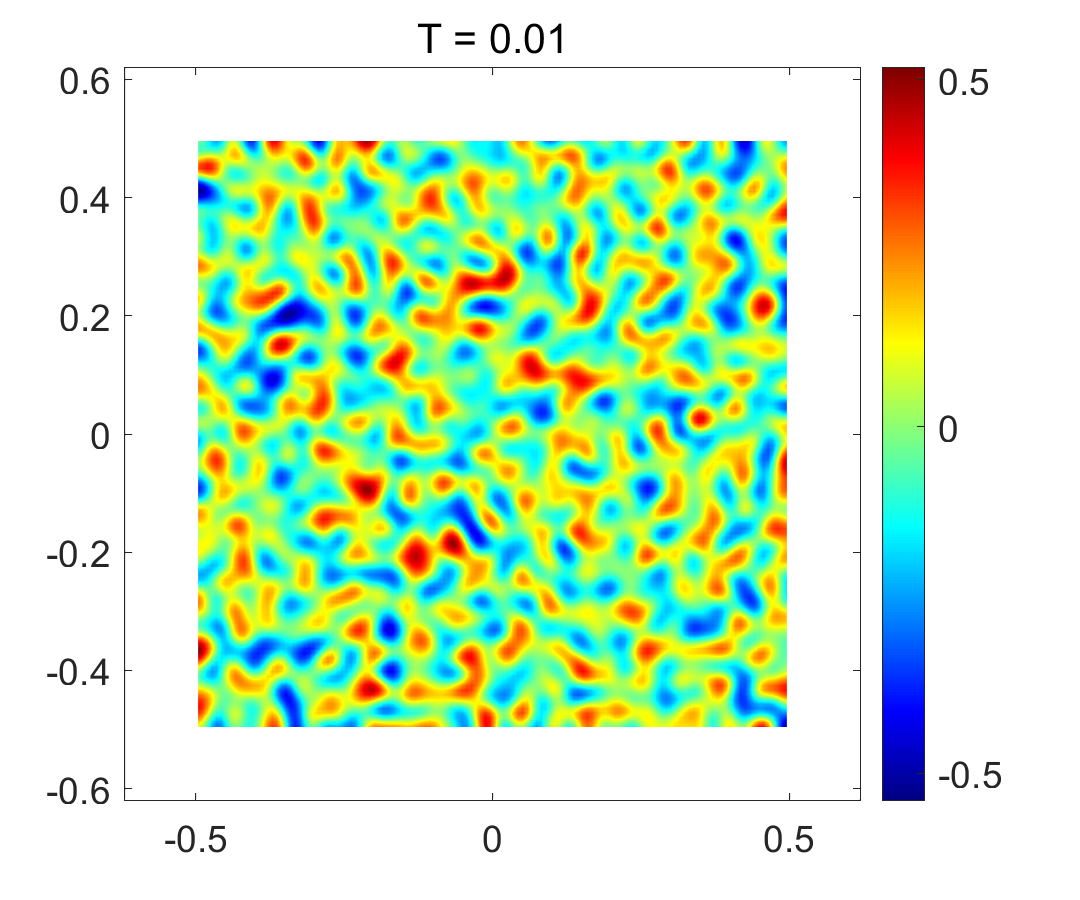}
			\includegraphics[width = 1.5in]{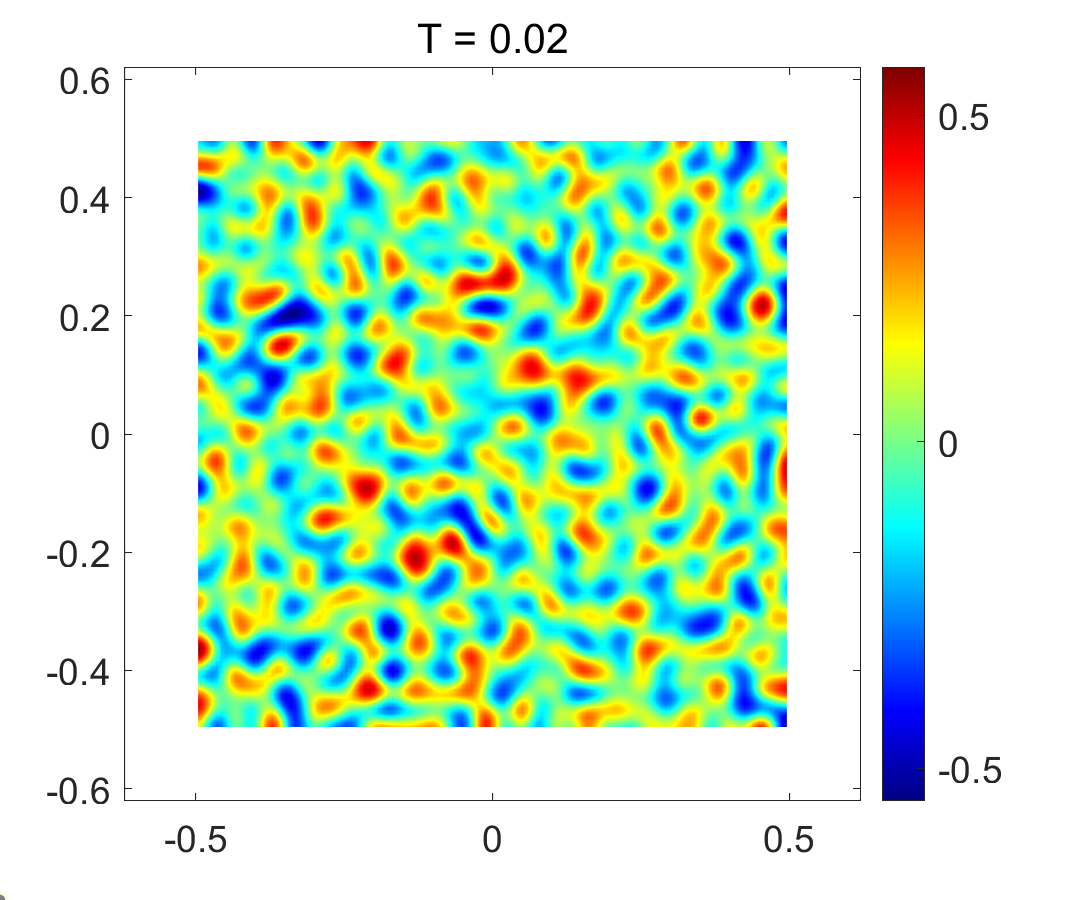}
			\includegraphics[width = 1.5in]{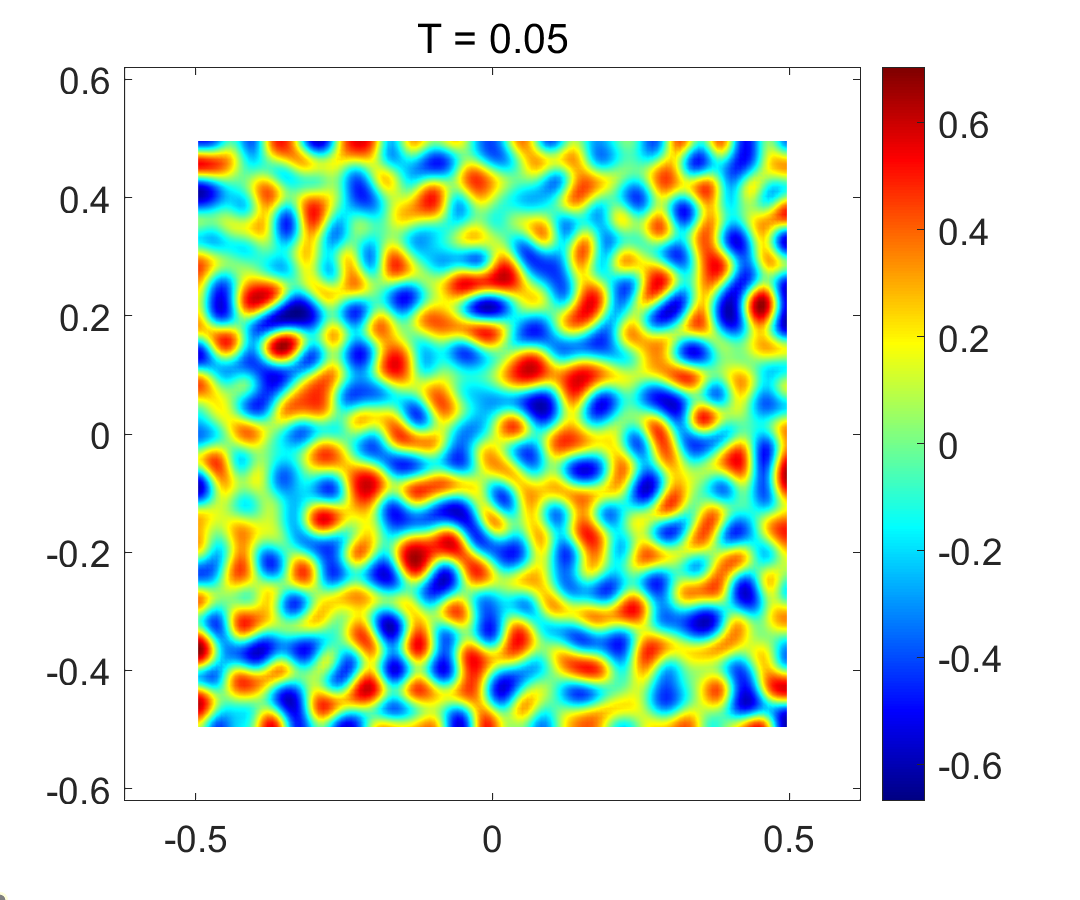}
			\includegraphics[width = 1.5in]{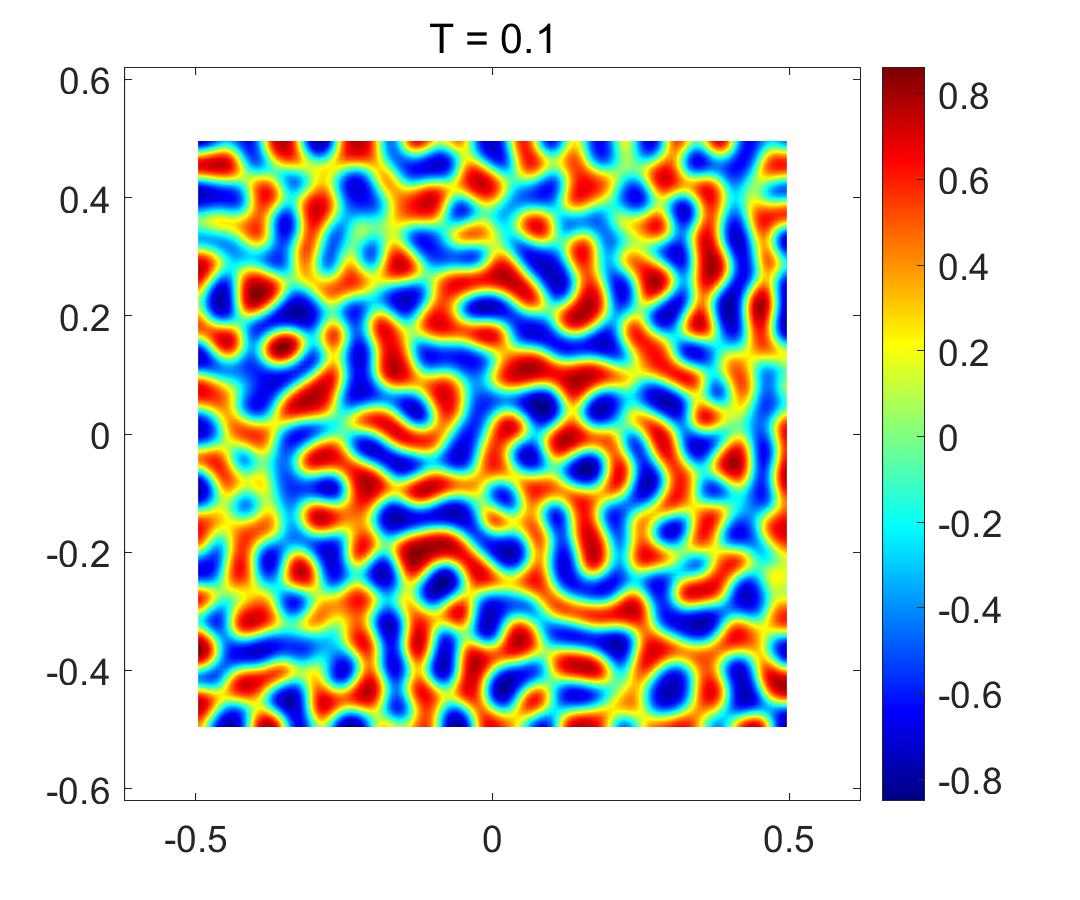}
		}
		\subfigure[]{
			\includegraphics[width = 1.5in]{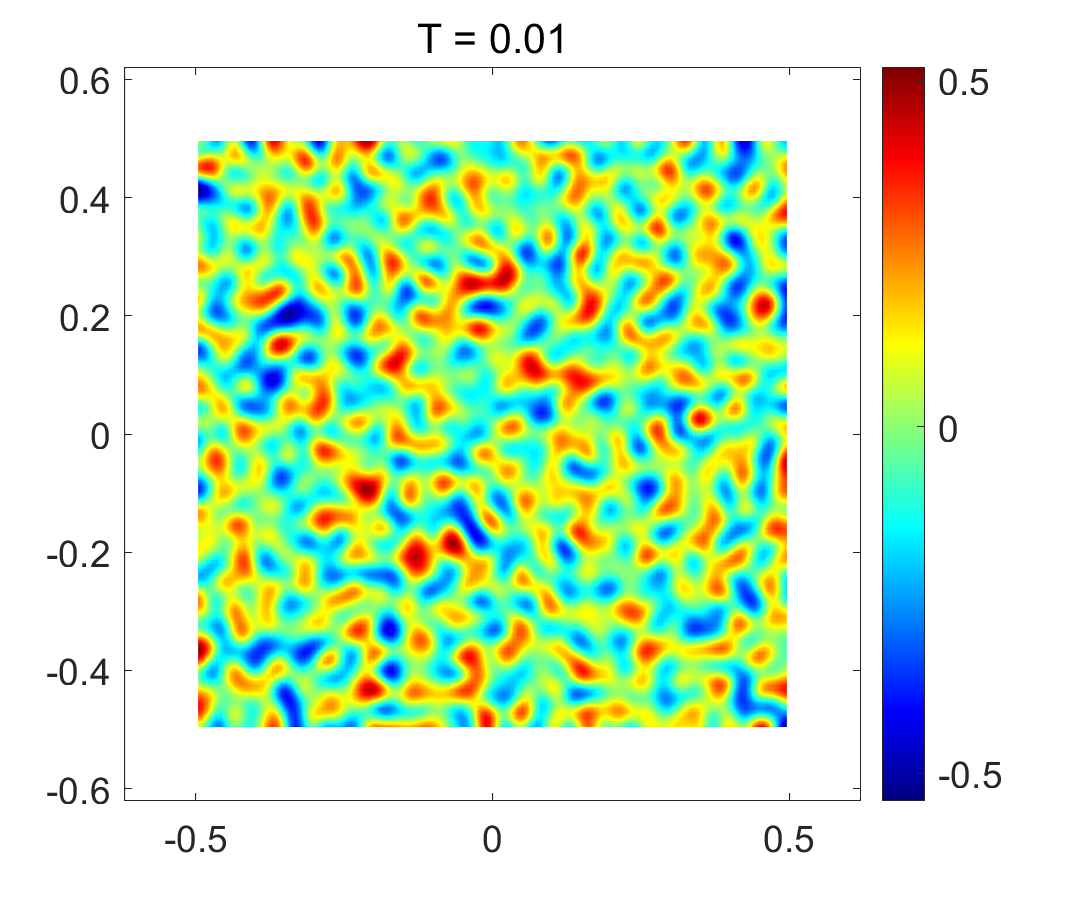}
			\includegraphics[width = 1.5in]{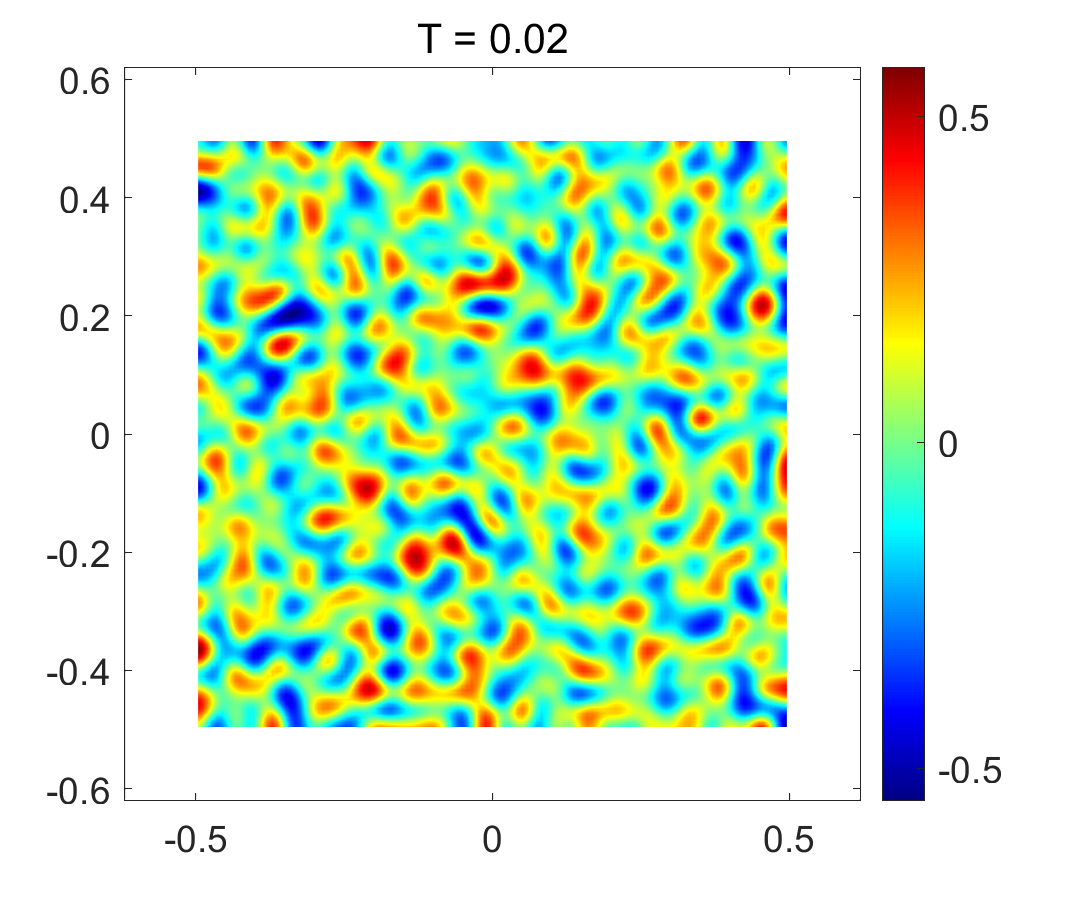}
			\includegraphics[width = 1.5in]{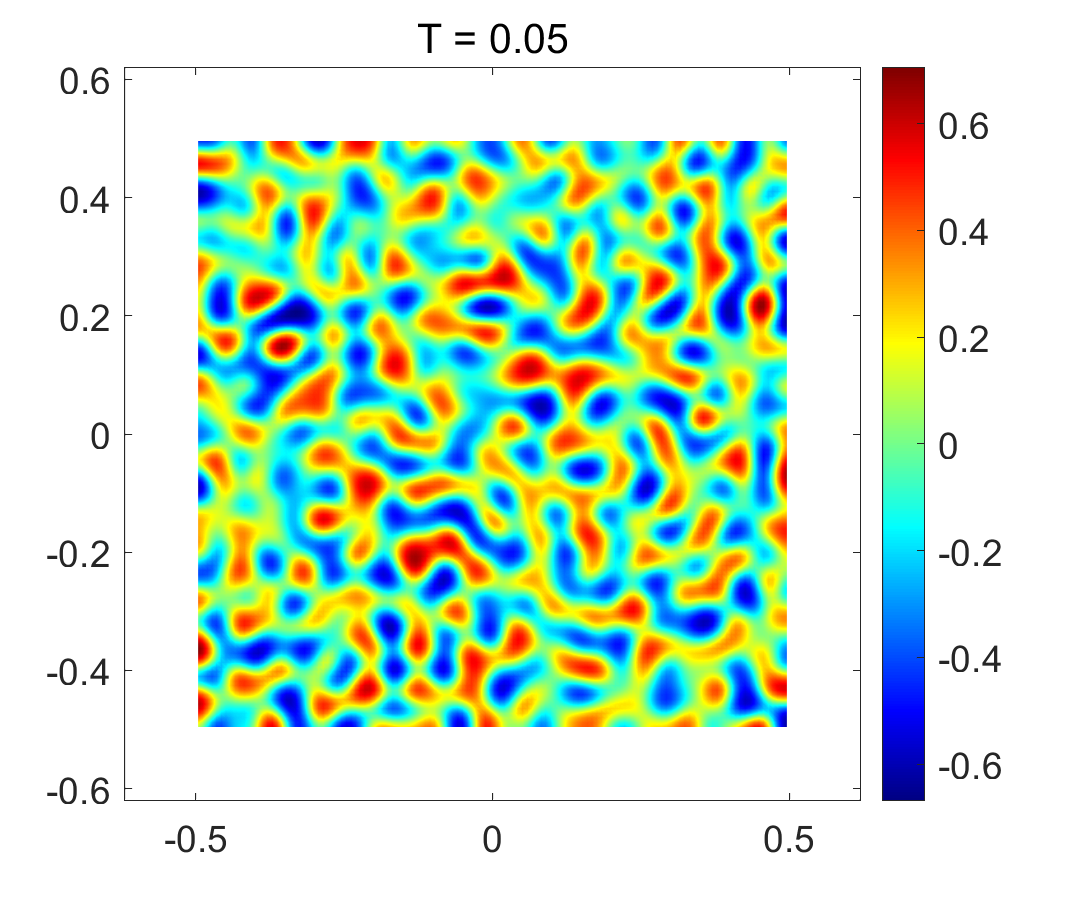}
			\includegraphics[width = 1.5in]{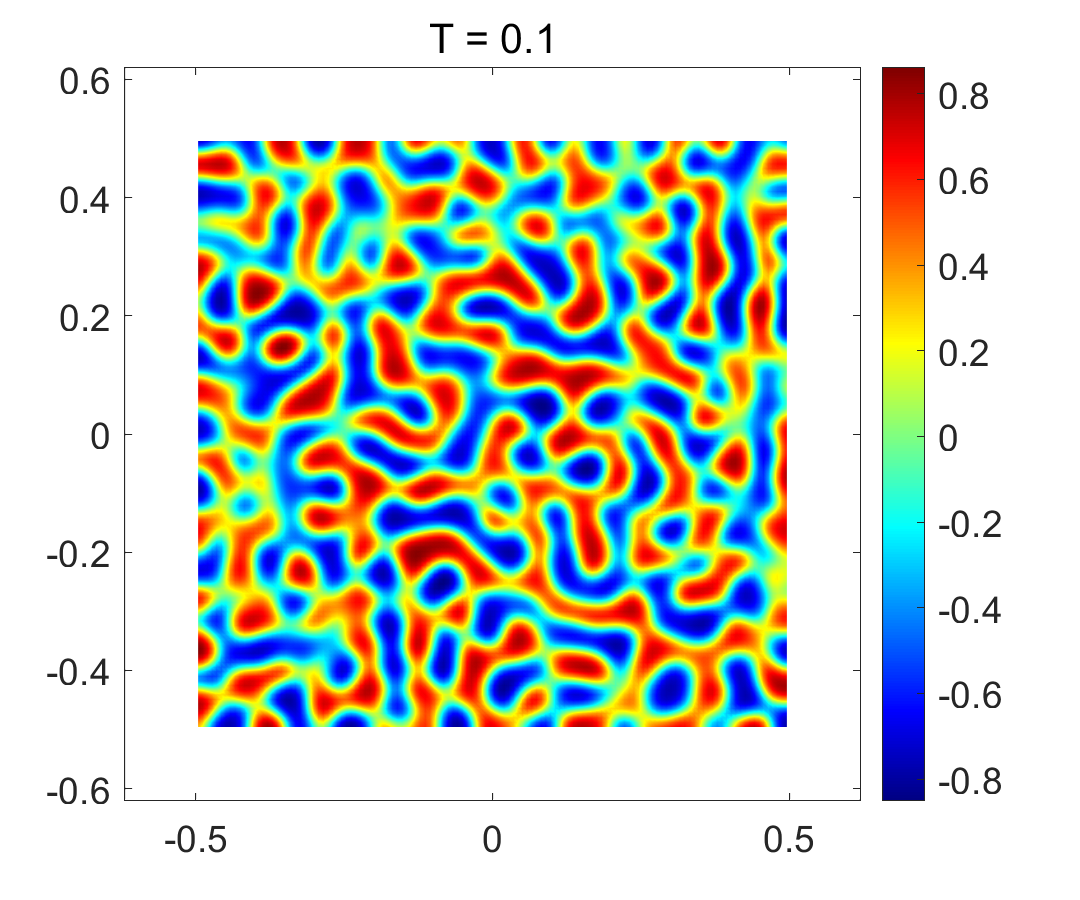}
		}
		\subfigure[]{
			\includegraphics[width = 1.5in]{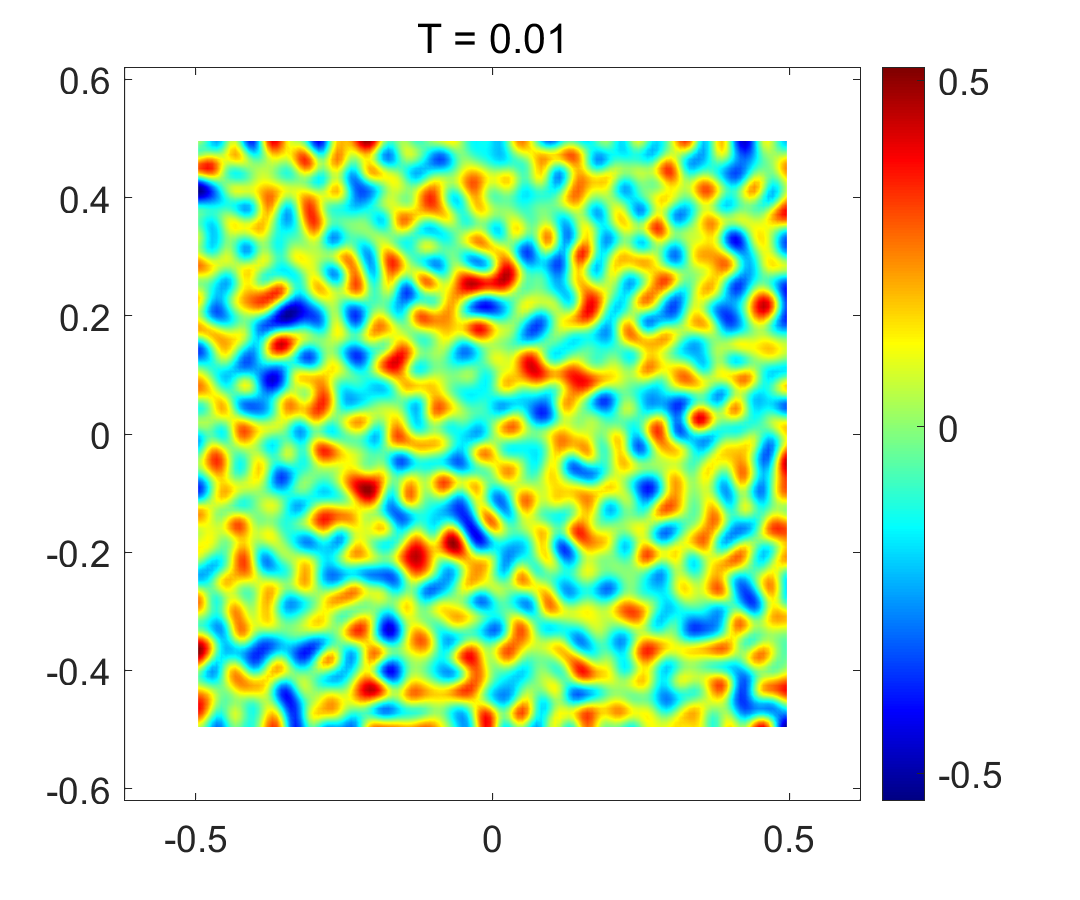}
			\includegraphics[width = 1.5in]{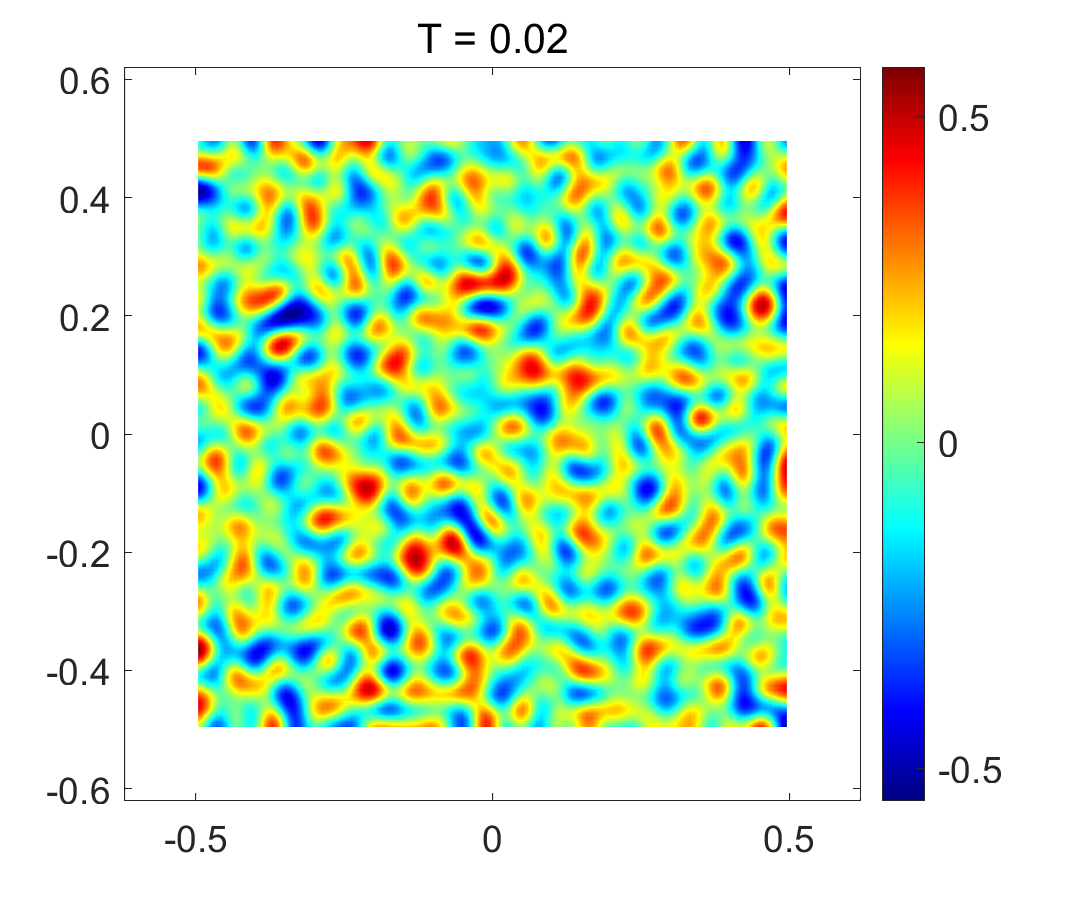}
			\includegraphics[width = 1.5in]{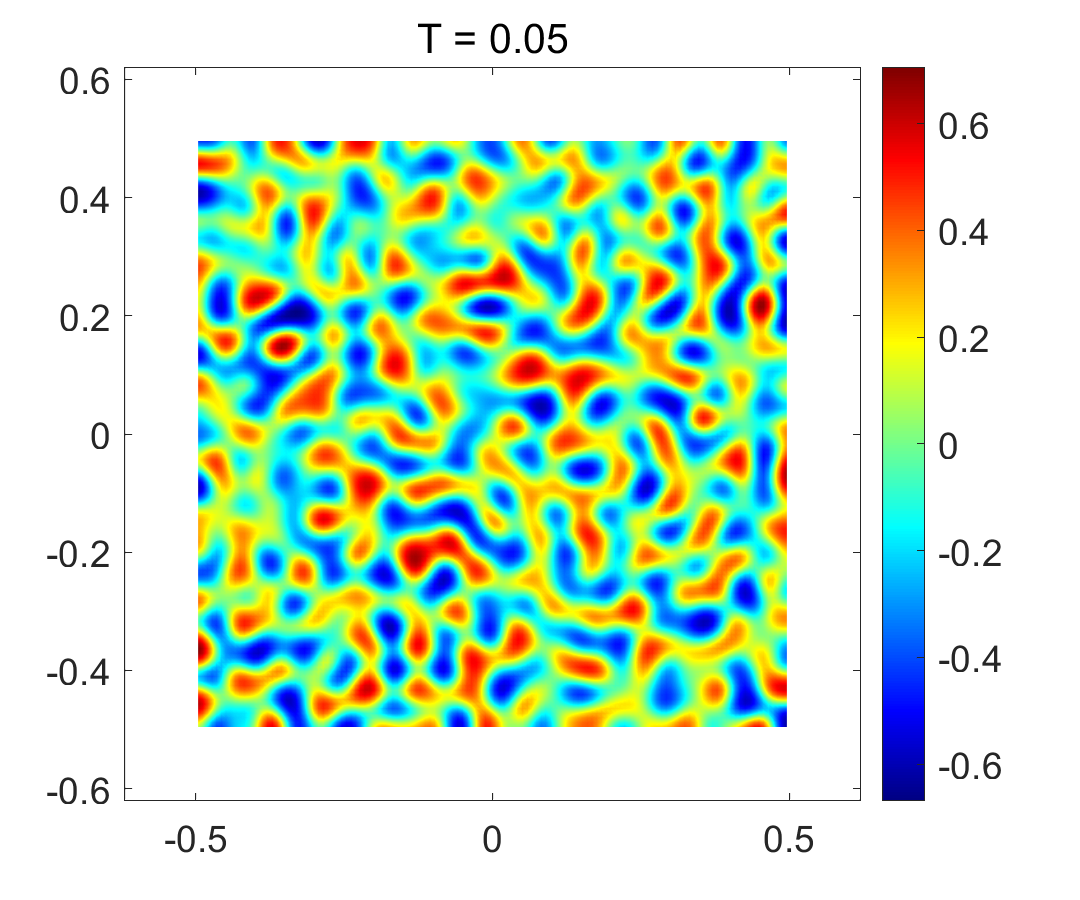}
			\includegraphics[width = 1.5in]{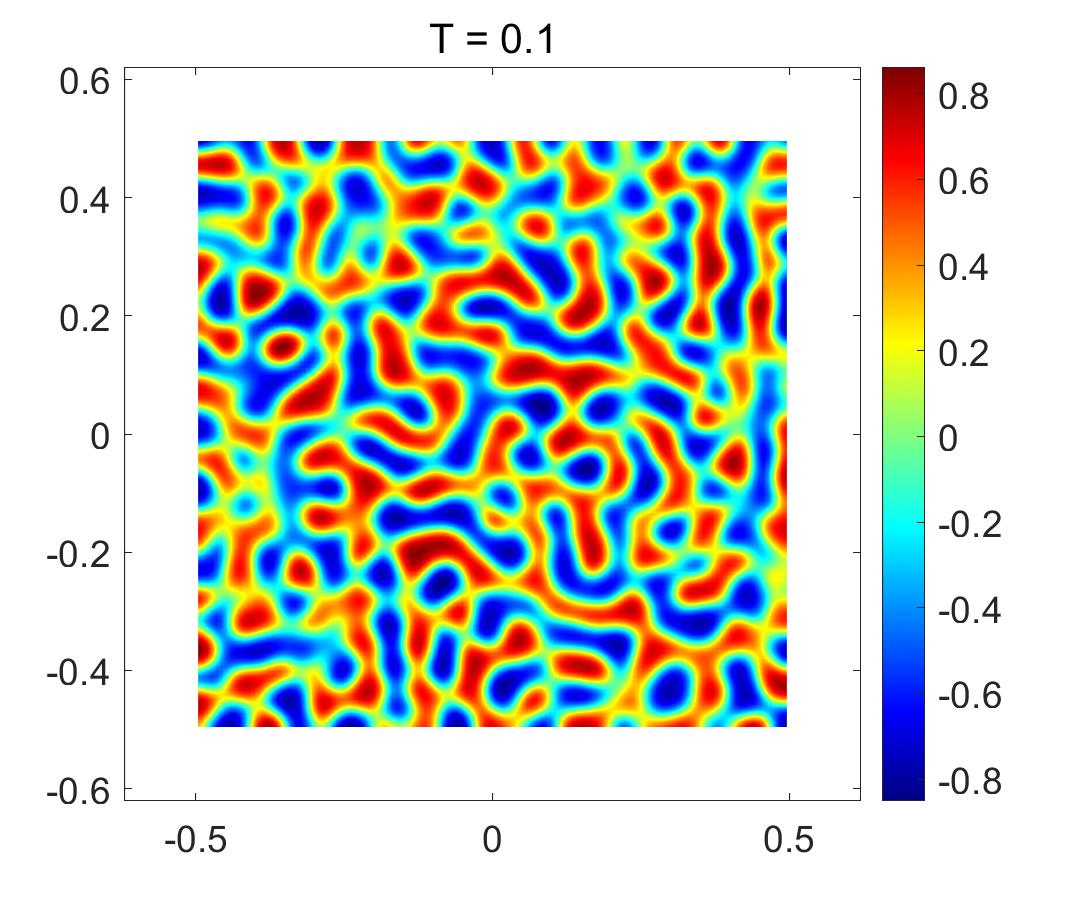}
		}
		\subfigure[]{
			\includegraphics[width = 1.5in]{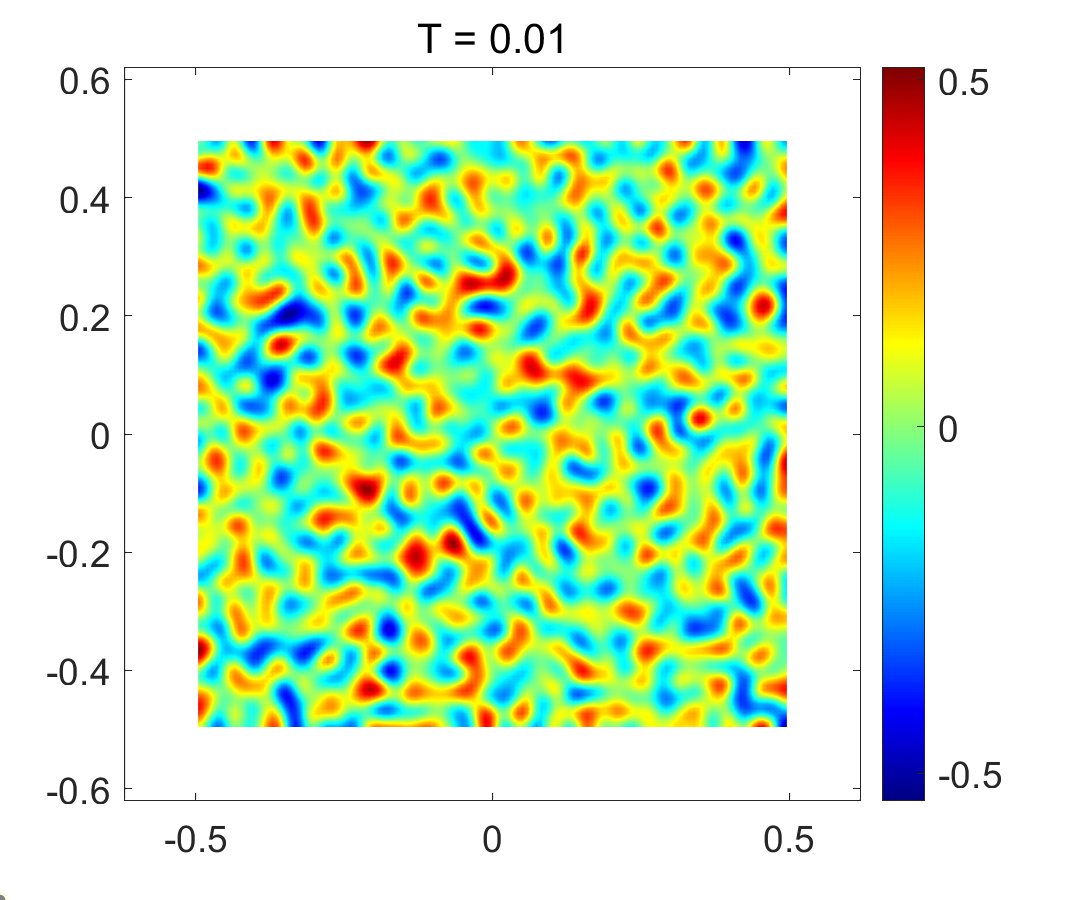}
			\includegraphics[width = 1.5in]{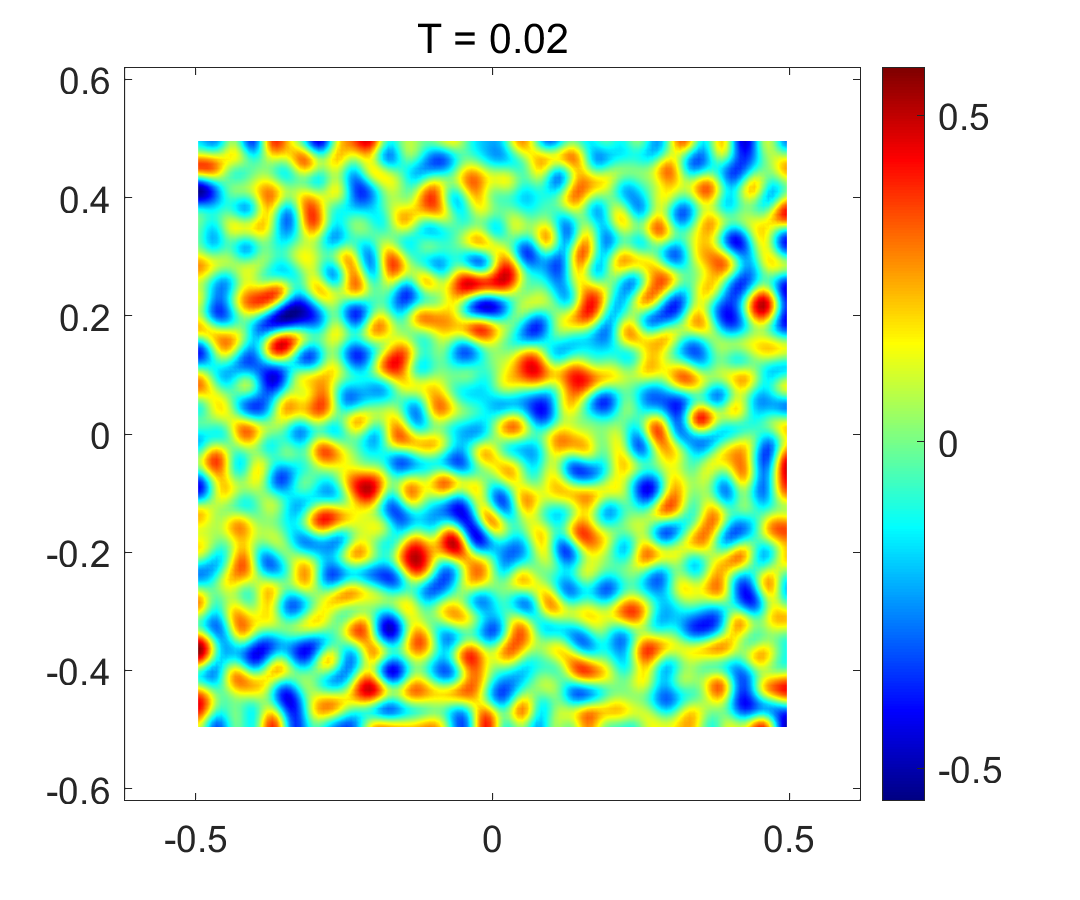}
			\includegraphics[width = 1.5in]{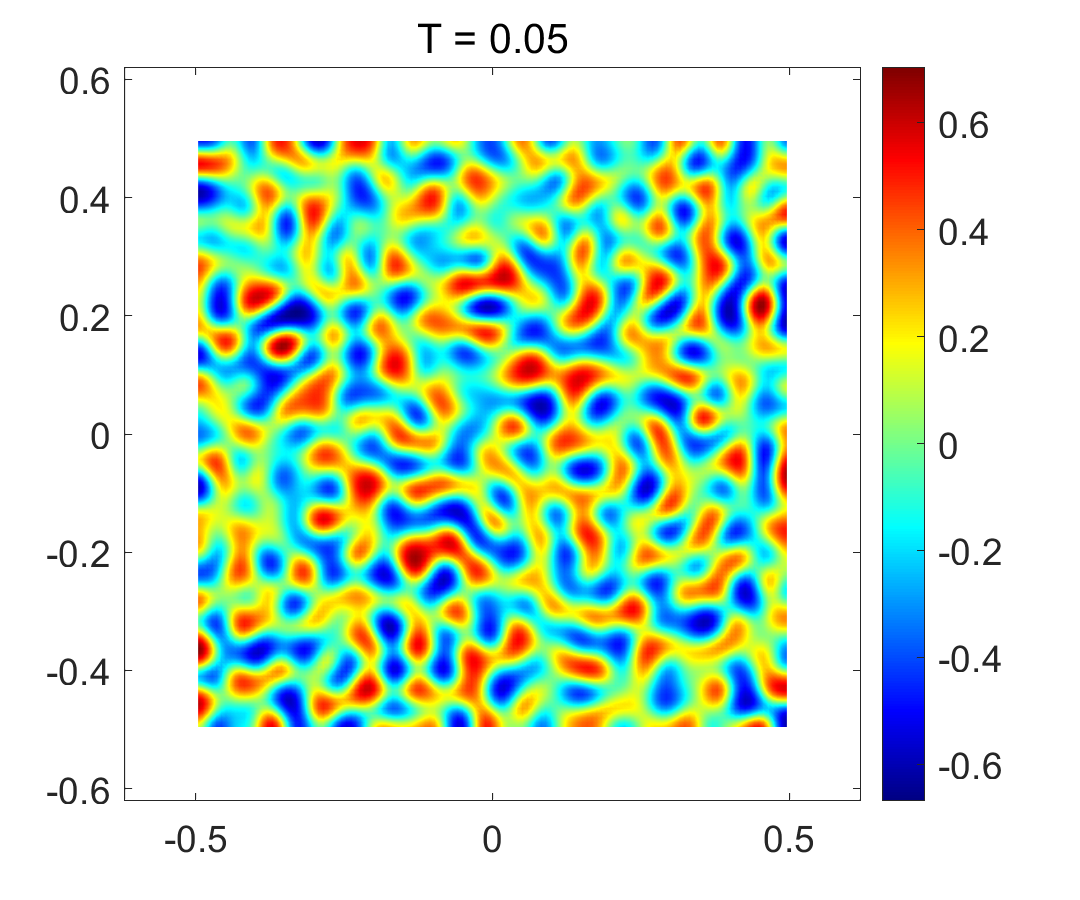}
			\includegraphics[width = 1.5in]{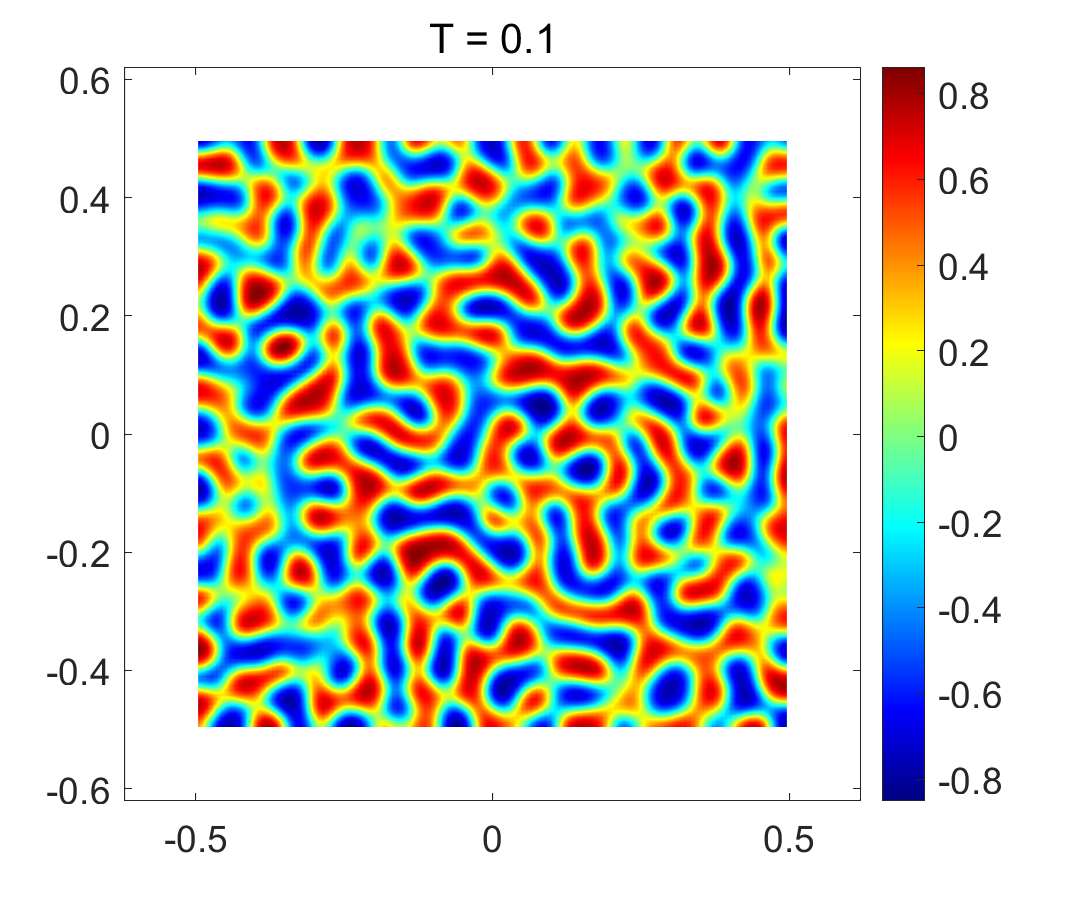}
		}
		\subfigure[]{
			\includegraphics[width = 1.5in]{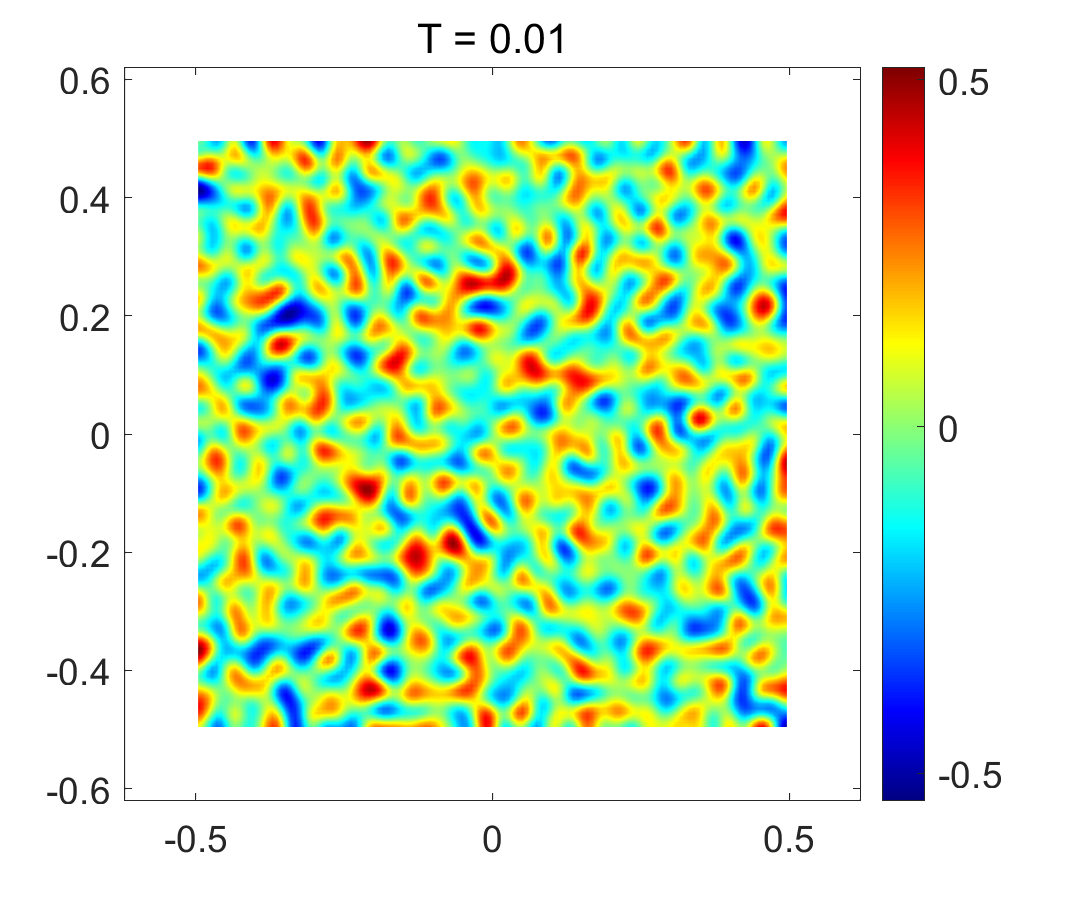}
			\includegraphics[width = 1.5in]{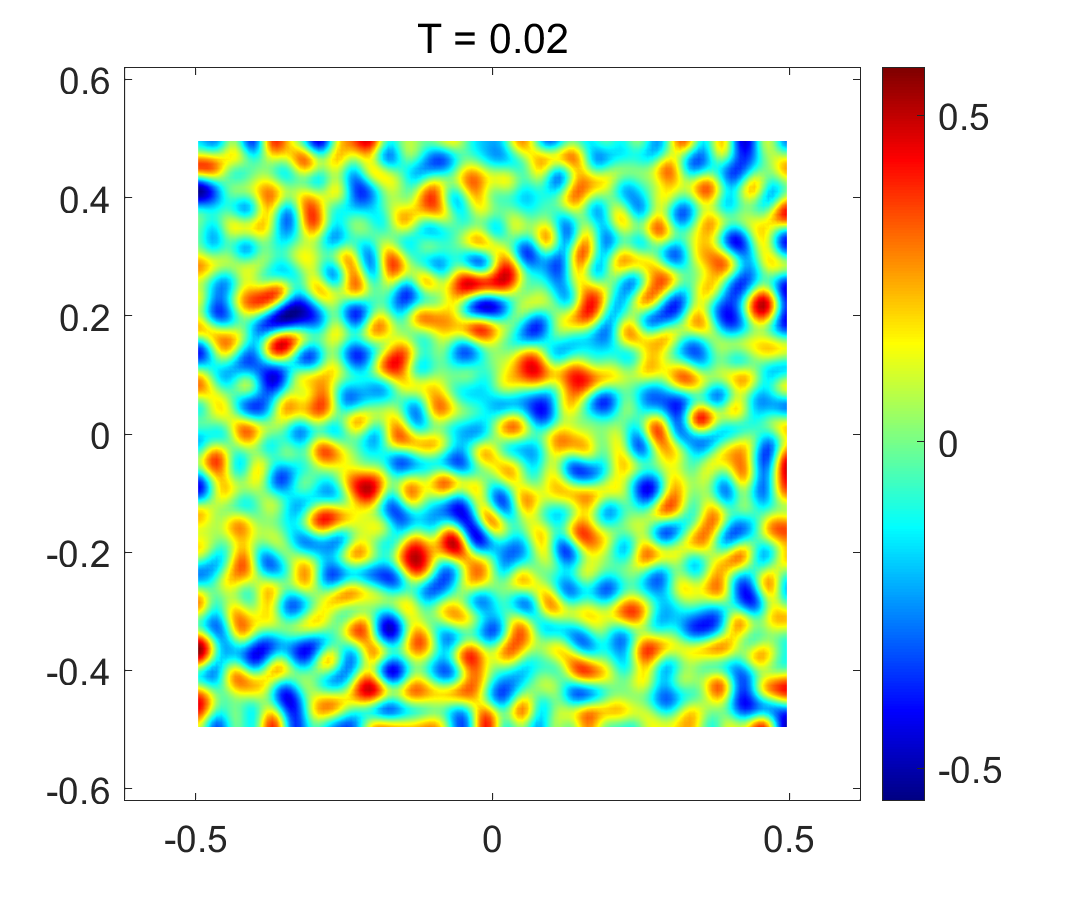}
			\includegraphics[width = 1.5in]{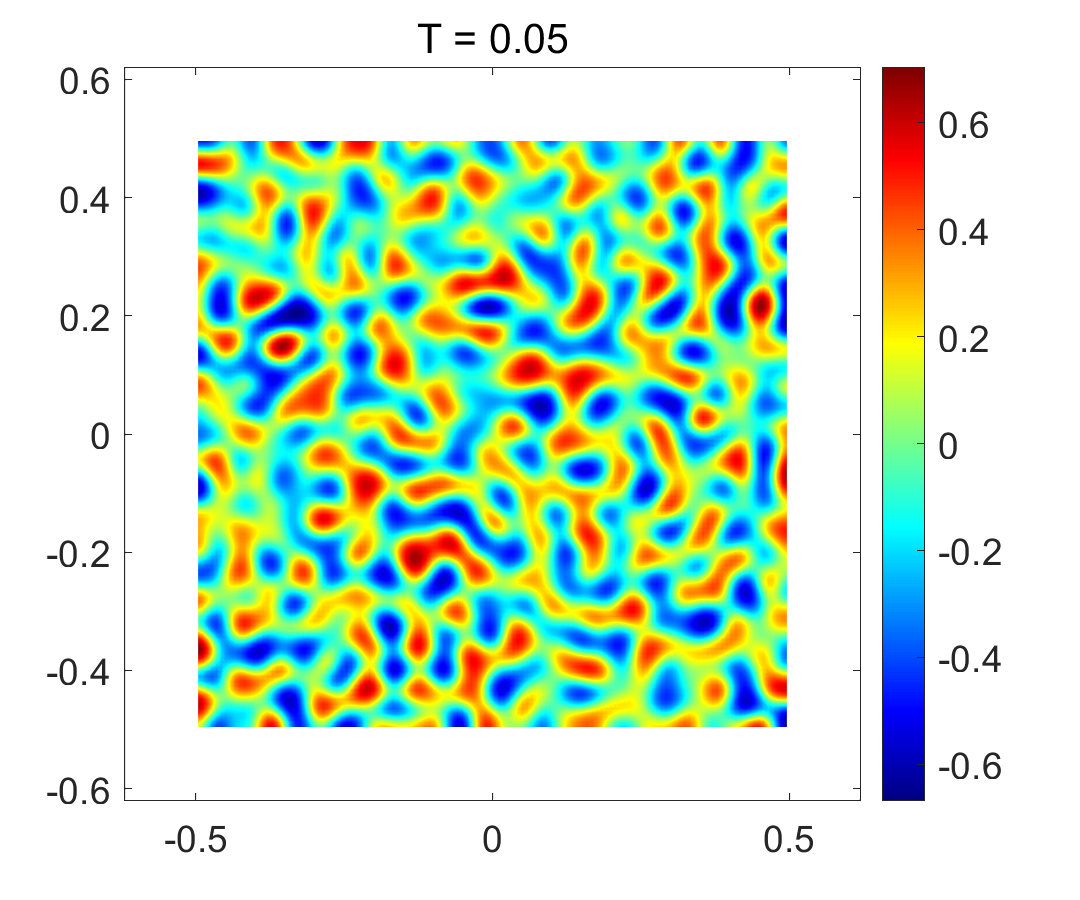}
			\includegraphics[width = 1.5in]{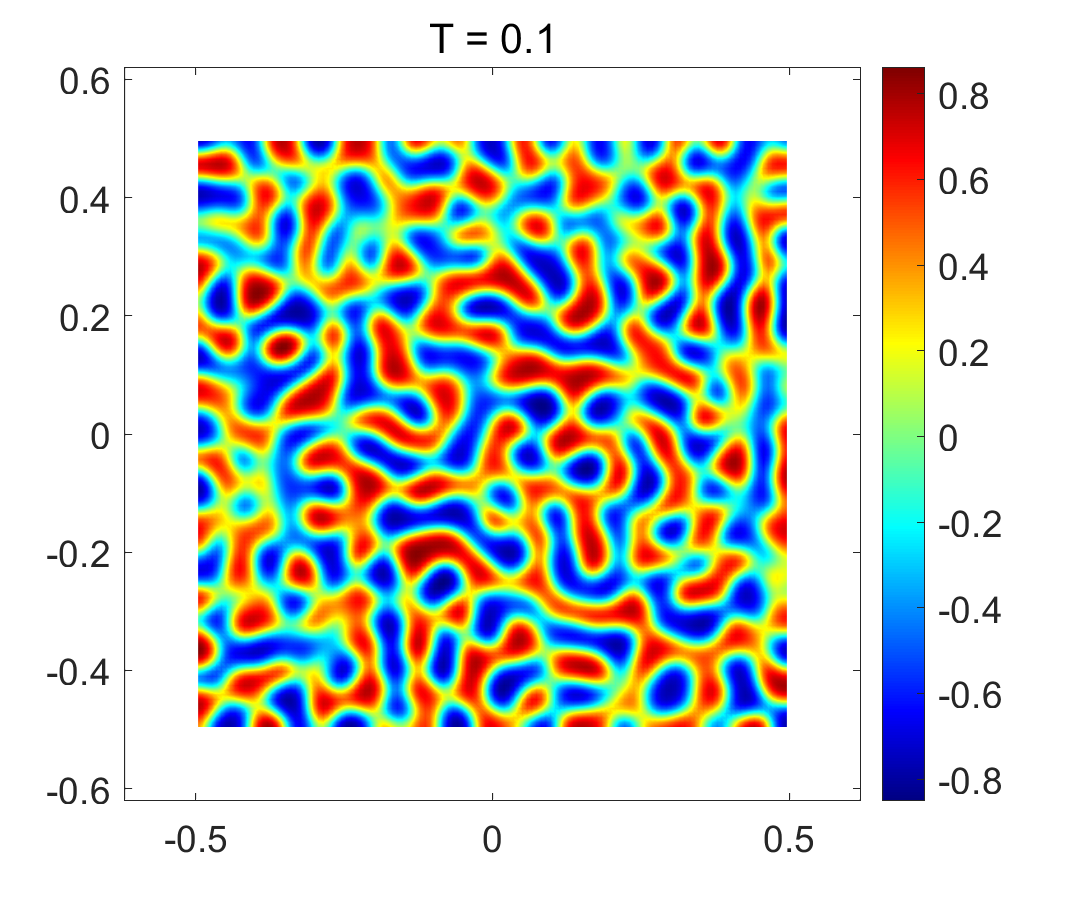}
		}
		\caption{Coarsening dynamics simulated using the Cahn-Hilliard model, the OPBDE Cahn-Hilliard model, and the DDM Cahn-Hilliard model, respectively. Profiles of $\phi$ and $\tilde{\phi}$ are shown at time instants $T = 0.01$, $0.02$, $0.05$, $0.1$. (a) The Cahn-Hilliard model results for $\phi$ simulated in $\Omega_1$; (b) The OPBDE Cahn-Hilliard model results for $\tilde{\phi}$ simulated in $\Omega$ with $\varepsilon = 10^{-2}$; (c) The DDM Cahn-Hilliard model results for $\tilde{\phi}$ simulated in $\Omega$ with $\varepsilon = 10^{-2}$; (d) The OPBDE Cahn-Hilliard model results for $\tilde{\phi}$ simulated in $\Omega$ with $\varepsilon = 2\times 10^{-3}$; (e) The DDM Cahn-Hilliard model results for $\tilde{\phi}$ simulated in $\Omega$ with $\varepsilon = 2\times10^{-3}$.}
		\label{Coarsening1}
	\end{figure}
	
	Figure \ref{Coarsening1} shows the coarsening dynamics produced from the three Cahn-Hilliard type models. Comparison among the solution profiles shows that they are almost identical.
	
	\begin{small}
		\begin{table}[H]
			\centering
			\setlength{\tabcolsep}{5pt}
			\begin{small}
				\begin{tabular}{ccccc}
					\toprule[1.5pt]
					& OPBDE Cahn-Hilliard & DDM Cahn-Hilliard & OPBDE Cahn-Hilliard & DDM Cahn-Hilliard \\
					&$\varepsilon = 10^{-2}$&$\varepsilon = 10^{-2}$&$\varepsilon = 2\times10^{-3}$&$\varepsilon = 2\times10^{-3}$\\
					\midrule[1pt]
					T = 0.01 & 0.0012 & 0.0012 & $2.9265\times 10^{-7}$ & $3.0149\times 10^{-7}$ \\
					T = 0.02 & 0.0013 & 0.0013 & $3.2299\times 10^{-7}$ & $3.3981\times 10^{-7}$ \\
					T = 0.05 & 0.0017 & 0.0017 & $4.6023\times 10^{-7}$ & $4.9291\times 10^{-7}$ \\
					T = 0.1  & 0.0027 & 0.0027 & $7.1844\times 10^{-7}$ & $7.6415\times 10^{-7}$ \\
					\bottomrule[1.5pt]
				\end{tabular}
			\end{small}
			\caption{Errors in $L^2$ norm for the two extended models versus the original Cahn-Hilliard model at $T = 0.01$, $0.02$, $0.05$, $0.1$ for interface width $\varepsilon = 10^{-2}$ and $2\times10^{-3}$, respectively.}
			\label{Error1}
		\end{table}
	\end{small}
	The table \ref{Error1} presents the $L^2$ error between the two extended models and the original Cahn-Hilliard model. It is observed that the error decreases with the decreasing width parameter $\varepsilon$. For $\varepsilon = 2\times10^{-3}$, the error between the OPBDE Cahn-Hilliard model and the original Cahn-Hilliard model is on the order of $\cO (10^{-7})$, which is actually smaller than the order of truncation error $\cO(\Delta x^2) \approx \cO(10^{-4})$ out of the discrete model. This demonstrates that the OPBDE Cahn-Hilliard performs well.
	
	\begin{figure}[H]
		\centering
		\subfigure[]{\includegraphics[width = 0.3 \textwidth]{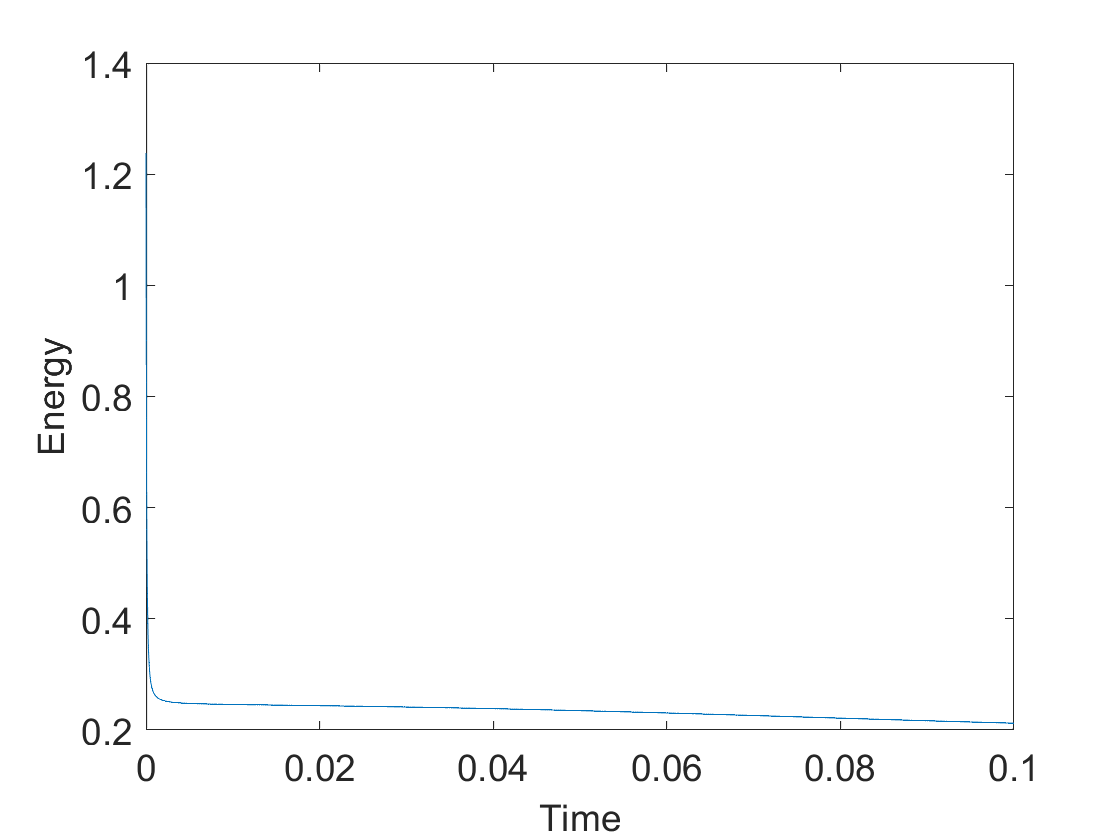}}
		\subfigure[]{\includegraphics[width = 0.3 \textwidth]{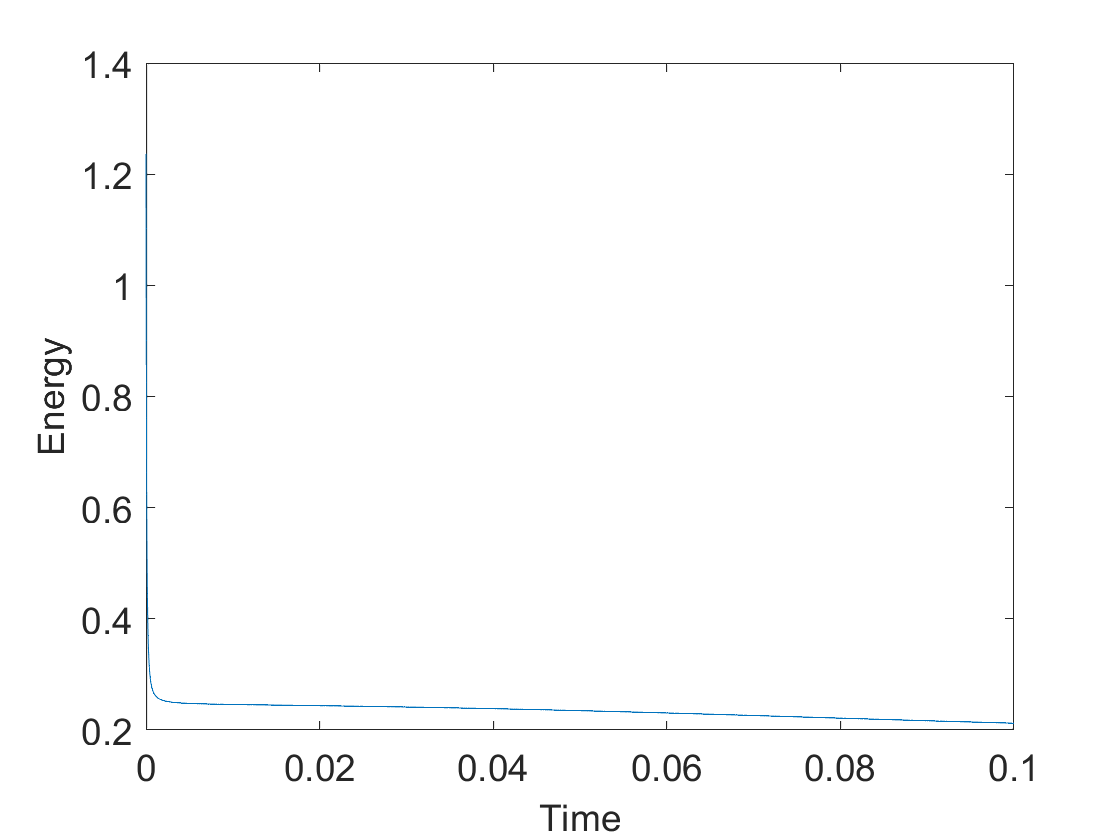}}
		\subfigure[]{\includegraphics[width = 0.3 \textwidth]{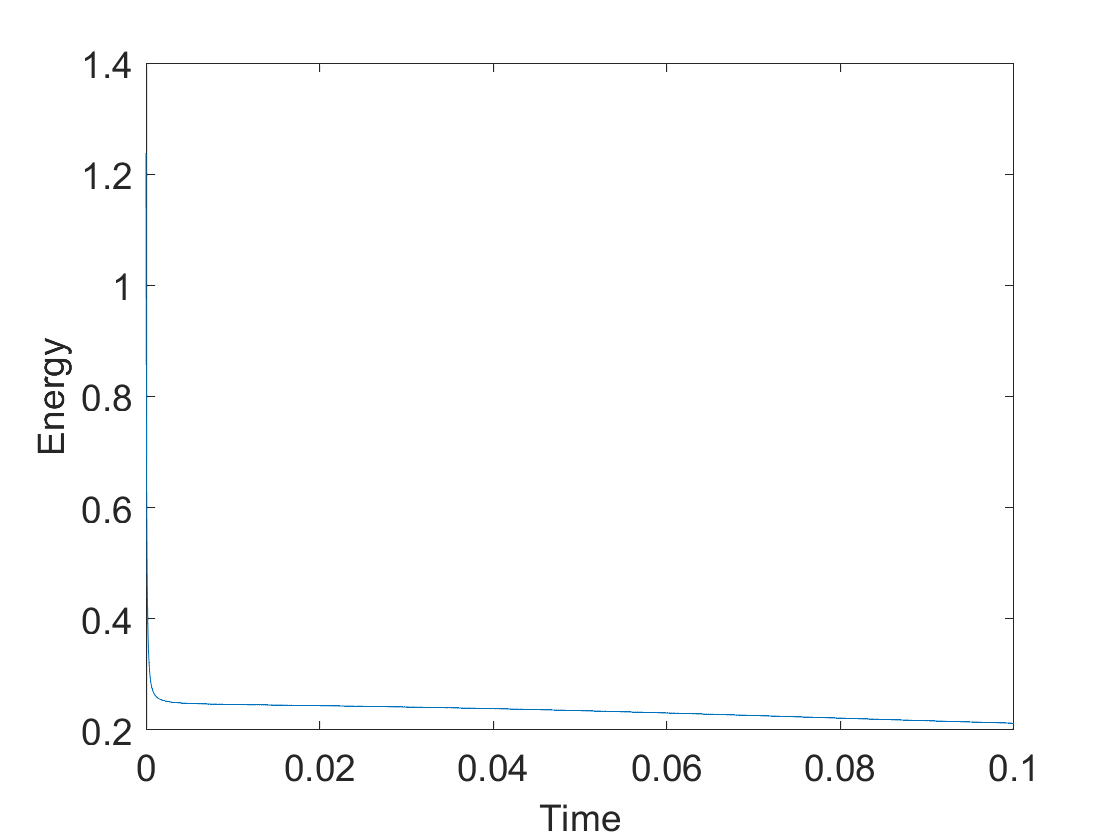}}
		\caption{Comparison of energy evolution between the Cahn-Hilliard model and the OPBDE Cahn-Hilliard model. (a) The Cahn-Hilliard model; (b) The OPBDE Cahn-Hilliard model with $\varepsilon = 10^{-2}$; (c) The OPBDE Cahn-Hilliard model with $\varepsilon = 2\times10^{-3}$. The curves of energy decay are nearly identical.}
		\label{Energy1}
	\end{figure}
	
	Figure \ref{Energy1} shows the energy evolution obtained from the Cahn-Hilliard model and the OPBDE Cahn-Hilliard model, with the energy decay over time being consistent with Thm. \ref{Thm2}.
	
	\begin{figure}[H]
		\centering
		\subfigure[]{\includegraphics[width = 0.3 \textwidth]{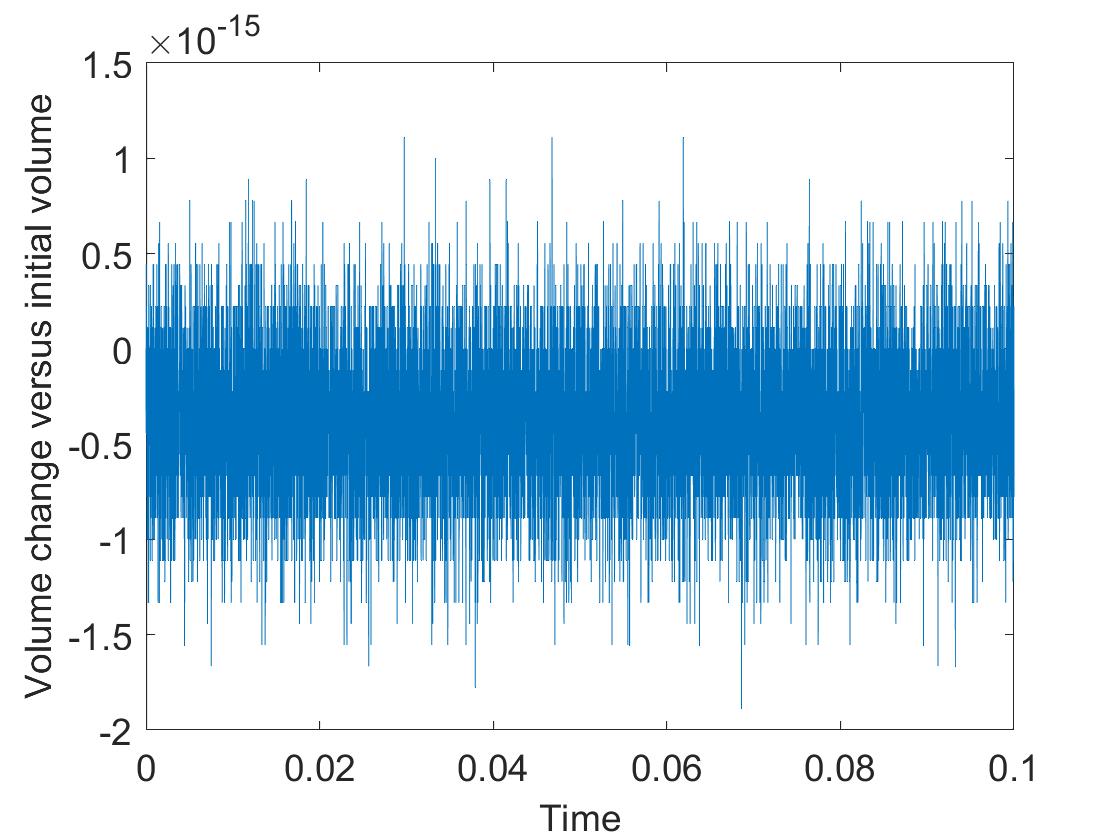}}
		\subfigure[]{\includegraphics[width = 0.3 \textwidth]{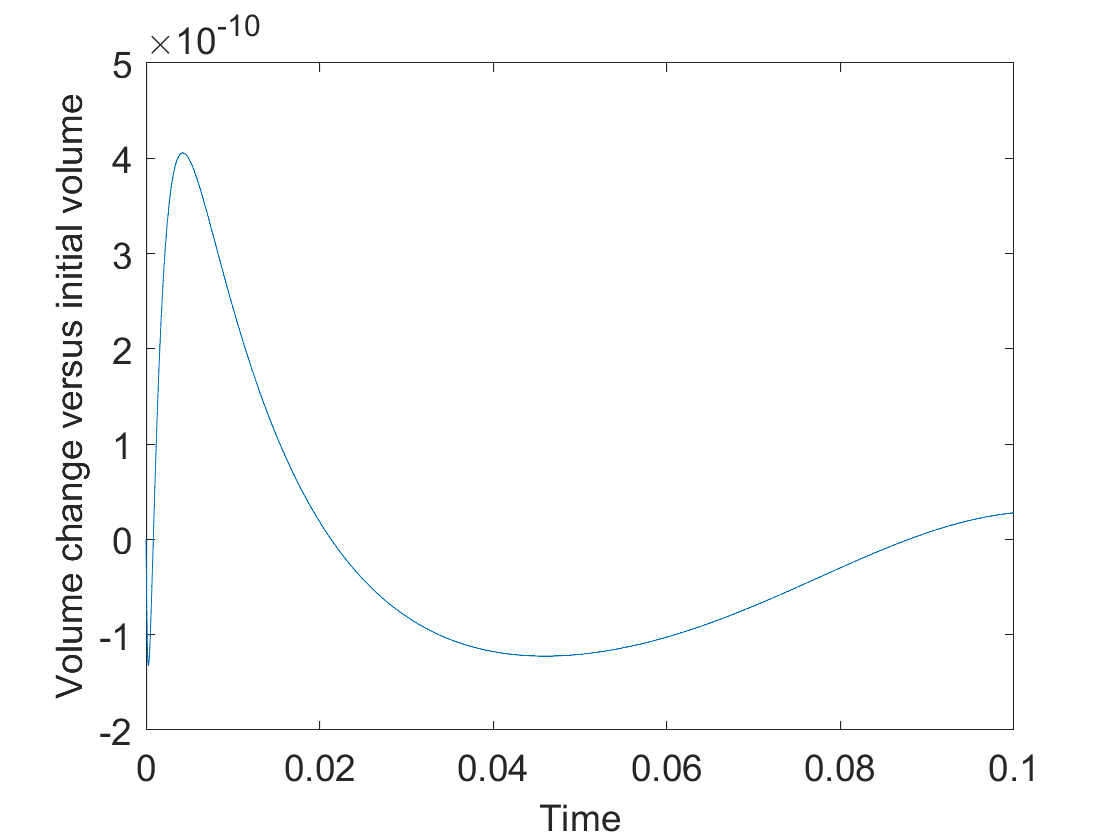}}
		\subfigure[]{\includegraphics[width = 0.3 \textwidth]{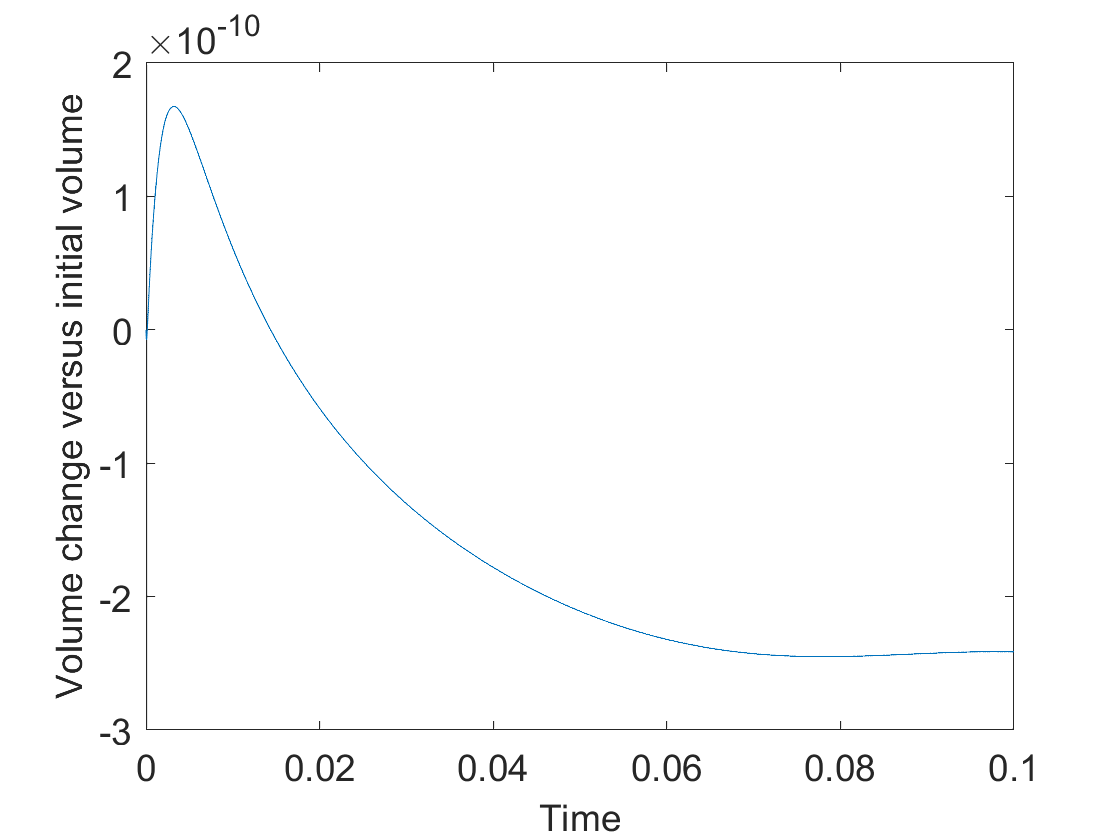}}
		\caption{Comparison of volume change versus initial volume between the Cahn-Hilliard model and the OPBDE Cahn-Hilliard model. (a) The Cahn-Hilliard model; (b) The OPBDE Cahn-Hilliard model with $\varepsilon = 10^{-2}$; (c) The OPBDE Cahn-Hilliard model with $\varepsilon = 2\times10^{-3}$.}
		\label{Volume1}
	\end{figure}
	
	Figure \ref{Volume1} shows the volume change versus initial volume, obtained from the Cahn-Hilliard model and the OPBDE Cahn-Hilliard model. Due to the no flux boundary condition, the results from the Cahn-Hilliard model show machine accuracy. The results from the OPBDE Cahn-Hilliard model show a slight increase in volume change, a consequence of the regularization of $\chi_{\varepsilon}$. 
It is noted that the magnitude of the computed volume change in $\int_{\Omega}\psi_{\varepsilon}\tilde{\phi}d\bx$, on the order of $\cO(10^{-10})$, is negligible in comparison to the magnitude of the truncation error $\cO(\Delta x^2)$.

	Below we simulate the coarsening dynamics in an irregular domain using the two extended models. The irregular domain is constructed with properties such as corner points and inner holes, which may produce different effects in different extended models. Here we only use $\varepsilon = 2\times10^{-3}$ to avoid the formation of a diffuse layer that is too wide around the corner points and inner holes. All the other parameters remain unchanged.
	
	\begin{figure}[H]
		\centering
		\subfigure[]{
			\includegraphics[width = 1.5in]{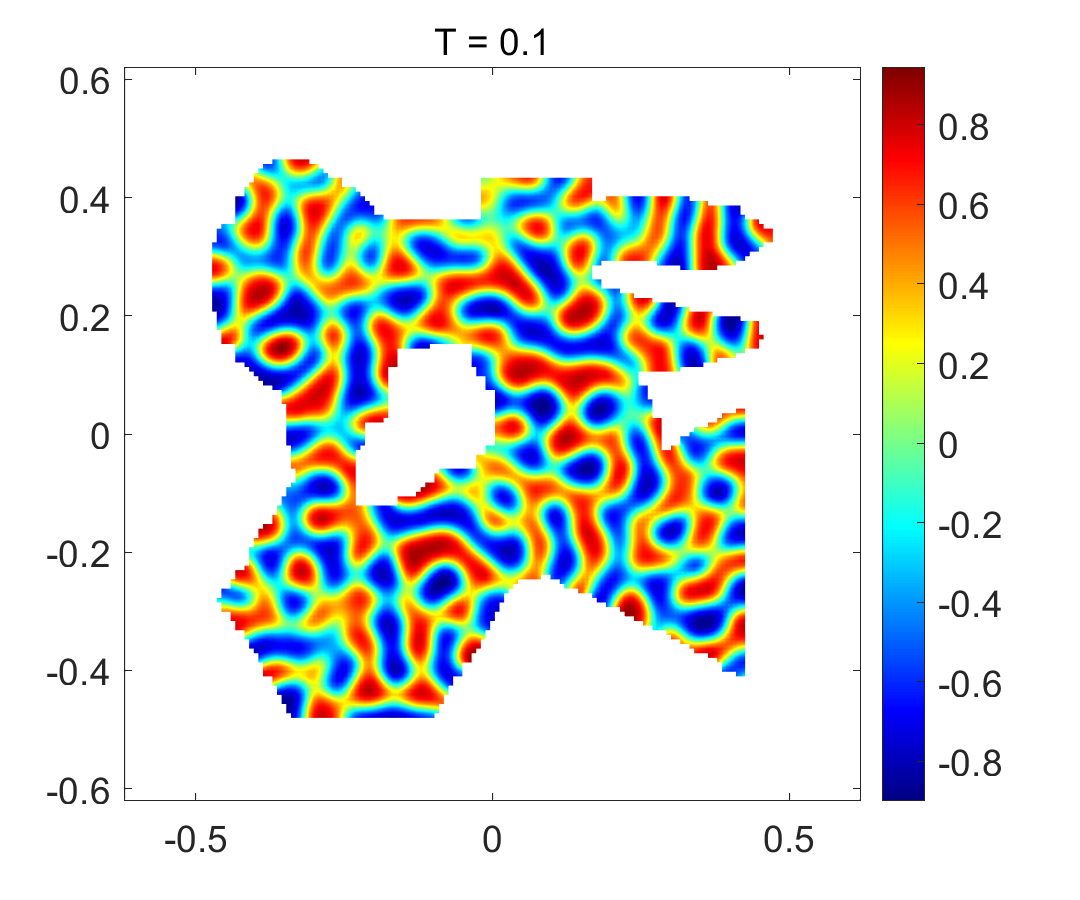}
			\includegraphics[width = 1.5in]{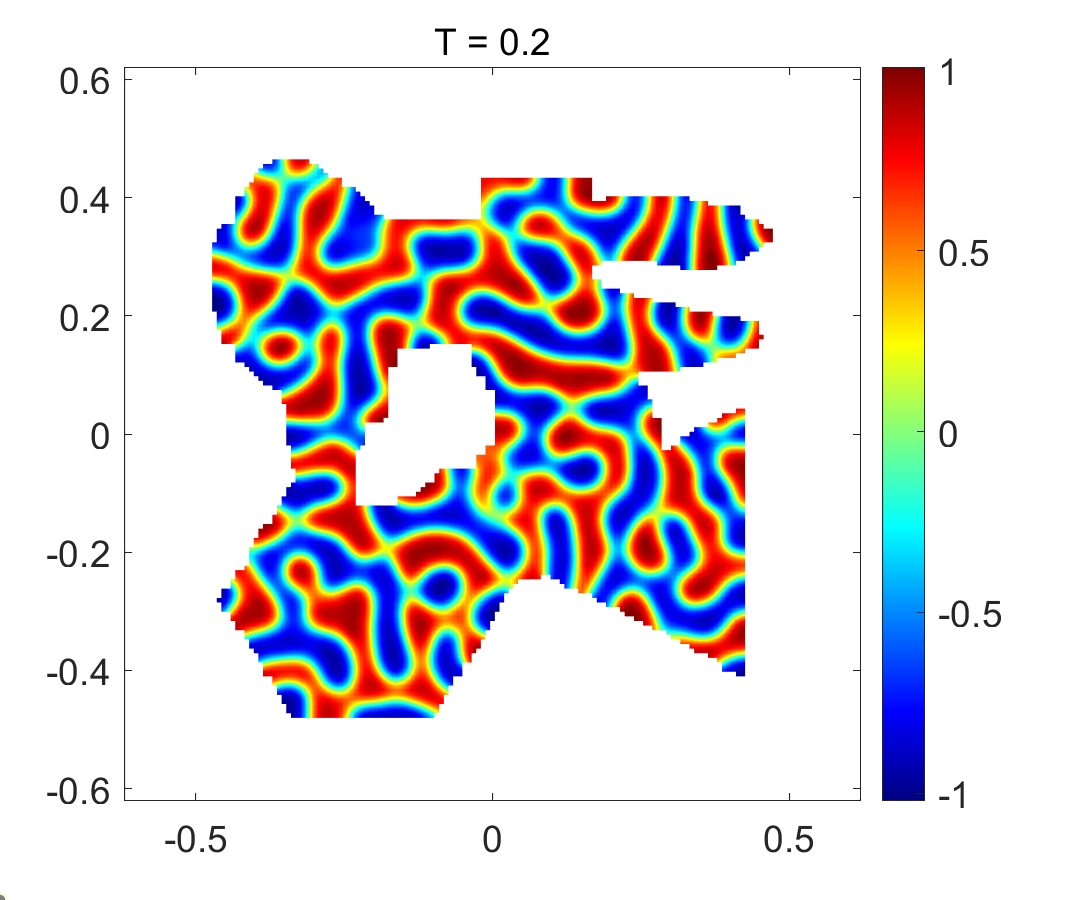}
			\includegraphics[width = 1.5in]{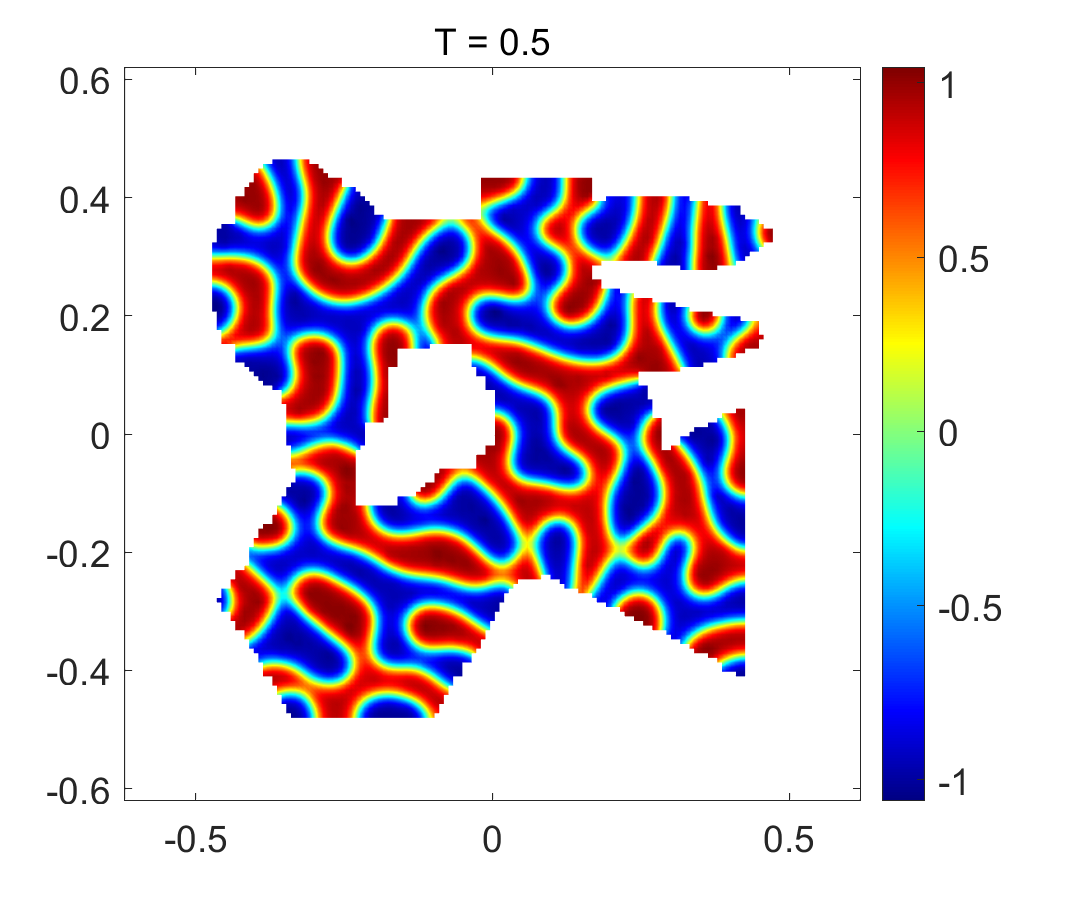}
			\includegraphics[width = 1.5in]{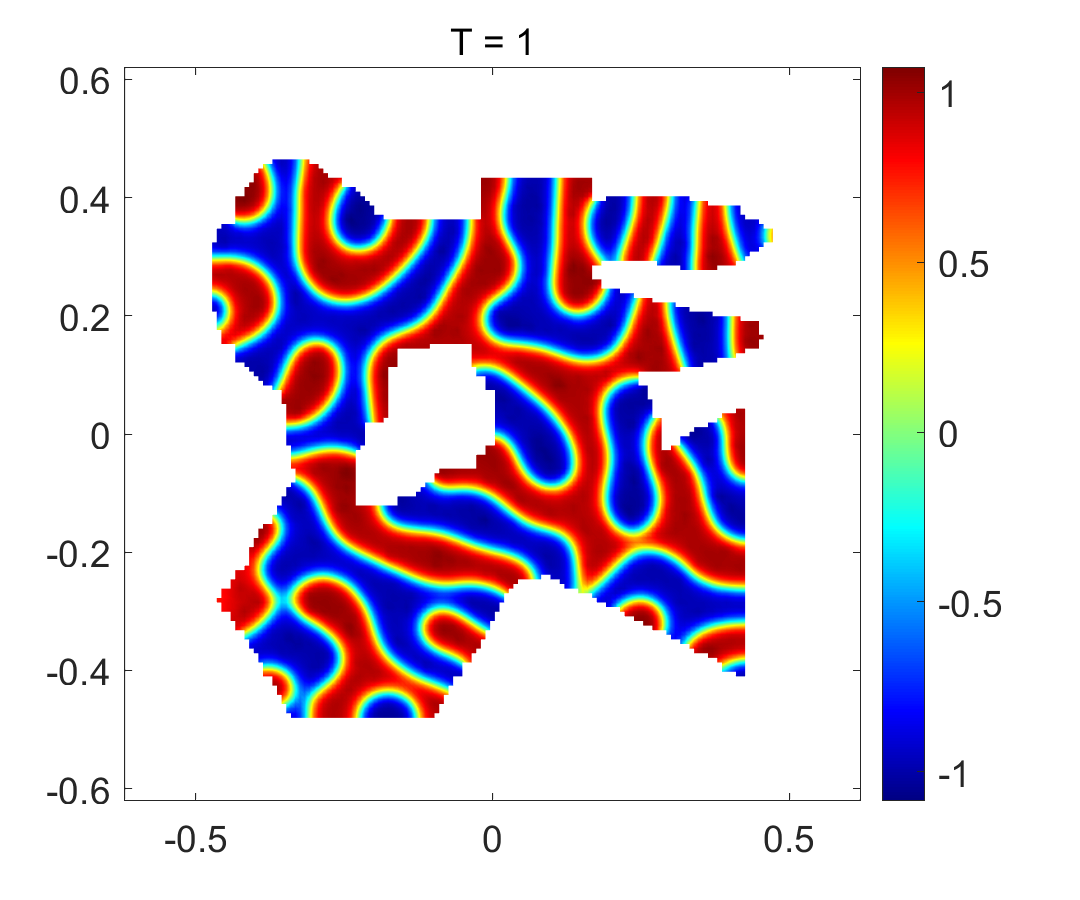}
		}
		\subfigure[]{
			\includegraphics[width = 1.5in]{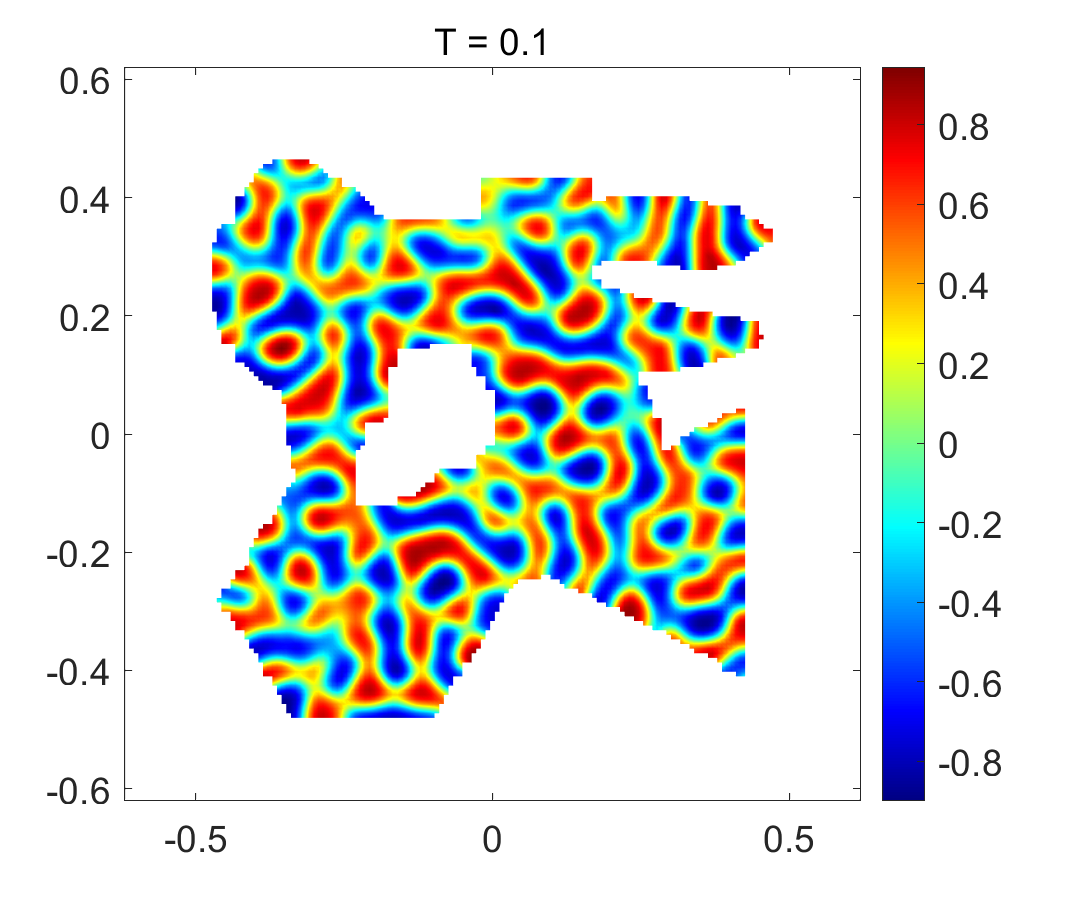}
			\includegraphics[width = 1.5in]{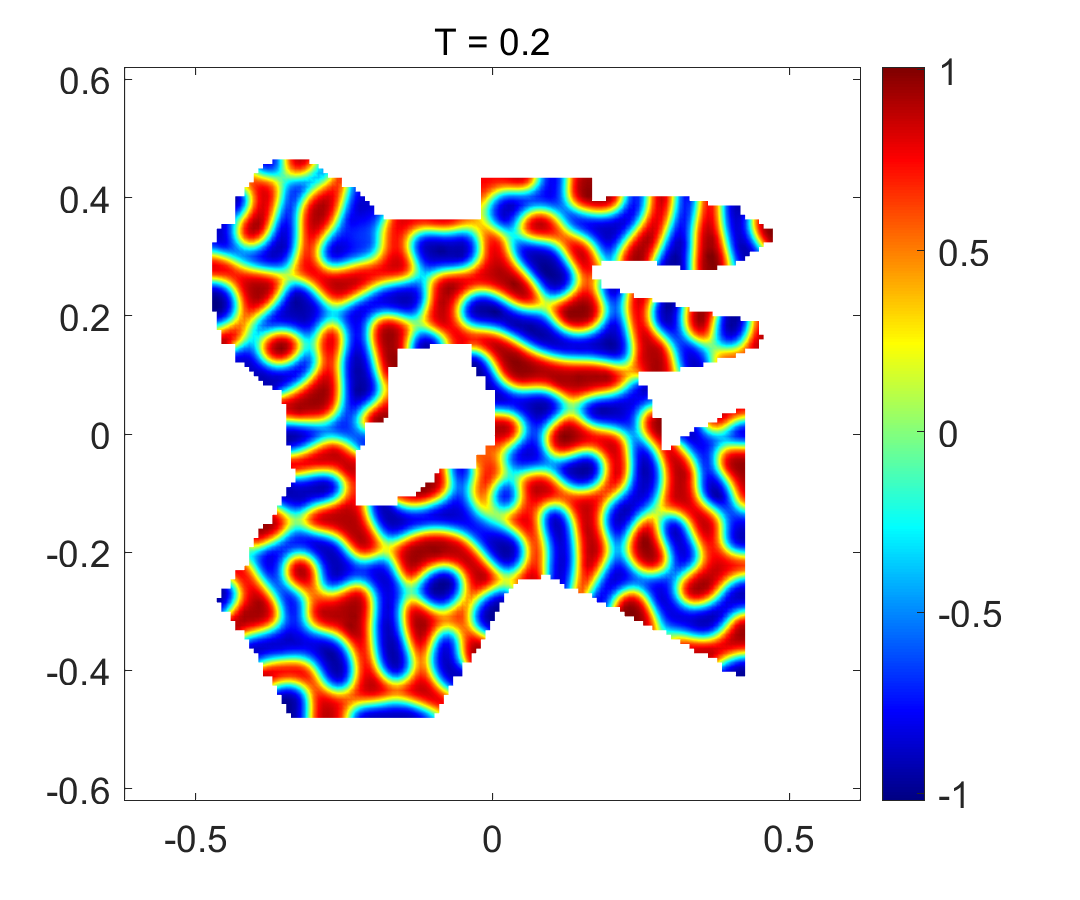}
			\includegraphics[width = 1.5in]{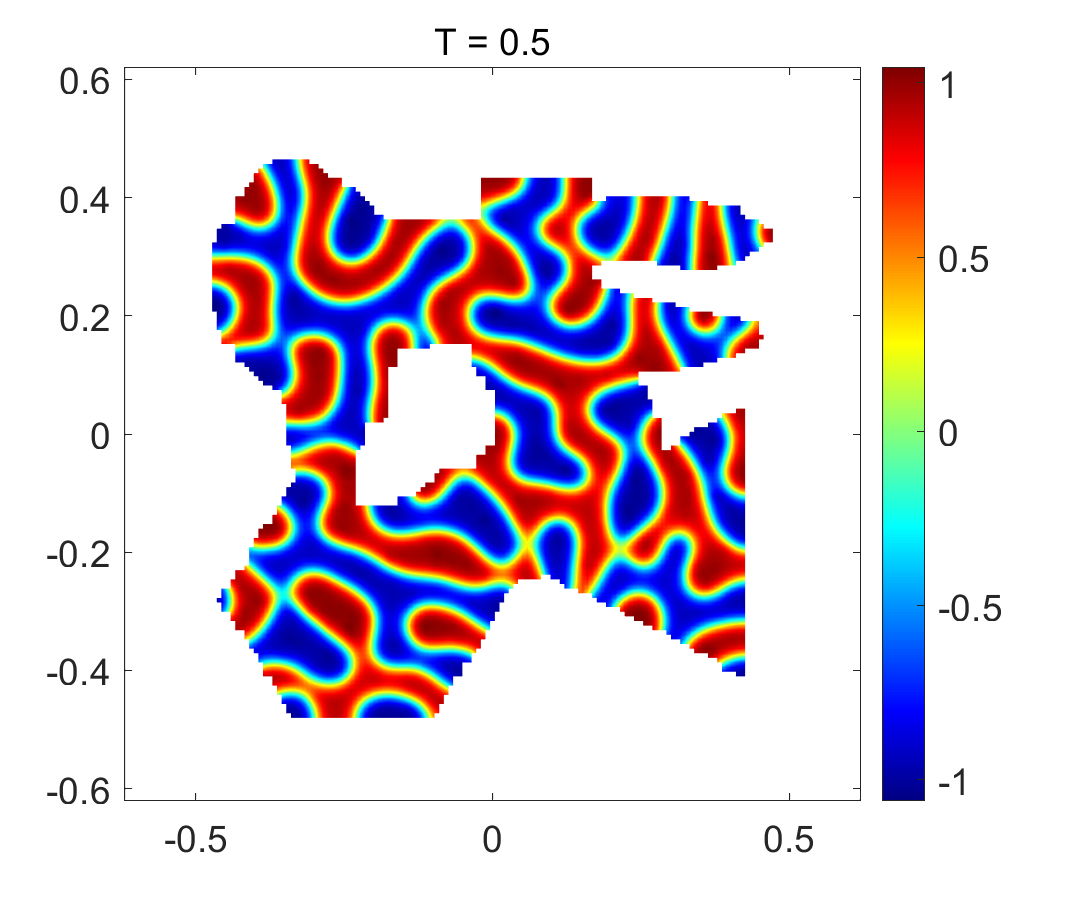}
			\includegraphics[width = 1.5in]{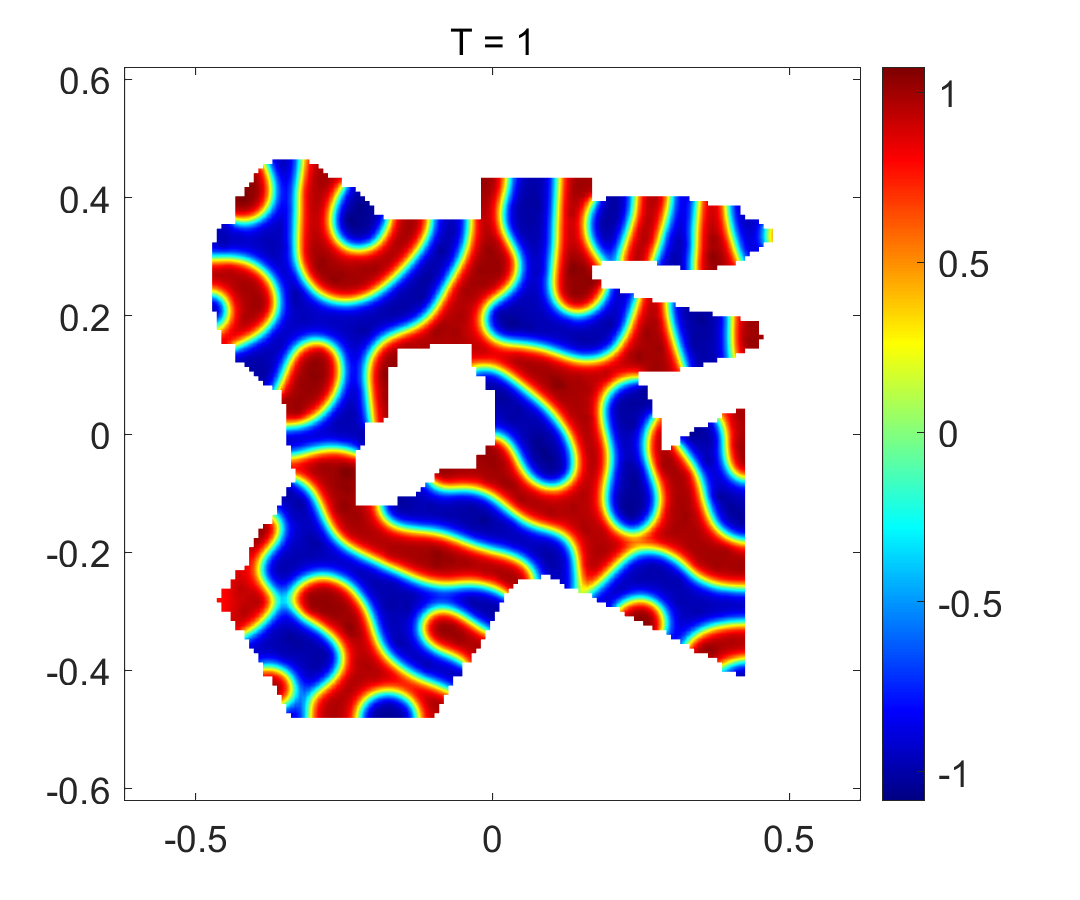}
		}
		\caption{Coarsening dynamics simulated using the OPBDE Cahn-Hilliard model and the DDM Cahn-Hilliard model. Profiles of $\tilde{\phi}$ are shown at
			time instants $t = 0.1$, $0.2$, $0.5$, $1$. (a) The OPBDE Cahn-Hilliard model results; (b) The DDM Cahn-Hilliard model results.}
		\label{Coarsening2}
	\end{figure}
	
	Figure \ref{Coarsening2} shows the coarsening dynamics produced from the OPBDE Cahn-Hilliard model and the DDM Cahn-Hilliard model using the same random initial values in an arbitrary irregular domain. In this case, we do not have the numerical results from the original model in the irregular $\Omega_1$. Therefore, comparison is made between the results from the two extended models, i.e., the OPBDE Cahn-Hilliard model and the DDM Cahn-Hilliard model. The profiles of $\tilde{\phi}$ in Fig. \ref{Coarsening2}(a) and (b) are almost identical.
	
	\begin{figure}[H]
		\centering
		\subfigure[]{\includegraphics[width = 0.45 \textwidth]{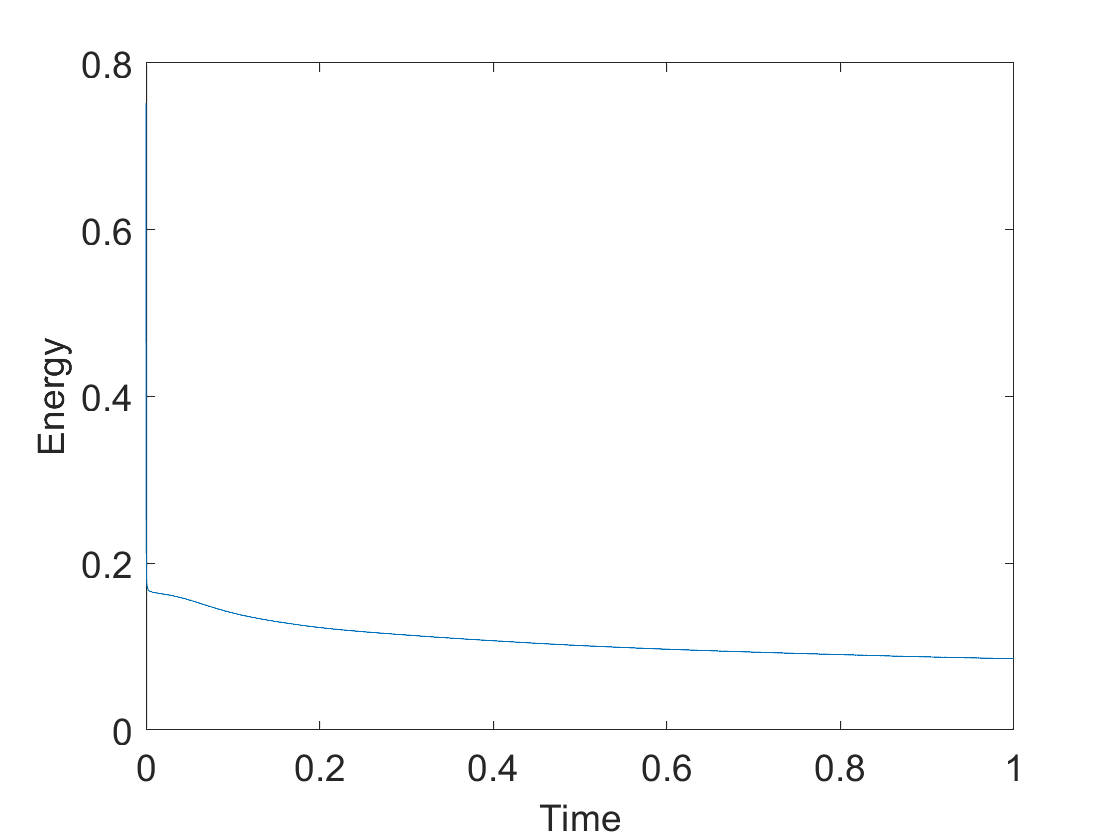}}
		\subfigure[]{\includegraphics[width = 0.45 \textwidth]{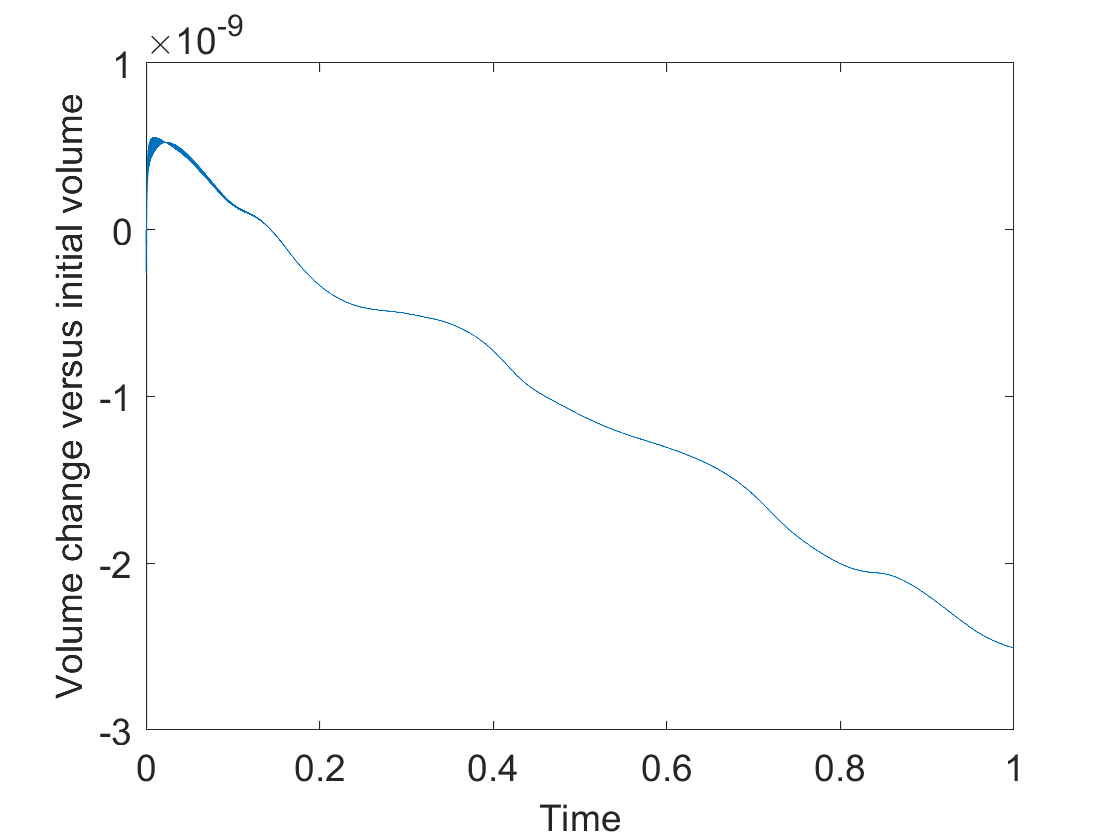}}
		\caption{Energy evolution and volume change versus initial volume simulated using the OPBDE Cahn-Hilliard model. (a) Energy evolution; (b) Volume change versus initial volume.}
		\label{Energy and Volume1}
	\end{figure}
	Figure \ref{Energy and Volume1} presents the energy evolution and volume change versus initial volume obtained from the OPBDE Cahn-Hilliard model. It is readily observed that the energy decreases in time and the volume remains conserved within the error range for $h_3 = 0$. These results meet the modeling requirements set forth. The energy and volume change from the DDM Cahn-Hilliard model have not been presented here since this model has not explicitly defined its volume and can not incorporate a definition of the free energy functional in the extended domain.
	
	\subsection{Droplet spreading on substrates}
	We perform simulations to show the drop spreading governed by the Cahn-Hilliard model with the dynamic boundary condition $\phi_t = -\Gamma(\bn_1 \cdot K\nabla \phi)$ and the no flux boundary condition $\bn \cdot M \nabla \mu= 0$ applied at $\partial \Omega_1$. At equilibrium, $\bn_1 \cdot \nabla \phi = 0$, and hence the equilibrium contact angle equals $90^\circ$. We use $f(\phi) = \frac{1}{4}(\phi^2-1)^2$, $A = 0$, $K = 10^{-4}$, $M = 0.01$, $\varepsilon = 2\times10^{-3}$, $\Delta x = \Delta y = \frac{1}{128}$, and $\Delta t = 10^{-3}$. The original domain is $\Omega_1 = [-0.5,0.5]\times [0,0.5]$, where the droplet spreading on the substrate $y=0$ is simulated using the original Cahn-Hilliard model. The domain is then extended to $\Omega = [-0.625,0.625]\times [-0.625,0.625]$, where the droplet spreading is simulated using the OPBDE Cahn-Hilliard model \eqref{OPBDE CH}. For the original model, the initial value for $\phi$ is given by
	\begin{equation}
		\phi_{0} = \tanh\left(\frac{0.2-\sqrt{x^2+(y-0.2)^2}}{0.01}\right).
	\end{equation}
	For the OPBDE Cahn-Hilliard model, the initial value for $\tilde{\phi}$ is
    \begin{equation}
        \tilde{\phi}_0 = \psi \times
        \begin{cases}
          \phi_0, & \mbox{if $(x,y) \in \Omega_1$},\\
          0, & \mbox{otherwise}.
        \end{cases}
    \end{equation}
	
	\begin{figure}[H]
		\centering
		\subfigure[]{
			\includegraphics[width = 1.2in]{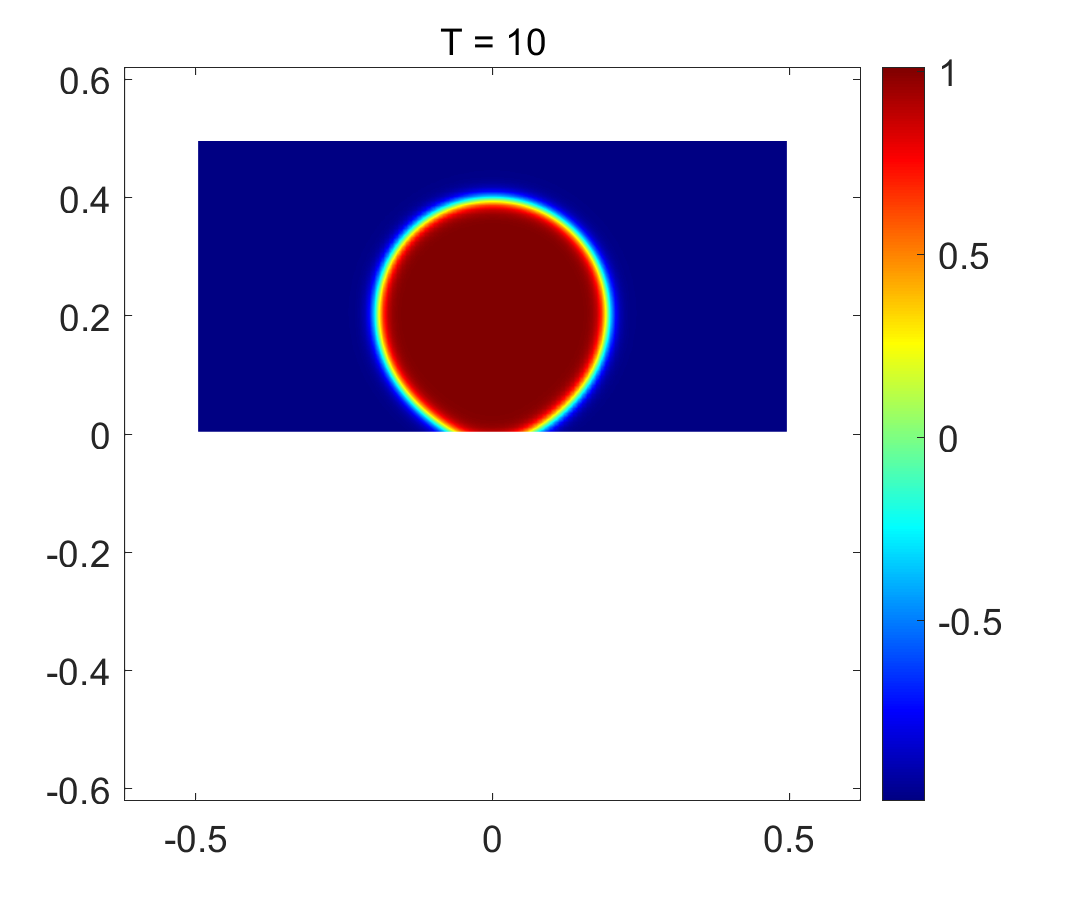}
			\includegraphics[width = 1.2in]{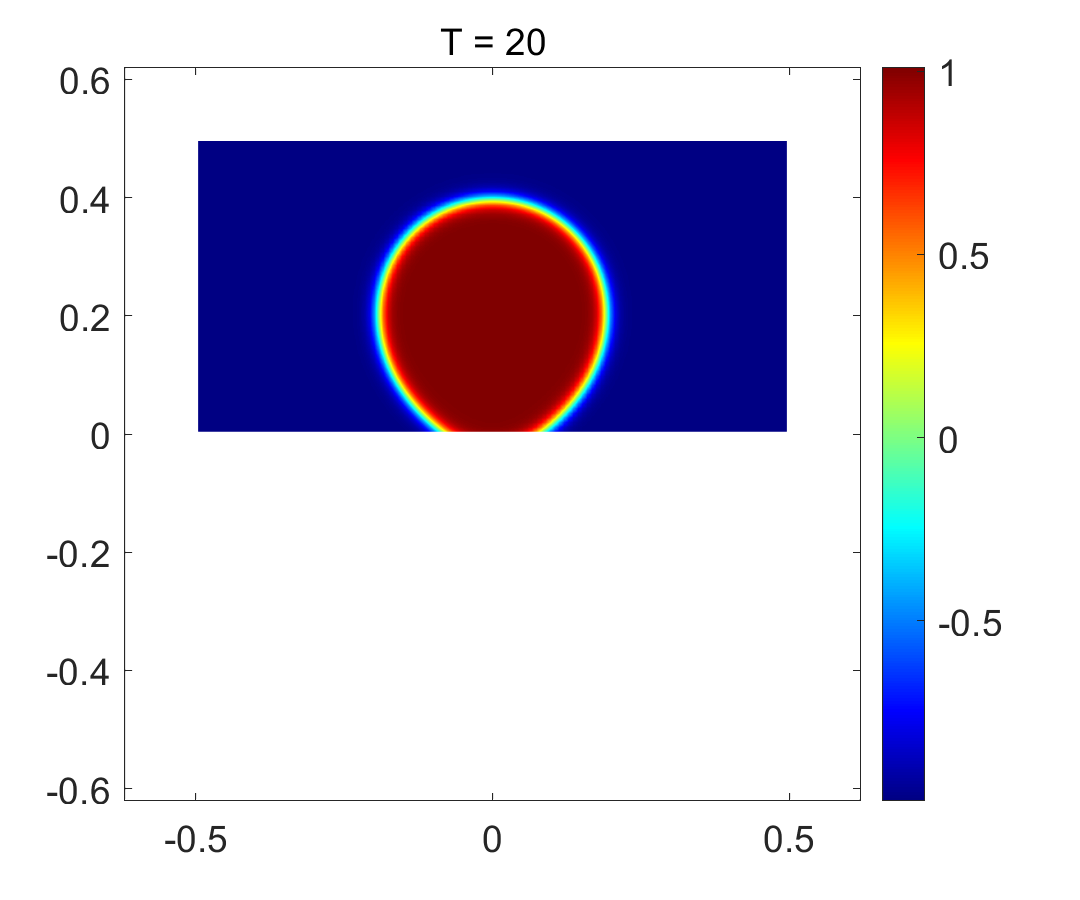}
			\includegraphics[width = 1.2in]{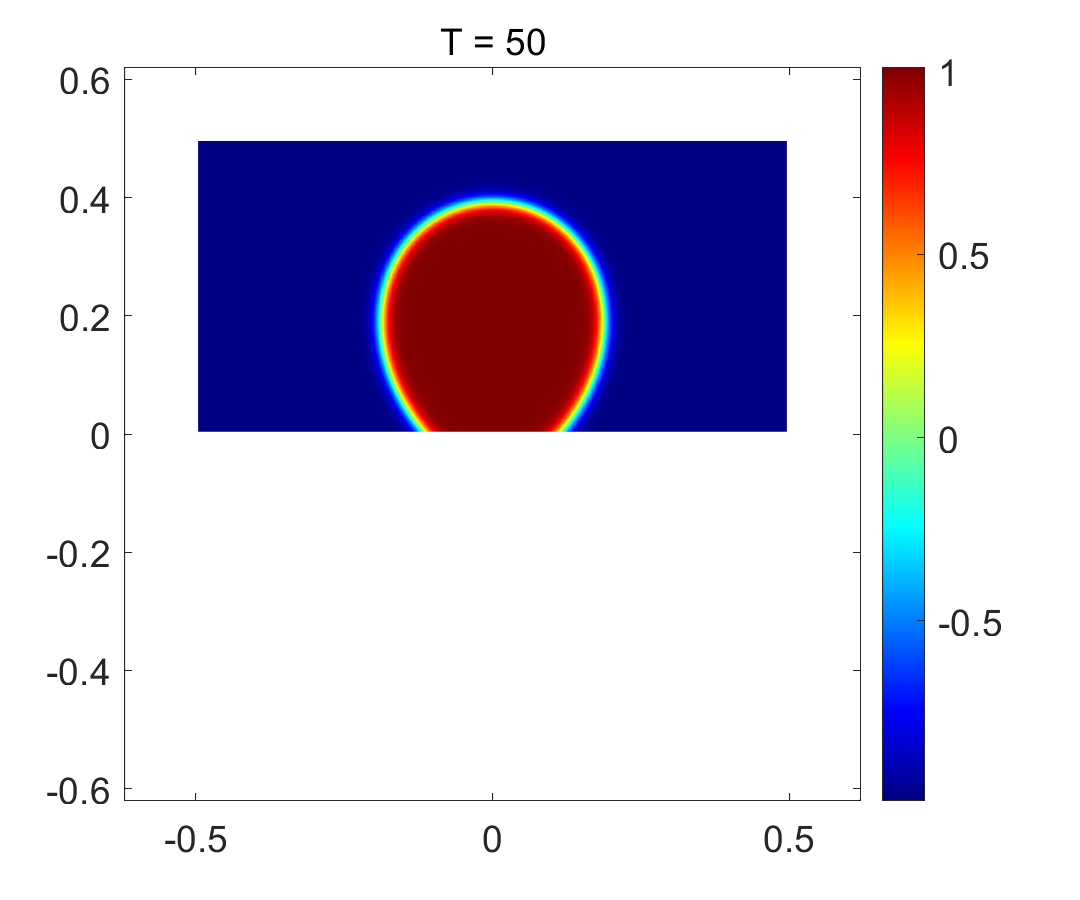}
			\includegraphics[width = 1.2in]{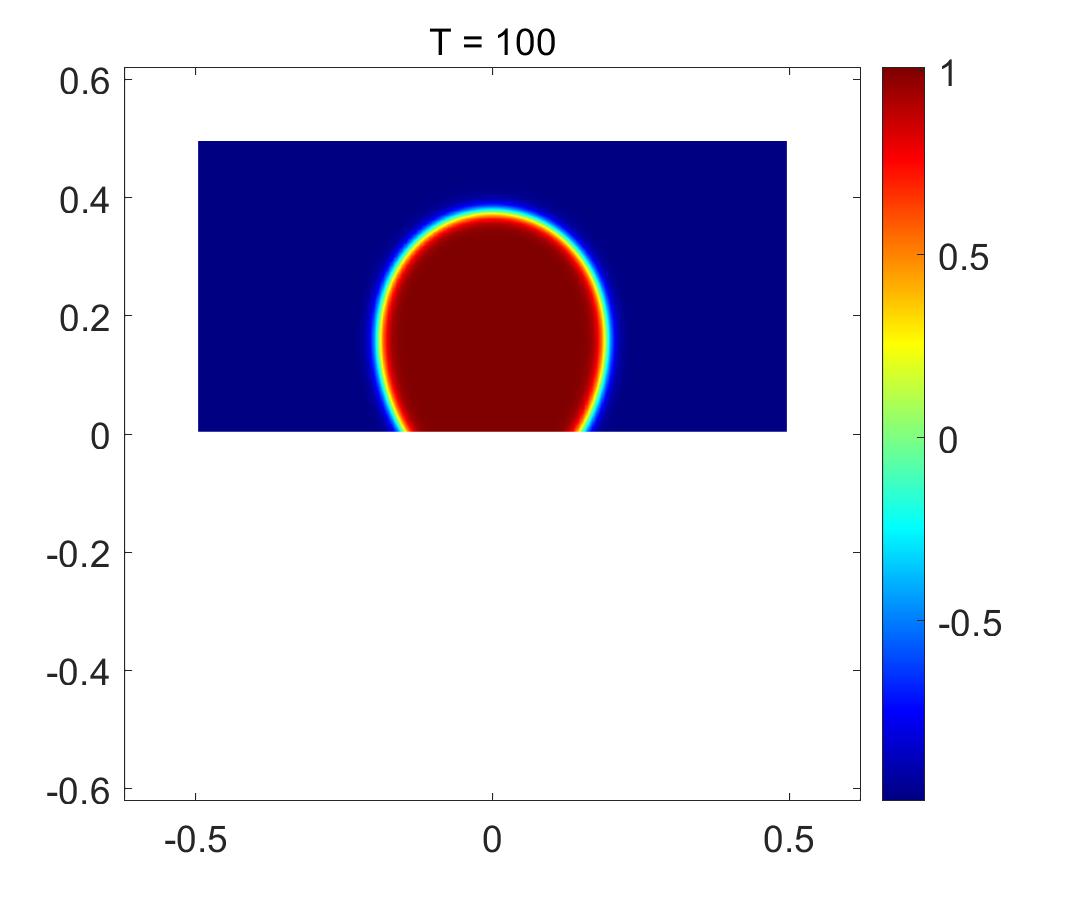}
			\includegraphics[width = 1.2in]{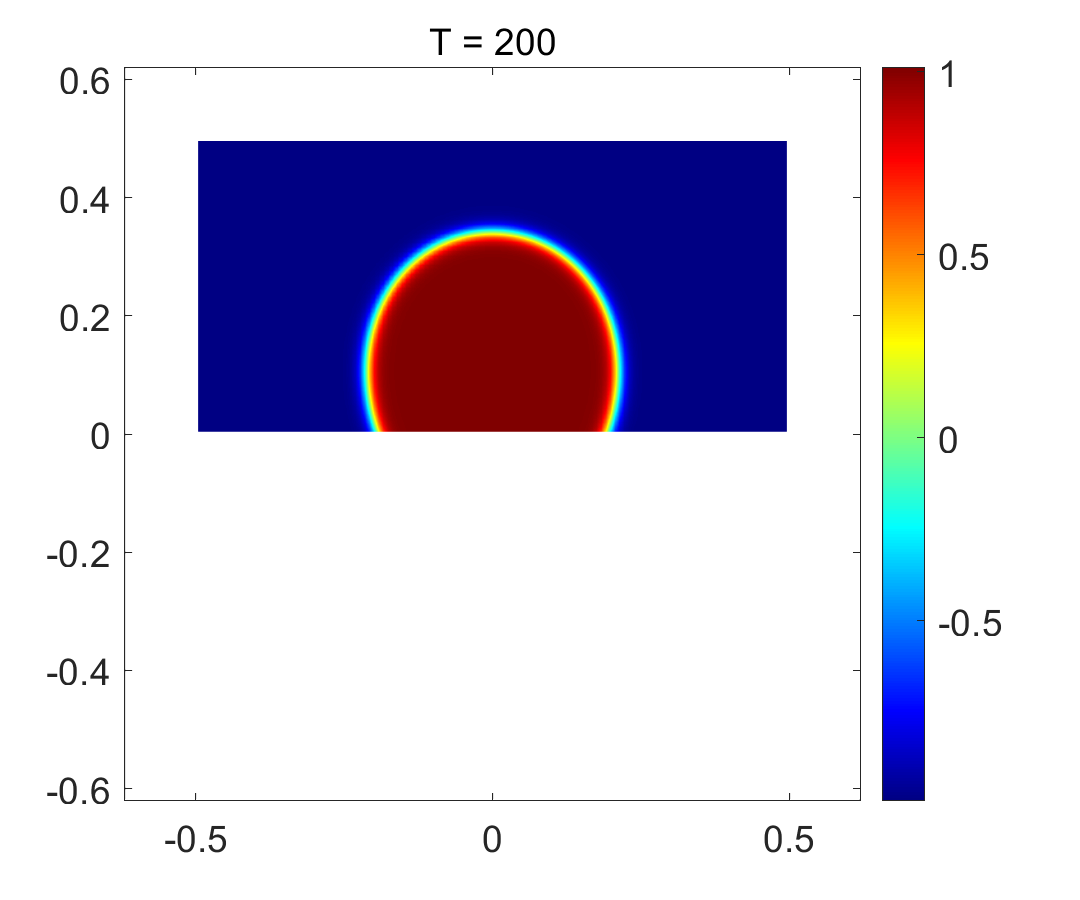}
		}
		\subfigure[]{
			\includegraphics[width = 1.2in]{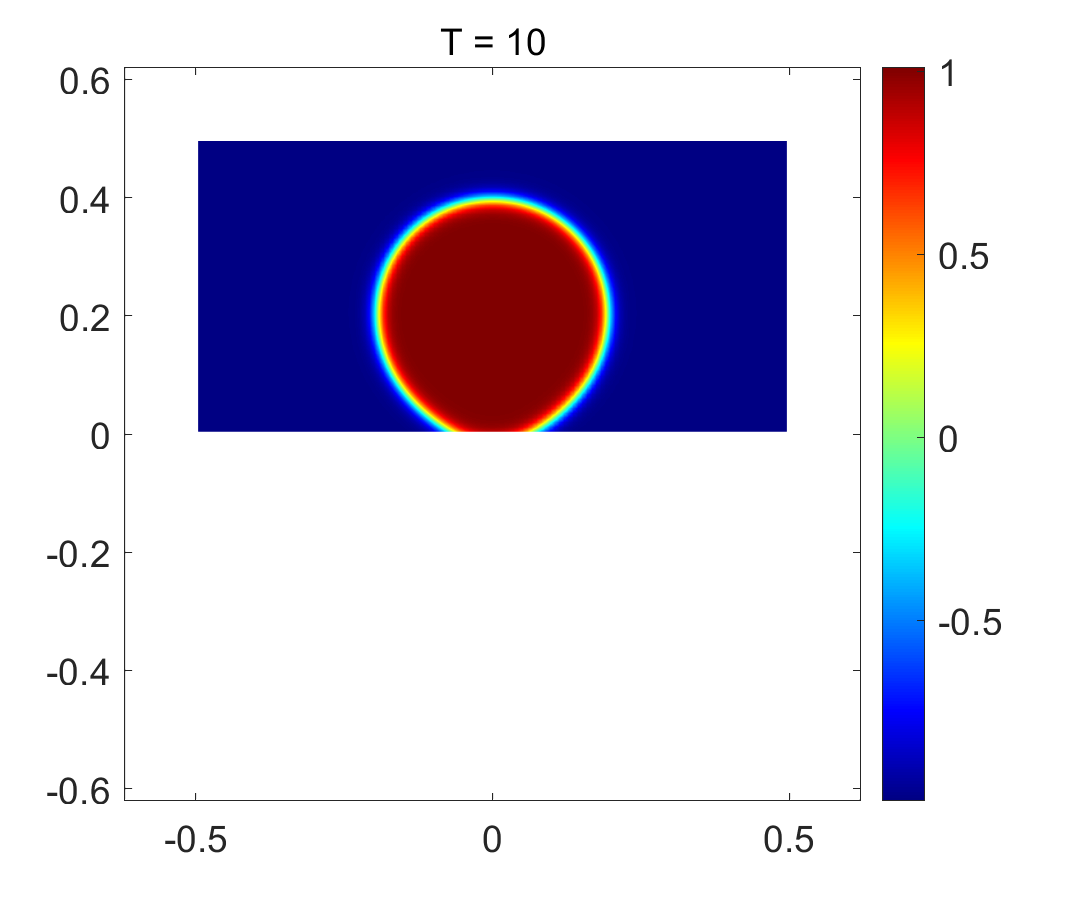}
			\includegraphics[width = 1.2in]{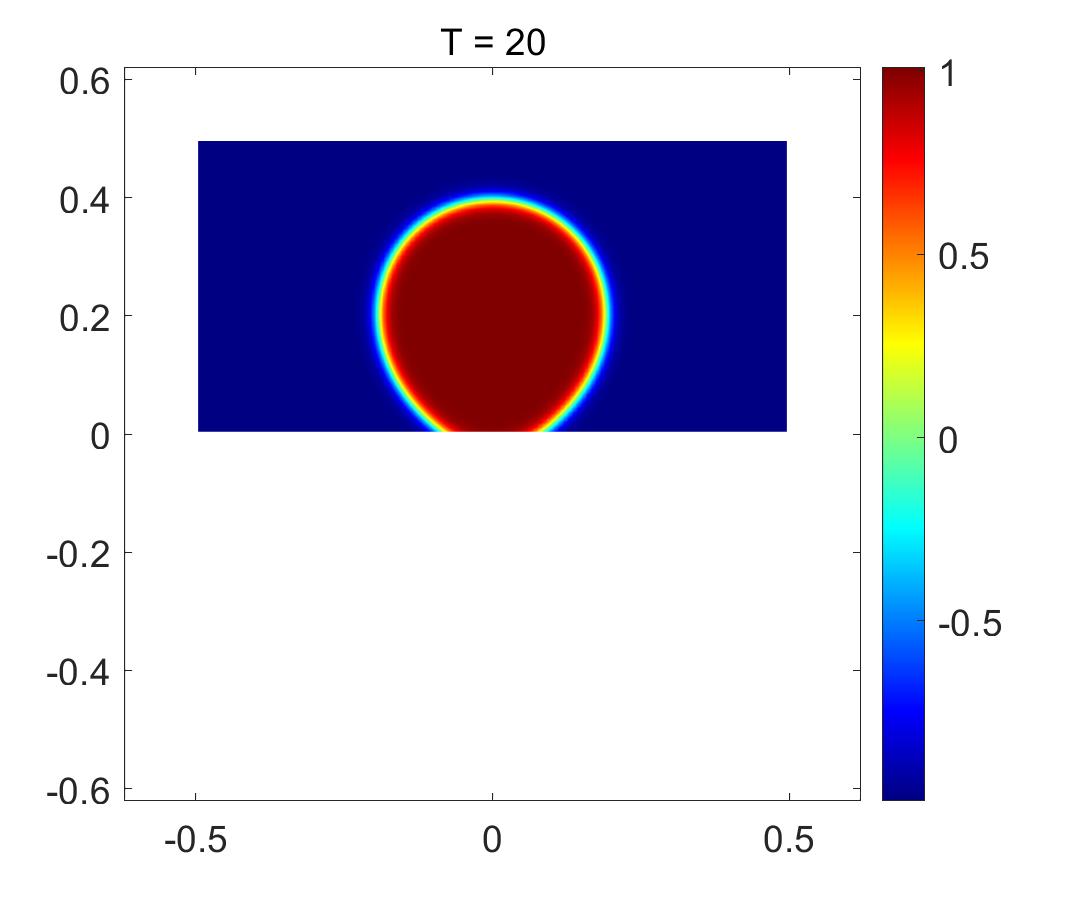}
			\includegraphics[width = 1.2in]{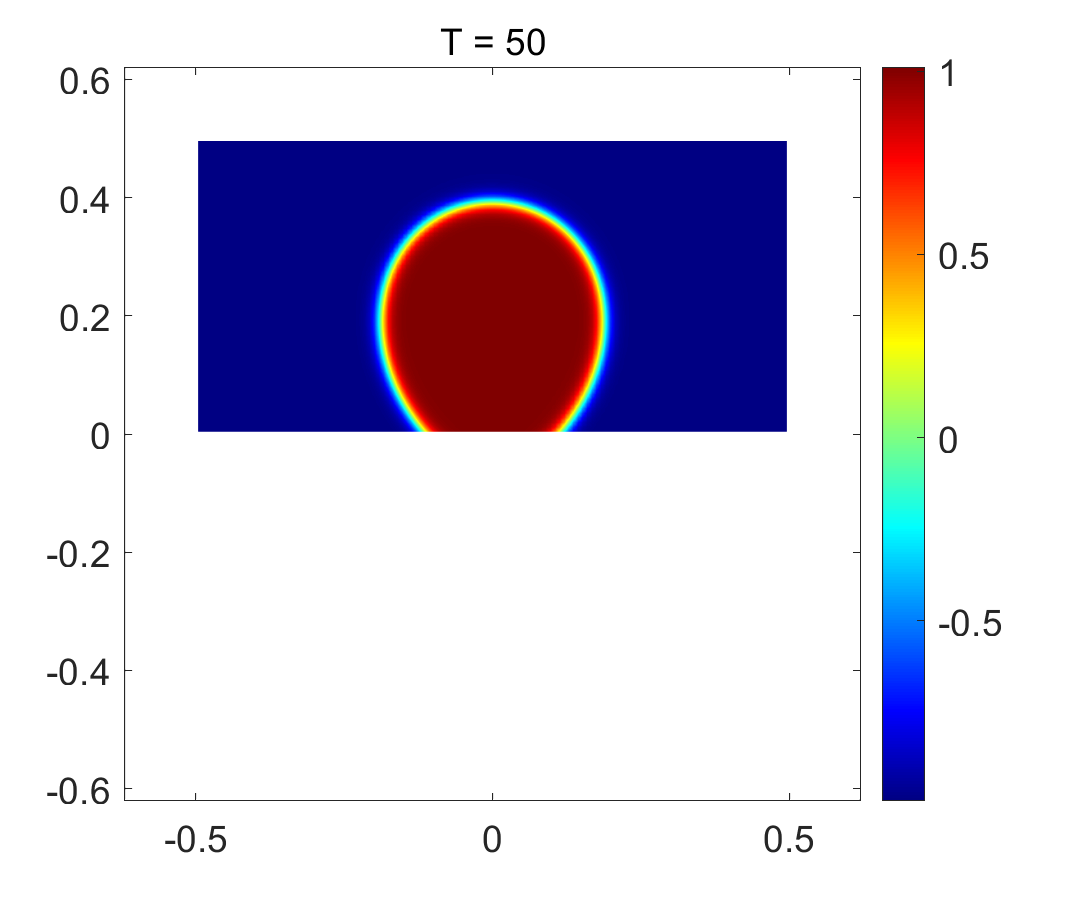}
			\includegraphics[width = 1.2in]{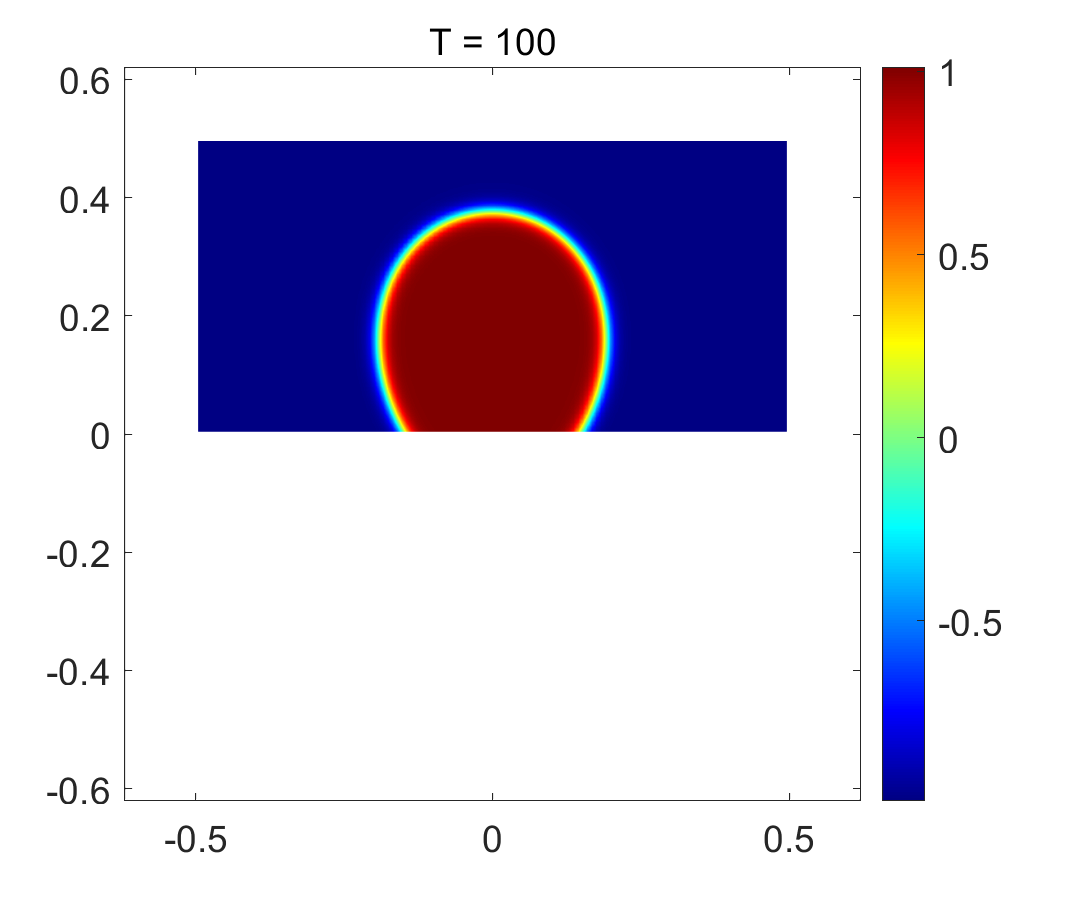}
			\includegraphics[width = 1.2in]{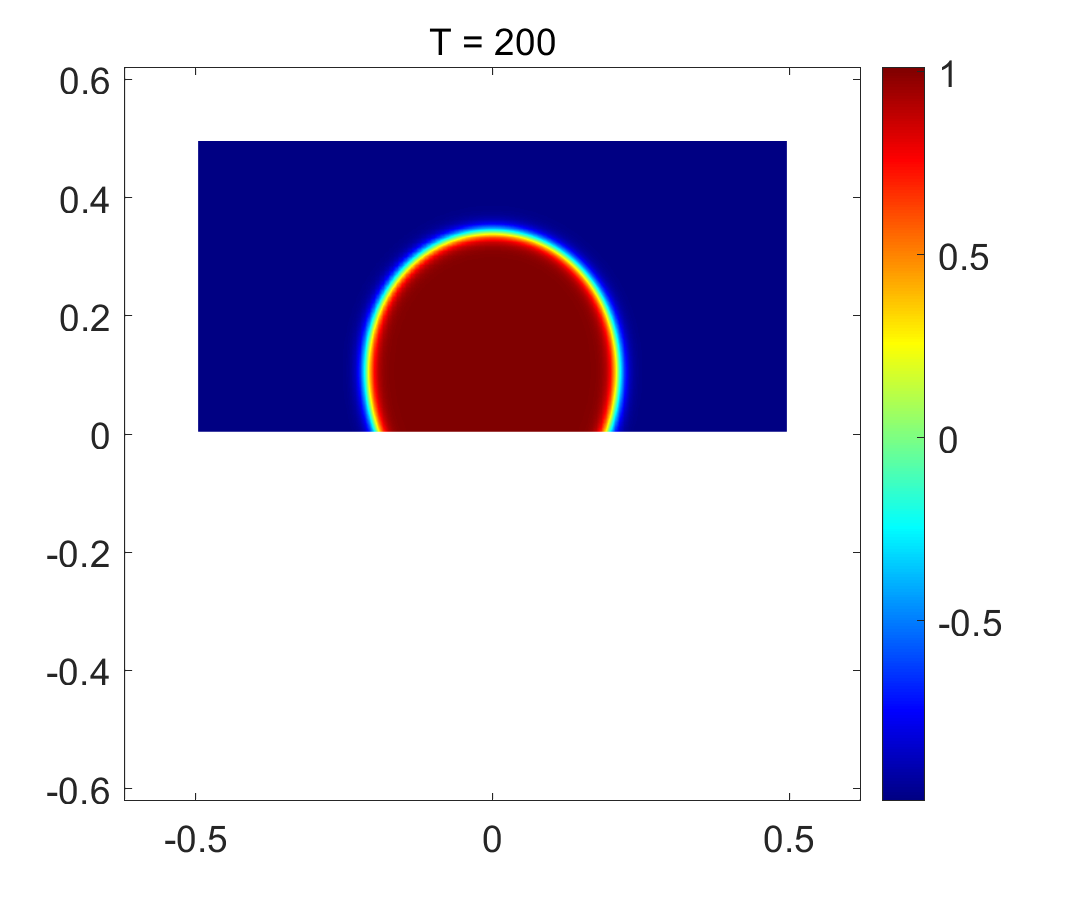}
		}
		\caption{Droplet spreading simulated using the Cahn-Hilliard model and the OPBDE Cahn-Hilliard model. Profiles of $\phi$ and $\tilde{\phi}$ are shown at time instants $T = 10$, $20$, $50$, $100$, $200$ for $\Gamma = 10$. (a) The Cahn-Hilliard model's $\phi$ simulated in $\Omega_1$; (b) The OPBDE Cahn-Hilliard model's $\tilde{\phi}$ simulated in $\Omega$. }
		\label{Droplet spreading1}
	\end{figure}
	
	\begin{figure}[H]
		\centering
		\subfigure[]{
			\includegraphics[width = 1.2in]{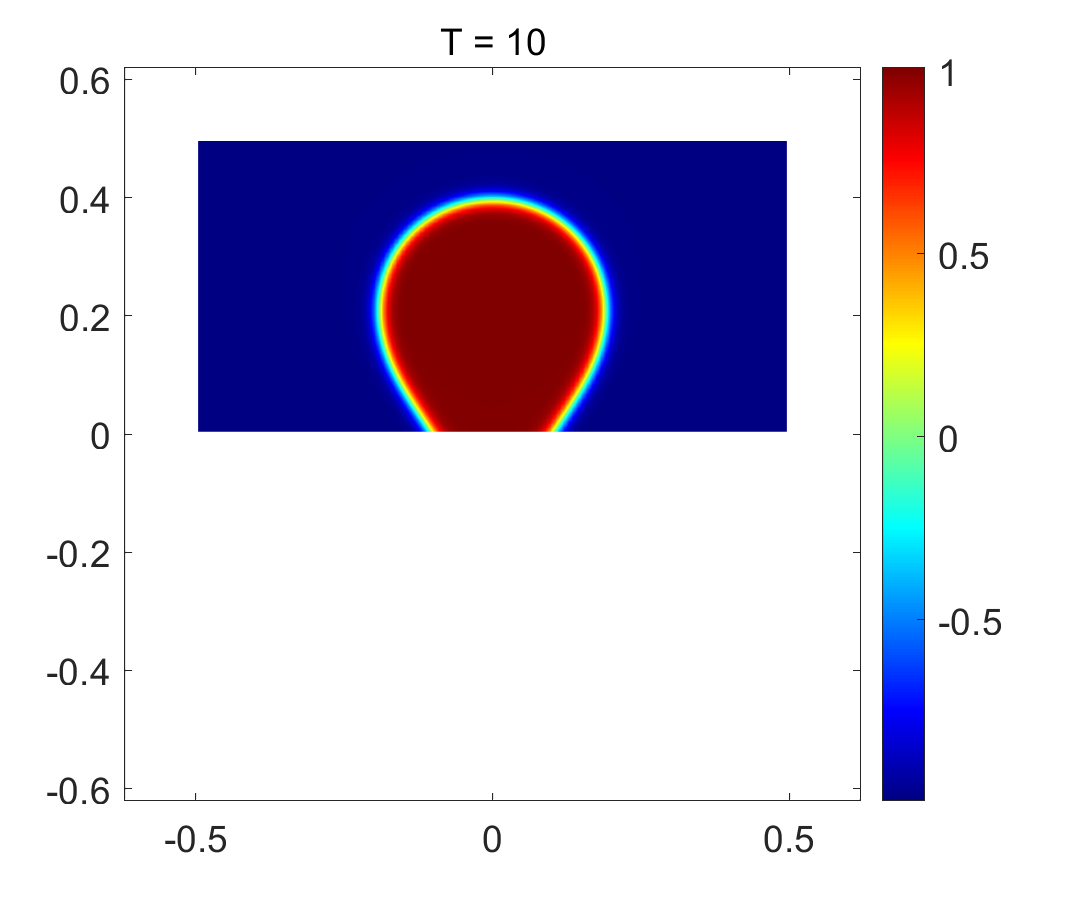}
			\includegraphics[width = 1.2in]{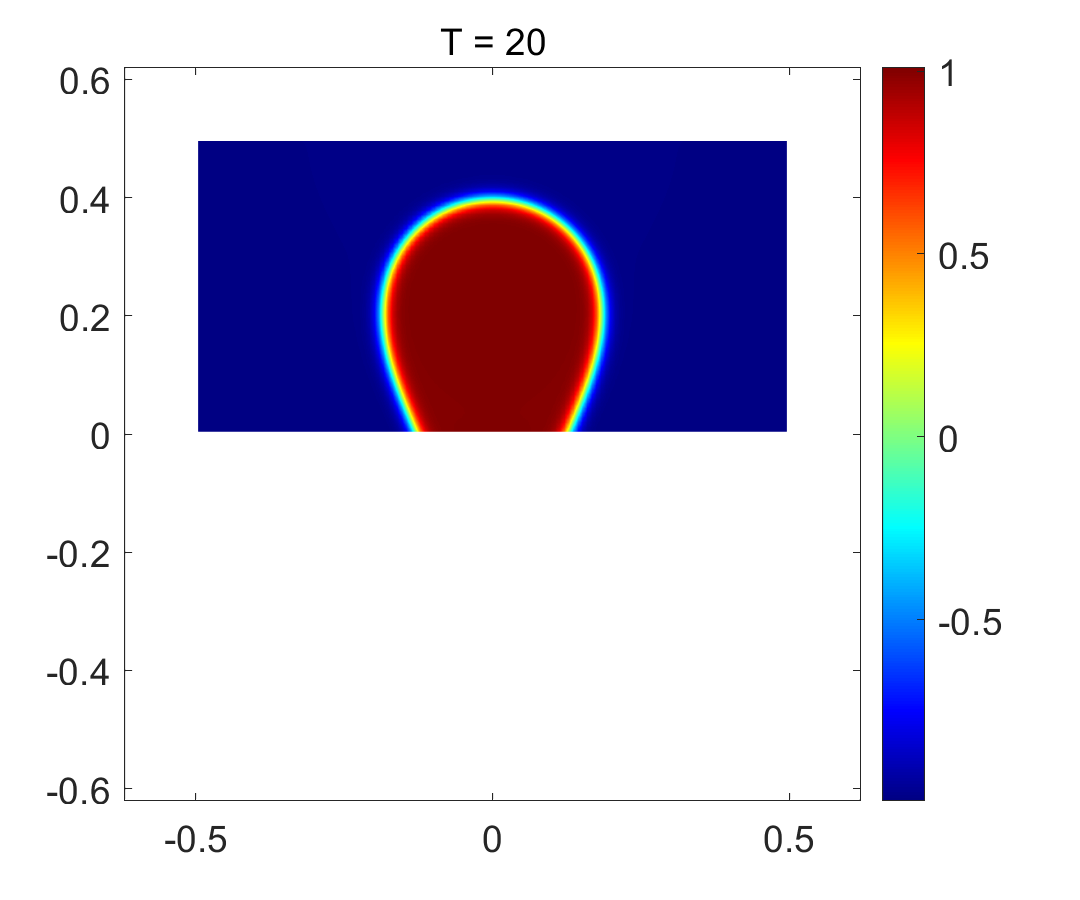}
			\includegraphics[width = 1.2in]{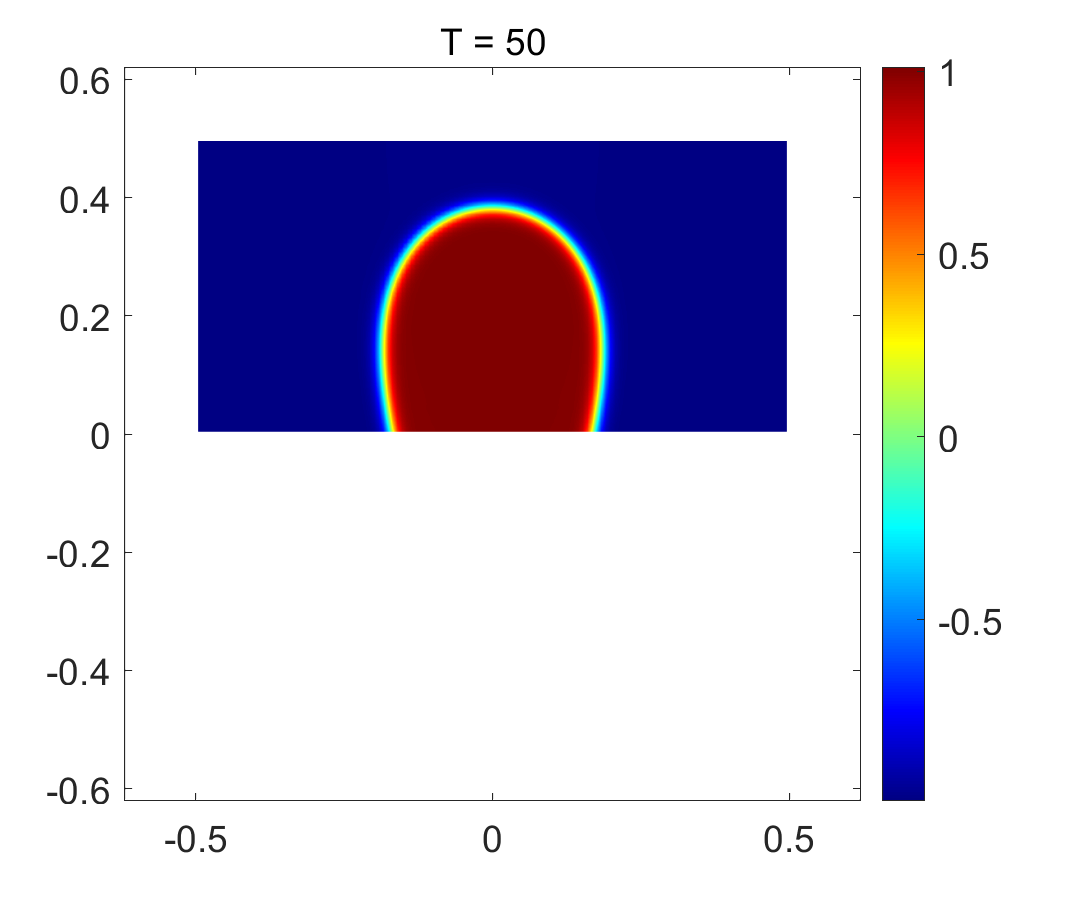}
			\includegraphics[width = 1.2in]{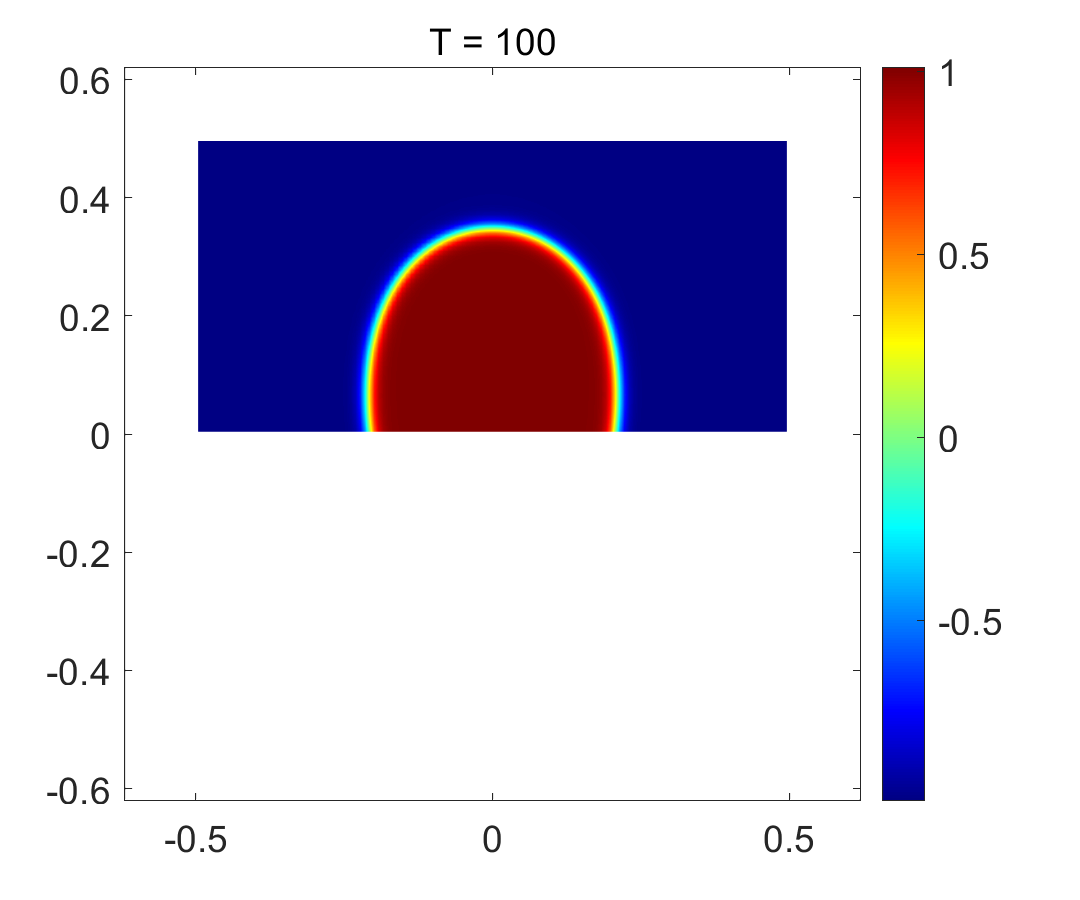}
			\includegraphics[width = 1.2in]{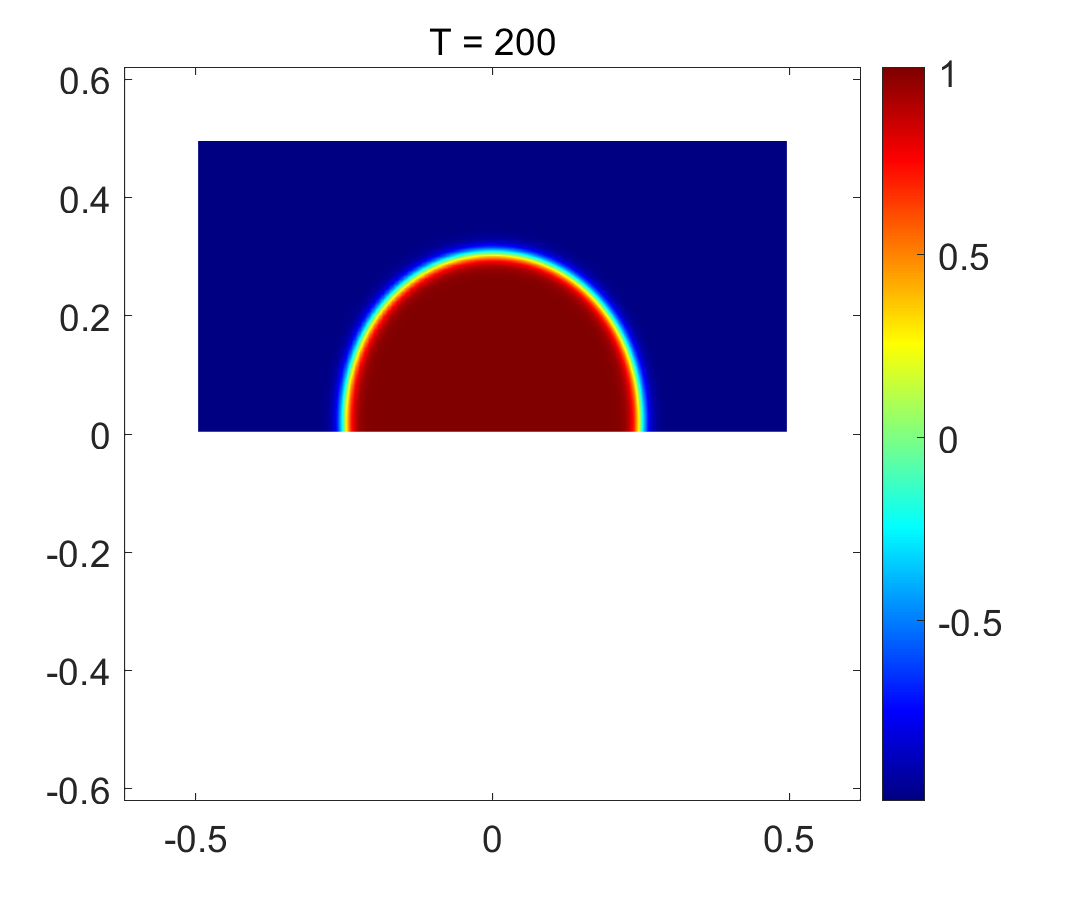}
		}
		\subfigure[]{
			\includegraphics[width = 1.2in]{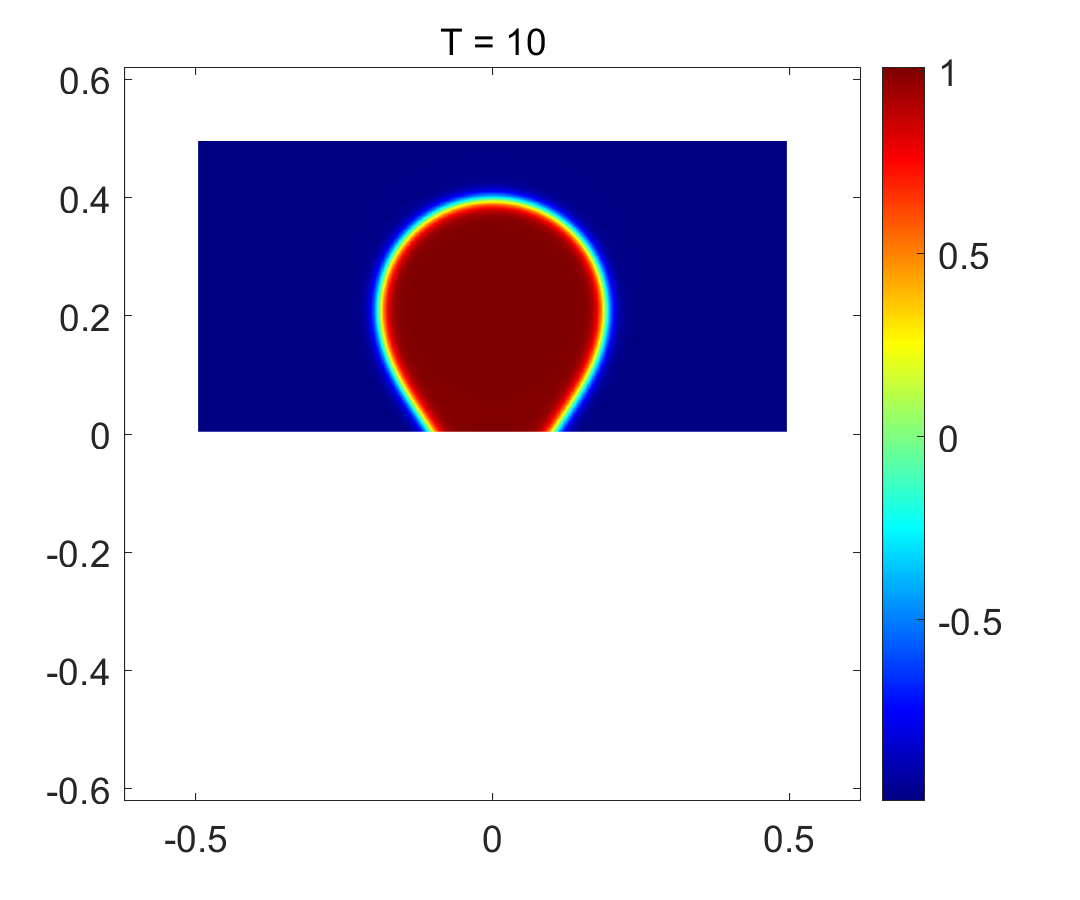}
			\includegraphics[width = 1.2in]{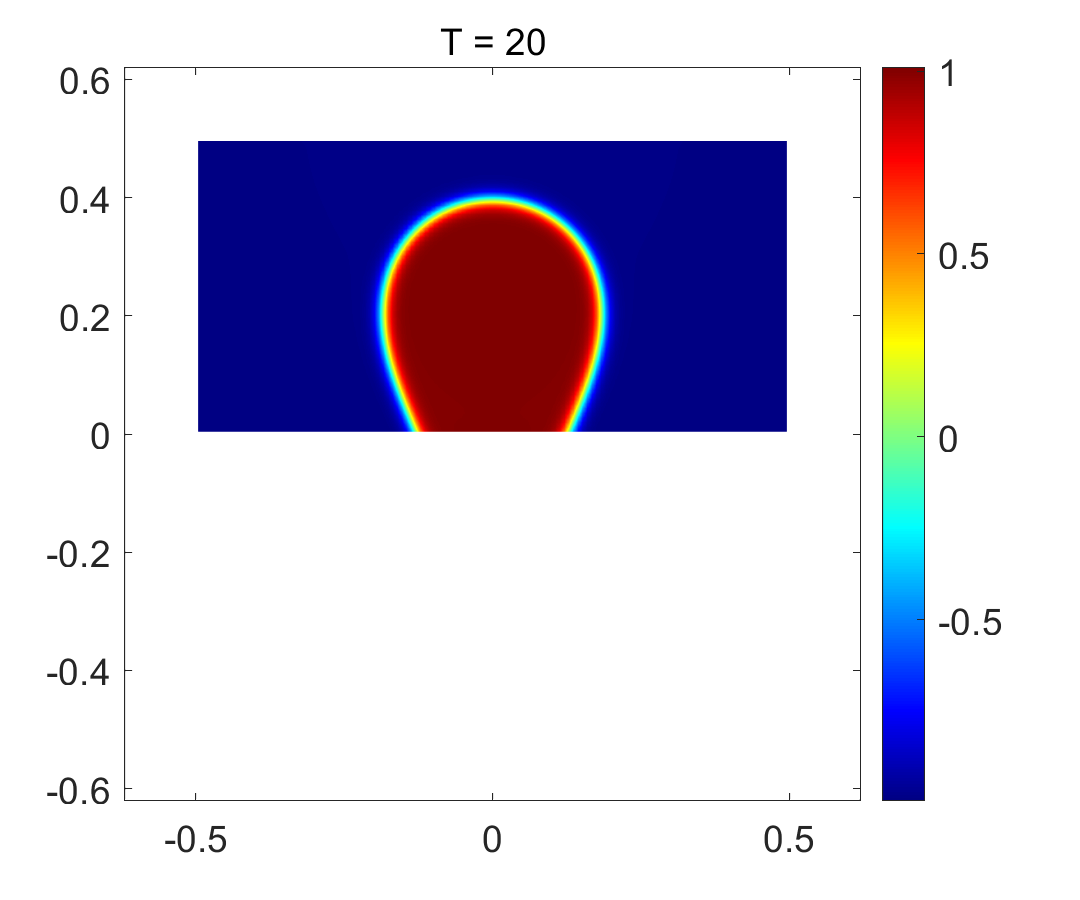}
			\includegraphics[width = 1.2in]{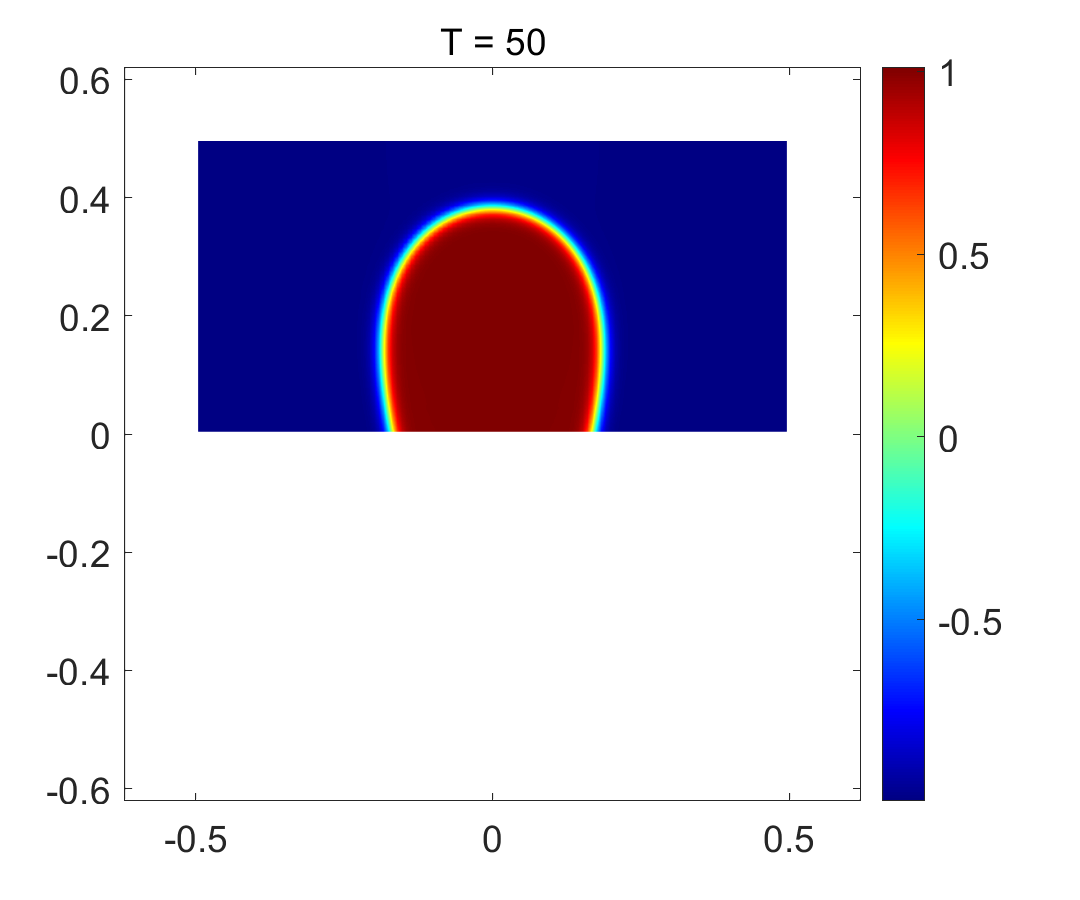}
			\includegraphics[width = 1.2in]{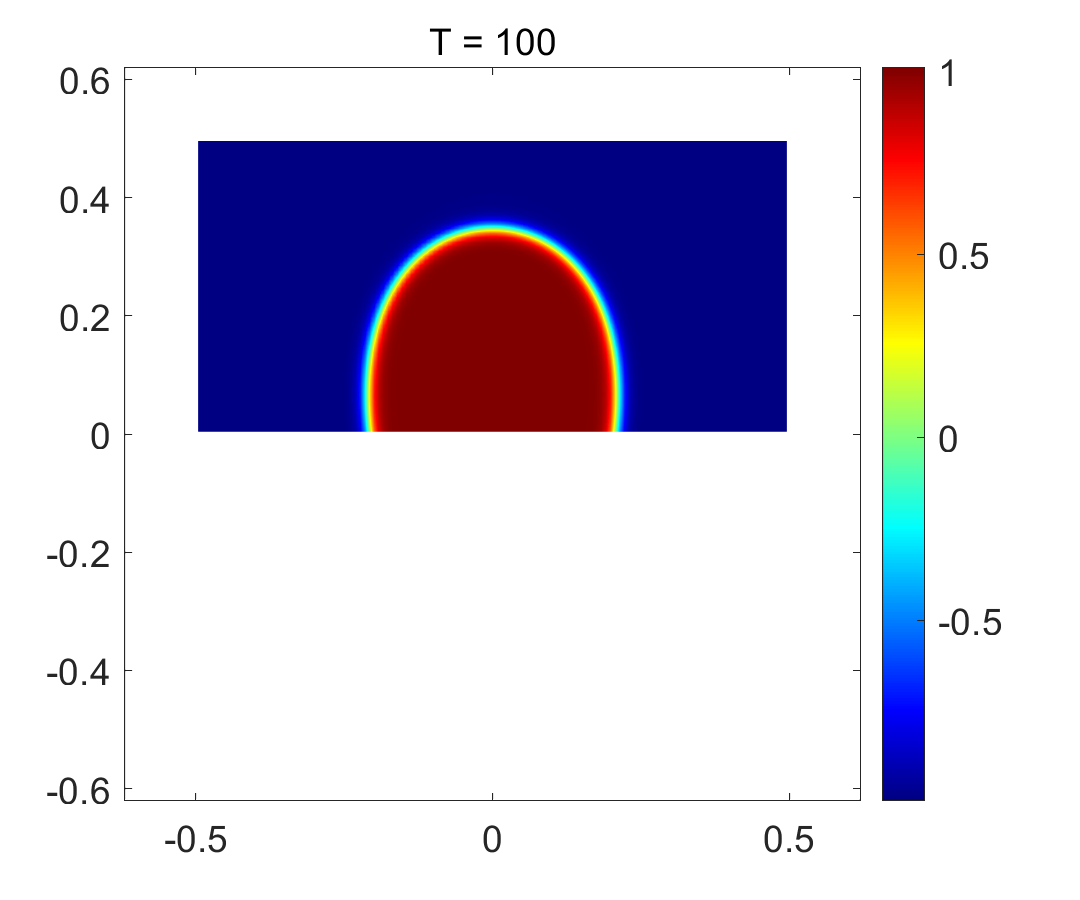}
			\includegraphics[width = 1.2in]{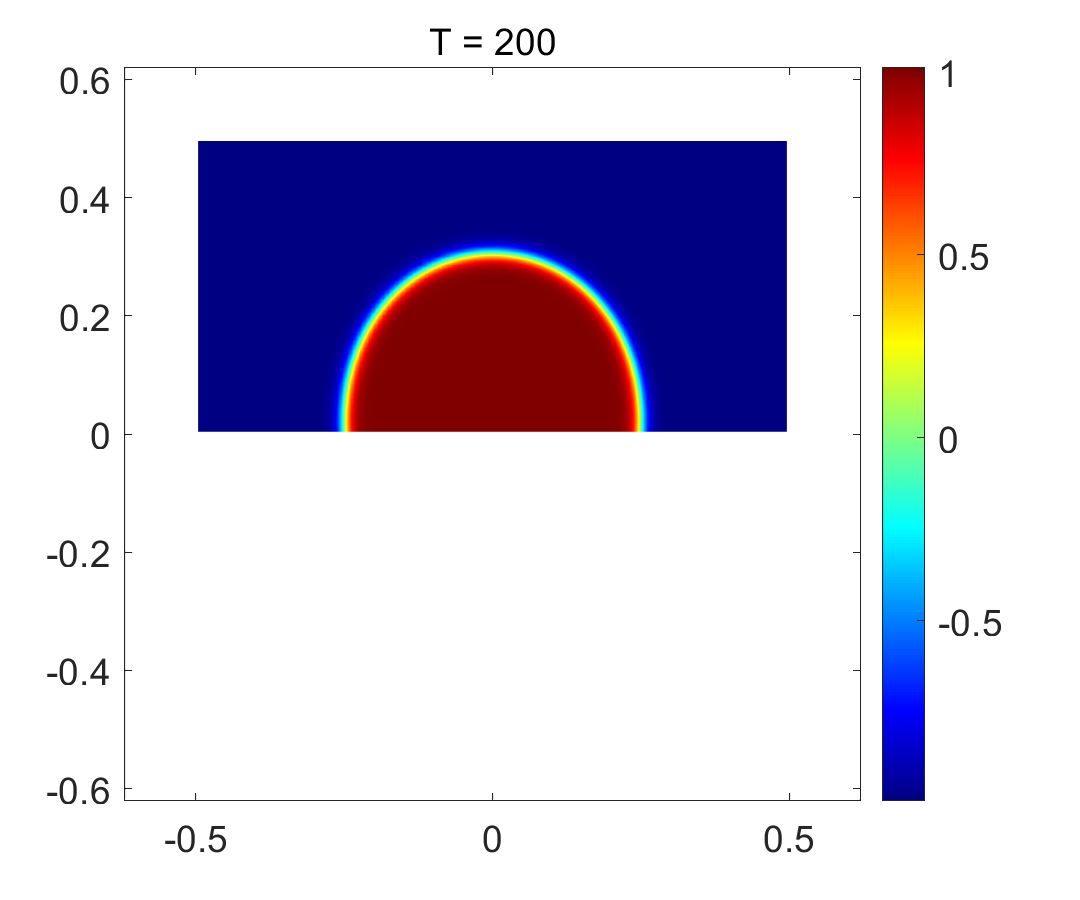}
		}
		\caption{Droplet spreading simulated using the Cahn-Hilliard model and the OPBDE Cahn-Hilliard model. Profiles of $\phi$ and $\tilde{\phi}$ are shown at time instants $T = 10$, $20$, $50$, $100$, $200$ for $\Gamma = 50$. (a) The Cahn-Hilliard model's $\phi$ simulated in $\Omega_1$; (b) The OPBDE Cahn-Hilliard model's $\tilde{\phi}$ simulated in $\Omega$.}
		\label{Droplet spreading2}
	\end{figure}
	
	\begin{figure}[H]
		\centering
		\subfigure[]{
			\includegraphics[width = 1.2in]{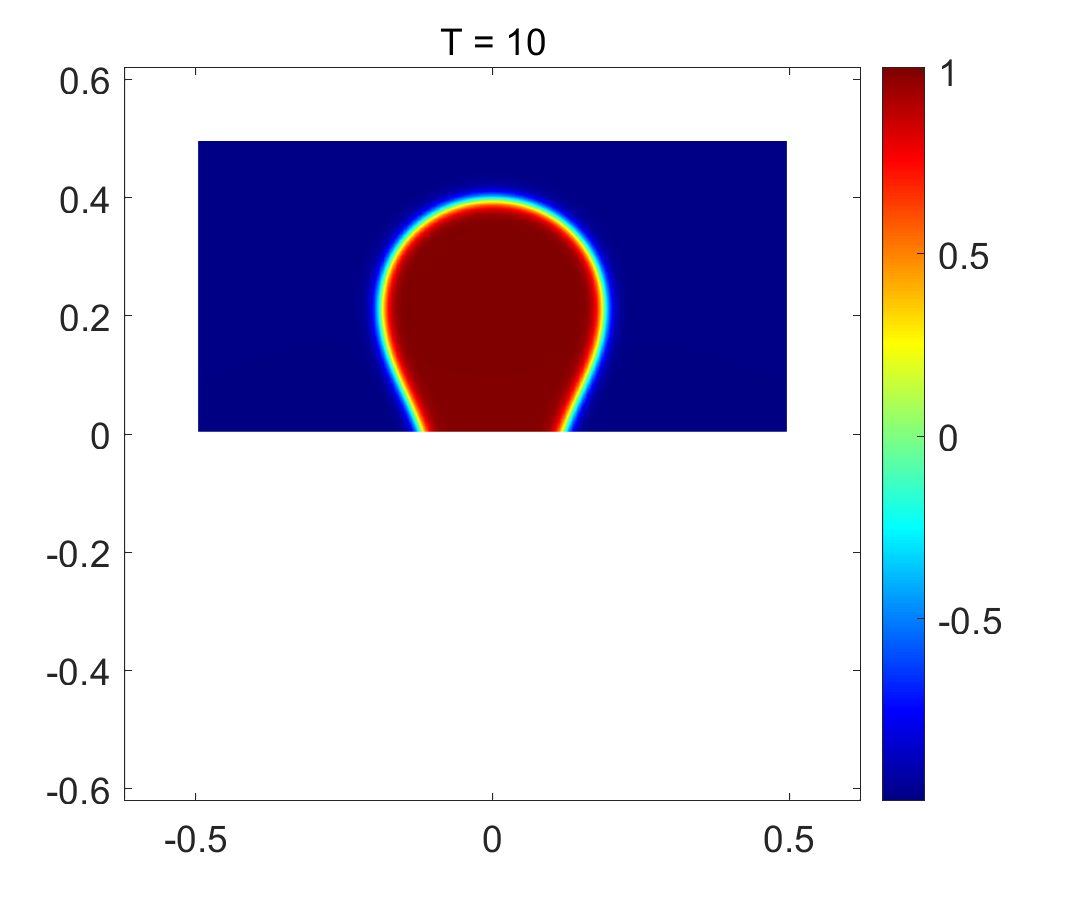}
			\includegraphics[width = 1.2in]{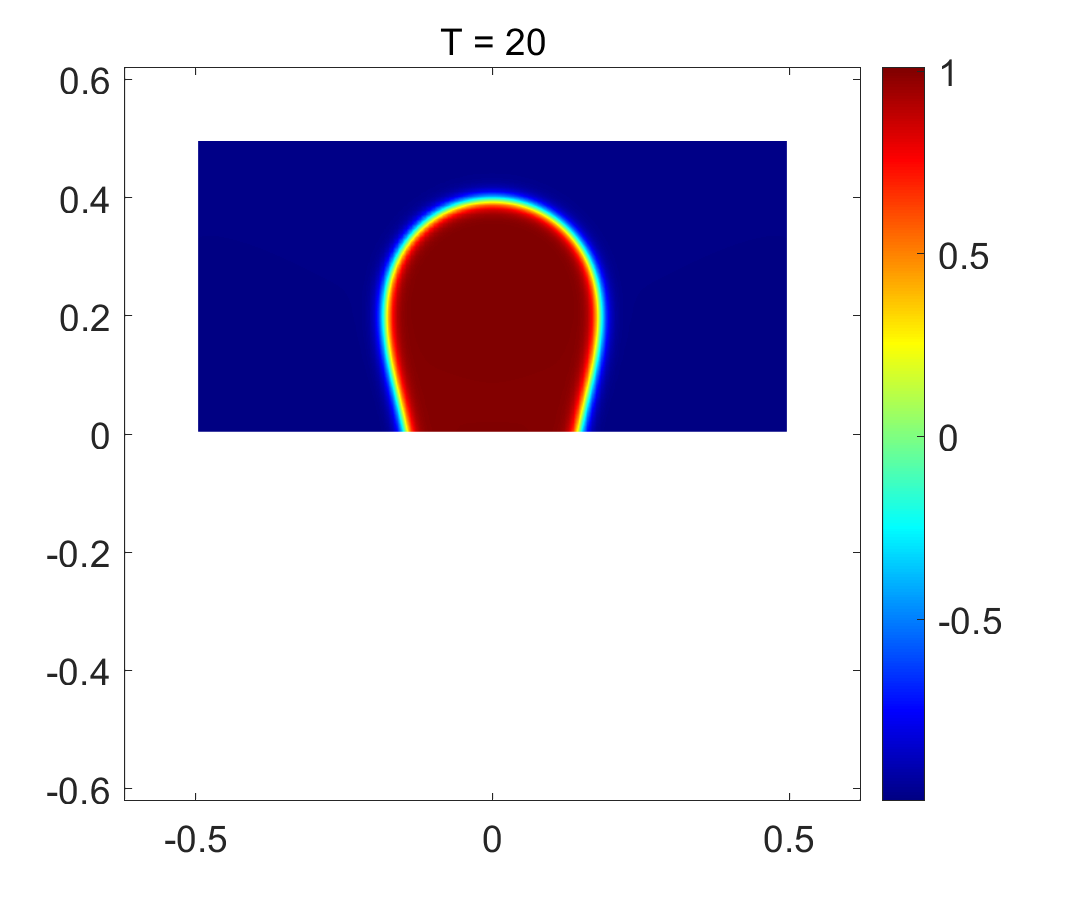}
			\includegraphics[width = 1.2in]{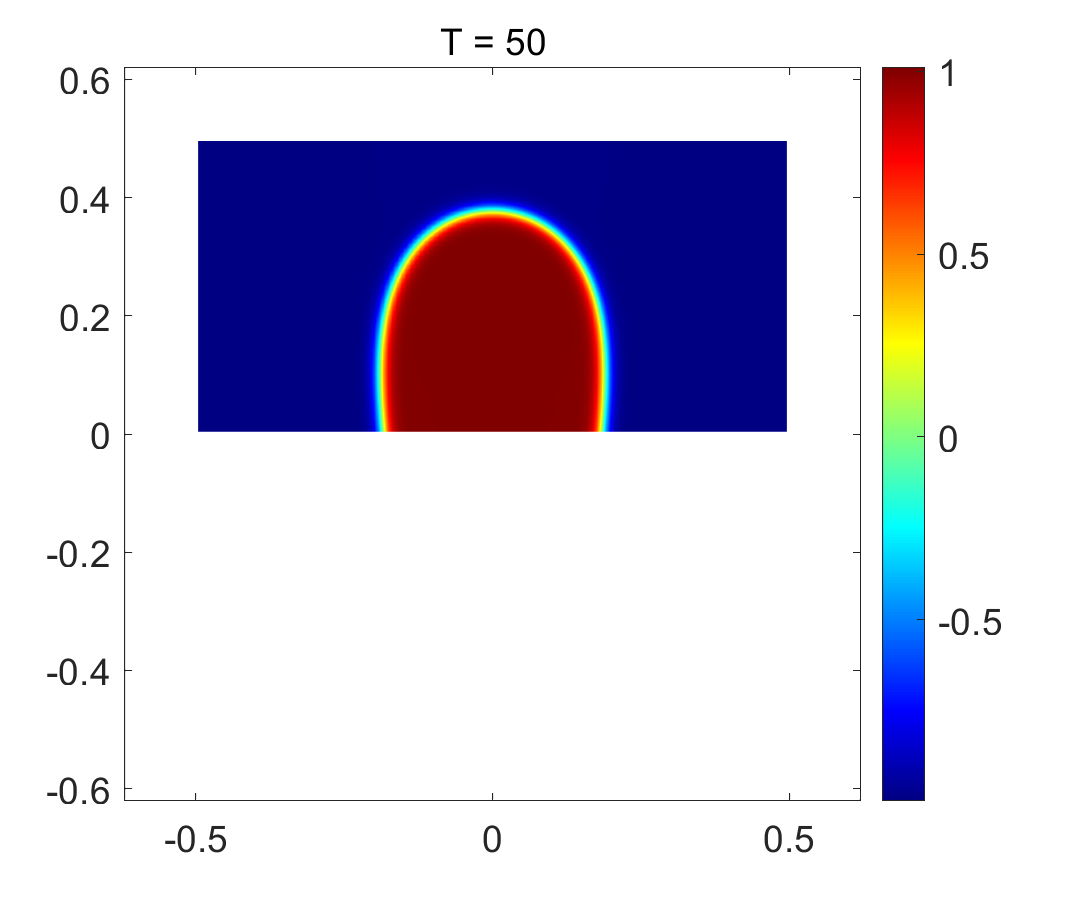}
			\includegraphics[width = 1.2in]{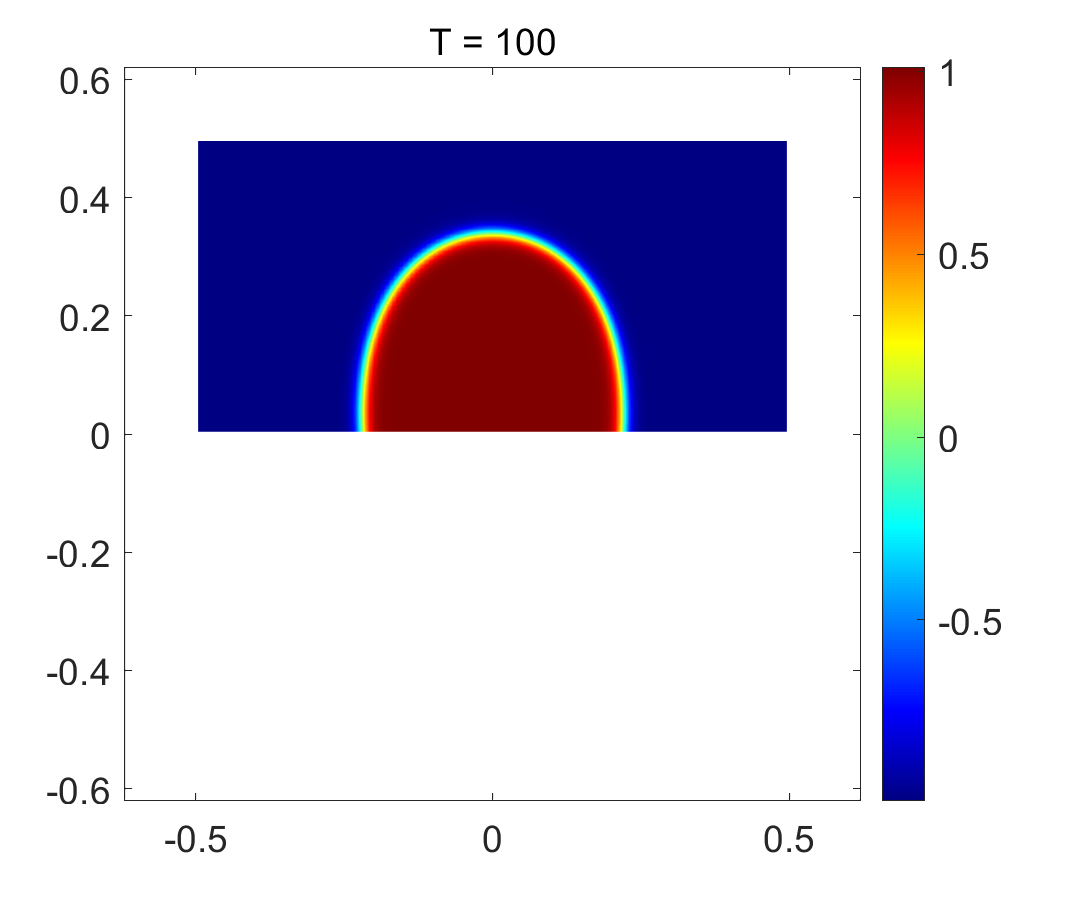}
			\includegraphics[width = 1.2in]{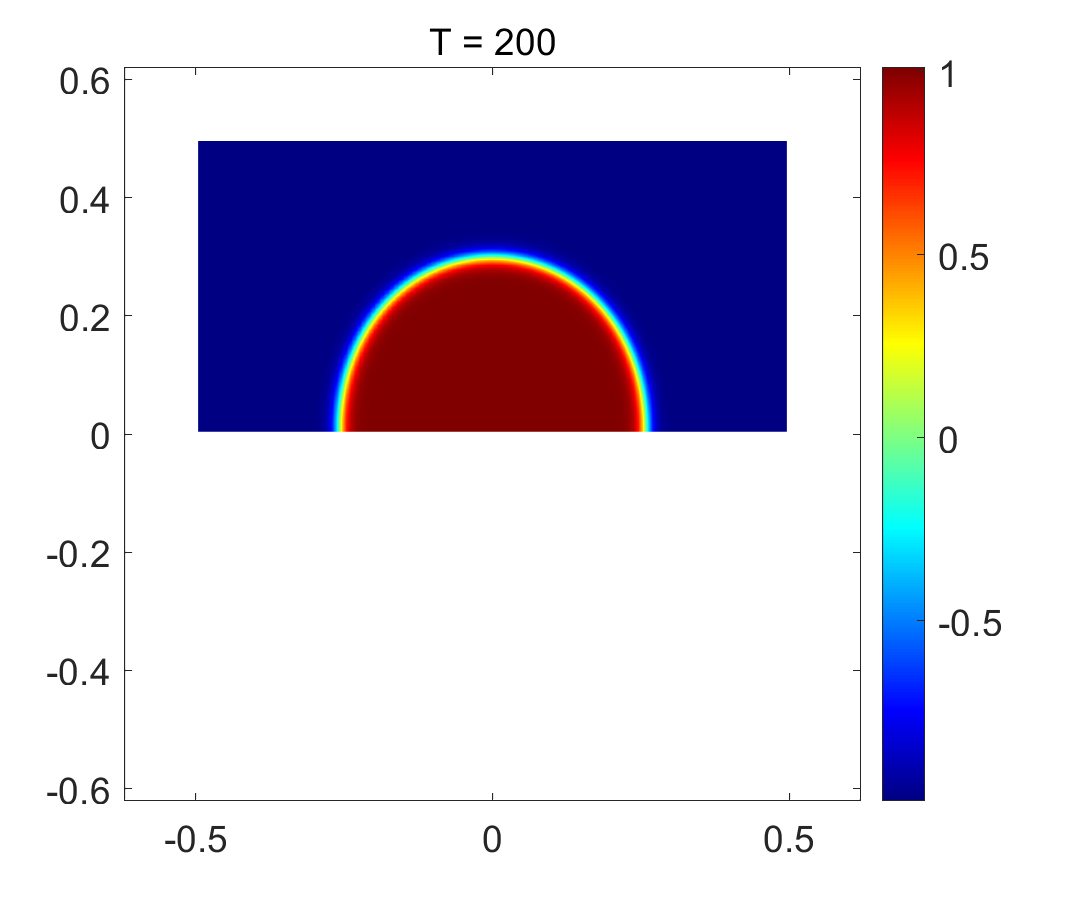}
		}
		\subfigure[]{
			\includegraphics[width = 1.2in]{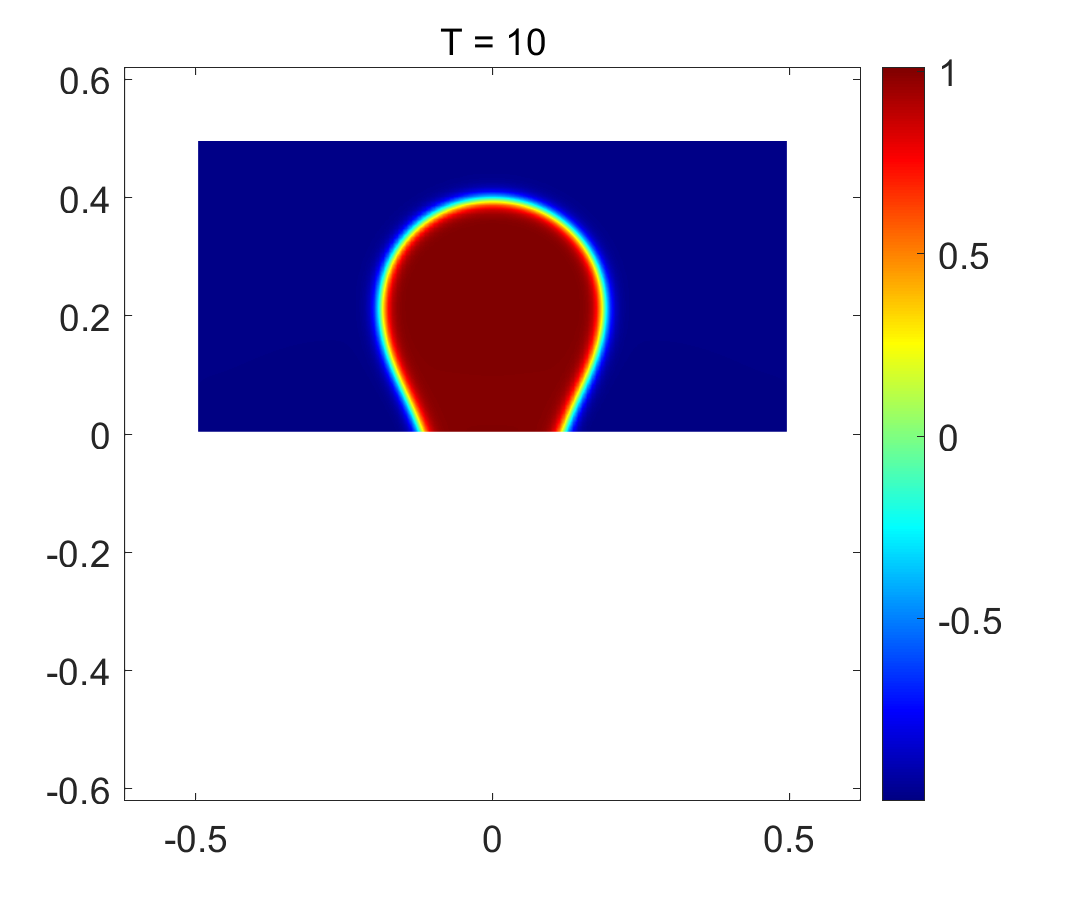}
			\includegraphics[width = 1.2in]{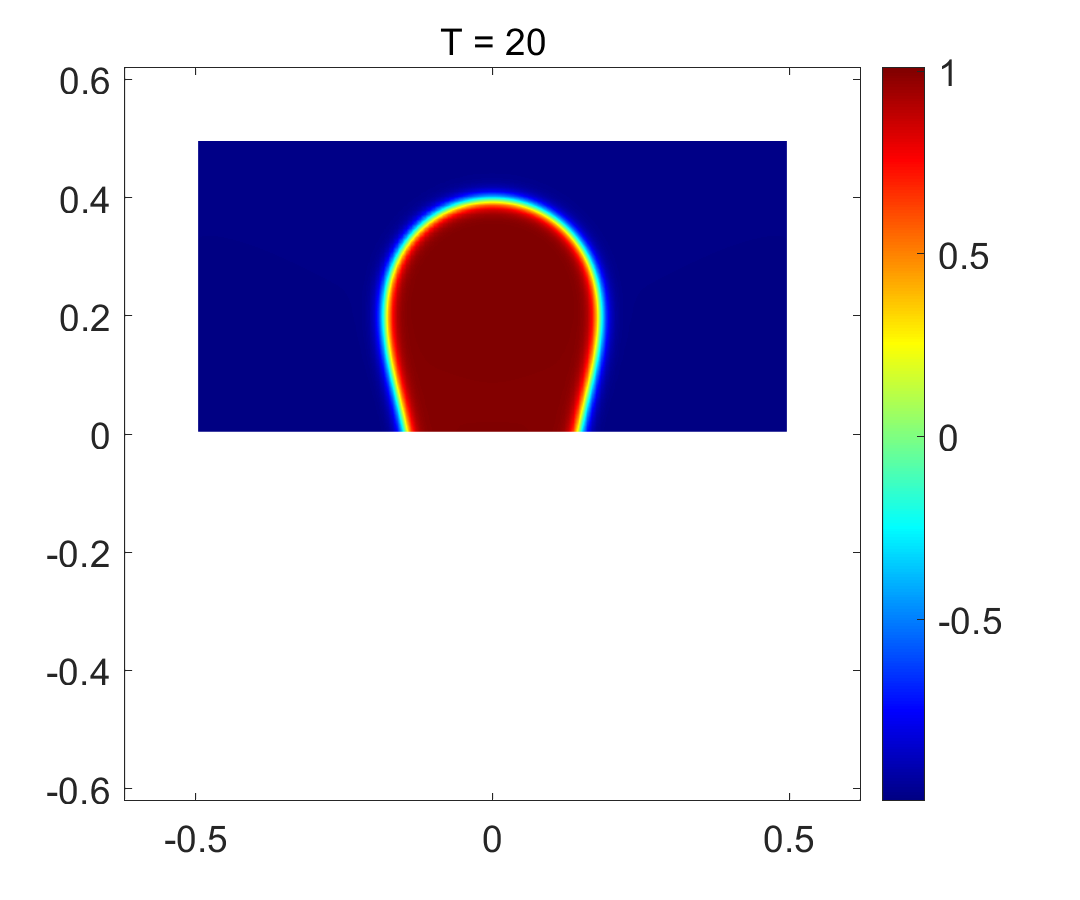}
			\includegraphics[width = 1.2in]{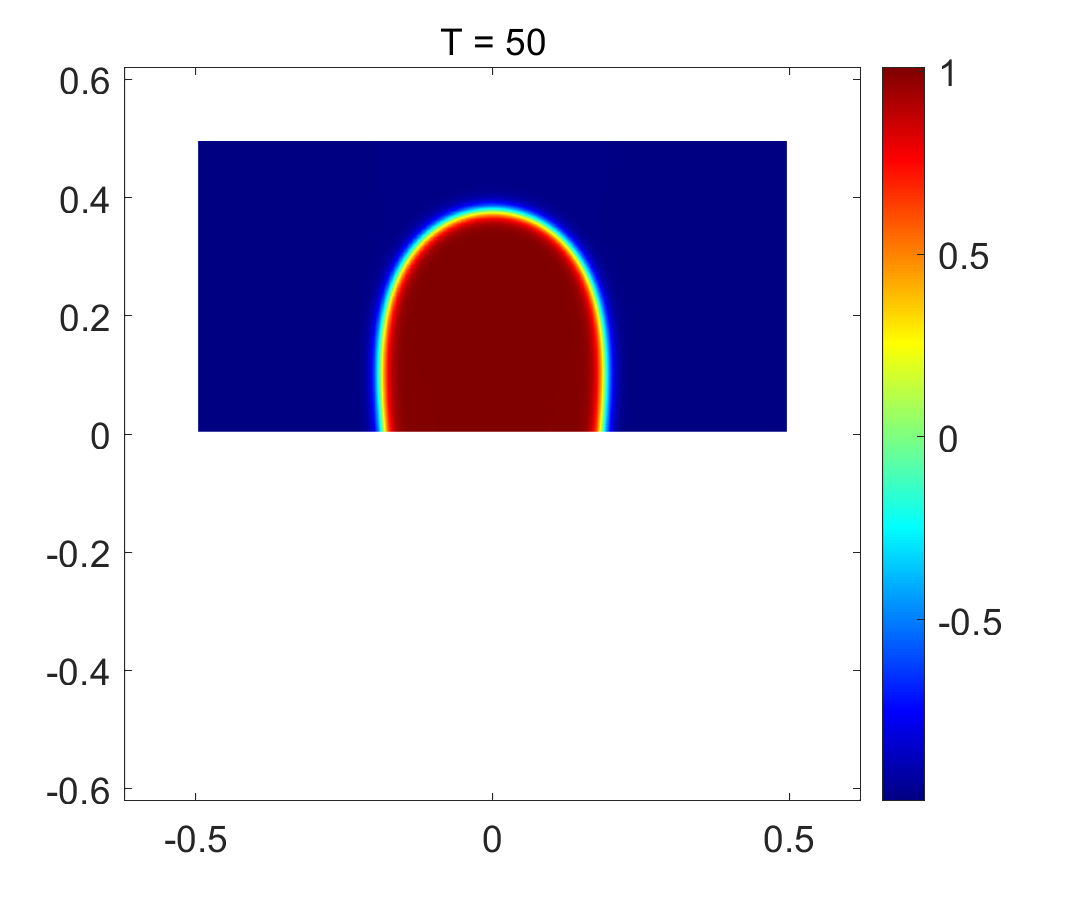}
			\includegraphics[width = 1.2in]{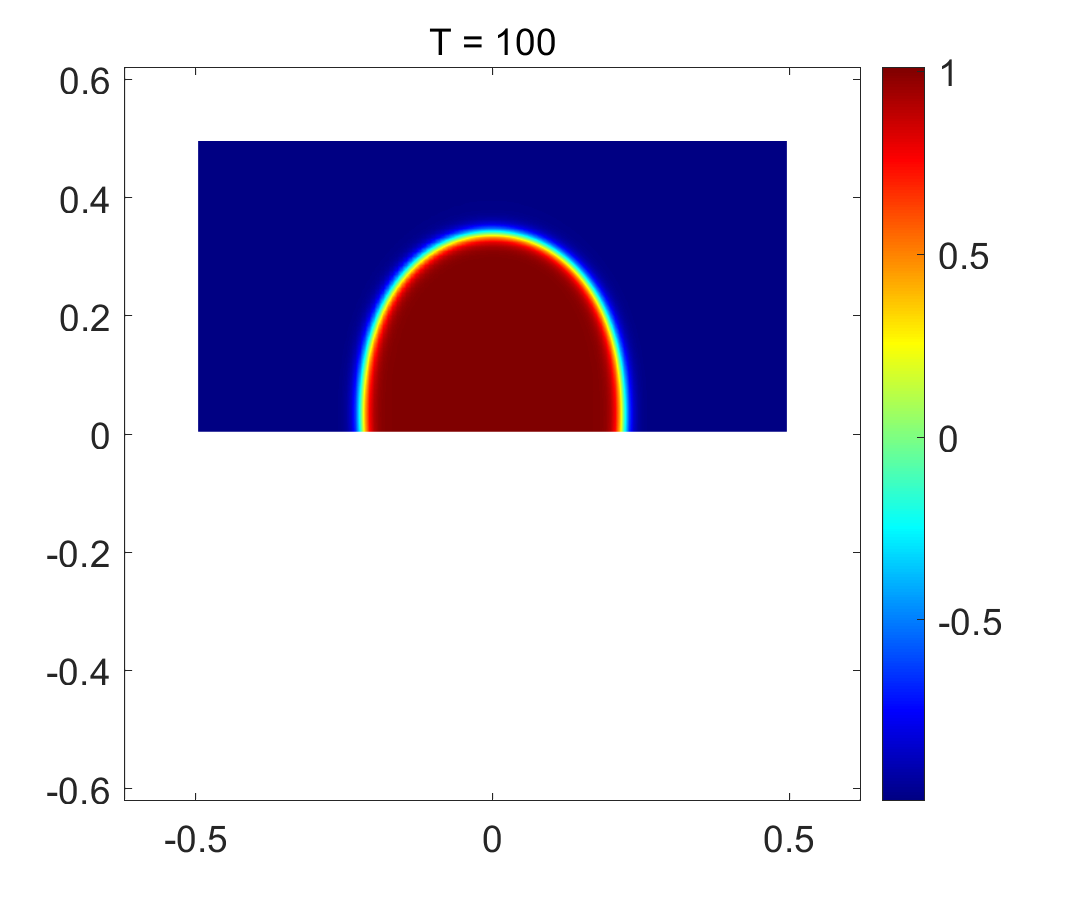}
			\includegraphics[width = 1.2in]{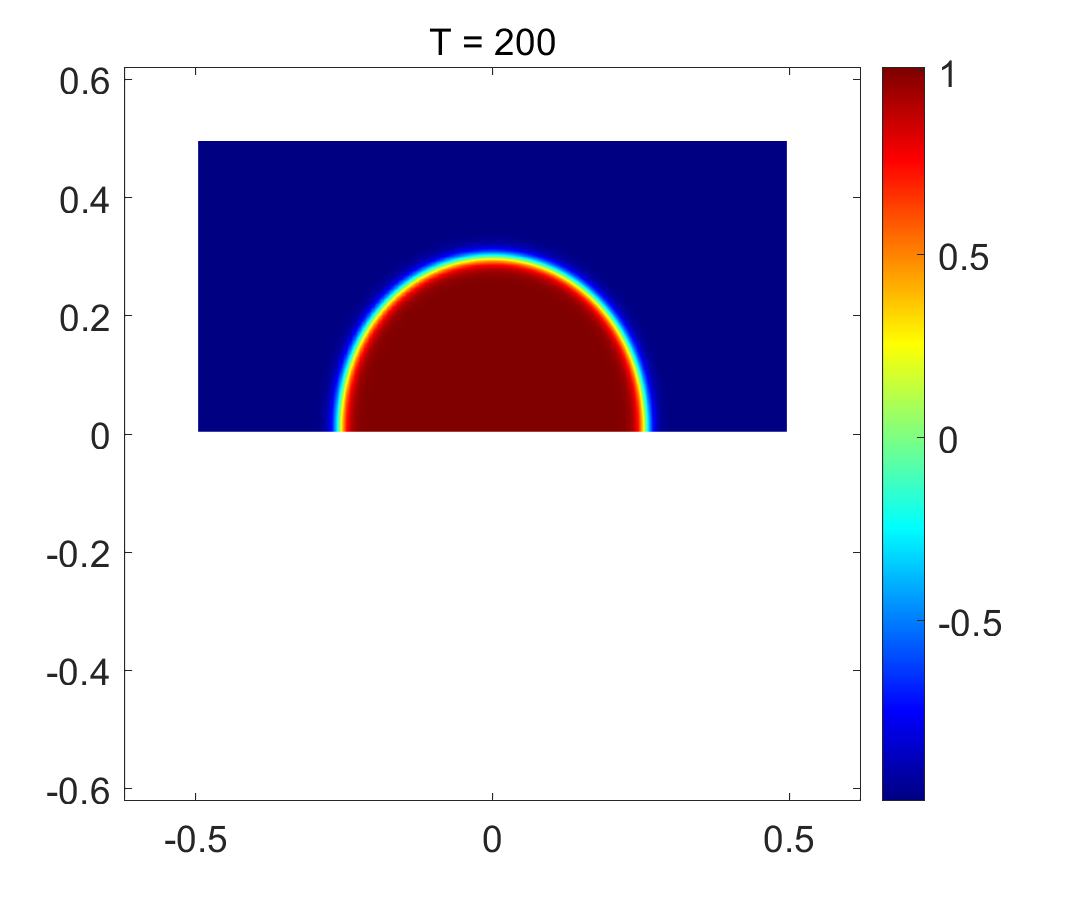}
		}
		\caption{Droplet spreading simulated using the Cahn-Hilliard model and the OPBDE Cahn-Hilliard model. Profiles of $\phi$ and $\tilde{\phi}$ are shown at time instants $T = 10$, $20$, $50$, $100$, $200$ for $\Gamma = 100$. (a) The Cahn-Hilliard model's $\phi$ simulated in $\Omega_1$; (b) The OPBDE Cahn-Hilliard model's $\tilde{\phi}$ simulated in $\Omega$.}
		\label{Droplet spreading3}
	\end{figure}
	
	Figures \ref{Droplet spreading1} to \ref{Droplet spreading3} show the droplet spreading simulated using the Cahn-Hilliard model and the OPBDE Cahn-Hilliard model. For each value of $\Gamma$, the two models produce nearly identical results. Furthermore, an increase in $\Gamma$ leads to a faster relaxation on the substrate with the contact angle getting closer to $90^\circ$.
	
	\begin{table}[H]
		\centering
		\begin{tabular}{cccc}
			\toprule[1.5pt]
			&$\Gamma = 10$ &$\Gamma = 50$ &$\Gamma = 100$ \\
			\midrule[1pt]
			T = 10 & $2.1861\times 10^{-5}$ & $1.8654\times 10^{-5}$ & $8.8009\times 10^{-6}$\\
			T = 20 & $3.7552\times 10^{-5}$ & $1.8201\times 10^{-5}$ & $2.4684\times 10^{-5}$\\
			T = 50 & $9.7414\times 10^{-5}$ & $8.1467\times 10^{-5}$ & $7.7564\times 10^{-5}$\\
			T = 100 &$2.2358\times 10^{-4}$ & $8.5598\times 10^{-5}$ & $5.6424\times 10^{-5}$\\
			T = 200 &$4.5011\times 10^{-4}$ & $3.1299\times 10^{-4}$ & $2.8396\times 10^{-4}$\\
			\bottomrule[1.5pt]
		\end{tabular}
		\caption{Errors in $L^2$ norm for the OPBDE Cahn-Hilliard model results versus the original Cahn-Hilliard model results at time instants $T = 10$, $20$, $50$, $100$, $200$ for $\Gamma = 10$, $50$, $100$.}
		\label{Error2}
	\end{table}
	
	Table \ref{Error2} presents the $L^2$ error between the Cahn-Hilliard model and the OPBDE Cahn-Hilliard model. In the presence of boundary relaxation at finite $\Gamma$, the greatest error is on the order of $10^{-4}$ at $T = 200$. Considering that the terminal time instant here is $T = 200$ compared to the previous results for the terminal time instant $T = 0.1$ in Table \ref{Error1}, we note that there is some accumulation of error in time but still within an acceptable range.
	
	\begin{figure}[H]
		\centering
		\subfigure[]{\includegraphics[width = 0.3 \textwidth]{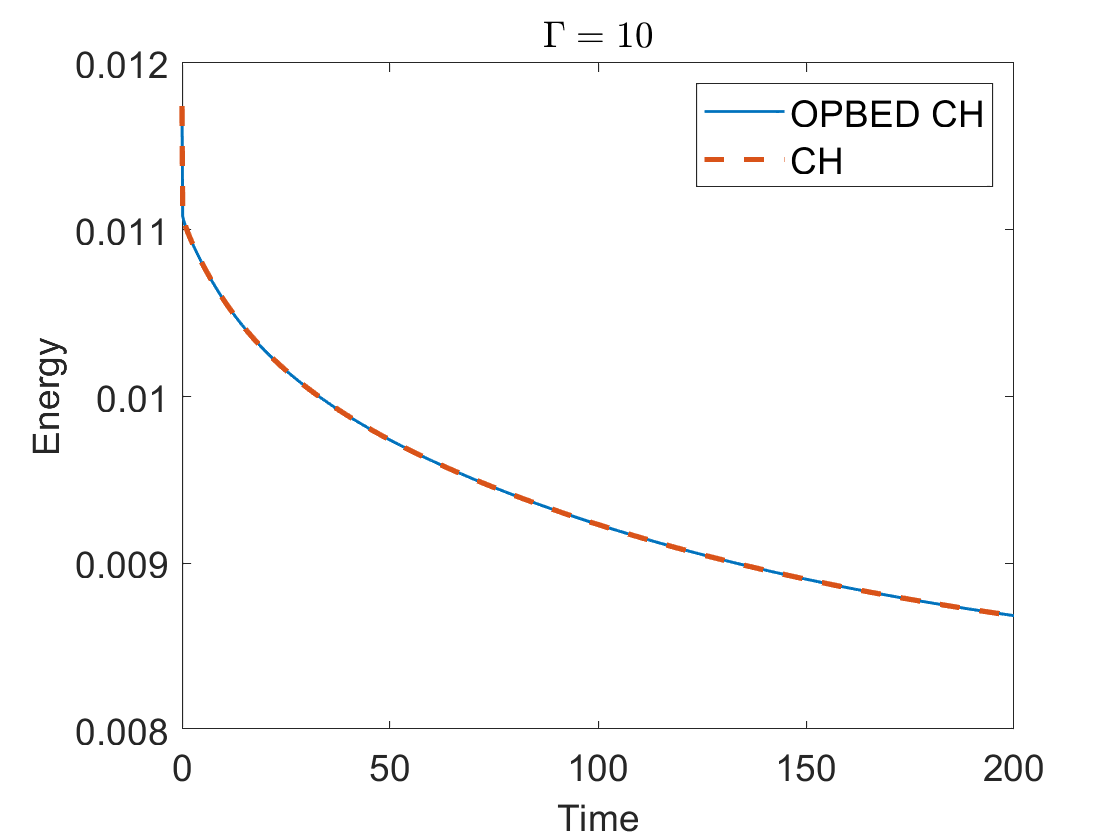}}
		\subfigure[]{\includegraphics[width = 0.3 \textwidth]{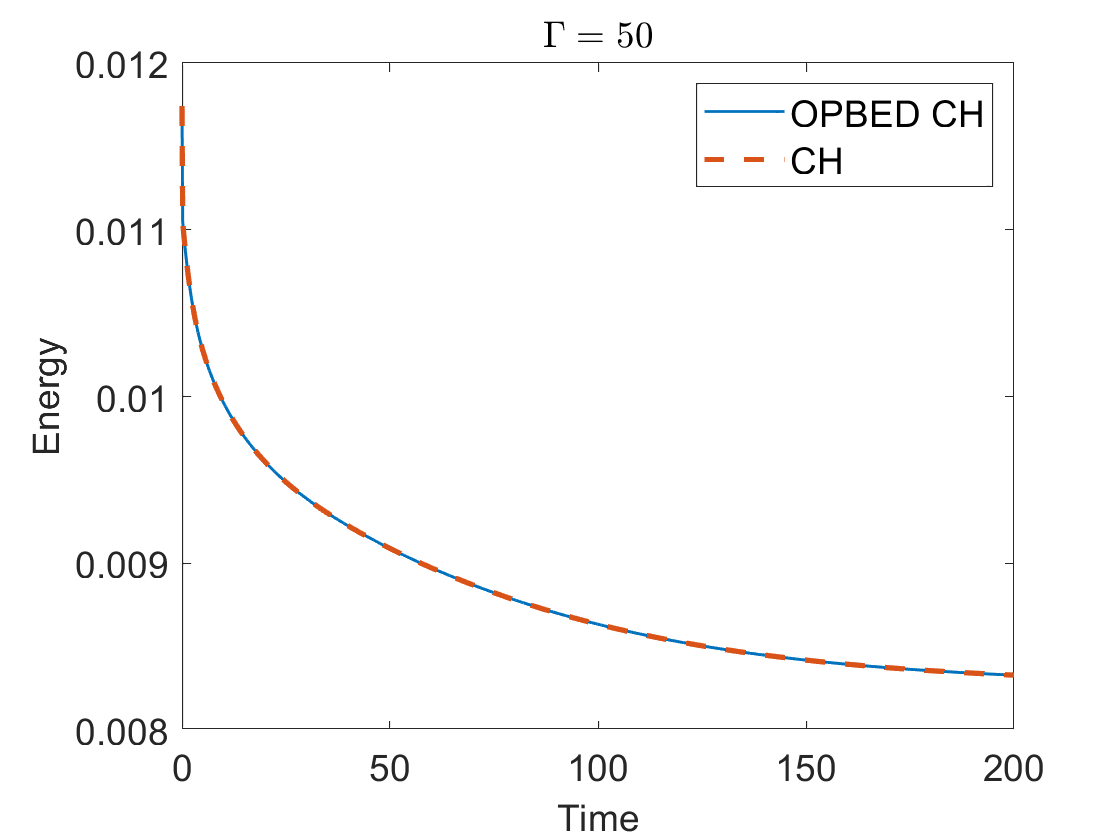}}
		\subfigure[]{\includegraphics[width = 0.3 \textwidth]{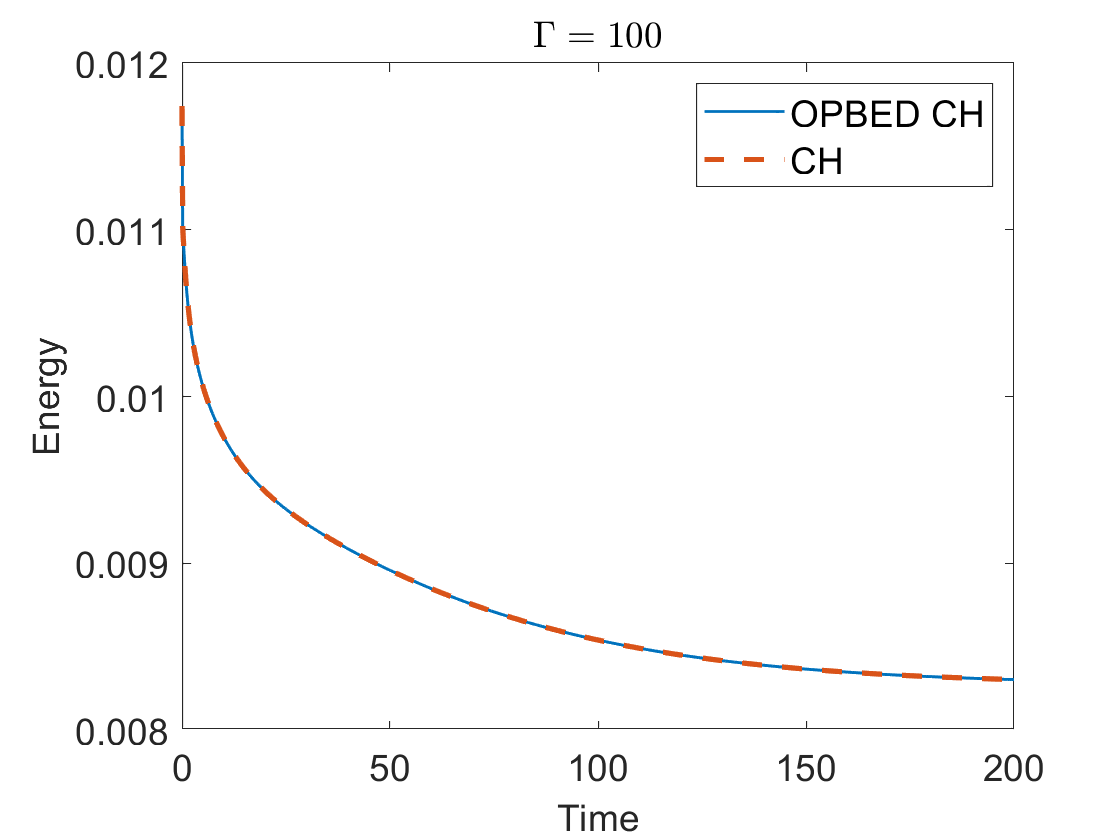}}
		\caption{Comparison of energy evolution between the Cahn-Hilliard model and the OPBDE Cahn-Hilliard model. (a) $\Gamma = 10$; (b) $\Gamma = 50$; (c) $\Gamma = 100$. Larger $\Gamma$ leads to faster energy decay. }
		\label{Energy2}
	\end{figure}
	
	\begin{figure}[H]
		\centering
		\subfigure[]{
			\includegraphics[width = 0.3 \textwidth]{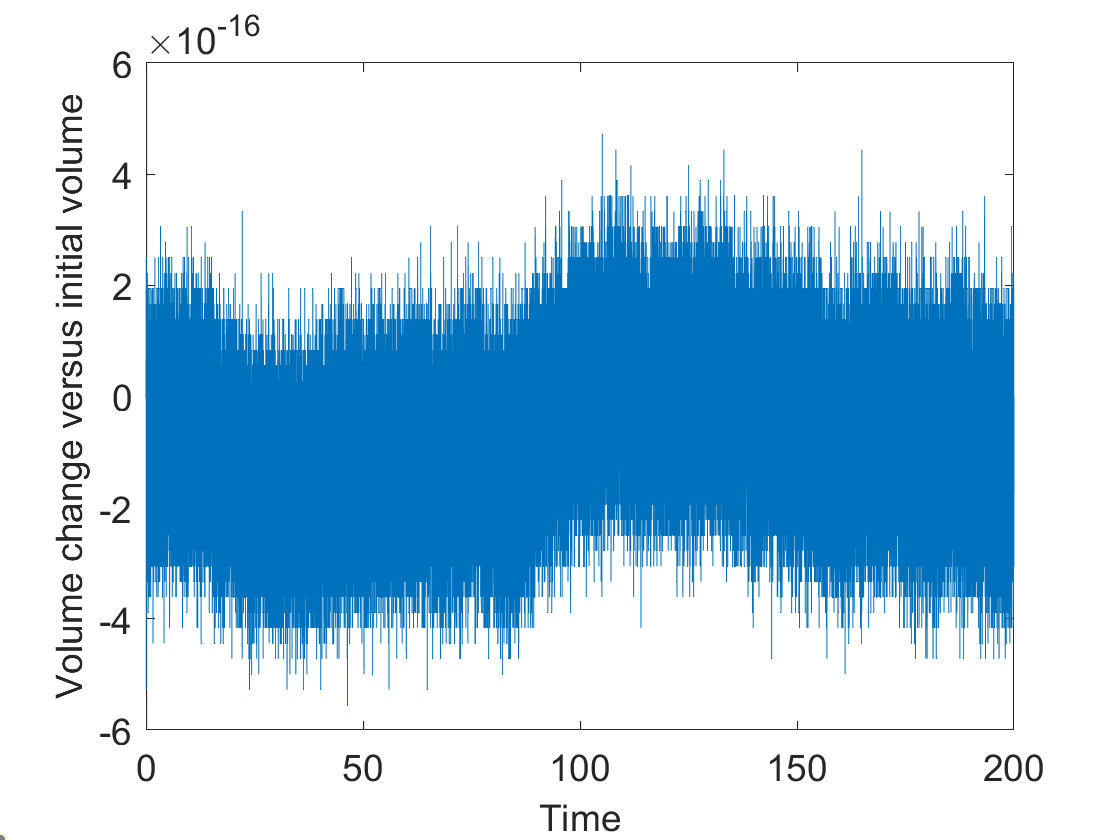}
			\includegraphics[width = 0.3 \textwidth]{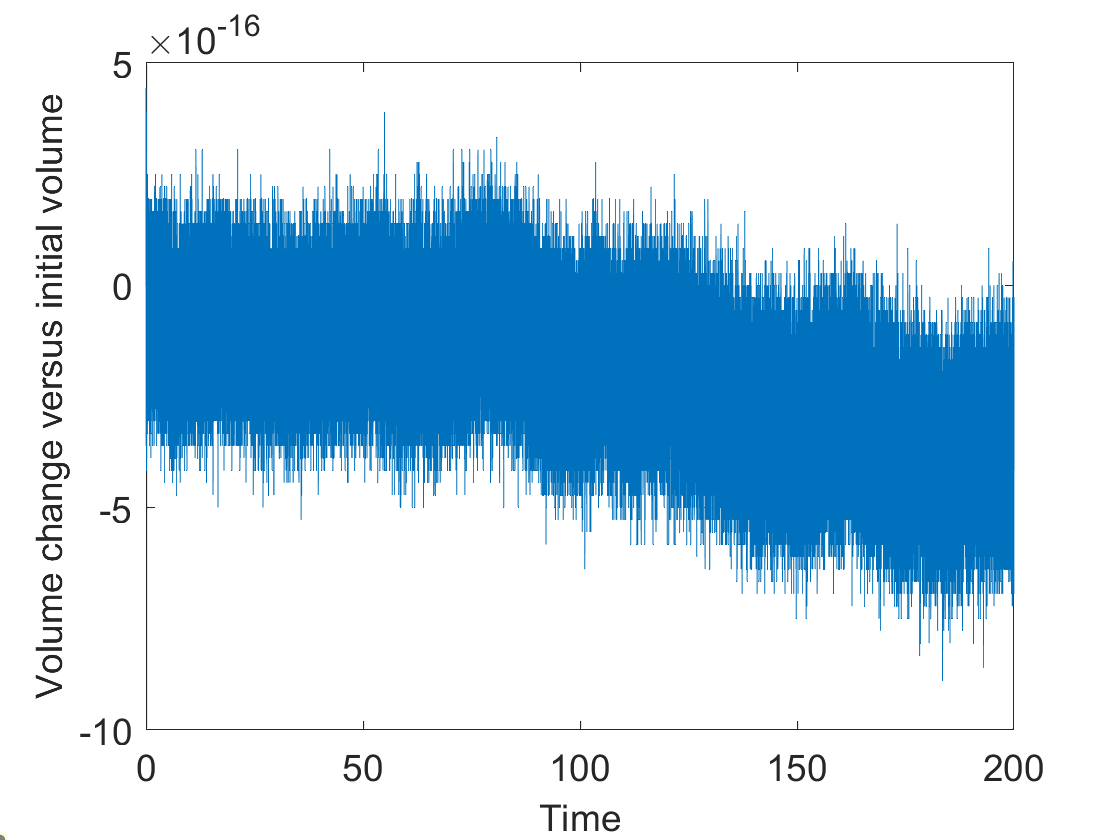}
			\includegraphics[width = 0.3 \textwidth]{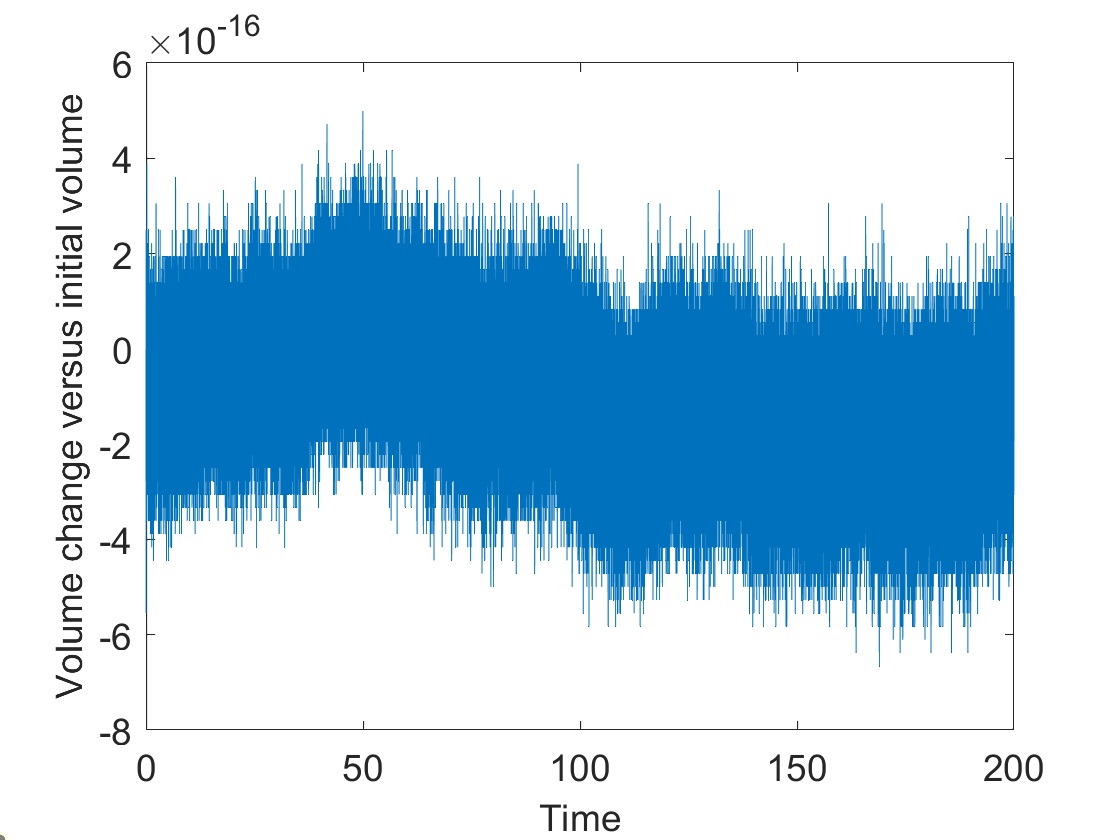}
		}
		
		\subfigure[]{
			\includegraphics[width = 0.3 \textwidth]{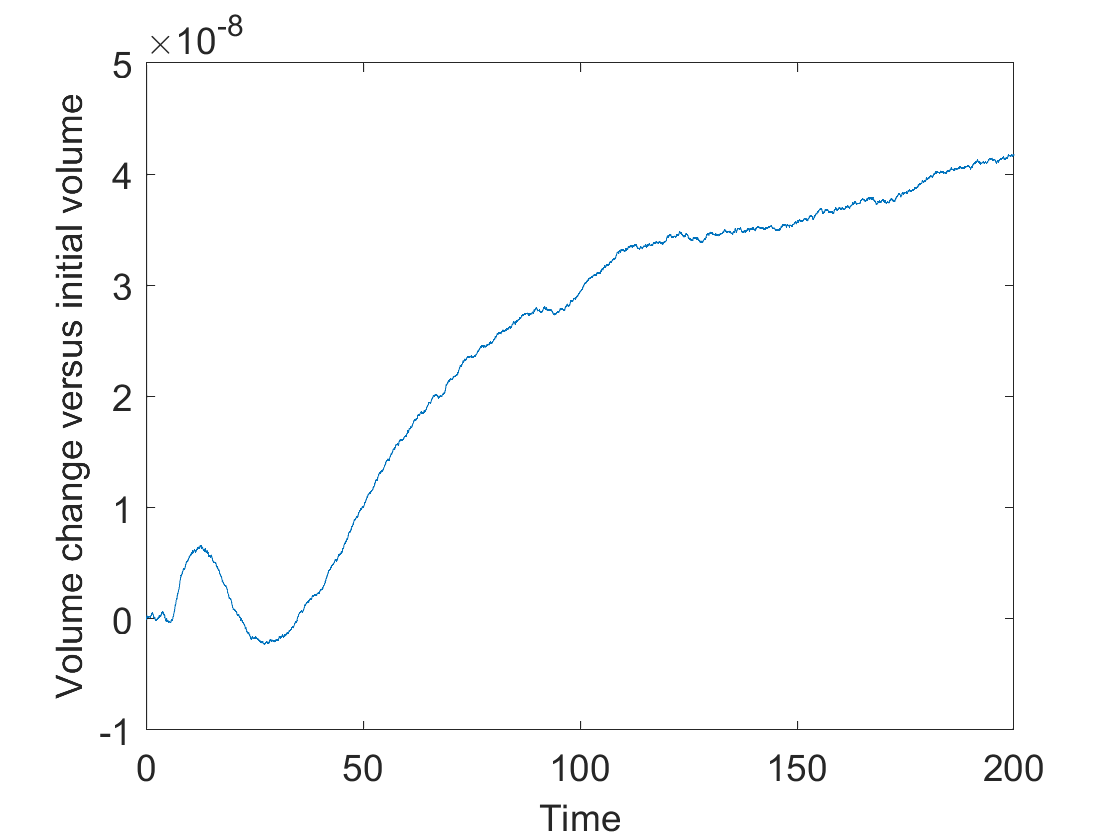}
			\includegraphics[width = 0.3 \textwidth]{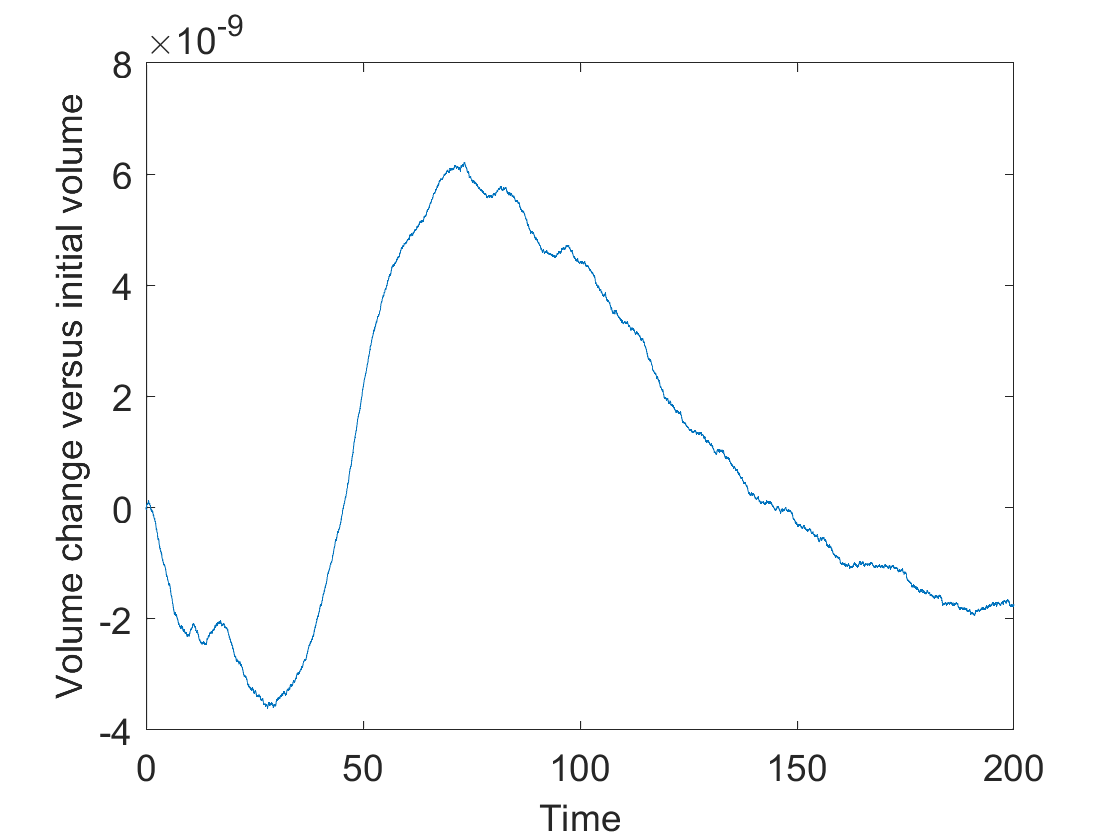}
			\includegraphics[width = 0.3 \textwidth]{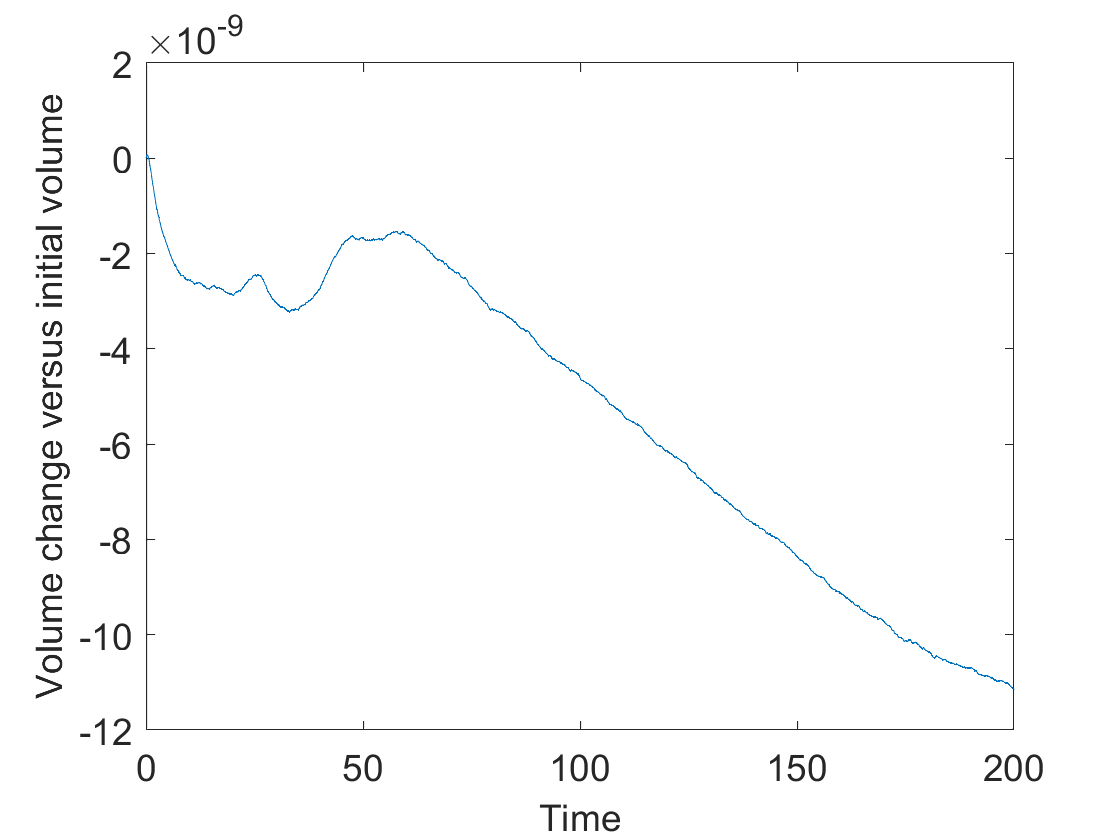}
		}
		\caption{Volume change versus initial volume as time proceeds. (a) The Cahn-Hilliard model; (b) The OPBDE Cahn-Hilliard model. From the left to the right are results for $\Gamma = 10$, $50$, $100$.}
		\label{Volume2}
	\end{figure}
	
	Figures \ref{Energy2} and \ref{Volume2} show the energy decay and volume conservation produced from the Cahn-Hilliard model and the OPBDE Cahn-Hilliard model for $\Gamma = 10$, $50$, $100$. Remarkable agreement between the two models is obtain for the energy decay. As to the volume change, the OPBDE Cahn-Hilliard model produces results within an acceptable range. These are consistent with theorems \ref{Thm1} and \ref{Thm2}.
	
	Finally we simulate the drop spreading on a curved substrate, where the shape of a droplet deforms to some extent in response to the shape of the substrate. All parameter values remain the same except for the profile of $\psi$ which describes the curved substrate. Simulations are performed using the OPBDE Cahn-Hilliard model for $\Gamma = 20$, $50$, $100$, respectively.
	
	\begin{figure}[H]
		\centering
		\includegraphics[width = 1.2in]{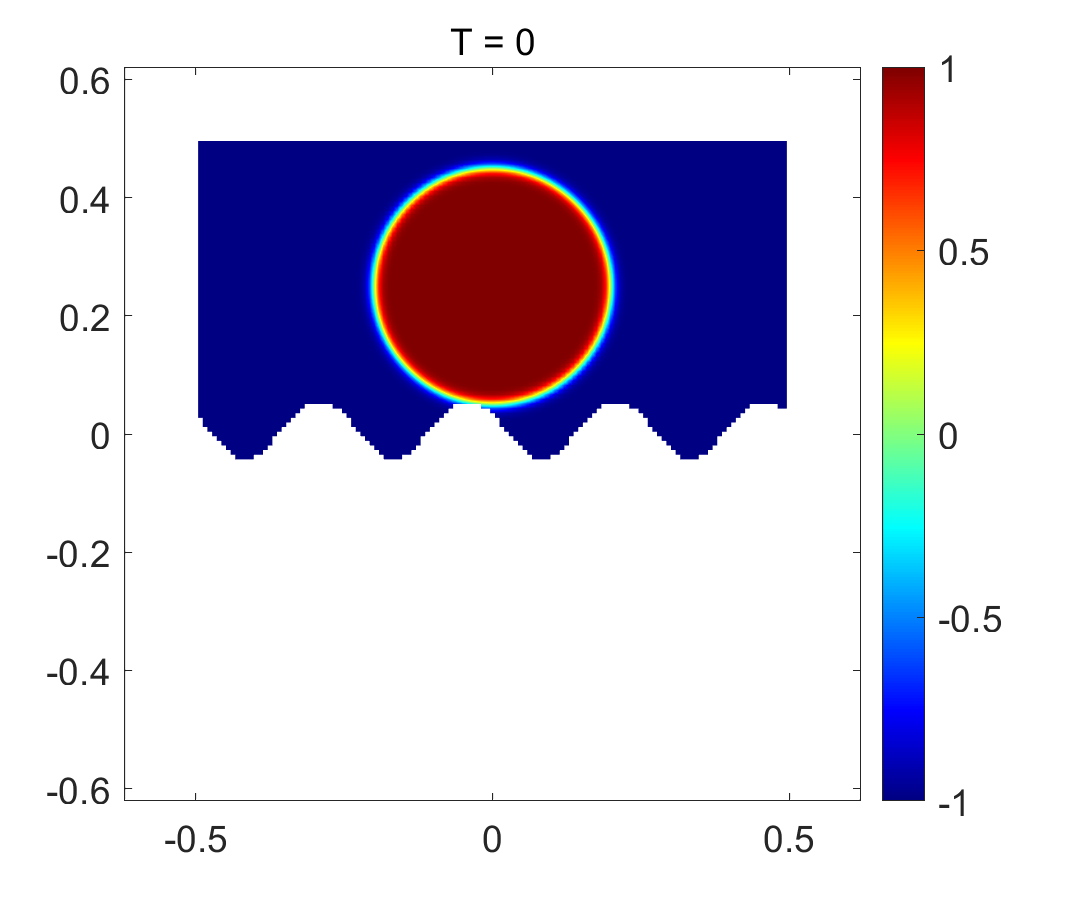}
		\includegraphics[width = 1.2in]{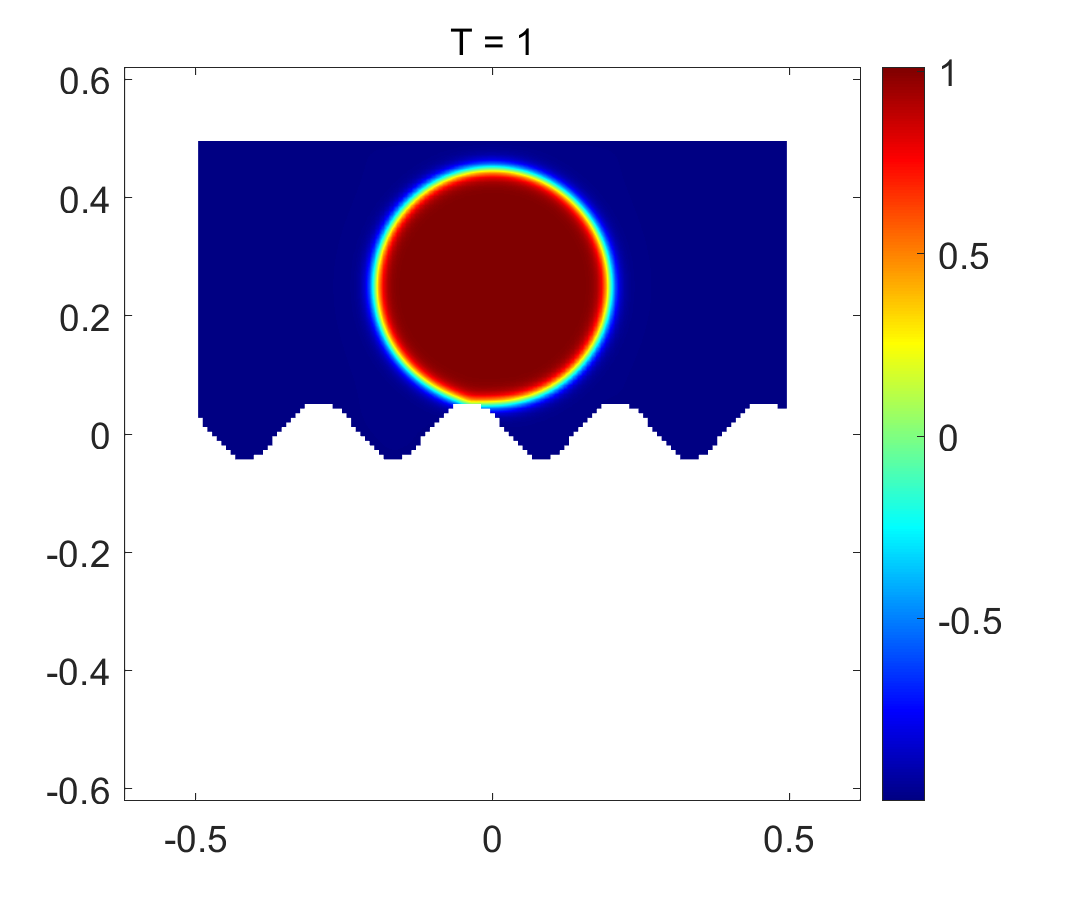}
		\includegraphics[width = 1.2in]{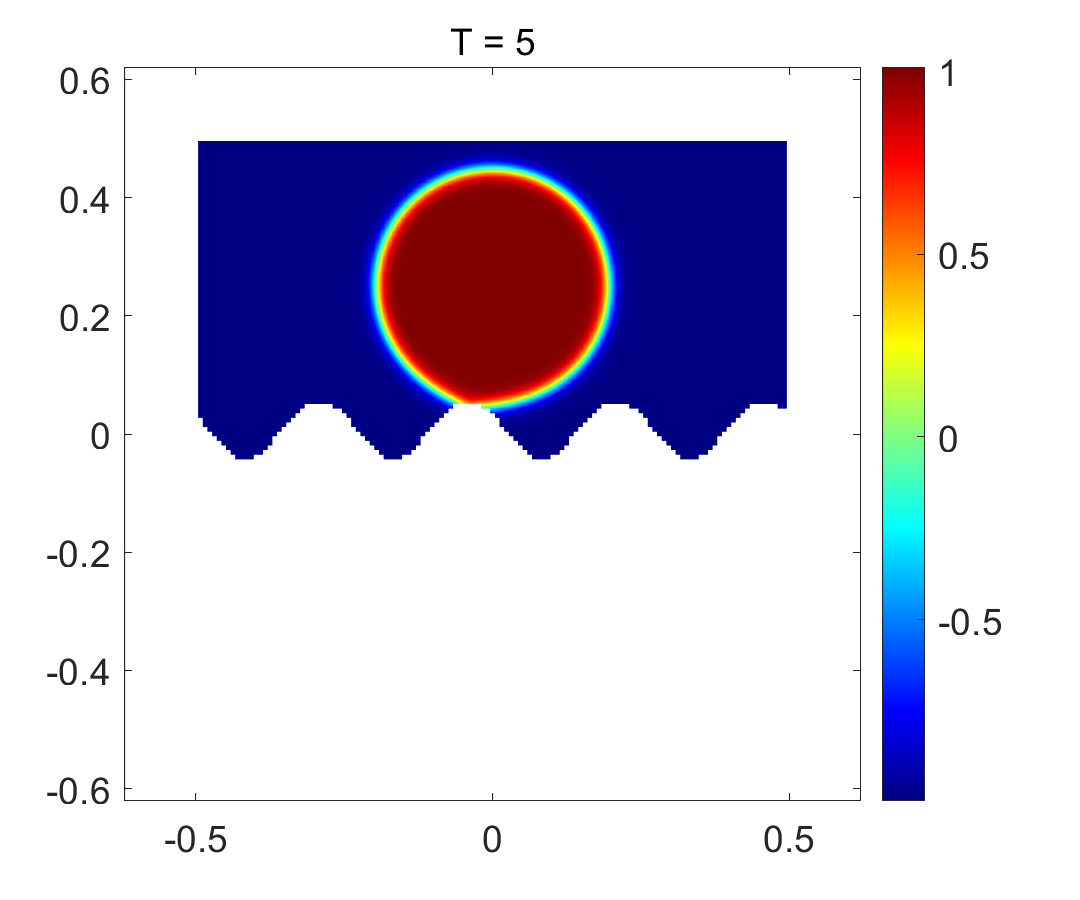}
		\includegraphics[width = 1.2in]{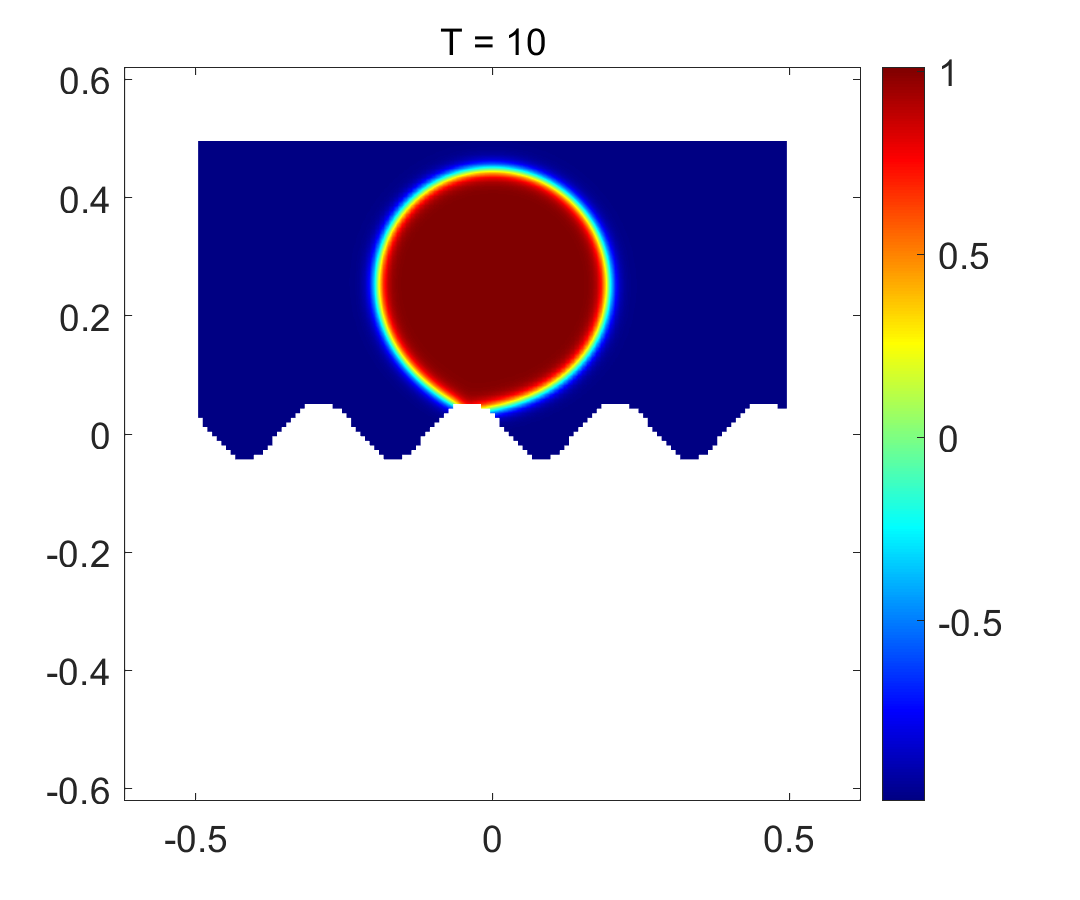}
		\includegraphics[width = 1.2in]{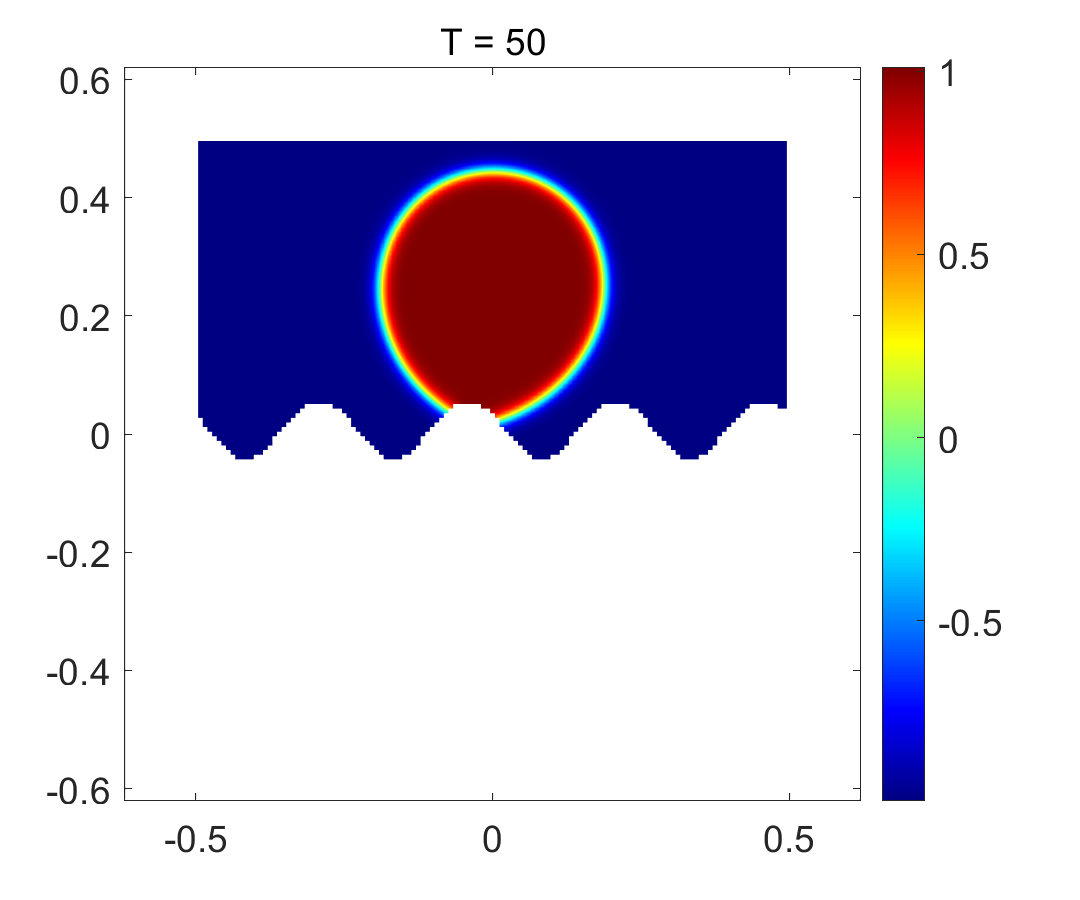}
		
		\includegraphics[width = 1.2in]{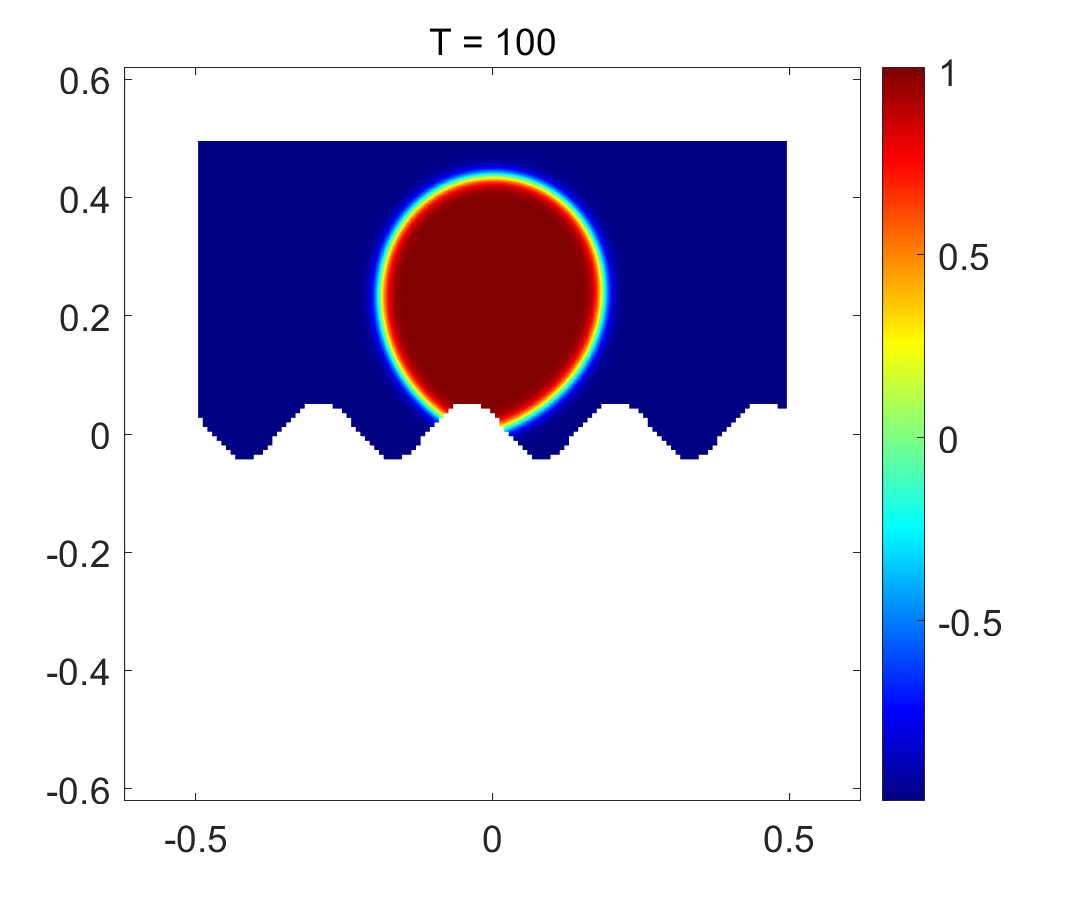}
		\includegraphics[width = 1.2in]{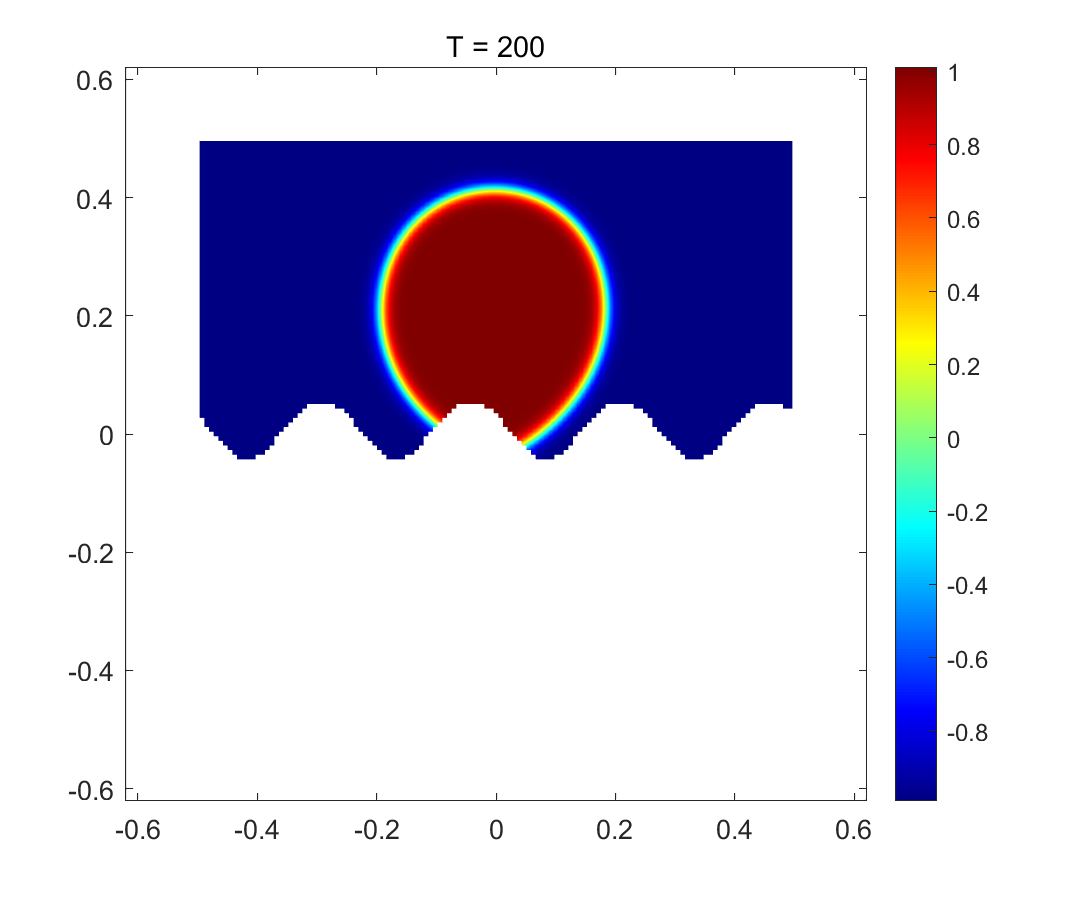}
		\includegraphics[width = 1.2in]{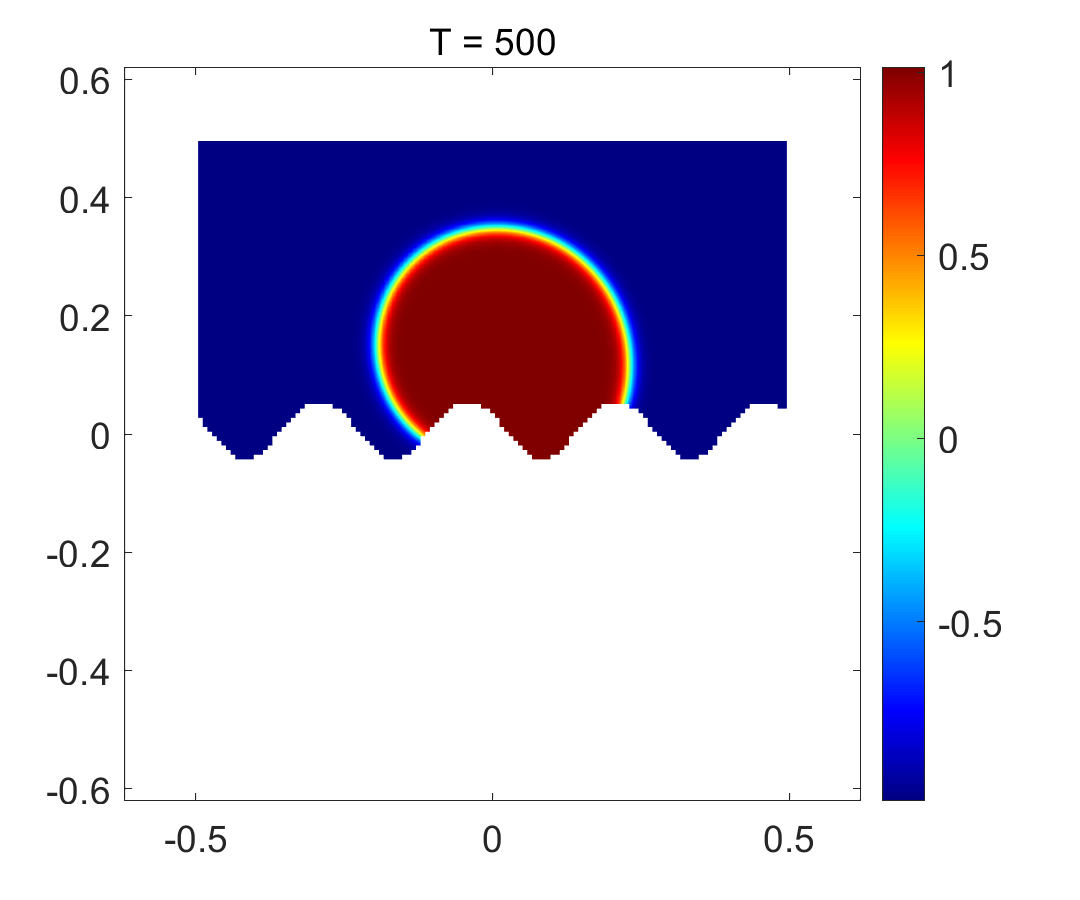}
		\includegraphics[width = 1.2in]{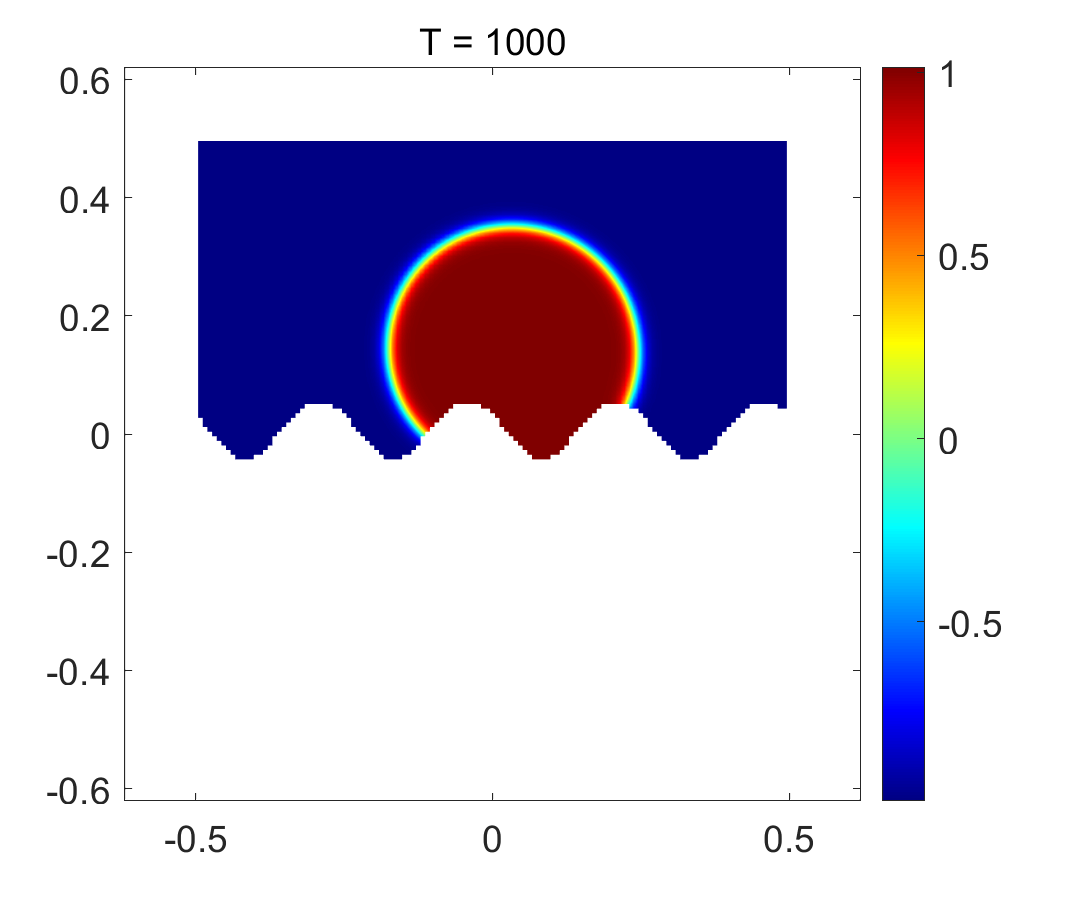}
		\includegraphics[width = 1.2in]{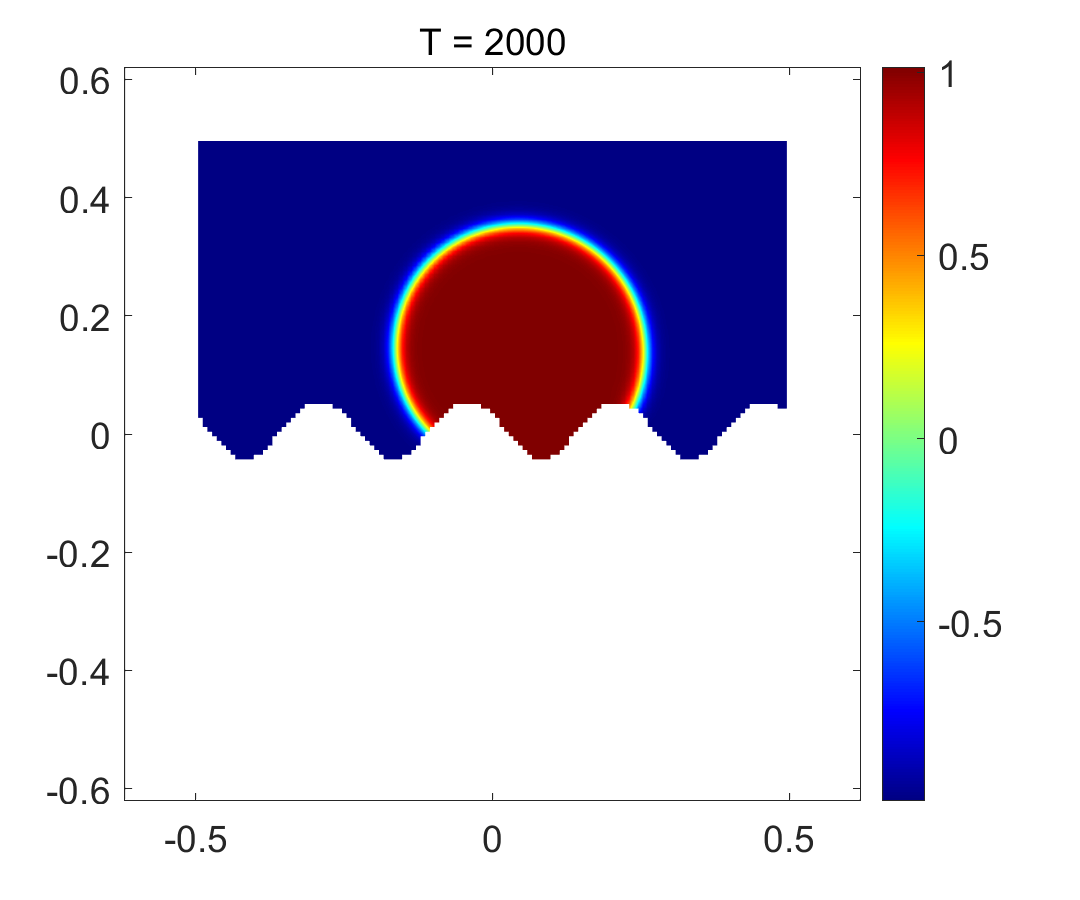}
		\caption{ Droplet spreading simulated using the OPBDE Cahn-Hilliard model for $\Gamma = 20$. Profiles of $\tilde{\phi}$ are shown at time instants $T = 0$, $1$, $5$, $10$, $50$, $100$, $200$, $500$, $1000$, $2000$.}
		\label{Droplet spreading41}
	\end{figure}
	\begin{figure}[H]
		\centering
		\includegraphics[width = 1.2in]{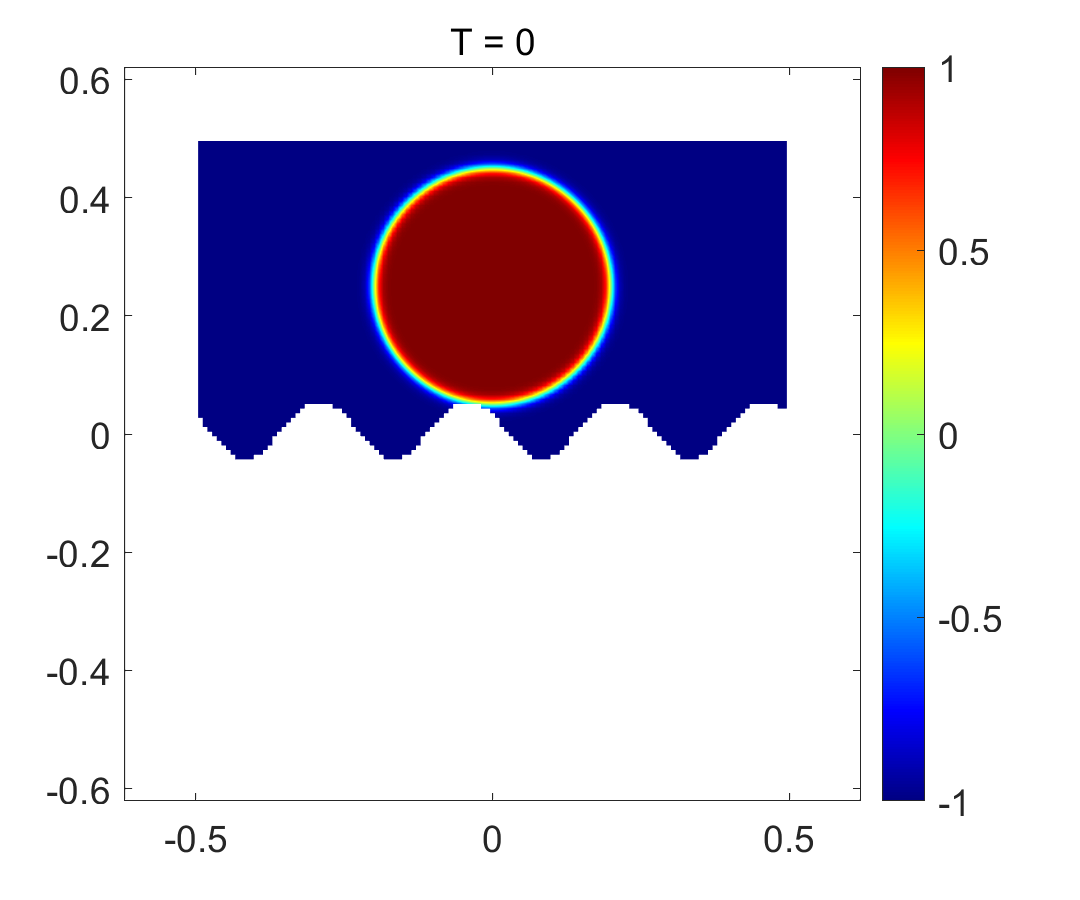}
		\includegraphics[width = 1.2in]{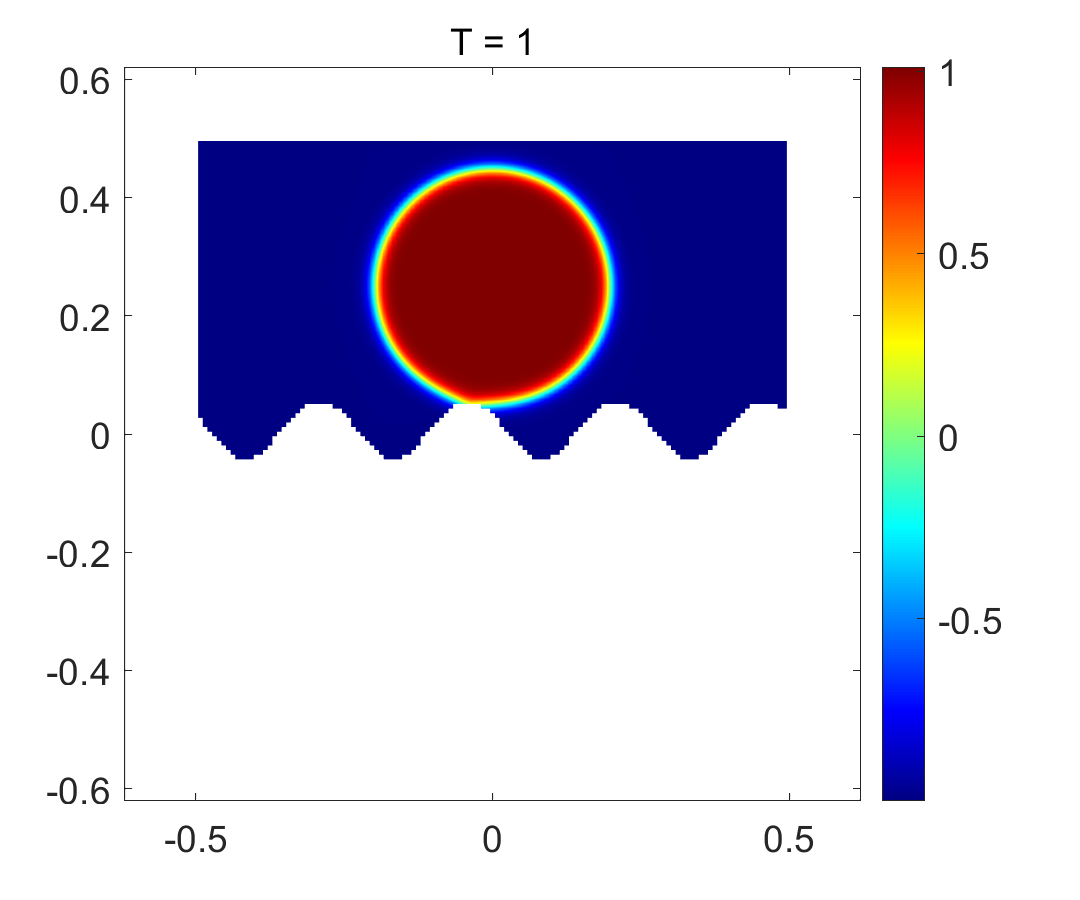}
		\includegraphics[width = 1.2in]{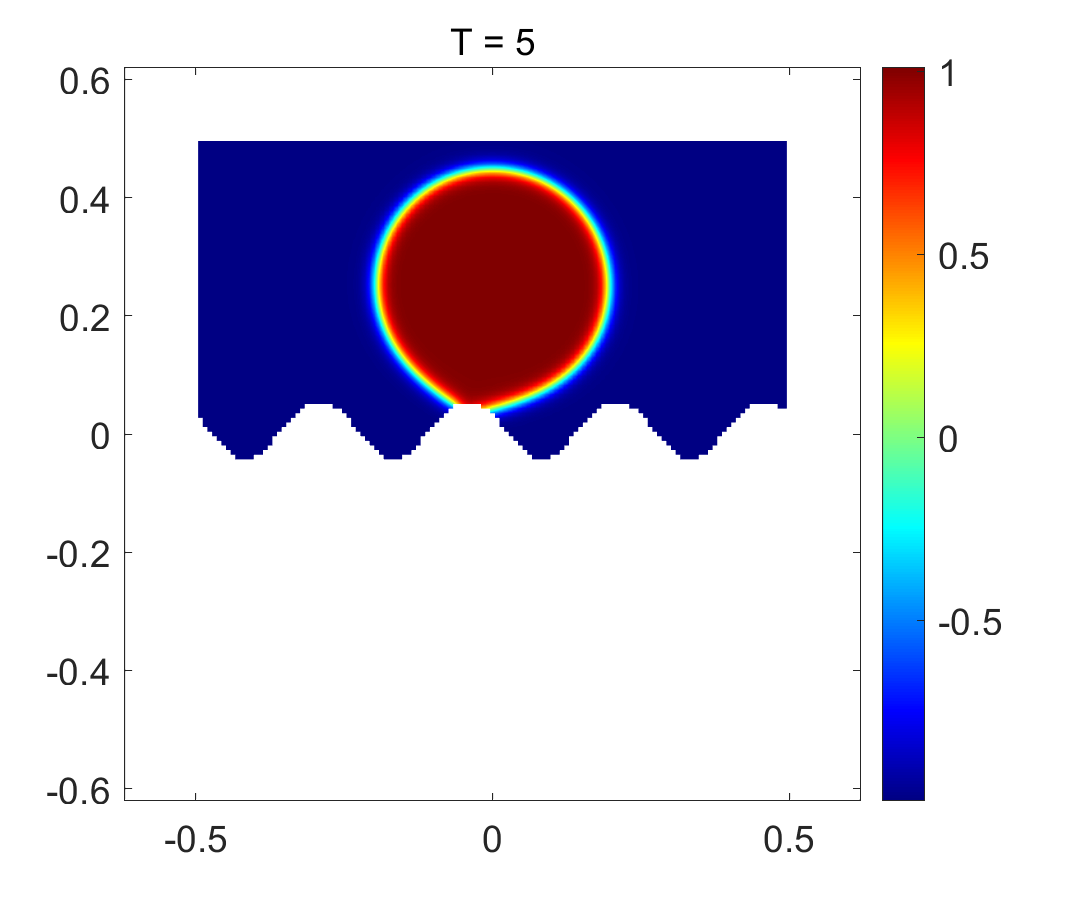}
		\includegraphics[width = 1.2in]{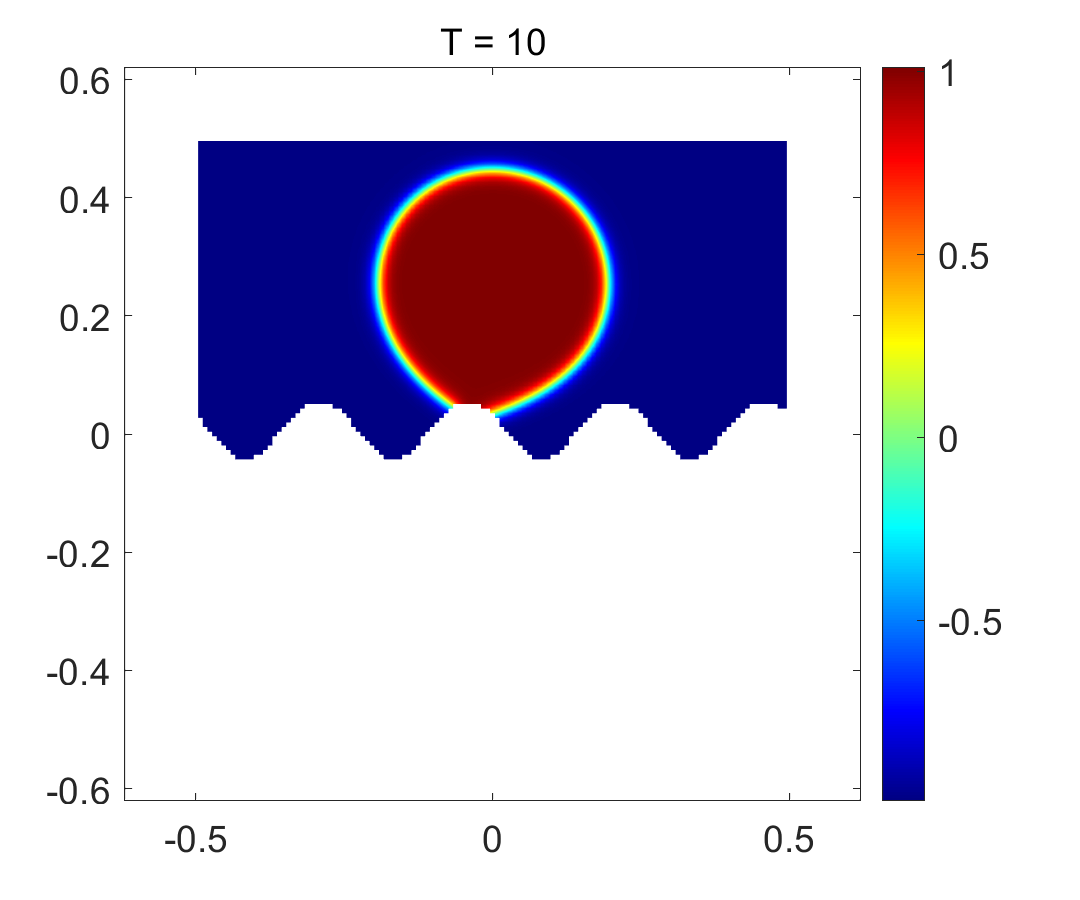}
		\includegraphics[width = 1.2in]{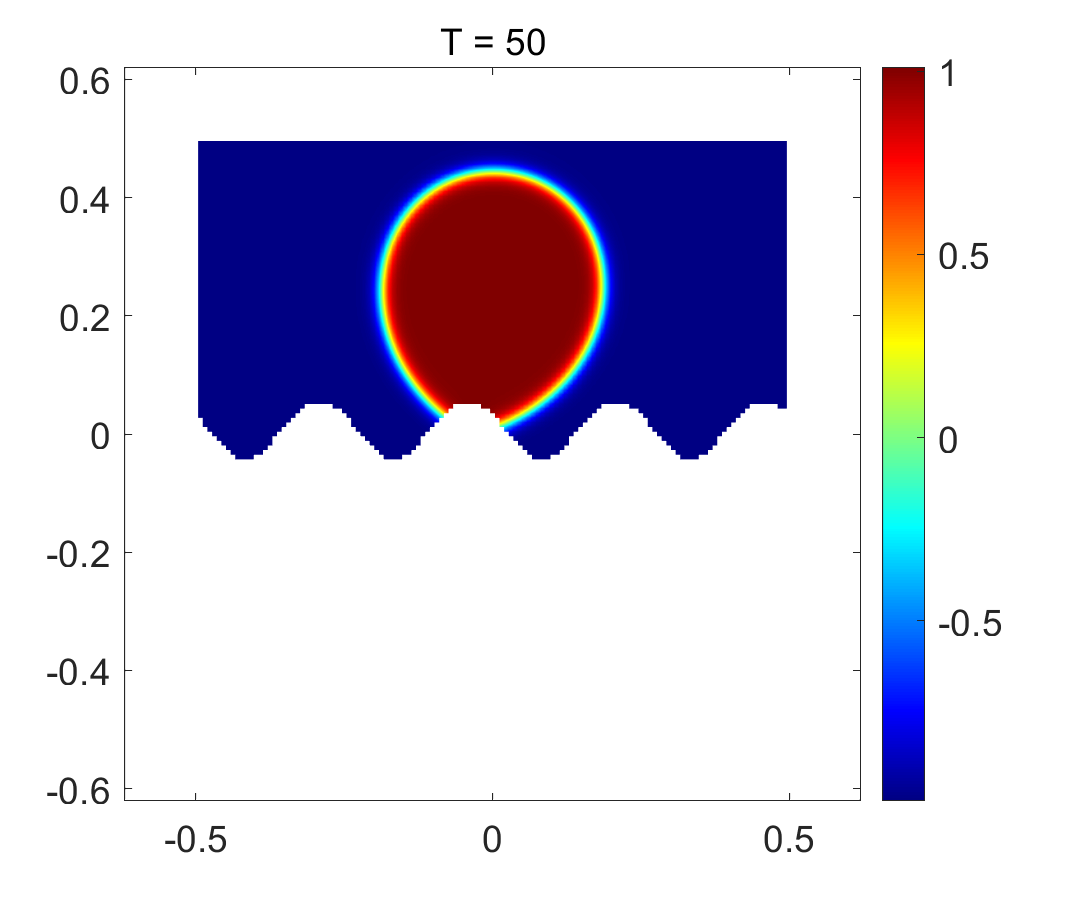}
		
		\includegraphics[width = 1.2in]{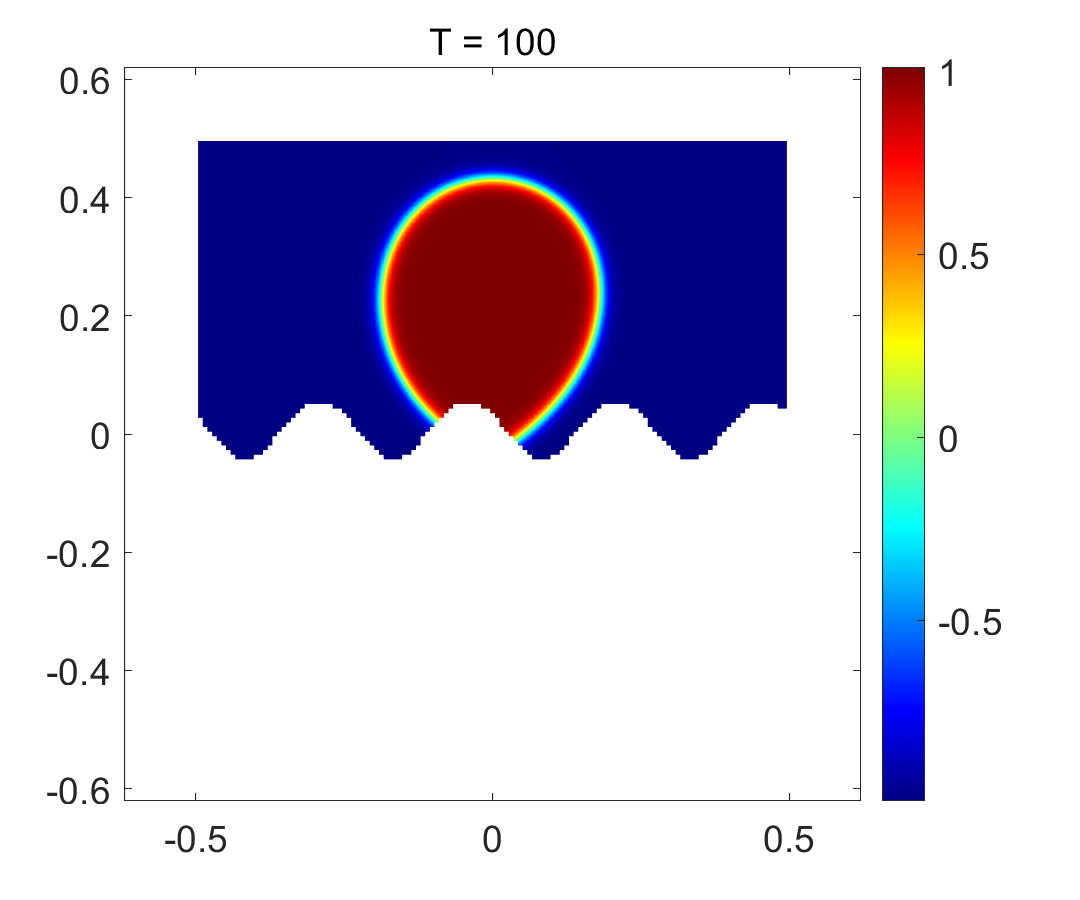}
		\includegraphics[width = 1.2in]{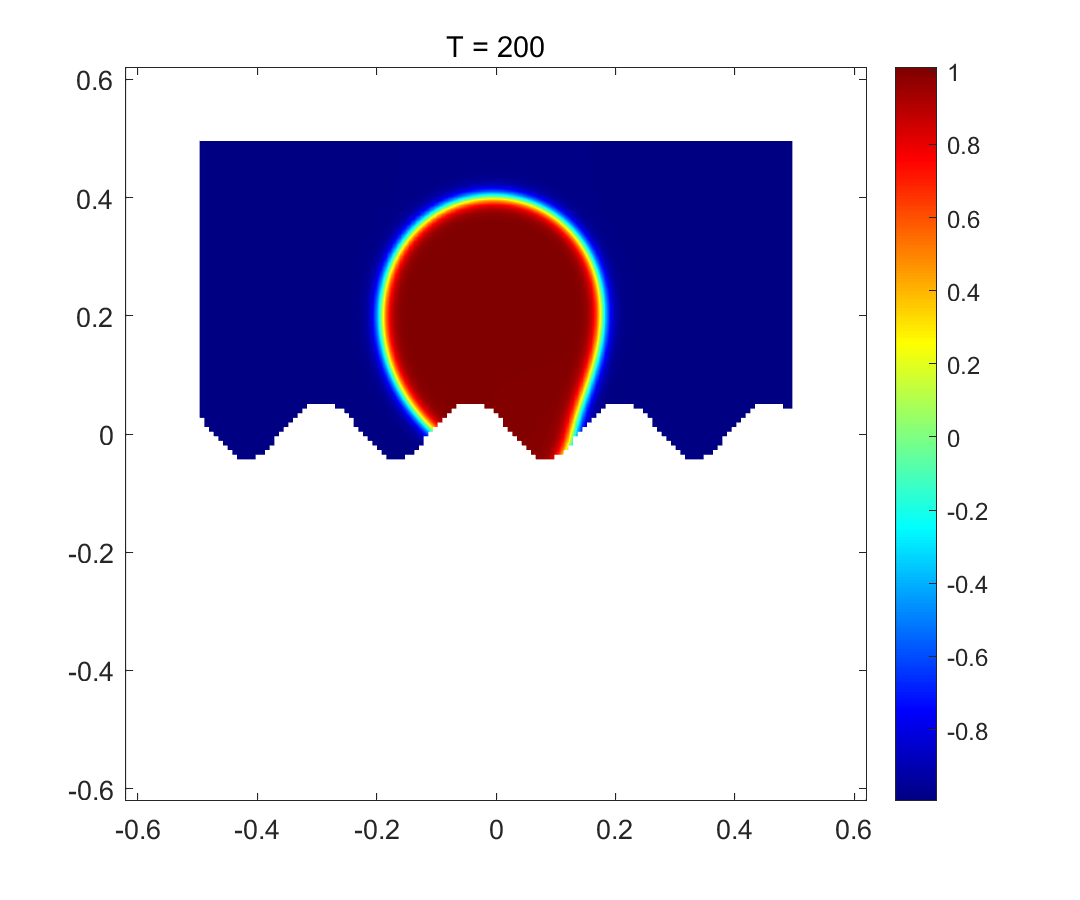}
		\includegraphics[width = 1.2in]{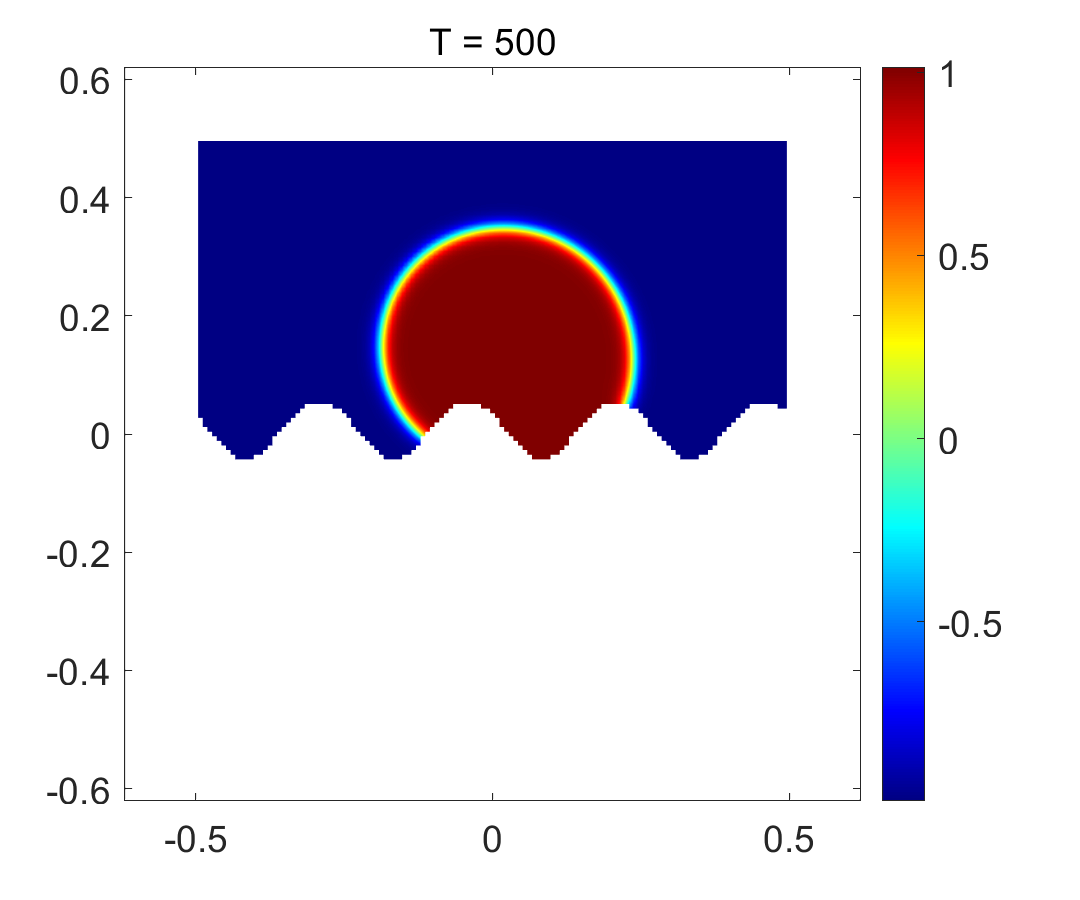}
		\includegraphics[width = 1.2in]{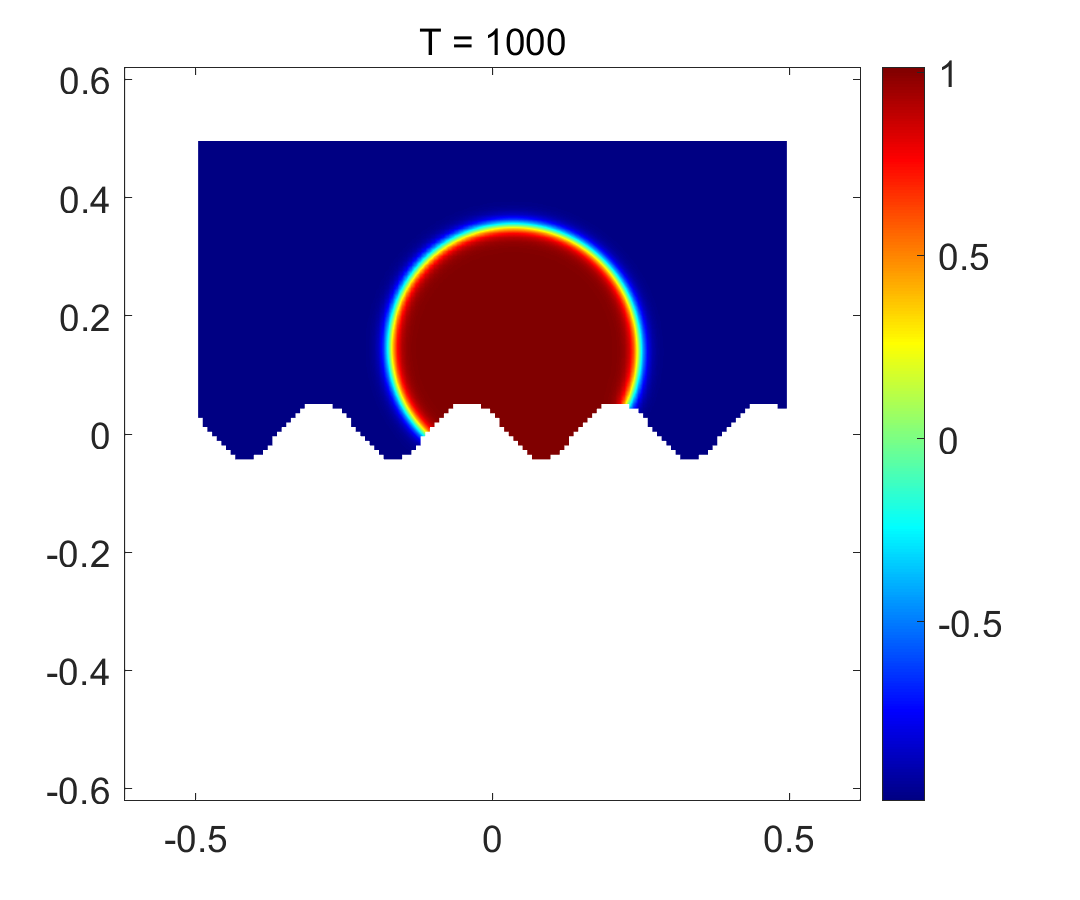}
		\includegraphics[width = 1.2in]{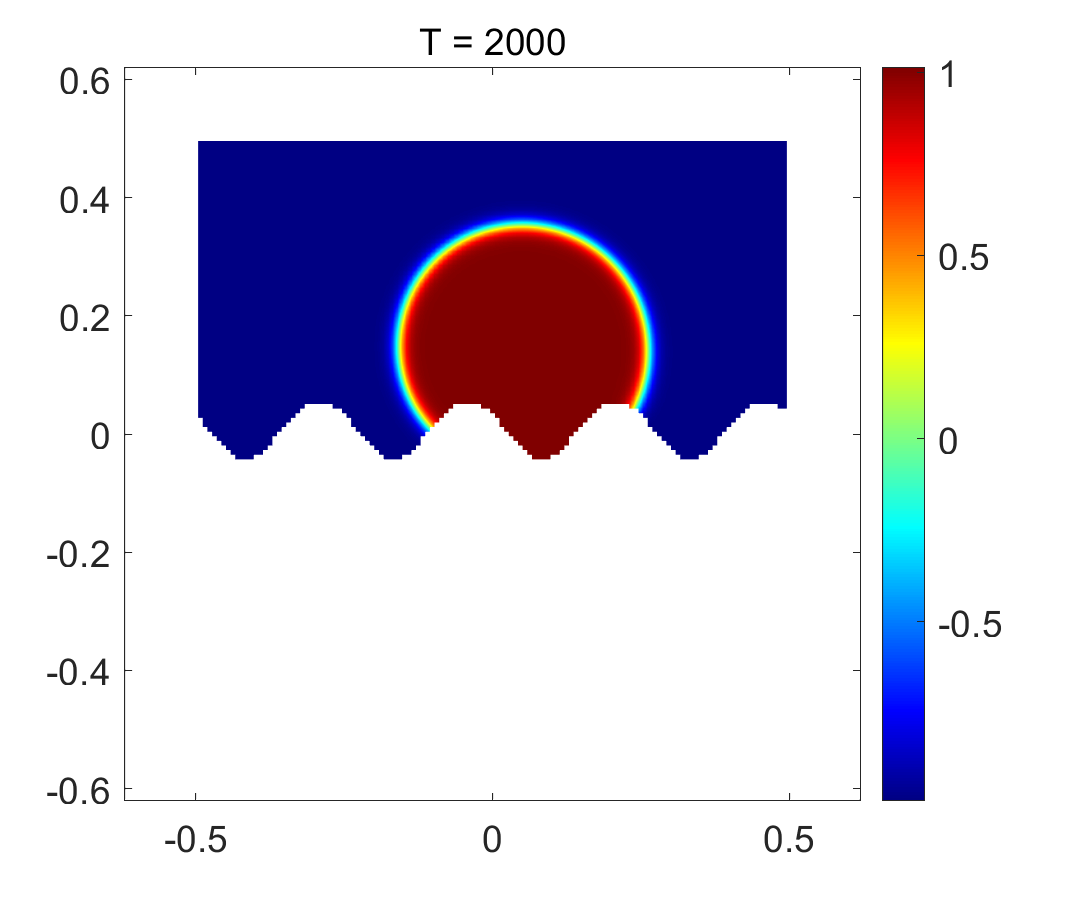}
		\caption{ Droplet spreading simulated using the OPBDE Cahn-Hilliard model for $\Gamma = 50$. Profiles of $\tilde{\phi}$ are shown at time instants $T = 0$, $1$, $5$, $10$, $50$, $100$, $200$, $500$, $1000$, $2000$.}
		\label{Droplet spreading42}
	\end{figure}
	\begin{figure}[H]
		\centering
		\includegraphics[width = 1.2in]{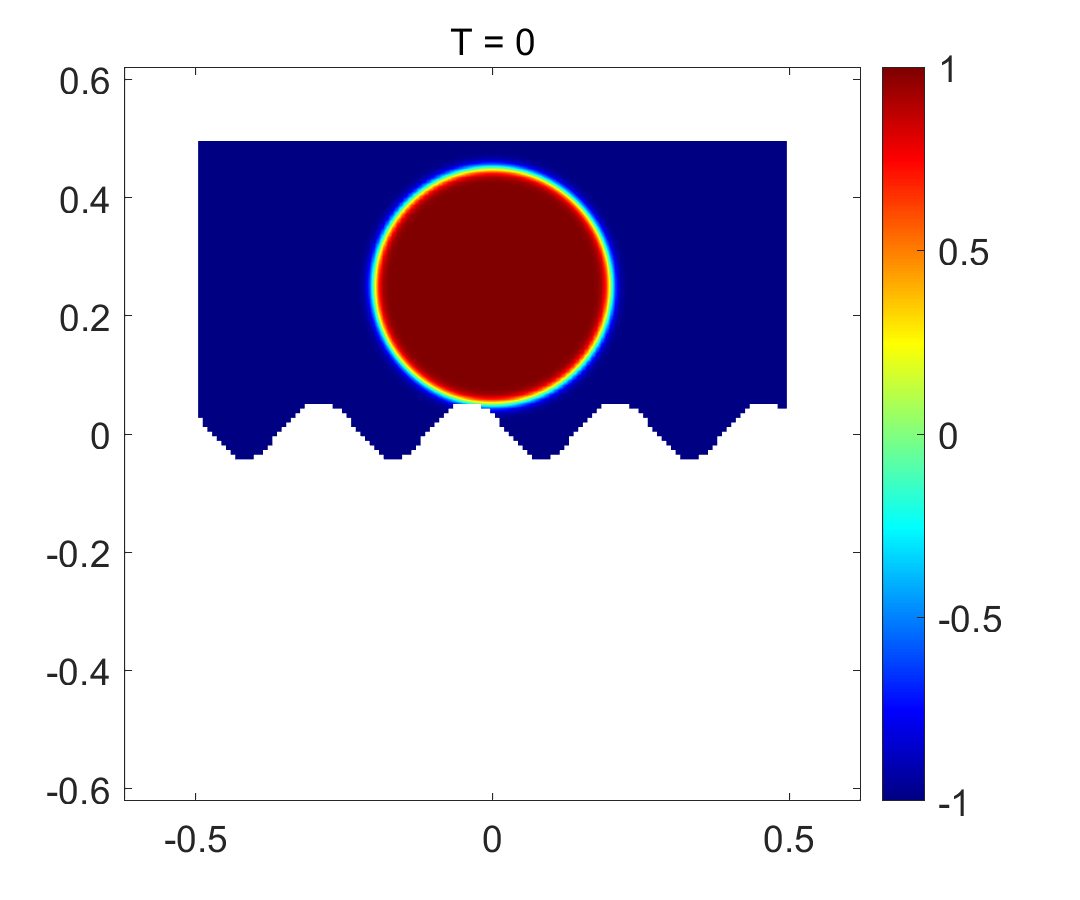}
		\includegraphics[width = 1.2in]{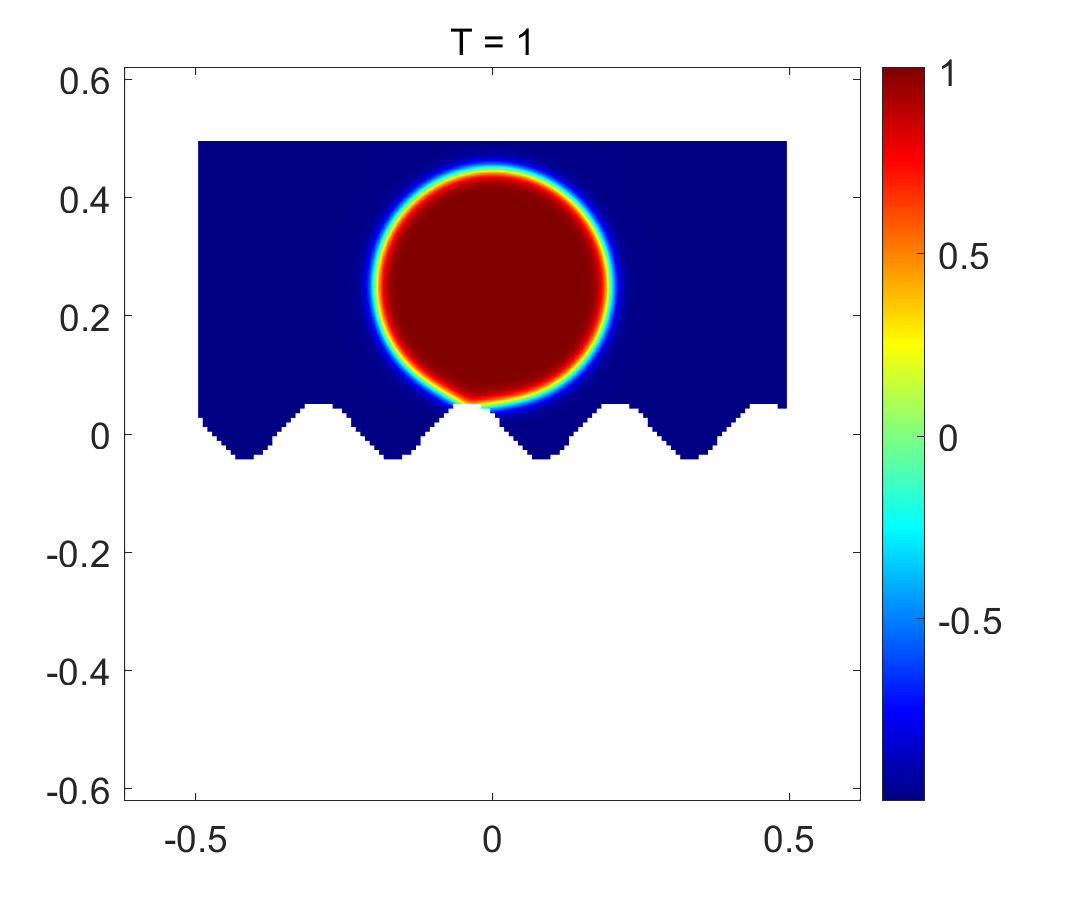}
		\includegraphics[width = 1.2in]{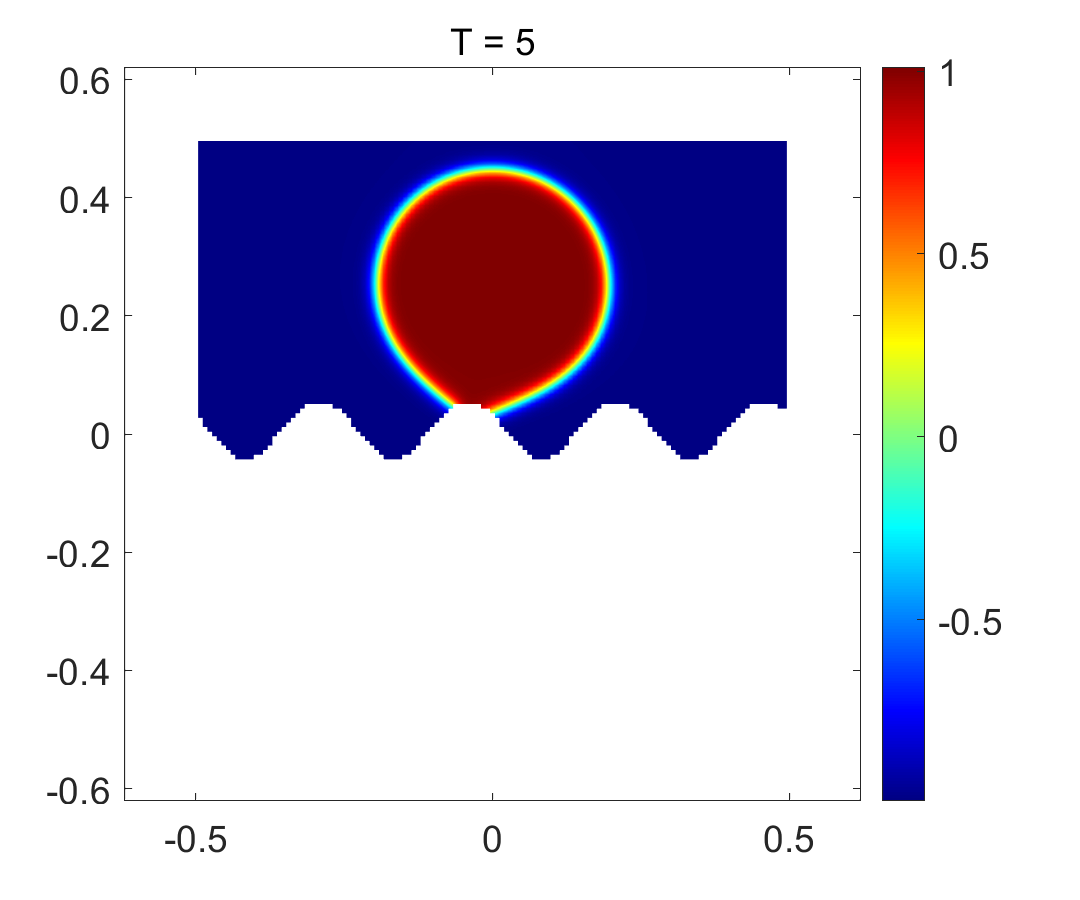}
		\includegraphics[width = 1.2in]{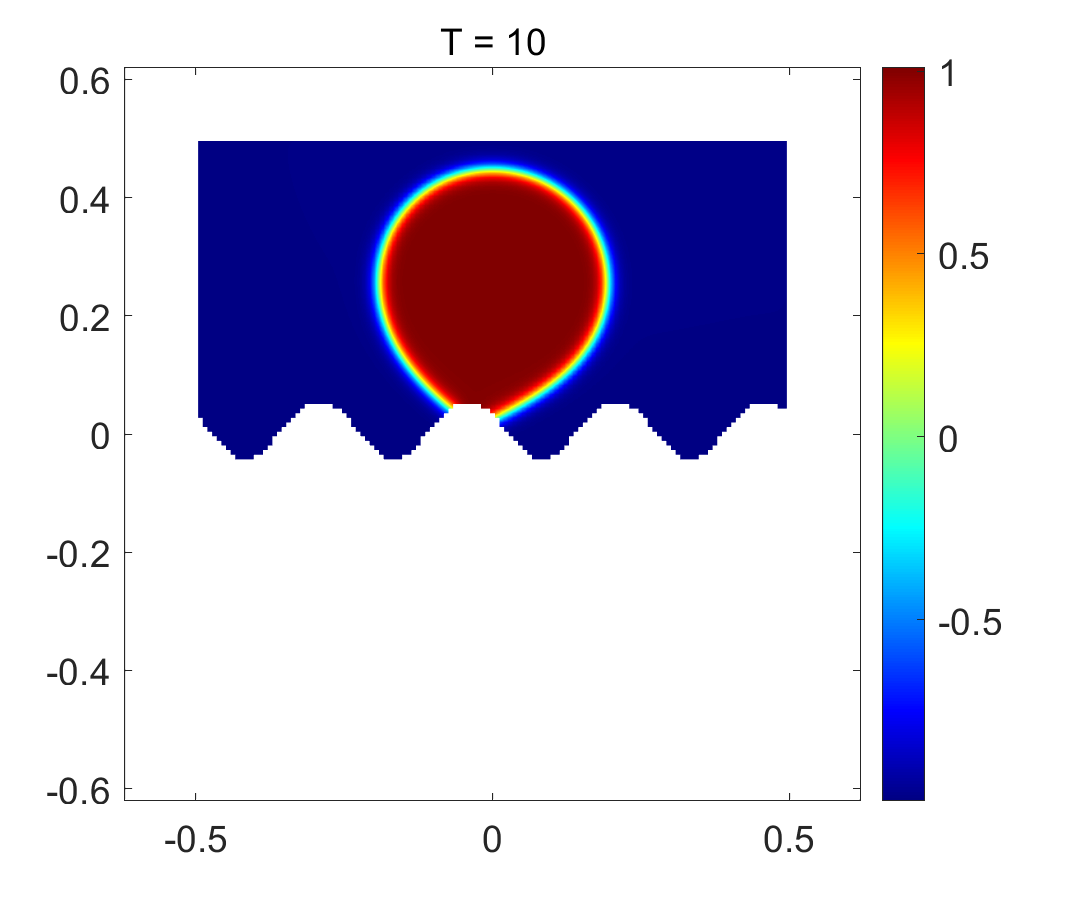}
		\includegraphics[width = 1.2in]{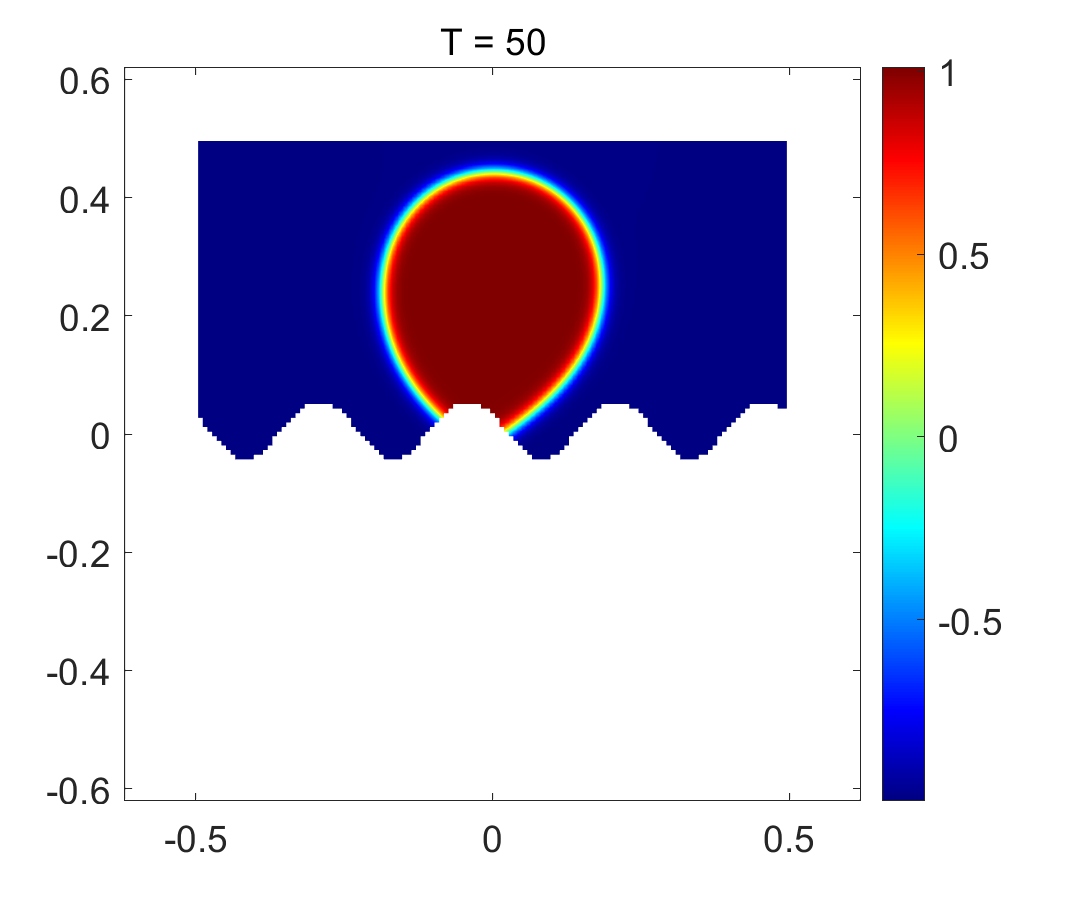}
		
		\includegraphics[width = 1.2in]{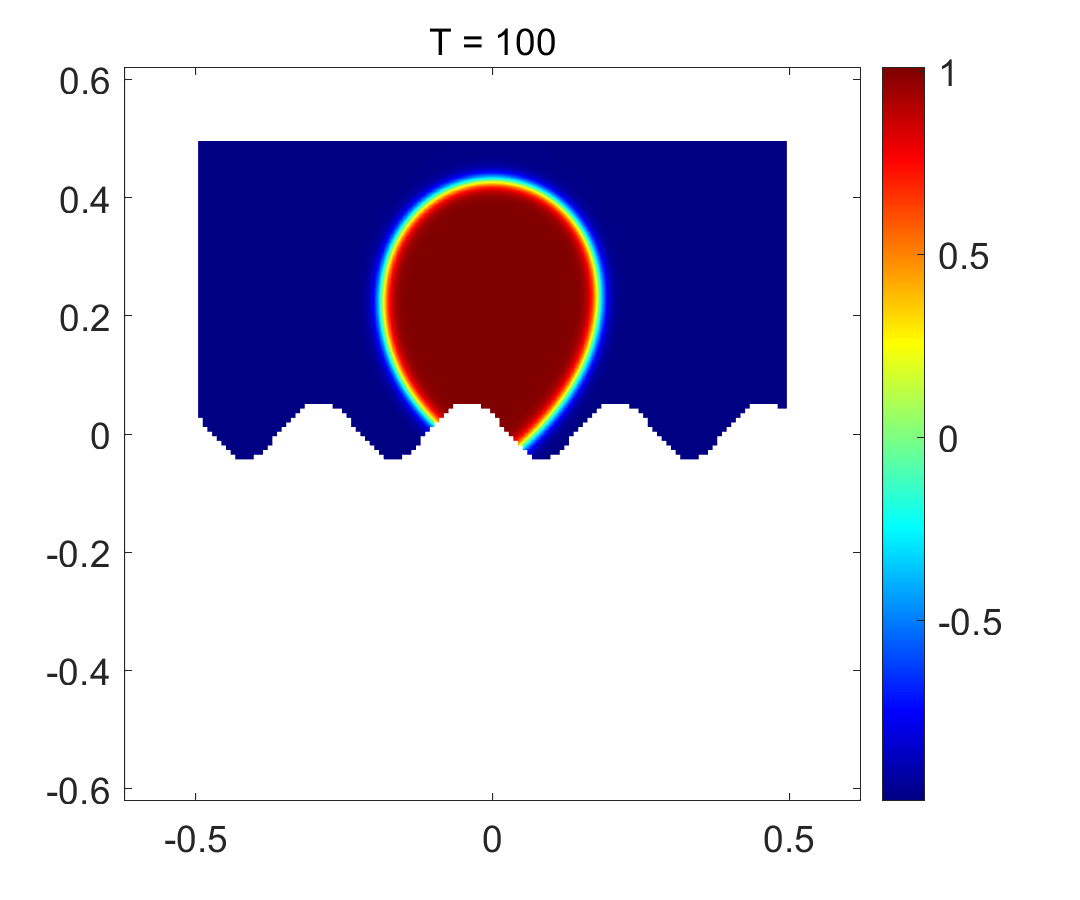}
		\includegraphics[width = 1.2in]{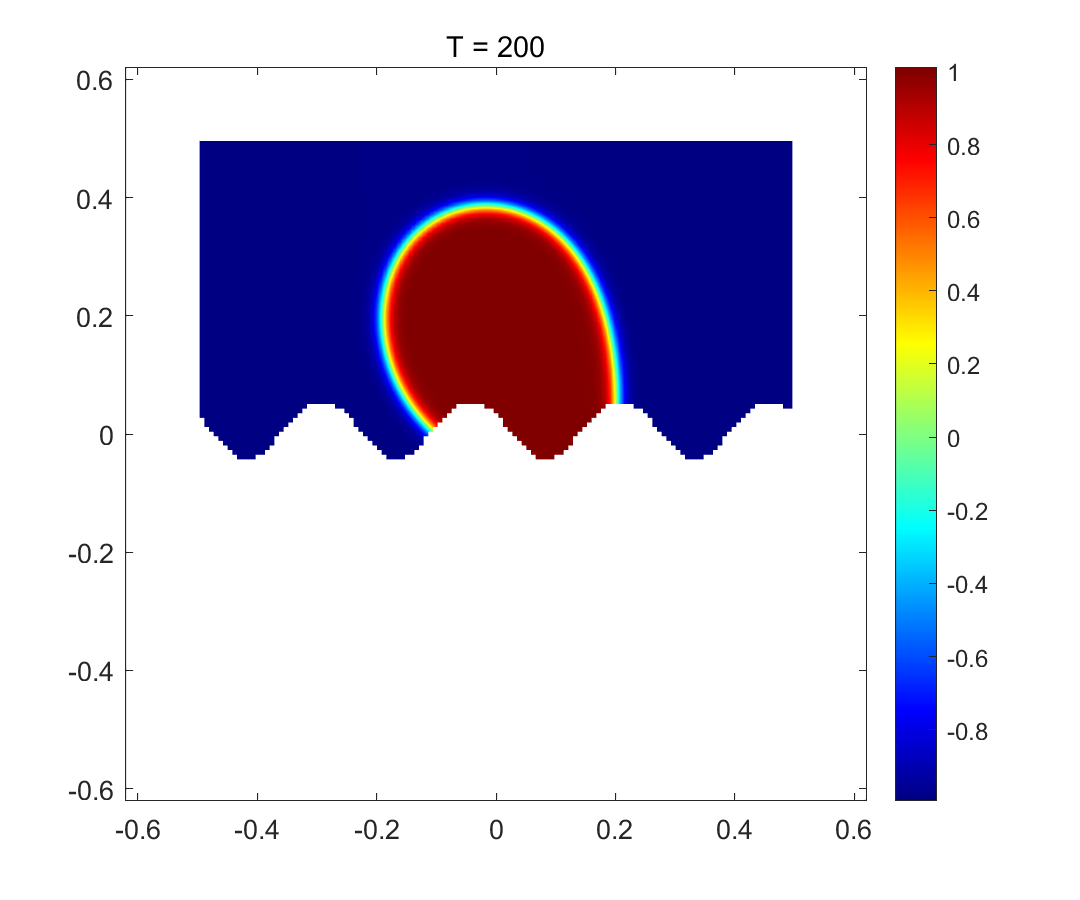}
		\includegraphics[width = 1.2in]{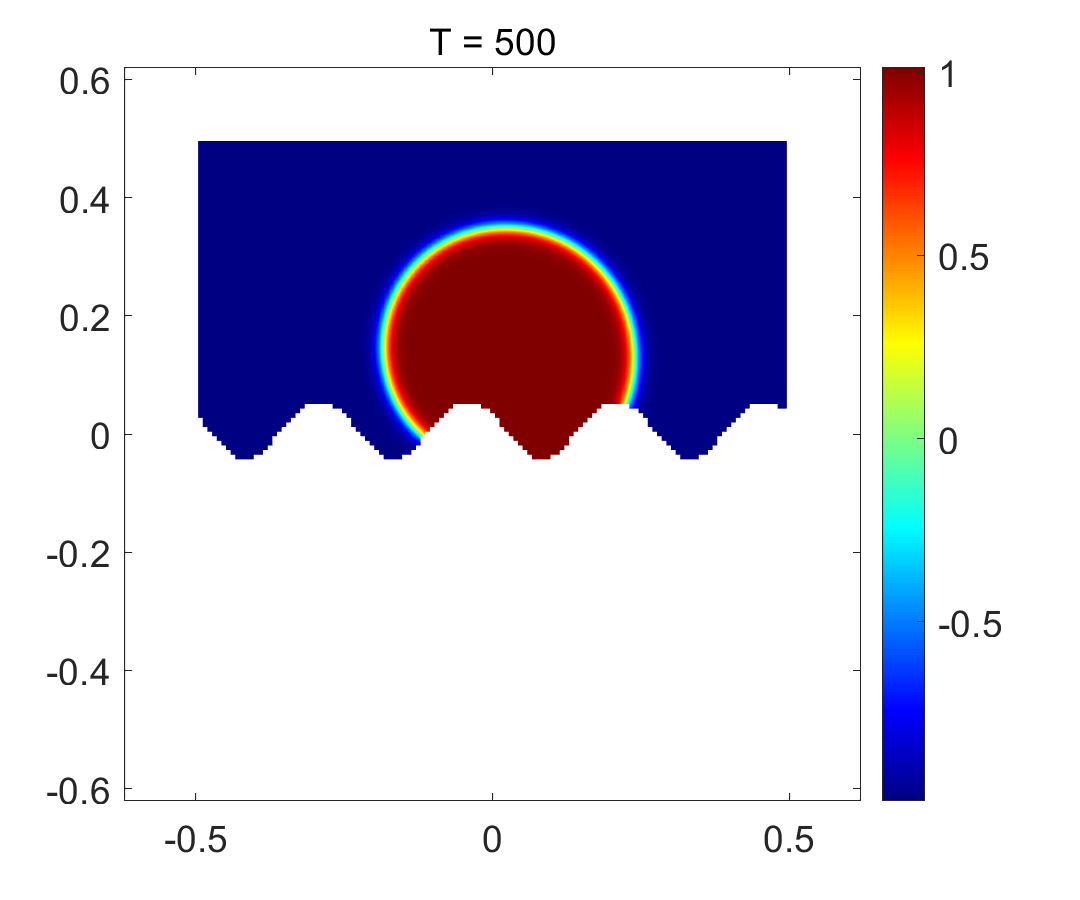}
		\includegraphics[width = 1.2in]{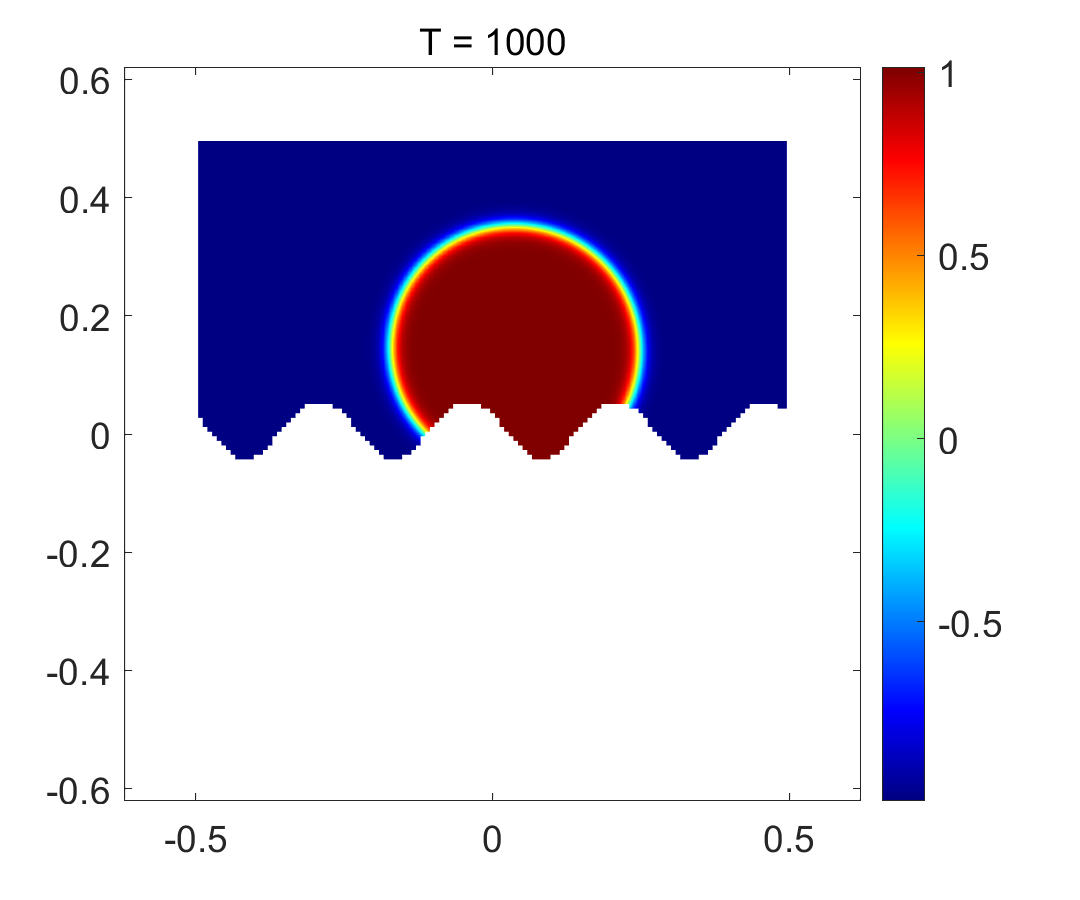}
		\includegraphics[width = 1.2in]{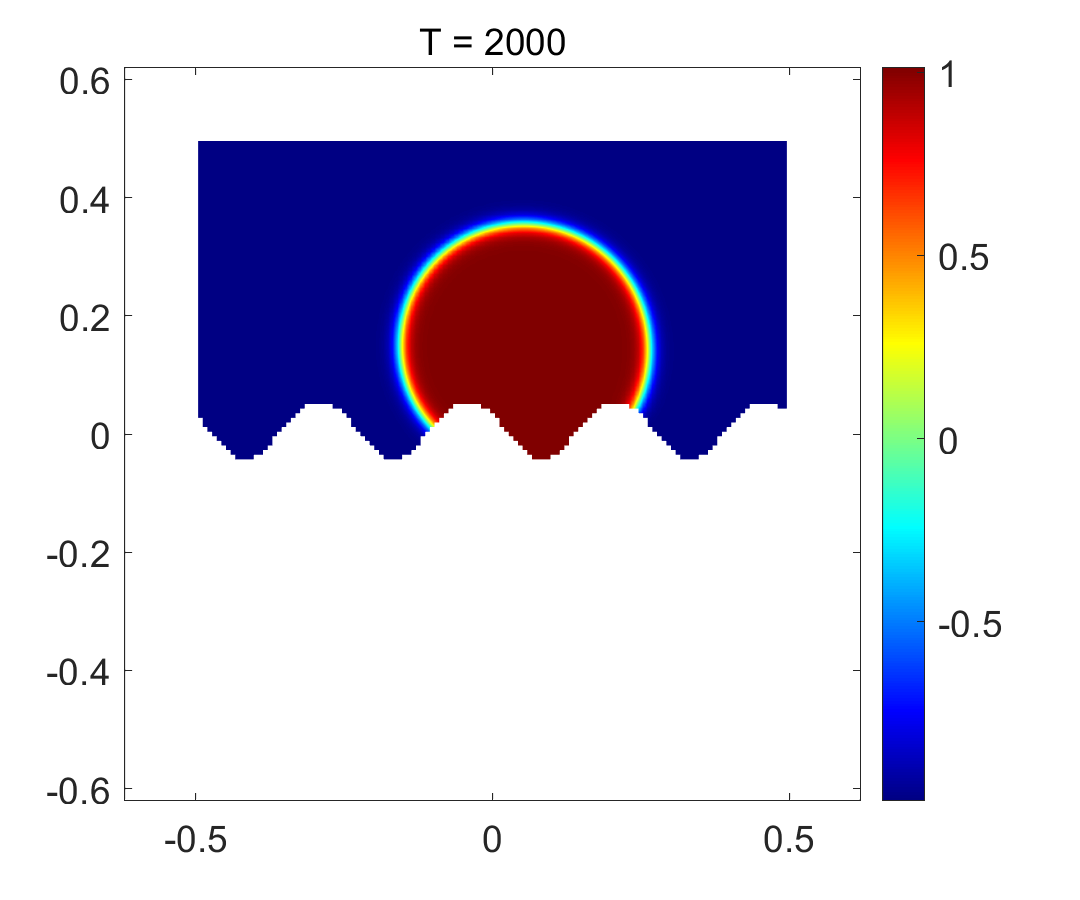}
		\caption{ Droplet spreading simulated using the OPBDE Cahn-Hilliard model for $\Gamma = 100$. Profiles of $\tilde{\phi}$ are shown at time instants $T = 0$, $1$, $5$, $10$, $50$, $100$, $200$, $500$, $1000$, $2000$.}
		\label{Droplet spreading43}
	\end{figure}
	
	Figures \ref{Droplet spreading41} to \ref{Droplet spreading43} show the droplet spreading dynamics for $\Gamma = 20$, $50$, $100$. The initial position of the droplet is slightly to the right of a peak on the substrate curve. The free interface evolves under the influence of the substrate geometry, leading to a rightward droplet movement. At time instants $T=50$, $100$, $200$, figure \ref{Droplet spreading43} shows that the left contact line remains relatively stationary, while the right contact line exhibits a rapid slip to the right. At the same time instants, figures \ref{Droplet spreading41} and \ref{Droplet spreading42} also show similar behavior though less significant for smaller $\Gamma$. It is observed that as $\Gamma$ increases, the boundary dynamics becomes faster, in consistency with the model prediction.
	
	\begin{figure}[H]
		\centering
		\subfigure[]{
			\includegraphics[width = 0.3 \textwidth]{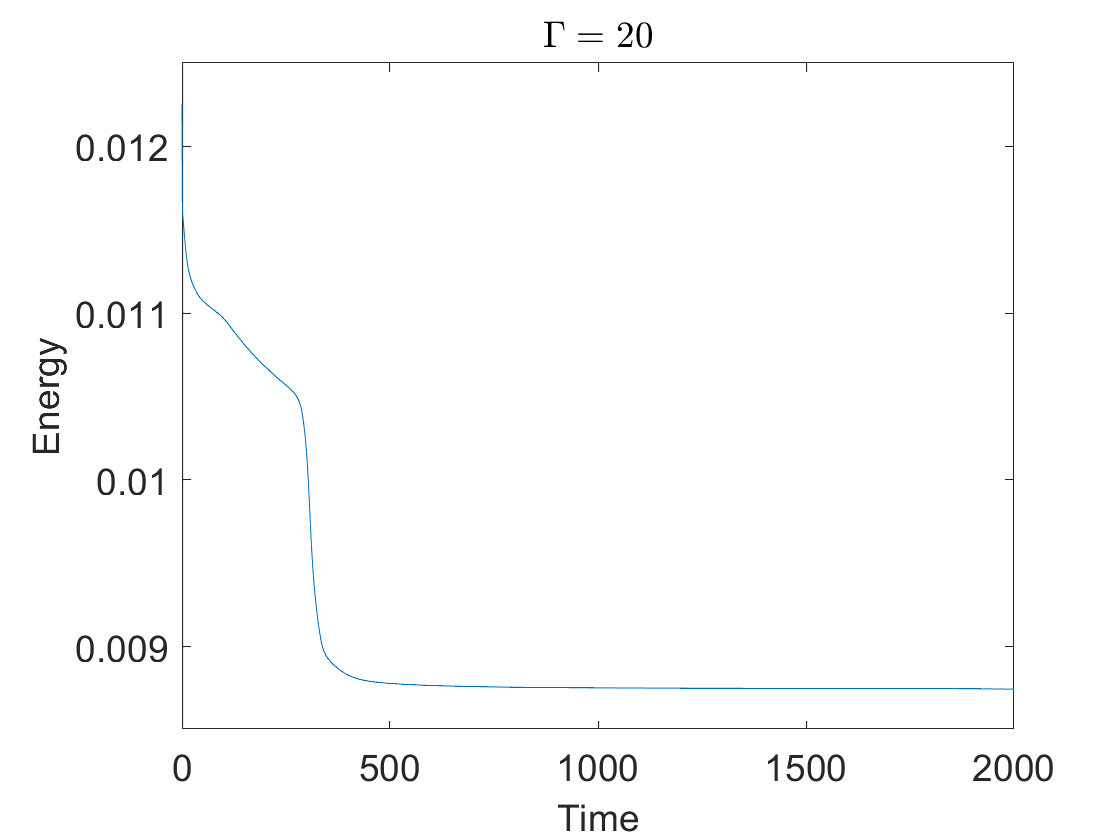}
			\includegraphics[width = 0.3 \textwidth]{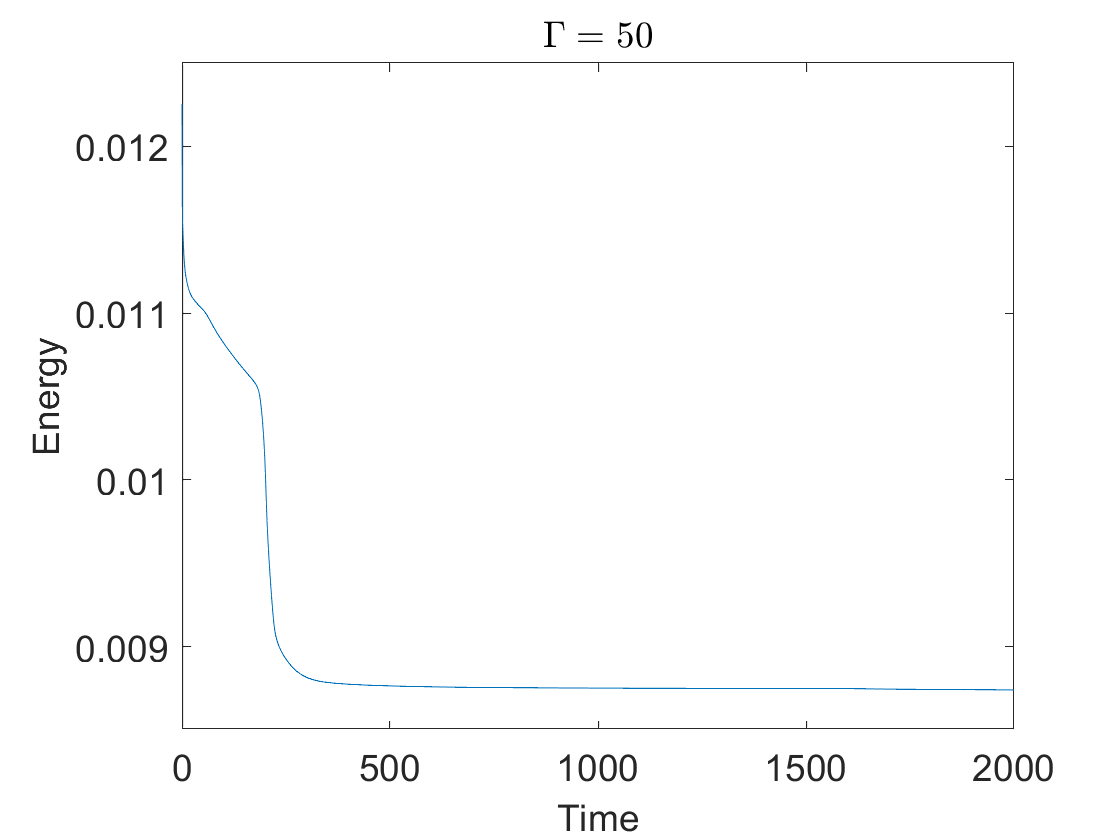}
			\includegraphics[width = 0.3 \textwidth]{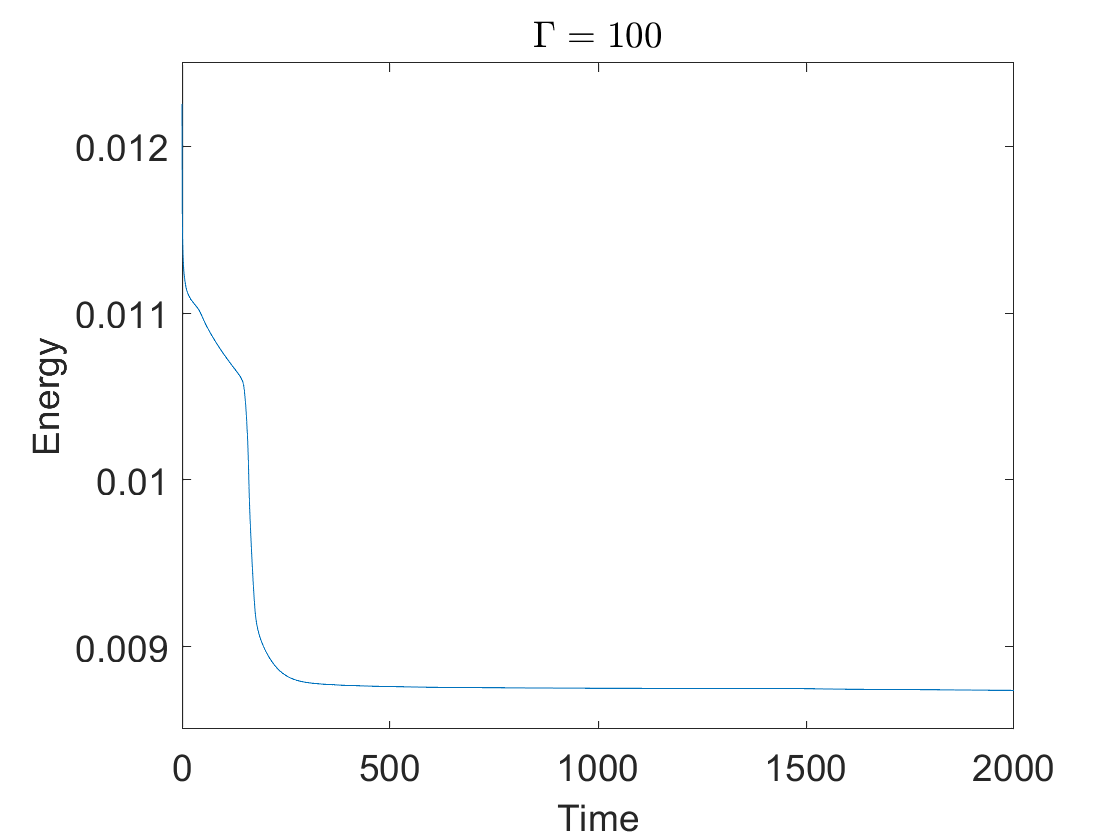}
		}
		
		\subfigure[]{
			\includegraphics[width = 0.3 \textwidth]{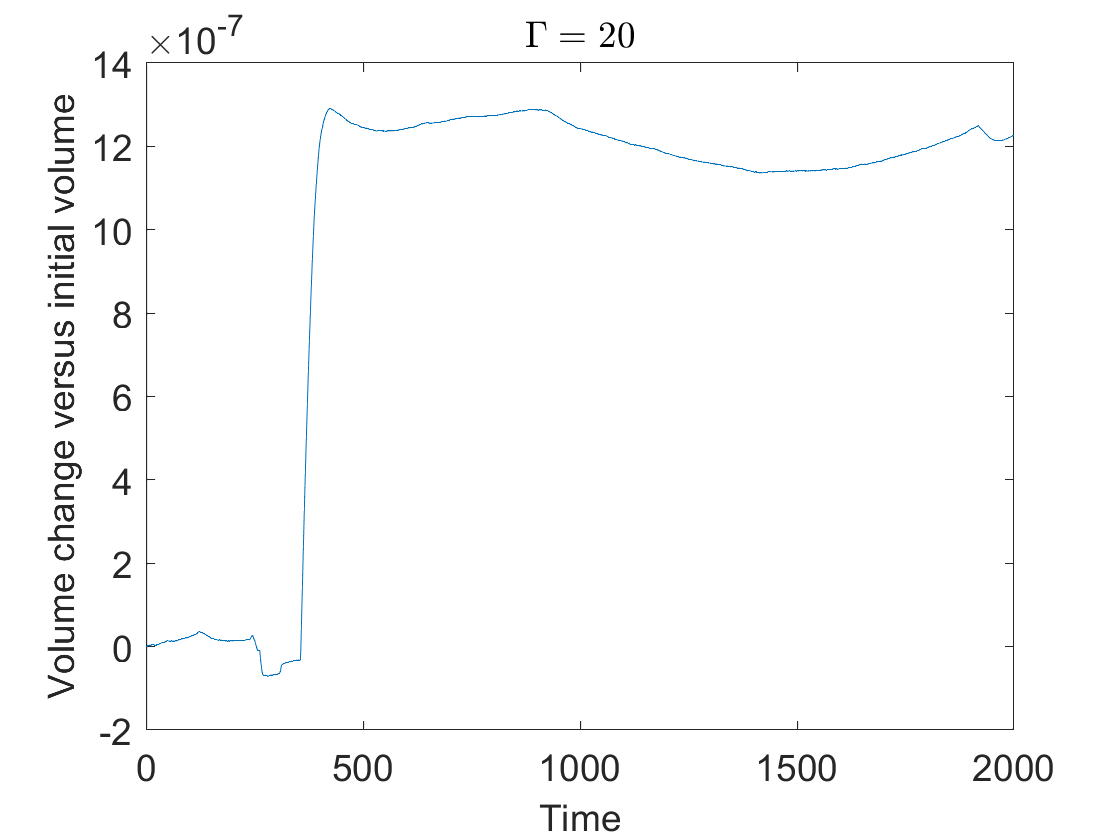}
			\includegraphics[width = 0.3 \textwidth]{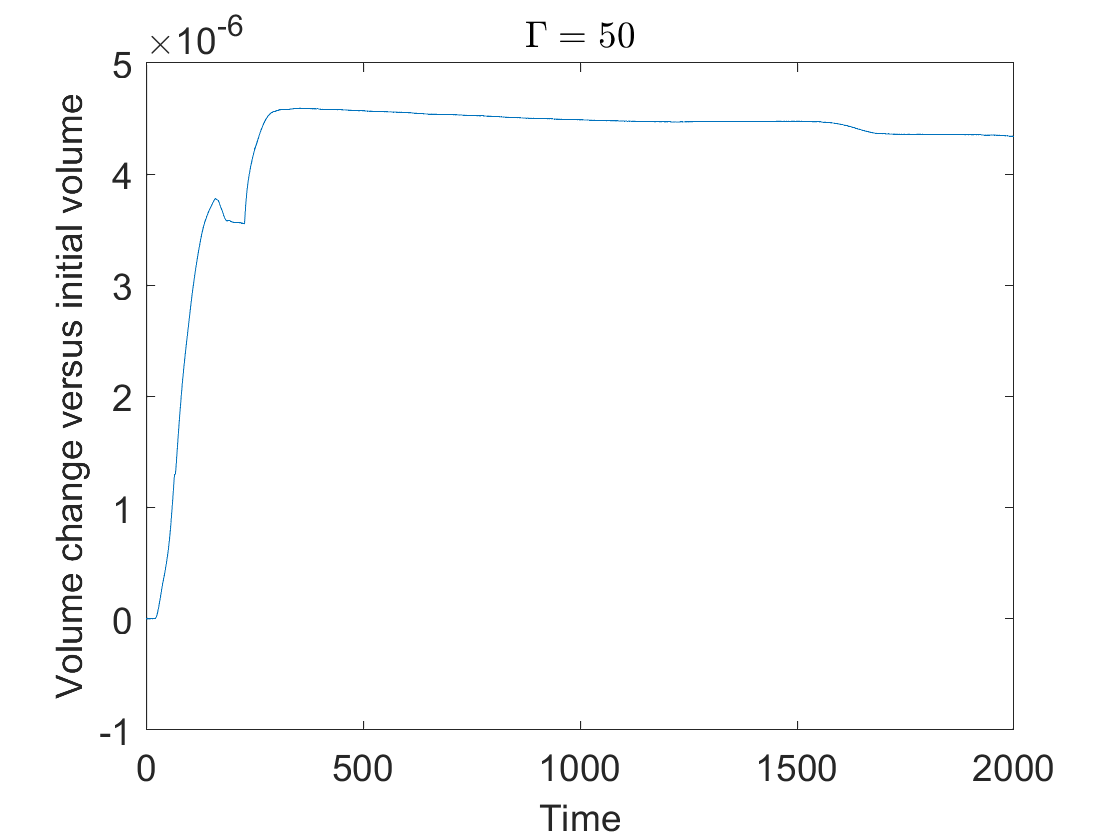}
			\includegraphics[width = 0.3 \textwidth]{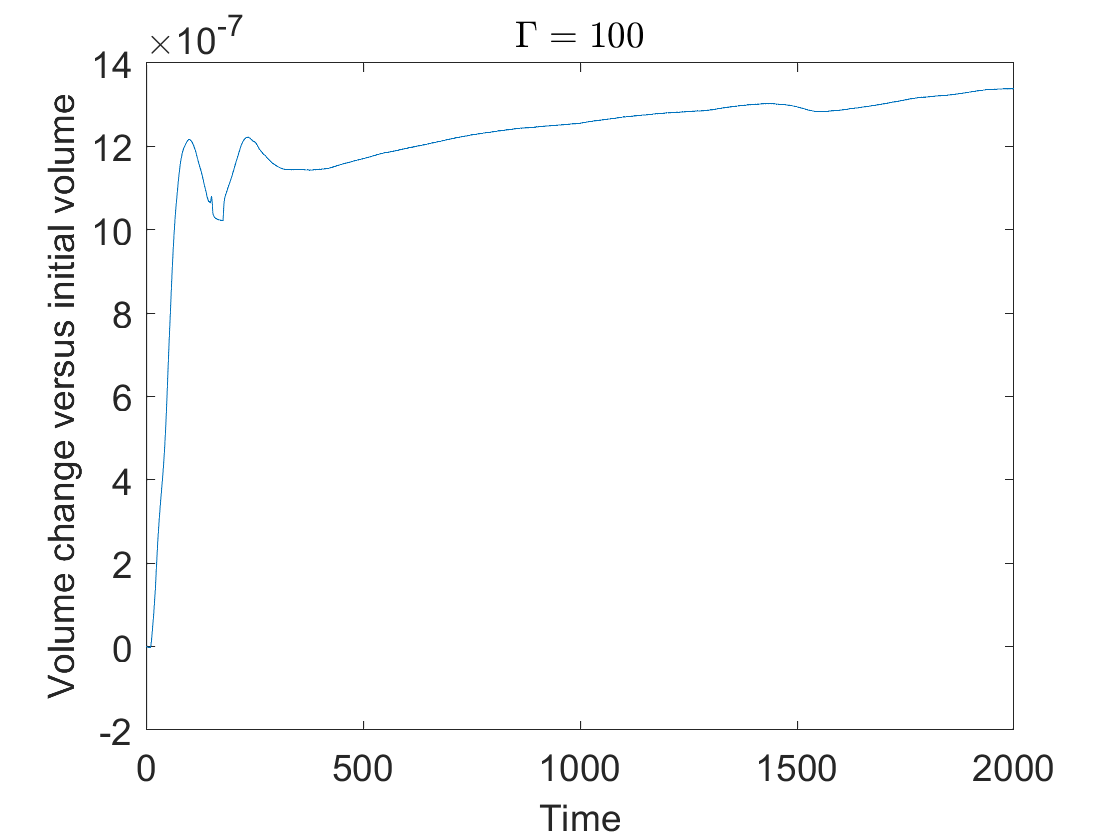}
		}
		\caption{Energy evolution and volume change versus initial volume, computed for droplet spreading on curved substrate. (a) Energy decay; (b) Volume change versus initial volume. From the left to the right are results for $\Gamma = 20$, $50$, $100$.}
		\label{Energy and Volume2}
	\end{figure}
	
	Figure \ref{Energy and Volume2} shows the energy decay and volume conservation simulated using the OPBDE Cahn-Hilliard model with $\Gamma = 20$, $50$, $100$. The energy decay becomes faster for larger $\Gamma$, and the relative volume change is within an acceptable range.
	
	\section{Concluding remarks}
	For the Cahn-Hilliard model in an arbitrary domain, an extended Cahn-Hilliard-type model in a larger, regular domain has been derived based on the OPBDE method. By the use of a source term, a modified conservation law is employed to incorporate the flux at the original boundary. In the framework of the variational Onsager principle, the free energy functional, the dissipation functional, and the rate of the free energy pumped into the system have been modified consistently from the original arbitrary domain to the extended regular domain. An asymptotic analysis has been carried out for the OPBDE Cahn-Hilliard model to show that the original model, including its boundary conditions, can be fully recovered. The consistency between the original and extended models is therefore established in the framework of the variational Onsager principle by which thermodynamic consistency is ensured. It is worth noting that, even without a prior knowledge on the specific form of the rate of free energy pumped into the system, the Onsager principle remains effective and instrumental in deriving the constitutive equation of the extended system, subject to a intrinsic property that can be validated later. The OPBDE Cahn-Hilliard model can be solved by a structure-preserving numerical scheme that upholds volume conservation and energy dissipation properties. The physical authenticity and numerical stability of the discretized system are therefore demonstrated. Numerical results from the extended OPBDE model have been compared with those from the original model and the DDM model with remarkable agreement, demonstrating the accuracy, effectiveness and robustness of the extended OPBDE model and its implementation in handling arbitrary domain geometries.
	
	This work has developed a general strategy for extending gradient flow models with physical boundary conditions from arbitrary domains to piecewise smooth, regular domains. In the framework of the OPBDE method, thermodynamic consistency is preserved, efficient structure-preserving algorithms can be developed, and physical authenticity is ensured. Our modeling and simulation results convincingly add the OPBDE method into the computational toolkit for computing initial-boundary value problems in complex geometries.

	\section*{Acknowledgement}
Zhen Zhang was partially supported by National Key R\&D Program of China
(2023YFA1011403), the NSFC grant (92470112), the Shenzhen Sci-Tech Inno-Commission Fund
(20231120102244002), and the Guangdong Provincial Key Laboratory of Computational Science
and Material Design (No. 2019B030301001).		
Tiezheng Qian was partially supported by the Hong Kong RGC General Research Fund (Grant No. 16306121) and the Key Project of the National Natural Science Foundation of China (Grant No. 12131010).

	\appendix
	\section*{Appendix I. Model derivation through the use of $\tilde{\phi}_t$ as flux variable }
	In this part we will show that the extended model can be derived as well if we use $\tilde{\phi}_t$ as the only flux variable for variation. For the extended model in $\Omega$, the energy dissipation rate of energy $\tilde F$ is given by
	\begin{equation}
		\tilde{F}_t = \int_{\Omega} \tilde{\mu}\tilde{\phi}_t d\bx+\int_{\partial \Omega} \bn \cdot \psi_{\varepsilon} K \nabla \tilde{\phi}\tilde{\phi}_t ds = \int_{\Omega} \frac{\delta \tilde{F}}{\delta P}P_t d\bx+\int_{\partial \Omega} \bn \cdot \psi_{\varepsilon} K \nabla \tilde{\phi}\tilde{\phi}_t ds.
	\end{equation}
	With the modified conservation law $P_t = -\nabla\cdot \bQ-|\nabla \psi_{\varepsilon}|h_3$, the dissipation functional $\Phi_{\tilde{F}}$ can be written as
	\begin{equation}
		\begin{split}
			\Phi_{\tilde{F}}
			&= \int_{\Omega} (\psi_{\varepsilon} \frac{\tilde{\bJ}^2}{2M}+|\nabla \psi_{\varepsilon}| \frac{\tilde{\phi}_t^2}{2\Gamma})d\bx \\
			&= \int_{\Omega} (\frac{\bQ^2}{2\psi_{\varepsilon}M}+|\nabla \psi_{\varepsilon}| \frac{P_t^2}{2\psi_{\varepsilon}^2\Gamma})d\bx \\
			&= \int_{\Omega} \left\{\frac{[(-\nabla\cdot)^{-1}(P_t + |\nabla \psi_{\varepsilon}|h_3)]^2}{2\psi_{\varepsilon}M}+|\nabla \psi_{\varepsilon}| \frac{\tilde{\phi}_t^2}{2\psi_{\varepsilon}^2\Gamma}\right\}d\bx \\
			&= \int_{\Omega} \half(P_t + |\nabla \psi_{\varepsilon}|h_3)\{-\nabla\cdot[\psi_{\varepsilon} M\nabla(\cdot)]\}^{-1}(P_t + |\nabla \psi_{\varepsilon}|h_3)+|\nabla \psi_{\varepsilon}| \frac{P_t^2}{2\psi_{\varepsilon}^2\Gamma})d\bx.
		\end{split}
	\end{equation}
		With the boundary condition $\bn \cdot K \nabla \tilde{\phi} = 0$, the Rayleighian is defined as
	\begin{equation}
		\begin{split}
			\tilde{R} &= \tilde{F}_t + \Phi_{\tilde{F}} - \tilde{W}_t \\
			&= \int_{\Omega}\frac{\delta \tilde{F}}{\delta P} P_t d\bx + \int_{\Omega} \half(P_t + |\nabla \psi_{\varepsilon}|h_3)\{-\nabla\cdot[\psi_{\varepsilon} M\nabla(\cdot)]\}^{-1}(P_t + |\nabla \psi_{\varepsilon}|h_3)+|\nabla \psi_{\varepsilon}| \frac{P_t^2}{2\psi_{\varepsilon}^2\Gamma})d\bx\\
			&- \int_{\Omega}(\tilde{\mu}\beta - \chi_{\varepsilon}|\nabla \psi_{\varepsilon}|h_3|\nabla \psi_{\varepsilon}|\Gamma^{-1}\beta) d\bx.
		\end{split}
	\end{equation}
	Applying the variational Onsager principle, we have
	\begin{equation}
		\begin{split}
			\frac{\delta \tilde{R}}{\delta P_t} = 0 
			&\implies \frac{\delta \tilde{F}}{\delta P} + [- \nabla\cdot[\psi_{\varepsilon} M\nabla(\cdot)]]^{-1}(P_t + |\nabla \psi_{\varepsilon}|h_3) + |\nabla \psi_{\varepsilon}|\Gamma^{-1}\frac{P_t}{\psi_{\varepsilon}^2\Gamma} = 0\\
			&\implies P_t = \nabla\cdot [\psi_{\varepsilon} M \nabla(\frac{\delta \tilde{F}}{\delta P} +|\nabla \psi_{\varepsilon}| \frac{P_t}{\psi_{\varepsilon}^2\Gamma})]- |\nabla\psi_{\varepsilon}|h_3,
		\end{split}
	\end{equation}
	which is identical to \eqref{Transport equation1}. Here we want to point out that the explicit form for $\tilde{W}_t$ used here is not necessary as
	\begin{equation}
		\frac{\delta \tilde{W}_t}{\delta P_t} = \chi_{\varepsilon} \frac{\delta \tilde{W}_t}{\delta \tilde{\phi}_t} = 0.
	\end{equation}

	\section*{Appendix II. The OPBDE Cahn-Hilliard model derived by using the Lagrange method}
	In this part we show a brief derivation of model \eqref{OPBDE CH} based on Lagrange method.
	Firstly, the rate of free energy change over time is given by
	\begin{equation}
		\tilde{F}_t = \int_{\Omega} \tilde{\mu}\tilde{\phi}_t d\bx+\int_{\partial \Omega} \bn \cdot \psi_{\varepsilon} K \nabla \tilde{\phi}\tilde{\phi}_t ds = \int_{\Omega} \frac{\delta \tilde{F}}{\delta P}P_t d\bx+\int_{\partial \Omega} \bn \cdot \psi_{\varepsilon} K \nabla \tilde{\phi}\tilde{\phi}_t ds.
	\end{equation}
	With the boundary condition $\bn \cdot K \nabla \tilde{\phi} = 0$ at $\partial \Omega$, we consider the variational Onsager principle with the Rayleighian $\tilde{R}$ defined by
	\begin{equation}
		\tilde{R} = \tilde{F}_t+\Phi_{\tilde{F}} - \tilde{W}_t - \lambda (P_t + \nabla\cdot \bQ + |\nabla \psi_{\varepsilon}|h_3).
	\end{equation}
	where $\lambda$ is the Lagrange multiplier to locally enforce the modified conservation law as a constraint. We use the same dissipation functional
	\begin{equation}
		\Phi_{\tilde{F}} = \int_{\Omega} (\psi_{\varepsilon} \frac{\tilde{\bJ}^2}{2M}+|\nabla \psi_{\varepsilon}| \frac{\tilde{\phi}_t^2}{2\Gamma})d\bx = \int_{\Omega} ( \frac{\bQ^2}{2\psi_{\varepsilon}M}+|\nabla \psi_{\varepsilon}| \frac{P_t^2}{2\psi_{\varepsilon}^2\Gamma})d\bx,
	\end{equation}
	then the model is given by
	\begin{align}
		\frac{\delta \tilde{R}}{\delta \bQ} = 0 &\implies \frac{\bQ}{\psi_{\varepsilon}M} + \nabla \lambda = 0, \label{eqn4}\\
		\frac{\delta \tilde{R}}{\delta P_t} = 0 &\implies \frac{\delta \tilde{F}}{\delta P} + |\nabla \psi_{\varepsilon}|\frac{P_t}{\psi_{\varepsilon}^2\Gamma} -\lambda = 0.\label{eqn5}
	\end{align}
	Similar to the previous derivations, here $\frac{\delta \tilde{W}_t}{\delta P_t} = 0$ and $\frac{\delta \tilde{W}_t}{\delta \bQ} = 0$ have been utilized.
	Combining \eqref{eqn4} and \eqref{eqn5} to eliminate $\lambda$, we have
	\begin{equation}
		\bQ = - \psi_{\varepsilon} M \nabla\lambda = -\psi_{\varepsilon}M\nabla(\frac{\delta \tilde{F}}{\delta P} +|\nabla \psi_{\varepsilon}| \frac{P_t}{\psi_{\varepsilon}^2\Gamma}),
	\end{equation}
	which is consistent with \eqref{Constitutive equation1}. Here $\lambda$ plays the role of $\chi_{\varepsilon}\mu_*$ introduced in \eqref{mu_star}.
	
	We omit the subsequent derivation because it is of no difference from the previous section with $\tilde{\bJ}\cdot \bn = 0$ at $\partial \Omega$. This derivation using the Lagrange multipliers' method is considerably concise.
	
	\bibliographystyle{plain}
	\bibliography{myref}
	
\end{document}